%% file: main.tex
\pdfoutput=1

\documentclass[graybox,envcountchap]{svmono}

\usepackage[framed,numbered,autolinebreaks,useliterate]{mcode}

\usepackage{amsmath,amsfonts,mathrsfs,amssymb}
\usepackage{bm}

\usepackage{amsthm}
\usepackage{siunitx}

\usepackage{algorithm,subfigure}
\usepackage{textcomp}

\usepackage{cite}
\usepackage{enumerate}
\usepackage{pdfpages}

 \def\p{\partial} 
\def\to{\rightarrow}

\def\bb{\begin{equation}} \def\ee{\end{equation}}
\def\beqn{\begin{eqnarray}}  \def\eqn{\end{eqnarray}}
\def\beqnx{\begin{eqnarray*}} \def\eqnx{\end{eqnarray*}}

\newcommand{\be}{\begin{eqnarray}}
\newcommand{\ben}{\begin{eqnarray*}}
\newcommand{\en}{\end{eqnarray}}
\newcommand{\enn}{\end{eqnarray*}}

\newcommand{\la}{\lambda}

\usepackage{color}

\usepackage{booktabs}

\newcommand{\norm}[1]{\left\lVert #1 \right\rVert}

\newtheorem{thm}{Theorem}[section]
\newtheorem{cor}{Corollary}[section]
\newtheorem{lem}{Lemma}[section]
\newtheorem{prop}{Proposition}[section]
\newtheorem{defn}{Definition}[section]
\newtheorem{assum}{Assumption}[section]
\theoremstyle{remark}
\newtheorem{rmk}{Remark}[section]
\newtheorem*{remark*}{Remark}
\numberwithin{equation}{section}

\usepackage{mathptmx}
\usepackage{helvet}
\usepackage{courier}
\usepackage{type1cm}

\usepackage[totoc,initsep=15pt,unbalanced=true]{idxlayout}

 \allowdisplaybreaks
 
\begin{document}

\author{Hongyu Liu, Catharine W.K. Lo, Shen Zhang}
\title{Inverse Problems for Mean Field Games}
\maketitle

\frontmatter

\include{preface}

\tableofcontents

\mainmatter


\chapter{Introduction}

The theory of Mean Field Games (MFGs) has its roots in the groundbreaking contributions of Peter Caines, Minyi Huang, and Roland Malham\'e \cite{huang2006large,huang2007invariance,huang2007large}, as well as the work of Pierre-Louis Lions and Jean Michel Lasry \cite{LasryLions1, LasryLions2,MFG-book}, who independently established the framework in 2006 to tackle the complexities of modeling and studying these systems. Over time, MFGs have garnered considerable interest and have evolved into a thriving area of study.

This theory pertains to analyzing the limit behavior of systems in which a very large number of homogeneous strategic players interact with each other in a specific symmetric manner under partial information. It provides a powerful tool to analyze complex systems involving a large number of rational decision-makers, such as crowds \cite{Achdou2021econs,MFGCrowd,MFGCrowd+Econs}, financial markets \cite{MFGCrowd+Econs,CarmonaDelarue2018_1,Lacker2019finance}, traffic flows \cite{MFGCar2,MFGAutoCar2}, or social networks \cite{MFGSocialNetwork}. 

At the heart of MFGs is the concept of considering the averaged-out (mean-field) behaviors of agents as the system approaches the limit of an infinitely large population. By doing so, MFGs provide a macroscopic perspective that captures the collective dynamics and strategic interactions of the agents, while abstracting away from individual idiosyncrasies. By focusing on averaged characteristics, this macroscopic prespective of MFGs allows for a simplified yet insightful analysis of complex systems, enabling the study of equilibrium states, the interplay of optimal strategies, and the emergence of global patterns, in a computationally tractable manner.

In a game, each individual player makes decisions based on their own optimization problem while taking into account the decisions made by other players. A key feature of MFGs is the presence of an adversarial regime, where the agents' decisions are influenced by the actions of others. In this regime, a Nash equilibrium exists, representing a state in which no agent can unilaterally improve their own outcome by deviating from their chosen strategy, and is unique within the so-called monotone regime. 
Nash equilibria of the game can be characterized in
terms of a system of differential equations.

Let $n\in\mathbb{N}$ and $\mathbb{R}^n$ be the Euclidean space.
In the continuous state space, the mean field
equilibrium can be characterized by the following MFG system

\begin{equation}\label{eq:MFG0}
	\begin{cases}
		-\partial_t u(x, t)-\Delta u(x, t)+H(x, \nabla_x u(x, t))=F(x, m(x, t)),\quad & (x, t)\in \mathbb{R}^n\times (0, T),\medskip \\
		\partial_tm(x, t)-\Delta m(x, t)-{\rm div} \left(m(x, t)\nabla_p H\big(x, \nabla_x u(x, t)\big) \right)=0,\quad & (x, t)\in \mathbb{R}^n\times (0, T),\medskip\\
		m(x, 0)=m_0(x), \quad u(x, T)=G(x,m_T),\quad & x\in\mathbb{R}^n, 
	\end{cases}
\end{equation}
In $\eqref{eq:MFG0}$, $\Delta$ and ${\rm div}$  are the Laplacian and divergent operators with respect to the $x$-variable, respectively. In general, $u$ is the value function of each player; $ m$ is
the population distribution.  Let $\mathcal{P}$ stands for the set of Borel probability measures on $\mathbb{R}^n$ ,
$F:\mathbb{R}^n\times\mathcal{P}\to\mathbb{R} $ is the running cost function which signifies the interaction
between the agents and the population; $m_0$ is the initial population distribution
and $G:\mathbb{R}^n\times\mathcal{P}\to\mathbb{R}$ signifies the terminal cost.

\section{The Expression of MFG System in Control Theory} 

We first briefly introduce the mathematical background (in control theory) of this MFG system. As we mentioned above, the MFG system can be interpreted as a Nash equilibrium for a system for non-atomic agents with a cost depending of the density of the other agents. More precisely,
at the initial time $t = 0$, let $m_0$ be the probability density of the distribution of agents.  The agents are assumed to be non-atomic (at least within their own populations, c.f. Section \ref{sec:multipop}) which means that they share a common belief on the future behavior of the density of agents $m(t)$. Let the Hamiltonian $H$ be convex with respect to the second variable. Then, starting from a position $x$ at time $t = 0$, each player solves a problem of the form
\begin{equation}\label{eq:mfgproblem}
\inf_{\alpha}\,\mathbb E\left\{\int_0^T L(X_s,\alpha_s)+F(X_s,m(s))ds+G(X_T,m_T)\right\},
\end{equation}
where $L$
is the Legendre transform of $H$ with respect to the last variable:
$$H(p, x) = \inf\, [H^*(x,a)+a\cdot p],\quad p\in\mathbb{R}^n.$$
and $X_s$ is the solution to the stochastic differential equation(SDE)
$$d X_s = \alpha_s ds + \sqrt{2}dB_s,\quad X_0 = x.$$
Here, $B_s$ is a standard $n$-dimensional Brownian motion (also viewed as individual/idiosyncratic noise of each player) and the infimum is taken over
controls $\alpha : [0, T] \to \mathbb{R}^n$ adapted to the filtration generated by $B_s$.

 As a optimization problem, we may introduce the value function
$u(x,t)$ for this problem:
\begin{equation}\label{eq:mfgproblem2}
	u(x,t):=\inf_{\alpha}\mathbb E\left\{\int_0^T L(X_s,\alpha_s)+F(X_s,m(s))ds+G(X_T,m_T)\right\},
\end{equation}
where

$$d X_s = \alpha_s ds + \sqrt{2}dB_s,\quad X_t = x.$$

Then if $m_t$
is known, then one can conclude (see, for instance, \cite{Linear_in_Tn}) that $u$ is a classical solution to the Hamilton-Jacobi
equation 
\begin{equation}\label{HJ}
	\begin{cases}
		-\partial_t u(x, t)-\Delta u(x, t)+H(x, \nabla_x u(x, t))=F(x, m(x, t))\\
		 u(x, T)=G(x,m_T).
	\end{cases}
\end{equation}

 Moreover, the optimal
strategy  of each agent is given by
$$  \alpha^*(t,x):=-\nabla H_p(x,\nabla_x u(x,t)).$$
Hence, the best strategy for each individual agent at position $x$ at time $t$, is to play $\alpha^*(t,x)$. Therefore the actual density $m(x,t)$ of agents is described by the  
 the Fokker-Planck equation
\begin{equation}\label{FP}
	\begin{cases}
	\partial_tm(x, t)-\Delta m(x, t)-{\rm div} \left(m(x, t)\nabla_p H\big(x, \nabla_x u(x, t)\big) \right)=0,\\
	m(x, 0)=m_0(x).
	\end{cases}
\end{equation}
 We say that the pair $(u,m)$ is a Nash
equilibrium of the game if the pair $(u,m)$ satisfies the MFG system $\eqref{eq:MFG0}$. 
\section{The Inverse Problems}

The well-posedness of the MFG system \eqref{eq:MFG0} has already been studied in many different settings. Under different assumptions, one can consider $F$ and $G$  depending on $m$ non-locally or locally. The well-posedness of the MFG system \eqref{eq:MFG0} is known in Cardaliaguet-Porretta \cite{CarPor}, Meszaros-Mou \cite{MM} in the case of nonlocal data $F$ and $G$. In the case that $F,G$ are locally dependent on the measure variable $m$, we refer to  Ambrose\cite{Amb:18}, Cardaliaguet-Porretta \cite{CarPor}, Porretta \cite{Por}.

In this book, we are mainly concerned with the inverse problem of determining unknown functions in the MFG system, including the running cost $F$ or/and the terminal cost $G$ or/and the Hamiltonian $H$ by the knowledge of boundary data. To that end, we introduce the data set in the following forms: 
Let $\Omega$ be a bounded domain in $\mathbb{R}^n$ with smooth boundary $\Sigma:=\p\Omega.$ Let $Q:=\Omega\times(0,T).$ Consider the measurement 
\begin{equation}\label{eq:M}
	\mathcal{M}_{F,G,H}:=(u(x,t),m(x,t))\big|_{\p Q},
\end{equation}
where $(u,m)$ are solutions of system $\eqref{eq:MFG0}$ with functions (probably unknown) $F,G,H$. The inverse problem mentioned above can be formulated as the following measurement map:
\begin{equation}\label{eq:ip1}
	\mathcal{M}_{F,G,H}\longrightarrow F\ \mbox{and/or}\ G\, \mbox{and/or}\ H. 
\end{equation}
Notice that we shall not use all information of $\mathcal{M}_{F,G,H}$. In different setups, we may only need partial information to obtain our results.

In the mean field game theory, the running cost $F$ and the terminal cost $G$ are critical for the agents to decide the strategies. However, in practice they are often partially known or totally unknown for the agents, while the boundary data can be measured. This is a major motivation for us to study the inverse problem \eqref{eq:ip1}. We believe our study could have many applications in the areas mentioned above.

\section{Technical Challenges}

This book is based on a series of work on the study of inverse problems of the MFG system \cite{LiuMouZhang2022InversePbMeanFieldGames,LiuZhang2022-InversePbMFG,LiuZhangMFG3,LiuZhangMFG4,LiuLoZhang2024decodingMFG,LiuLo2025decodingMFG,ding2023determining,liulo2024determiningstatespaceanomalies,ding2024determininginternaltopologicalstructures,klibanov2023holder,klibanov2023mean2,liu2023stability,imanuvilov2023lipschitz1,imanuvilov2023unique}. As previously mentioned, the forward problem of MFG systems has been thoroughly examined in recent literature. Conversely, the inverse problems for MFGs is a new topic that is at the frontier of research and highly interesting. Hence, we aim to collect existing results, analyze their implications, and explore methodologies that have been developed to address fundamental questions in this area. Through this book, we hope to provide insights into the structure of MFGs and open new directions for both theoretical advancements and practical applications.

Comparing to the traditional inverse problem, MFG inverse problems possess unique structures, resulting in several notable technical features that make the study of these problems incredibly fascinating and challenging. First, the MFG systems are coupled nonlinear parabolic systems, with one equation progressing forward in time and another backward in time. This coupled nonlinearity is tackled using a powerful strategy based on linearisation, which we will detail in Section \ref{sec:linearize}. Second, there is a probability density constraint
\begin{equation}
    \begin{aligned}
    &m(x,t)\geq 0, \quad\forall x\in\mathbb{R}^n ,\, t\in(0,T)\\
    &\int m(x,t) \,dx=0,\quad \forall t\in(x,T)
    \end{aligned}
\end{equation}
which poses significant difficulties for constructing proper probing modes for the inverse problems. This constraint forces us to use less input to determine the unknowns, and requires further modification of the technique of the successive linearisation, as we will see in the later chapters. 

In fact, there have been several works in other coupled or non-coupled PDE systems where the probability density constraint naturally arises, such as \cite{liu2023determining,li2023inverse,LiLoCAC2024,li2024inverse}. We think that the mathematical approaches and methods presented in this book provide fresh insights into inverse boundary problems in these fresh and fascinating settings, with the potential to provide highly significant theoretical and practical outcomes. 
    
In the following chapters, we investigate the MFG system defined in a variety of settings. Specifically, we focus on reconstructing cost functions, interaction terms and dynamics from observed equilibria, developing theoretical frameworks to solve inverse problems in MFGs. We thoroughly examine the essential requirements for the recovery of different sets of unknowns.

This book is organized as follows. In Chapter \ref{prelim}, we present some classical results of the MFG system, as well as some basic tools that we need for our inverse problems study. In Chapters \ref{InverseCoefBdry}--\ref{chap:Cauchy}, we obtain the unique recovery of the unknown coefficients using certain boundary data in different setups. Chapter \ref{chap:internaltop} studies situations where there are anomalies present in the MFG state space. Finally, in Chapter \ref{chap:Carleman}, we derive stability results using single measurement.

\chapter{Preliminaries}\label{prelim}
\section{The Review of Classical Results for the MFG system}\label{reviw}
For our inverse problems, we require previous knowledge or assumptions about the unknown functions. To demonstrate the necessity or suitability of these requirements, we first review a few significant classical results for the MFG system.. We may give a brief proof here and we refer to \cite{note_MFG} for details. For simplicity, we first work with the functions which are periodic in space, which means we consider the following system in this chapter:
\begin{equation}\label{eq:mfg-periodic}
	\begin{cases}
		-\partial_t u(x,t) -\Delta u(t,x)+ H\big(x,\nabla u(x,t)\big)-F(x,m(x,t))=0,&  {\rm{in}}\ \mathbb T^n\times (0,T),\\
		\partial_tm(x,t)-\Delta m(x,t)-{\rm div} \big(m(x,t) \nabla_pH(x, \nabla u(x,t)\big)=0, & {\rm{in}}\ \mathbb T^n\times(0,T),\\
		u(x,T)=G(x,m(x,T)),\ m(x,0)=m_0(x), & {\rm{in}}\ \mathbb T^n,
	\end{cases}
\end{equation}
where $\Omega=\mathbb{T}^n:=\mathbb{R}^n/\mathbb{Z}^n$ in this chapter and all functions in system $\eqref{eq:mfg-periodic}$ are assumed to be periodic in space. We shall point out that the 
 target of this periodic assumption is that we want to simplify the proofs and avoids the technical discussion on the boundary conditions. We can get same results in a general domain with suitable boundary conditions and regularity assumptions. We will have more discussions on this when we work in general domains.
\subsection{Notations and Assumptions for Well-posedness}
Let  $\mathcal{P}(\Omega)$ denotes the set of probability measures on $\Omega=\mathbb{T}^n$ and let $U$ be a real function defined on $\mathcal{P}(\Omega)$.
we  define the Wasserstein distance between $m_1$ and $m_2$ in $\mathcal{P}(\Omega)$ as following:
\begin{defn}\label{W_distance}
	Let $m_1,m_2$ be two Borel probability measures on $\Omega$. Define
	\begin{equation}
		d_1(m_1,m_2):=\sup_{Lip(\psi)\leq 1}\int_{\Omega}\psi(x)d(m_1-m_2)(x),
	\end{equation}
	where $Lip(\psi)$ denotes the Lipschitz constant for a Lipschitz function, i.e., 
	\begin{equation}\label{eq:Lip1}
		Lip(\psi)=\sup_{x, y\in\Omega, x\neq y}\frac{|\psi(x)-\psi(y)|}{|x-y|}. 
	\end{equation}
\end{defn}
Next we introduce the regularity assumptions. For $k\in\mathbb{N}$ and $0<\alpha<1$, the H\"older space $C^{k+\alpha}(\Omega)$ is defined as the subspace of $C^{k}(\Omega)$ such that $\phi\in C^{k+\alpha}(\Omega)$ if and only if $D^l\phi$ exist and are H\"older continuous with exponent $\alpha$ for all $l\in \mathbb{N}^n$ with $|l|\leq k$. The norm is defined as
\begin{equation}
	\|\phi\|_{C^{k+\alpha}(\Omega) }:=\sum_{|l|\leq k}\|D^l\phi\|_{\infty}+\sum_{|l|=k}\sup_{x\neq y}\frac{|D^l\phi(x)-D^l\phi(y)|}{|x-y|^{\alpha}}.
\end{equation}
\begin{assum}\label{hypo1}
	(i) The functions $F(x, m)$ and $G(x, m)$ are Lipschitz continuous in $\mathbb{T}^n\times\mathcal{P}(\mathbb{T}^n)$.\\
	(ii) $F(\cdot, m)$ and $G(\cdot, m)$ are bounded in $C^{1+\alpha}(\mathbb{T}^n)$ and $C^{2+\alpha}(\mathbb{T}^n)$ (for some $\alpha\in (0,1)$) uniformly with respect to $m\in\mathcal{P}(\mathbb{T}^n)$.\\
	(iii) The Hamiltonian $H :\mathbb{T}^n\times\mathbb{R}\to\mathbb{R}$ is locally Lipschitz continuous, $\nabla_p H$ exists and is
	continuous on $\mathbb{T}^n\times\mathbb{R}\to\mathbb{R}$
	, and $H$ satisfies the growth condition
	\begin{equation}
		\left<\nabla_p H(x,p),p \right>\geq C(1+|p|^2)
	\end{equation}
	for some constant $C>0$.\\
	(iv) The initial distribution is absolutely continuous with respect to the Lebesgue measure, and has a $C^{2+\alpha}$
	continuous density $m_0$.
\end{assum}
\begin{rmk}
	\begin{itemize}
		\item We may construct $F,G$ that satisfy these conditions in the following ways. For the non-local dependent case, consider $F(x,m)=\int_{\mathbb{T}^n}f(x)\, dm$ and $f$ is Lipschitz and smooth. For the local-dependent case, let $m$ be the density of the distribution and $F(x,m)=g(x)m$, $g$ is smooth.
		\item The most important example of the Hamiltonian is just $H(x,p)=\frac{1}{2}p^2.$
		\item 	In  Definition \ref{W_distance}, $m$ ( $m_1$ ) is viewed as a distribution. However, we also use $m$ to  denote the density of  a distribution  in some cases if no confusion as in part one.
	\end{itemize}  
\end{rmk}
Then we introduce the monotonicity assumptions, which is necessary for the uniqueness and stability of the solution.
\begin{assum}\label{hypo-monoton}
	For any $m_1,m_2\in\mathcal{P}(\Omega)$, we have 
	\begin{equation}\label{mono1}
		\int_{\Omega} \big( F(x,m_1)-F(x,m_2)\big)d(m_1-m_2)(x)\geq 0,
	\end{equation}
	and 
	\begin{equation}
		\int_{\Omega} \big( G(x,m_1)-G(x,m_2)\big)d(m_1-m_2)(x)\geq 0,
	\end{equation}
	for all $m_1,m_2\in \mathcal{P}(\Omega).$
\end{assum}
Physically, the monotonicity condition can be interpreted as the players favoring more scattered configurations in comparison to congested areas.\cite{Linear_in_Tn}

\subsection{Existence and Uniqueness of Classical Solution}
Based on the assumptions above, we give a rough proof of existence and uniqueness of classical solution to MFG system.
\begin{thm}
	Under assumptions $\ref{hypo1}$, the system $\eqref{eq:mfg-periodic}$ admits a classical solution $(u,m)\in C^{2+\alpha,1+\frac{\alpha}{2}}$.
\end{thm}
\begin{proof}
	For a fixed large constant $C>0.$ Let 
	\begin{equation}
		\mathcal{J}:=\{\rho\in C^0([0,T],\mathcal{P}(\mathbb{T}^n)):\sup_{s\neq t}\frac{d_1(\rho(s)-\rho(t))}{s-t}\leq C \}.
	\end{equation}
	For any $\rho\in \mathcal{J}$, let $u$ be the unique solution to the following system:
	\begin{equation}
		\left\{
		\begin{array}{ll}
			-\partial_t u(x,t) -\Delta u(t,x)+ H\big(x,\nabla u(x,t)\big)-F(x,\rho(x,t))=0,&  {\rm{in}}\ \mathbb T^n\times (0,T),\medskip\\
			u(x,T)=G(x,\rho(x,T)), & {\rm{in}}\ \mathbb T^n.
		\end{array}
		\right.
	\end{equation}
	Then we define $m=\psi (\rho)$ be the solution of the Fokker-Planck
	equation
	\begin{equation}
		\left\{
		\begin{array}{ll}
			\partial_tm(x,t)-\Delta m(x,t)-{\rm div} \big(m(x,t) \nabla_pH(x, \nabla u(x,t)\big)=0, & {\rm{in}}\ \mathbb T^n\times(0,T),\medskip\\
			\ m(x,0)=m_0(x), & {\rm{in}}\ \mathbb T^n.
		\end{array}
		\right.
	\end{equation}
	Then the regularity assumptions $\ref{hypo1}$ implies that $\psi$ is a continuous map form $ \mathcal{J}$ to itself. Then the Schauder fixed point theorem implies that $\psi$ get a fixed point, which is a solution of $\eqref{eq:mfg-periodic}$.
\end{proof}
\begin{rmk}
	This theorem shows that as long as $m_0$ is a density of a distribution function, then the solution $m(x,t)$ must be positive for all $(x,t)\in\mathbb{T}^n\times(0,T)$ and
	$$\int_{\mathbb{T}^n}  m(x,t)\,dx=1,\, \text{ for all } t\in (0,T)$$
\end{rmk}
\begin{thm}
	Assume that $H(x,p)$ is uniformly convex with respect to $p$.
	Then if assumption $\ref{hypo-monoton}$ holds,  the solution of the system $\eqref{eq:mfg-periodic}$ is unique.
\end{thm}
\subsection{The Ergodic MFG System}
One may be interested in the large time average of the MFG system  as the horizon $T$
tends to infinity. It turns out that the limit system takes the following form:
\begin{equation}\label{static-mfg}
	\begin{cases}
		\lambda -\Delta u+ H\big(x,\nabla u\big)-F(x,m)=0,&  {\rm{in}}\ \mathbb T^n\times (0,T),\medskip\\
		-\Delta m(x,t)-{\rm div} \big(m \nabla_pH(x, \nabla u\big)=0, & {\rm{in}}\ \mathbb T^n\times(0,T).\medskip\\
	\end{cases}
\end{equation}
We have the following results\cite{note_MFG}.
\begin{thm}\label{m=e^-u}
	For the classical case $H(x,p)=\frac{1}{2}p^2$, the system
	\begin{equation}\label{static-mfg'}
		\begin{cases}
			\lambda -\Delta u+ \frac{1}{2}|\nabla u|^2-F(x,m)=0,&  {\rm{in}}\ \mathbb T^n,\medskip\\
			-\Delta m(x,t)-{\rm div} \big(m \nabla u\big)=0, & {\rm{in}}\ \mathbb T^n.\medskip\\
			\int_{\mathbb{T}^n} m\, dx=1.
		\end{cases}
	\end{equation}
	get a solution $(u,m)$ which satisfies
	\begin{equation}
		m=\frac{e^{-u}}{\int_{\mathbb{T}^n} e^{-u}\, dx}.
	\end{equation}
	If we also assume $F$ is increasing with respect to the second variable and $\int_{\mathbb{T}^n} u\, dx=0$, then the solution $(\lambda,u,m)$ is unique.
\end{thm}
\begin{proof}
	It can be proved by a similar fixed point argument.
\end{proof}

\section{Basic Ideas and Results for Inverse Problems}
\subsection{A preliminary result}
Before we begin, we first give a preliminary result for the uniqueness of $F$, in the most general case.

Consider the measurement map $\mathcal{M}=\mathcal{M}_F$ given by \begin{equation}\label{eq:measure1gen}\mathcal{M}_F=\left(\left.\left(u(x,t),m(x,t)\right)\right|_{t=\{0,T\},x\in\Omega},\left.\left(u,m,\partial_\nu u, \partial_\nu (\sigma m),\nabla_p \mathcal{H}\right)\right|_{t\in(0,T),x\in\partial\Omega}\right)\to F. \end{equation}
Then the following result holds:

\begin{thm}
    Suppose there exists a (weak) solution $u,m$ to the MFG system 
    \begin{equation}\label{eq:MFG1}
    \begin{cases}
        -\partial_t u(x,t) -\sigma\Delta u(x,t) + \mathcal{H}(x,t,\nabla u,m) = F(x,t,u,m) &\quad \text{in }Q:=\overline{\Omega}\times[0,T],\\
        \partial_t m(x,t) -\Delta(\sigma m(x,t)) - \nabla\cdot(m\nabla_p \mathcal{H}(x,t,\nabla u,m)) = 0  &\quad \text{in }Q.
    \end{cases}
\end{equation} 
Let $\mathcal{M}$ be the associated measurement map \eqref{eq:measure1gen}.
If $\mathcal{M}$ is known, then, the integral
\[\int_Q F(x,t,u,m)m\,dx\,dt\] is known for any solution $m$ solving \eqref{eq:MFG1}.

In the case that $m=m_0=c\in \mathbb{R}^+_0$, the integral \[c\int_Q F(x,t,u,c)\,dx\,dt\] is known.
\end{thm}

\begin{proof}

Multiplying the first equation in \eqref{eq:MFG1} of $u$ by the solution $m$ of the second equation in \eqref{eq:MFG1}, and integrating over $Q$, we have
\begin{align*}
    &\int_Q u\partial_t m -\int_\Omega u(T)m(T) + \int_\Omega u(0)m(0) \\
    &-\int_Q u\Delta(\sigma m) - \int_\Gamma m\sigma\partial_\nu u + \int_\Gamma u \partial_\nu (\sigma m)\\
    &-\int_Q u\nabla\cdot(m\nabla_p \mathcal{H}(x,t,\nabla u,m)) + \int_\Gamma m\nabla_p \mathcal{H}(x,t,\nabla u,m)\partial_\nu u = \int_Q F(x,t,u,m) m.
\end{align*}
But $m$ satisfies its own equation, so we obtain that
\begin{multline*}
     -\int_\Omega u(T)m(T) + \int_\Omega u(0)m(0) - \int_\Gamma m\sigma\partial_\nu u + \int_\Gamma u \partial_\nu (\sigma m) + \int_\Gamma m\nabla_p \mathcal{H}(x,t,\nabla u,m)\partial_\nu u \\= \int_Q F(x,t,u,m) m.
\end{multline*}

Observe that the terms on the left hand side are all known when $\mathcal{M}$ is measured. Correspondingly, the right hand side is also known.

The second result follows simply by considering the measurement map $\mathcal{M}$ for the constant solution $m=m_0=c$.

\end{proof}

As a corollary, we have the following result:
\begin{corollary}\label{Cor:GeneralCase}
    Suppose that $F(x,t,u,m)=\alpha m^k$ is a power of $m$ for a given $k$. Then for the measurement map $\mathcal{M}^1$ associated to $m=m_0=c\in\mathbb{R}_0^+$, $F$ can be uniquely determined.
\end{corollary}

\begin{proof}
    By the previous theorem, 
    \[\int_Q F(x,t,u,m)m\,dx\,dt=\int_Q \alpha [m(x,t)]^k\,dx\,dt=\int_Q \alpha c^k\,dx\,dt=c^k\alpha T |\Omega|\] is known. Therefore, $\alpha$ is uniquely determined, and thus $F$ is uniquely determined.
\end{proof}

\begin{remark}
    Observe that in this case, we do not impose any boundary or initial/terminal conditions on the system \eqref{eq:MFG1}, but instead measure these data.
\end{remark}

Corollary \ref{Cor:GeneralCase} strongly suggests that we will be able to obtain uniqueness results on $F$ which are more direct, if $F$ is in polynomial form and depends only on $m$. Consequently, in the remainder of this subsection, we will consider $F$ and $G$ to be analytic, and generalize Corollary \ref{Cor:GeneralCase} to obtain their uniqueness results.

\subsection{Non-uniqueness and Discussion on the Zero Admissibility Conditions}
In this section, we show that the zero admissibility conditions, namely $F(x,t, 0)=0$ and $G(x,0)=0$ are unobjectionably necessary if one intends to uniquely recover $F$ or $G$ by knowledge of the measurement operator 
\[\mathcal{M}_{F, G}:m_0\mapsto u(x,t)|_{t=0}.\] For simplicity, we only consider the case that the space dimension $n=1$ without the periodic boundary conditions. That is, we consider the following MFG system:
\begin{equation}\label{dim1}
	\begin{cases}
		-\p_tu_j(x,t)-\p_{xx} u_j(x,t)+\frac 1 2 {|\p_x u_j(x)|^2}= F_j(x,t,u_j(x,t)),& \text{ in } \mathbb{R}\times (0,T),\medskip\\
		\p_t u_j(x,t)-\p_{xx} u_j(x,t)-\p_x(u_j(x,t)\p_x u_j(x,t))=0,&\text{ in } \mathbb{R}\times(0,T),\medskip\\
		u_j(x,T)= G_j(x,u_j(x,T)), & \text{ in } \mathbb{R},\medskip\\
		u_j(x,0)=m_0(x), & \text{ in } \mathbb{R}.\\
	\end{cases}  		
\end{equation}
Furthermore, we assume $T$ is small enough such that the solution of the MFG system \eqref{dim1} is unique \cite{Amb:18,LasryLions1}. In what follows, we construct examples to show that if the zero admissibility conditions are violated then the corresponding inverse problems do not have uniqueness. 

\begin{prop}
	Consider the system $\eqref{dim1}$. There exist $F_1=F_2\in C^{\infty}(\mathbb{R}\times\mathbb{R}\times\mathbb{R})$ and $G_1\neq G_2\in C^{\infty}(\mathbb{R}\times\mathbb{R})$ (but we do not have $G_1(x,0)=G_2(x,0)=0$) such that  the corresponding two systems admit the same measurement map, i.e. $\mathcal{M}_{G_1}=\mathcal{M}_{G_2}$.
\end{prop}
\begin{proof}
	Set 
	\[
	F_1=F_2=-\sin(x)+\frac{1}{4}(e^t-1)^2\cos^2(x),
	\] 
	and
	\[
	G_1=(e^T-1)\sin(x),\quad G_2=(1-e^T)\sin(x).
	\]
	It can be directly verified that 
	\[
	u_1(x,t)=(e^t-1)\sin(x)\quad\mbox{and}\quad u_2(x,t)=(1-e^t)\sin(x),
	\]  
	satisfy the corresponding  system. In this case, we have $\mathcal{M}_{G_1}(m_0)=\mathcal{M}_{G_2}(m_0)=0$ for any admissible $m_0$. 
\end{proof}

\begin{prop}\label{prop2}
	Consider the system $\eqref{dim1}$. There exist $G_1=G_2\in C^{\infty}(\mathbb{R}\times\mathbb{R})$ and $F_1\neq F_2\in C^{\infty}(\mathbb{R}\times\mathbb{R}\times\mathbb{R})$ (but we do not have $F_j(x,t,0)=0$, $j=1,2$) such that  the corresponding two systems admit the same measurement map, i.e. $\mathcal{M}_{F_1,G_1}=\mathcal{M}_{F_2,G_2}$.
\end{prop}
\begin{proof}
	Set
	\[
	F_1=-x(2t-T)+\frac{t^2(t-T)^2}{2},\quad F_2=-2x(2t-T)+2t^2(t-T)^2,
	\]
	and
	\[
	G_1=G_2=0. 
	\]
	Here, it is noted that $F_1$ and $F_2$ are independent of $u$. In such a case, it is straightforward to verify that $u_j(x,t)=jxt(t-T)$ is the solution of  the corresponding  system \eqref{dim1}. Clearly, one has $\mathcal{M}_{F_1}(m_0)=\mathcal{M}_{F_2}(m_0)=0$ for any admissible $m_0$. 
\end{proof}

Moreover, we can find $F_1, F_2\in C^{\infty}(\mathbb{R}\times\mathbb{R})$ which are independent of $t$ such that Proposition \ref{prop2} holds.

\begin{proof}

	Define
	$$Lu_j:=-\p_tu(x,t)-\p_{xx} u(x,t)+\frac{|\p_x u|^2}{2}.$$ 
	It is sufficient for us to show that there exist $u_1(x,t),u_2(x,t)$ such that
	\begin{enumerate}		
		\item[(1)] $L u_1\neq L u_2$ and $\p_t (Lu_j)=0 $ for $ j=1,2$;
		
		\item[(2)] $u_1(x,0)=u_2(x,0)$ and $u_1(x,T)=u_2(x,T)$.
	\end{enumerate}
	In fact, if this is true, we can set $F_j= Lu_j$ and $G(x)=u_1(x,T)$. Then one has $G_1=G_2.$
	
	Without loss of generality, we assume $T=1.$	
	Let $p(t)$ be a non-zero solution of the following ordinary differential equation (ODE):
	\begin{equation*}
		(	\ln( p'(t)))'=\frac{\sqrt{1+4t}}{2},
	\end{equation*}
	and $q(t)$ be a solution of the ODE:
	\begin{equation*}
		\begin{cases}
			&2q'(t)+\sqrt{1+4t}\, q''(t)=p(t)p'(t)\sqrt{1+4t},\medskip\\
			&q(0)=0.
		\end{cases}
	\end{equation*}
	With $p(t)$ and $q(t)$ given above, we can set 
	\[
	u_1(x,t)=p(t(t-1))x+q(t(t-1))\quad\mbox{and}\quad u_2(x,t)=q(t(t-1))x+2q(t(t-1)).
	\]
	It can be directly verified that $u_1$ and $u_2$ fulfil the requirements (1) and (2) stated above. 
\end{proof}

\section{An Important Tool to Treat Nonlinearity - High-order Linearization}\label{sec:linearize}

In the study of inverse problems for mean-field game (MFG) systems, nonlinearity presents a significant challenge that must be addressed to extract meaningful information. A well-established technique to handle such difficulties is linearization, which allows for the reduction of a nonlinear inverse problem to its linear counterpart.
\subsection{The Linearization System in Probability Space}

Since we need to study the linearized system  for our inverse problem study, we define the variation of a function defined on $\mathcal{P}(\Omega) $. Recall that  $\mathcal{P}(\Omega)$ denotes the set of probability measures on $\Omega$ and let $U$ be a real function defined on $\mathcal{P}(\Omega) $.

\begin{defn}\label{def_der_1}
		Let $U :\mathcal{P}(\Omega)\to\mathbb{R}$. We say that $U$ is of class $C^1$ if there exists a continuous map $K:  \mathcal{P}(\Omega)\times \Omega\to\mathbb{R}$ such that, for all $m_1,m_2\in\mathcal{P}(\Omega) $ we have
		\begin{equation}\label{derivation}
			\lim\limits_{s\to 0^+}\frac{U\left(m_1+s(m_2-m_1)\right)-U(m_1)}{s}=\int_{\Omega} K(m_1,x)d(m_2-m_1)(x).
		\end{equation}
	\end{defn}
	Note that the definition of $K$ is up to additive constants. We define  the derivative
	$\dfrac{\delta U}{\delta m}$ (called the ``flat derivative") as the unique map $K$ satisfying $\eqref{derivation}$ and 
	\begin{equation}
		\int_{\Omega} K(m,x) dm(x)=0.
	\end{equation}
	Similarly, we can define higher order derivatives of $U$, and we refer to \cite{linear_in_Om} for more related discussion.

Finally, the relationship between the MFG system and its linearized system can be stated as follows:
 Let $(u_1, m_1)$ and $(u_2, m_2)$ be two solutions of the mean field games system, associated with the starting initial conditions $m_{0}^1$
 and $m_{0}^2$. Let $(s,\rho)$ be the solution of the linearized system  related to $(u_2, m_2)$, with initial condition $m_{0}^1-m_{0}^2$. Then we have 
the norms of $u_1-u_2-s$ and $m_1-m_2-\rho$ in suitable function spaces are bounded by 
$Cd_1(m^1_{0},m_0^2)$, where $d_1$ is the Wasserstein distance between $m_1$ and $m_2$ as defined in Definition \ref{W_distance}. For details of this theorem we also refer to \cite{linear_in_Om}. However, since we work in H\"{o}lder spaces in this paper and assume the measure $m$  always belongs to $C^{2+\alpha}(\Omega)$ at least, this result is replaced by local well-posedness results, which we will discuss in the next chapter.

Next, we shall discuss the  linearized systems in our setting and  their relationship in next subsection.

\subsection{Higher-order Linearization}\label{HLM}
We introduce the basic setting of this higher order
linearization method.
Let us consider the following system as an example:
\begin{equation}\label{main_linear example}
 \begin{cases}
        -\partial_t u(x,t) -\sigma(x,t)\Delta u(x,t) + \frac{1}{2}\kappa(x,t)|\nabla u(x,t)|^2 = F(x,t,m) &\quad \text{in }Q,\\
        \partial_t m(x,t) -\Delta(\sigma(x,t) m(x,t)) - \nabla\cdot(\kappa(x,t)m(x,t)\nabla u(x,t)) = 0  &\quad \text{in }Q,\\
        u(x,T)=G(x,m(x,T)),\quad m(x,0)=f(x) &\quad \text{in }\Omega,
    \end{cases}
\end{equation}
for bounded positive real functions $\sigma(x,t),\kappa(x,t)>0$, $\sigma(x)\in C^{2,2}(\bar{Q})$, $\kappa(x,t)\in C^{1,0}(\bar{Q})$, and
\[F(x,t,z)=\sum_{k=1}^{\infty} F^{(k)}(x,t)\frac{(z-m_0)^k}{k!}, \quad G(x,z)=\sum_{k=1}^{\infty} G^{(k)}(x)\frac{(z-m_0)^k}{k!}, \] 
where $ F^{(k)}(x,t)=\frac{\p^k F}{\p z^k}(x,t,m_0)\in C^{2+\alpha,1+\frac{\alpha}{2}}(Q)$ and $G^{(k)}(x)=\frac{\p^kG}{\p z^k}(x,m_0)\in C^{2+\alpha}(\mathbb{R}^n)$, with fixed solution $(u_0,m_0)$.

Let
\begin{equation}
\begin{aligned}
      &u(x,t)|_{\Gamma}=u_0|_{\Gamma}+\sum_{l=1}^{N}\varepsilon_lg_l|_{\Gamma}\\
      &m(x,t)|_{\Gamma}=m_0|_{\Gamma}+\sum_{l=1}^{N}\varepsilon_lh_l|_{\Gamma}
\end{aligned}
\end{equation}
where $g_l,h_l\in C^{2+\alpha,1+\frac{\alpha}{2}}(\mathbb{R}^n\times\mathbb{R})$  and $\varepsilon = (\varepsilon_1,\varepsilon_2, \dots,\varepsilon_N) \in \mathbb{R}^N$ with $|\varepsilon| = \sum_{l=1}^N |\varepsilon_l|$ small enough. Then, $(u(x,t;0), m(x,t;0))=(u_0,m_0)$ is the solution when $\varepsilon = 0$.

Let $S$ be the solution operator of the MFG system. Then there exists a bounded linear operator $A$ from $\mathcal{H}:=B_{\delta}( C^{2+\alpha,1+\frac{\alpha}{2}}(\Gamma))\times D_{\delta}( C^{2+\alpha,1+\frac{\alpha}{2}}(\Gamma))$ to $[C^{2+\alpha,1+\frac{\alpha}{2}}(Q)]^2$ such that
\begin{equation}
	\lim\limits_{\|(g,h)\|_{\mathcal{H}}\to0}\frac{\|S(g,h)-S(u_0,m_0)- A(g,h)\|_{[C^{2+\alpha,1+\frac{\alpha}{2}}(Q)]^2}}{\|(g,h)\|_{\mathcal{H}}}=0,
\end{equation} 
where $\|(g,h)\|_{\mathcal{H}}:=\|g\|_{ C^{2+\alpha,1+\frac{\alpha}{2}}(\Gamma)      }+\|h\|_{C^{2+\alpha,1+\frac{\alpha}{2}}(\Gamma)}$.
Now we consider $\varepsilon_l=0$ for $l=2,\dots,N$ and fix $f_1$.
Notice that if $F\in\mathcal{A}$ , then $F$ depends on the distribution $m$ locally. We have that 
	$$
	\dfrac{\delta F}{ \delta m}(x,m_0)(\rho(x,t)):=\left<\dfrac{\delta F}{ \delta m}(x,m_0,\cdot),\rho(x,t)\right>_{L^2}=
	F^{(1)}(x)\rho(x,t),
	 $$ 
	 up to a constant. 
Then it is easy to check that $A(g,h)\Big|_{\varepsilon_1=0}$ is the solution map of the following system which is called the first-order linearization system:
\begin{equation}\label{MFG2Linear1}
    \begin{cases}
        -\partial_t u^{(1)}(x,t) -\sigma \Delta u^{(1)}(x,t) + \kappa\nabla u_0\cdot \nabla u^{(1)}(x,t)= F^{(1)}(x,t)m^{(1)}(x,t) &\quad \text{in }Q,\\
        \partial_t m^{(1)}(x,t) -\Delta (\sigma m^{(1)}(x,t)) - \nabla\cdot(\kappa m_0\nabla u^{(1)}(x,t)) -\nabla \cdot (\kappa m^{(1)}(x,t)\nabla u_0)= 0  &\quad \text{in }Q,\\
        u^{(1)}(x,t)=g_1,\quad m^{(1)}(x,t)=h_1 &\quad \text{in }\Gamma,\\
        u^{(1)}(x,T)=G^{(1)}(x)m^{(1)}(x,T), \quad m^{(1)}(x,0) = 0  &\quad \text{in }\Omega.
    \end{cases}
\end{equation}

 In the following, we define
\begin{equation}\label{eq:ld1}
 (u^{(1)}, m^{(1)} ):=A(g,h)\Big|_{\varepsilon_1=0}. 
 \end{equation}
For notational convenience, we write
\begin{equation}\label{eq:ld2}
u^{(1)}=\partial_{\varepsilon_1}u(x,t;\varepsilon)|_{\varepsilon=0}\quad\text{and}\quad m^{(1)}=\partial_{\varepsilon_1}m(x,t;\varepsilon)|_{\varepsilon=0}.
\end{equation}
We shall utilize such notations in our subsequent discussion in order to simplify the exposition and their meaning should be clear from the context. In a similar manner, we can define, for all $l\in \mathbb{N}$, $u^{(l)} := \left. \partial_{\varepsilon_l} u \right\rvert_{\varepsilon = 0}$, $m^{(l)} := \left. \partial_{\varepsilon_l} m \right\rvert_{\varepsilon = 0}$, 
we obtain a sequence of similar systems.

For the higher orders, we consider
\[u^{(1,2)}:= \left. \partial_{\varepsilon_1} \partial_{\varepsilon_2} u \right\rvert_{\varepsilon = 0}, \quad m^{(1,2)}:= \left. \partial_{\varepsilon_1} \partial_{\varepsilon_2} m \right\rvert_{\varepsilon = 0}.\]
Similarly, $(u^{(1,2)},m^{(1,2)})$ can be viewed as the output of the second-order Fr\'echet derivative of $S$ at a specific point. By following similar calculations in deriving \eqref{MFG2Linear1}, one can show that the second-order linearization is given as follows:
\begin{equation}\label{MFGQuadraticLinear2}
    \begin{cases}
        -\partial_t u^{(1,2)} -\sigma\Delta u^{(1,2)} + \kappa\nabla u_0 \cdot \nabla u^{(1,2)} + \kappa\nabla u^{(1)} \cdot \nabla u^{(2)} \\\qquad\qquad\qquad\qquad = F^{(1)}m^{(1,2)} + F^{(2)}m^{(1)}m^{(2)} &\quad \text{in }Q,\\
        \partial_t m^{(1,2)} -\Delta (\sigma m^{(1,2)}) -\nabla\cdot(\kappa m_0\nabla u^{(1,2)}) - \nabla \cdot(\kappa m^{(2)} \nabla u^{(1)}) \\\qquad\qquad\qquad\qquad = \nabla\cdot(\kappa m^{(1)} \nabla u^{(2)}) + \nabla\cdot(\kappa m^{(1,2)} \nabla u_0)   &\quad \text{in }Q,\\
        u^{(1)}(x,T)=G^{(1)}(x)m^{(1,2)}(x,T)+G^{(2)}(x)m^{(1)}(x,T)m^{(2)}(x,T) &\quad \text{in }\Omega,\\
        m^{(1,2)}(x,0) = 0  &\quad \text{in }\Omega.
    \end{cases}
\end{equation}

Inductively, for $N\in \mathbb{N}$, we consider
\[u^{(1,2,\dots,N)}:= \left. \partial_{\varepsilon_1} \partial_{\varepsilon_2} \cdots \partial_{\varepsilon_N} u \right\rvert_{\varepsilon = 0}, \quad m^{(1,2,\dots,N)}:= \left. \partial_{\varepsilon_1} \partial_{\varepsilon_2} \cdots \partial_{\varepsilon_N} m \right\rvert_{\varepsilon = 0},\] and obtain a sequence of parabolic systems.









\chapter{Inverse Coefficient Problems Using (Time-Space) Boundary Measurements}\label{InverseCoefBdry}
In this chapter, our objective is to demonstrate that the running cost function and/or the total cost function can be uniquely determined by certain boundary data. Broadly speaking, we aim to establish a sequence of theorems of the following form:
\begin{thm}
    Assume $F_j,G_j (j=1,2)$ satisfy some suitable conditions. Let $\mathcal{M}_{F_j,G_j}$ be the measurement map associated to
	the following system:
	\begin{equation}
	\begin{cases}
		-\partial_t u(x, t)-\Delta u(x, t)+H(x, \nabla_x u(x, t))=F(x, m(x, t)),\quad & (x, t)\in \Omega\times (0, T),\medskip \\
		\partial_tm(x, t)-\Delta m(x, t)-{\rm div} \left(m(x, t)\nabla_p H\big(x, \nabla_x u(x, t)\big) \right)=0,\quad & (x, t)\in \Omega\times (0, T),\medskip\\
		m(x, 0)=m_0(x), \quad u(x, T)=G(x,m_T),\quad & x\in\Omega, 
	\end{cases}
\end{equation}
	If for any $m_0\in C^{2+\alpha}(\Omega)\cap\mathcal{O}_a$, one has 
	$$\mathcal{M}_{F_1,G_1}(m_0)=\mathcal{M}_{F_2,G_2}(m_0),$$    then it holds that	
	$$(G_1(x,z),F_1(x,z))=(G_2(x,z),F_2(x,z)) \ \text{  in  } \Omega\times \mathbb{R}.$$ 
\end{thm}
In this chapter, we will explore such problems in different domains $\Omega$ and under varying conditions.

\section{Local Well-posedness of the Forward Problems}\label{section wp}
In this section, we show the well-posedness of the MFG systems in our study. In fact, we shall use several quite similar well-posedness results in different set-ups. We only show the details of the proofs for the case $\Omega=\mathbb{T}^n$. 
We consider the case that $F$ and $G$ belong
to an analytical class. Henceforth, we set
\begin{equation}\label{eq:Q}
	Q=\overline{\mathbb{T}^n\times(0,T) },
\end{equation}
be the closure of $\mathbb{T}^n\times(0,T).$
Let us recall the admissible class we posed in the preview chapter and give some discussions here.
\begin{defn}\label{Admissible class1}
	We say $U(x,t,z):\mathbb{T}^n\times \mathbb{R}\times\mathbb{C}\to\mathbb{C}$ is admissible, denoted by $U\in \mathcal{A}$, if it satisfies the following conditions:
	\begin{enumerate}
		\item[(i)]~The map $z\mapsto U(\cdot,\cdot,z)$ is holomorphic with value in $C^{2+\alpha,1+\frac{\alpha}{2}}(Q)$ for some $\alpha\in(0,1)$;
		\item[(ii)] $U(x,t,0)=0$ for all $(x,t)\in\mathbb{T}^n\times (0,T).$ 	
	\end{enumerate}
	
	Clearly, if (1) and (2) are fulfilled, then $U$ can be expanded into a power series as follows:
	\begin{equation}\label{eq:F}
		U(x,t,z)=\sum_{k=1}^{\infty} U^{(k)}(x,t)\frac{z^k}{k!},
	\end{equation}
	where $ U^{(k)}(x,t)=\frac{\p^k U}{\p z^k}(x,t,0)\in C^{2+\alpha,1+\frac{\alpha}{2}}(Q).$
\end{defn}   
\begin{defn}\label{Admissible class2}
	We say $U(x,z):\mathbb{T}^n\times\mathbb{C}\to\mathbb{C}$ is admissible, denoted by $U\in\mathcal{B}$, if it satisfies the following conditions:
	\begin{enumerate}
		\item[(i)] The map $z\mapsto U(\cdot,z)$ is holomorphic with value in $C^{2+\alpha}(\mathbb{T}^n)$ for some $\alpha\in(0,1)$;
		\item[(ii)] $U(x,0)=0$ for all $x\in\mathbb{T}^n.$
	\end{enumerate}

	Clearly, if (1) and (2) are fulfilled, then $U$ can be expanded into a power series as follows:
	\begin{equation}\label{eq:G}
		U(x,z)=\sum_{k=1}^{\infty} U^{(k)}(x)\frac{z^k}{k!},
	\end{equation}
	where $ U^{(k)}(x)=\frac{\p^kU}{\p z^k}(x,0)\in C^{2+\alpha}(\mathbb{T}^n).$
\end{defn}

Even though we already give the global well-posedness result in section \ref{reviw}, the reason why we give a local well-posedness here is that
we also show the infinite differentiability of the equation with respect to a given (small) input $m_0(x).$ This is the foundation of the linearization method that we discussed in section \ref{HLM}.

As a preliminary, we recall the well-posedness result
for linear parabolic equations \cite{Lady}\cite[Lemma 3.3]{Linear_in_Tn}.
\begin{lem}\label{linear app unique}
	Consider the parabolic equation 
	\begin{equation}\label{linearapp wellpose}
		\begin{cases}
			-\p_tv(x,t)-\Delta v(x,t)+{\rm div} ( a(x,t)\cdot\nabla v(x,t))= f(x,t),& \text{ in }\mathbb{T}^n\times(0,T),\medskip\\
			v(x,0)=v_0(x), & \text{ in } \mathbb{T}^n, 
		\end{cases}  	
	\end{equation}
	where the periodic boundary condition is imposed on $v$. Suppose $a,f\in C^{\alpha,\frac{\alpha}{2}}(Q) $ and $v_0\in C^{2+\alpha}(\mathbb{T}^n)$, then \eqref{linearapp wellpose} has a unique classical solution $v\in C^{2+\alpha,1+\frac{\alpha}{2}}(Q).$
\end{lem}

The following result is somewhat standard (especially Theorem \ref{local_wellpose}-(a)), while our technical conditions could be different from those in the literature. For completeness we provide a proof here. The following proof is based on the implicit functions theorem for Banach spaces. One may refer to \cite{Pos.J} for more related details about the theory of maps between Banach spaces. 
\begin{thm}\label{local_wellpose}
	Suppose that $F\in\mathcal{A}$ and $G\in\mathcal{B}$. The following results holds:
	\begin{enumerate}
		
		\item[(a)]
		There exist constants $\delta>0$ and $C>0$ such that for any 
		\[
		m_0\in B_{\delta}(C^{2+\alpha}(\mathbb{T}^n)) :=\{m_0\in C^{2+\alpha}(\mathbb{T}^n): \|m_0\|_{C^{2+\alpha}(\mathbb{T}^n)}\leq\delta \},
		\]
		the MFG system \eqref{eq:mfg-periodic'} has a solution $u \in
		C^{2+\alpha,1+\frac{\alpha}{2}}(Q)$ which satisfies
		\begin{equation}\label{eq:nn1}
			\|(u,m)\|_{ C^{2+\alpha,1+\frac{\alpha}{2}}(Q)}:= \|u\|_{C^{2+\alpha,1+\frac{\alpha}{2}}(Q)}+ \|m\|_{C^{2+\alpha,1+\frac{\alpha}{2}}(Q)}\leq C\|m_0\|_{ C^{2+\alpha}(\mathbb{T}^n)}.
		\end{equation}
		Furthermore, the solution $(u,m)$ is unique within the class
		\begin{equation}\label{eq:nn2}
			\{ (u,m)\in  C^{2+\alpha,1+\frac{\alpha}{2}}(Q)\times C^{2+\alpha,1+\frac{\alpha}{2}}(Q): \|(u,m)\|_{ C^{2+\alpha,1+\frac{\alpha}{2}}(Q)}\leq C\delta \}.
		\end{equation}		
		\item[(b)] Define a function 
		\[
		S: B_{\delta}(C^{2+\alpha}(\mathbb{T}^n))\to C^{2+\alpha,1+\frac{\alpha}{2}}(Q)\times C^{2+\alpha,1+\frac{\alpha}{2}}(Q)\ \mbox{by $S(m_0):=(u,v)$}. 
		\] 
		where $(u,v)$ is the unique solution to the MFG system \eqref{eq:mfg-periodic'}.
		Then for any $m_0\in B_{\delta}(C^{2+\alpha}(\mathbb T^n))$, $S$ is holomorphic.
	\end{enumerate}
\end{thm}
\begin{proof}
	Let 
	\begin{align*}
		&X_1:= C^{2+\alpha}(\mathbb{T}^n ), \\
		&X_2:=C^{2+\alpha,1+\frac{\alpha}{2}}(Q)\times C^{2+\alpha,1+\frac{\alpha}{2}}(Q),\\
		&X_3:=C^{2+\alpha}(\mathbb{T}^n)\times C^{2+\alpha}(\mathbb{T}^n)\times C^{\alpha,\frac{\alpha}{2}}(Q )\times C^{\alpha,\frac{\alpha}{2}}(Q ),
	\end{align*} and we define a map $\mathscr{K}:X_1\times X_2 \to X_3$ by that for any $(m_0,\tilde u,\tilde m)\in X_1\times X_2$,
	\begin{align*}
		&
		\mathscr{K}( m_0,\tilde u,\tilde m)(x,t)\\
		:=&\big( \tilde u(x,T)-G(x,\tilde m(x,T)), \tilde m(x,0)-m_0(x) , 
		-\p_t\tilde u(x,t)-\Delta \tilde u(x,t)\\ &+\frac{|\nabla \tilde u(x,t)|^2}{2}- F(x,t,\tilde m(x,t)), 
		\p_t \tilde m(x,t)-\Delta \tilde m(x,t)-{\rm div}(\tilde m(x,t)\nabla \tilde u(x,t))  \big) .
	\end{align*}

	First, we show that $\mathscr{K} $ is well-defined. Since the
	H\"older space is an algebra under the point-wise multiplication, we have $|\nabla u|^2, {\rm div}(m(x,t)\nabla u(x,t))  \in C^{\alpha,\frac{\alpha}{2}}(Q ).$
	By the Cauchy integral formula,
	\begin{equation}\label{eq:F1}
		F^{(k)}\leq \frac{k!}{R^k}\sup_{|z|=R}\|F(\cdot,\cdot,z)\|_{C^{\alpha,\frac{\alpha}{2}}(Q ) },\ \ R>0.
	\end{equation}
	Then there is $L>0$ such that for all $k\in\mathbb{N}$,
	\begin{equation}\label{eq:F2}
		\left\|\frac{F^{(k)}}{k!}m^k\right\|_{C^{\alpha,\frac{\alpha}{2}}(Q )}\leq \frac{L^k}{R^k}\|m\|^k_{C^{\alpha,\frac{\alpha}{2}}(Q )}\sup_{|z|=R}\|F(\cdot,\cdot,z)\|_{C^{\alpha,\frac{\alpha}{2}}(Q ) }.
	\end{equation}
	By choosing $R\in\mathbb{R}_+$ large enough and by virtue of \eqref{eq:F1} and \eqref{eq:F2}, it can be seen that the series \eqref{eq:F} converges in $C^{\alpha,\frac{\alpha}{2}}(Q )$ and therefore $F(x,m(x,t))\in  C^{\alpha,\frac{\alpha}{2}}(Q ).$ Similarly, we have $G(x,m(x,T))\in C^{2+\alpha}(\mathbb{T}^n).$ 
	This implies that $\mathscr{K} $ is well-defined.

	Let us show that $\mathscr{K}$ is holomorphic. Since $\mathscr{K}$ is clearly locally bounded, it suffices to verify that it is weakly holomorphic; see \cite[P.133 Theorem 1]{Pos.J}. That is we aim to show the map
	$$\lambda\in\mathbb C \mapsto \mathscr{K}((m_0,\tilde u,\tilde m)+\lambda (\bar m_0,\bar u,\bar m))\in X_3,\quad\text{for any $(\bar m_0,\bar u,\bar m)\in X_1\times X_2$}$$
	is holomorphic. In fact, this follows from the condition that $F\in\mathcal{A}$ and $G\in\mathcal{B}$.

	Note that $  \mathscr{K}(0,0,0)=0$. Let us compute $\nabla_{(\tilde u,\tilde m)} \mathscr{K} (0,0,0)$:
	\begin{equation}\label{Fer diff}
		\begin{aligned}
			\nabla_{(\tilde u,\tilde m)} \mathscr{K}(0,0,0) (u,m) =(& u|_{t=T}-G^{(1)}m(x,T), m|_{t=0}, \\
			&-\p_tu(x,t)-\Delta u(x,t)-F^{(1)}m, \p_t m(x,t)-\Delta m(x,t)).
		\end{aligned}			
	\end{equation}

	By Lemma $\ref{linear app unique}$, 
	if $\nabla_{(\tilde u,\tilde m)} \mathscr{K} (0,0,0)=0$, we have 
	$ \tilde m=0$ and then $\tilde u=0$. Therefore, the map is injective. 
	
	On the other hand,  letting $ (r(x),s(x,t))\in C^{2+\alpha}(\mathbb{T}^n)\times C^{\alpha,\frac{\alpha}{2}}(Q ) $,
	and by Lemma $\ref{linear app unique}$, there exists $a(x,t)\in C^{2+\alpha,1+\frac{\alpha}{2}}(Q)$ such that
	\begin{equation*}
		\begin{cases}
			\p_t a(x,t)-\Delta a(x,t)=s(x,t)  &\text{ in } \mathbb{T}^n,\medskip\\
			a(x,0)=r(x)                       &  \text{ in } \mathbb{T}^n .
		\end{cases}
	\end{equation*}
	Then letting $ (r'(x),s'(x,t))\in C^{2+\alpha}(\mathbb{T}^n)\times C^{\alpha,\frac{\alpha}{2}}(Q ) $, one can show that there exists $ b(x,t)\in C^{2+\alpha,1+\frac{\alpha}{2}}(Q)$ such that
	\begin{equation*}
		\begin{cases}
			-\p_t b(x,t)-\Delta b(x,t)-F^{(1)}a=s'(x,t)  &\text{ in } \mathbb{T}^n,\medskip\\
			b(x,T)=G^{(1)}a(x,T)+r'(x)                  &  \text{ in } \mathbb{T}^n.
		\end{cases}
	\end{equation*}
	
	Therefore, $\nabla_{(\tilde u,\tilde m)} \mathscr{K} (0,0,0)$ is a linear isomorphism between $X_2$ and $X_3$. Hence, by the implicit function theorem, there exist $\delta>0$ and a unique holomorphic function $S: B_{\delta}(\mathbb{T}^n)\to X_2$ such that $\mathscr{K}(m_0,S(m_0))=0$ for all $m_0\in B_{\delta}(\mathbb{T}^n) $.
	
	By letting $(u,m)=S(m_0)$, we obtain the unique solution of the MFG system \eqref{eq:mfg-periodic'}. Let $ (u_0,v_0)=S(0)$.   Since $S$ is Lipschitz, we know that there exist constants $C,C'>0$ such that 
	\begin{equation*}
		\begin{aligned}
			&\|(u,m)\|_{ C^{2+\alpha,1+\frac{\alpha}{2}}(Q)^2}\\
			\leq& C'\|m_0\|_{B_{\delta}(\mathbb{T}^n)} +\|u_0\|_	{ C^{2+\alpha,1+\frac{\alpha}{2}}(Q)}+\|v_0\|_{ C^{2+\alpha,1+\frac{\alpha}{2}}(Q)}\\
			\leq& C \|m_0\|_{B_{\delta}(\mathbb{T}^n)}.
		\end{aligned}
	\end{equation*}
	
	The proof is complete. 
\end{proof}
\begin{rmk}
	Regarding the local well-posedness, several remarks are in order.
	\begin{enumerate}[(a)]
		
		\item 
        The admissibility conditions in Definitions~\ref{Admissible class1} and \ref{Admissible class2} shall be imposed as a-priori conditions on the unknowns $F$ and $G$ in what follows for our inverse problem study. It is remarked that as noted earlier that both $F$ and $G$ are functions of real variables. However, for technical reasons, we extend the functions to the complex plane with respect to the $z$-variable, namely $U(\cdot,z)$ and $ U(\cdot,\cdot,z)$, and assume that they are holomorphic as functions of the complex variable $z$. This also means that we shall assume $F$ and $G$ are restrictions of those holomorphic functions to the real line. This technical assumption shall be used to show the well-posedness of the MFG system in section $\ref{section wp}.$ Throughout the chapter, we also assume that in the series expansions \eqref{eq:F} and \eqref{eq:G}, the coefficient functions $U^{(k)}$ are real-valued. 
        \item The conditions on $F$ and $G$ (Definition \ref{Admissible class1}-(i) and $G$ satisfies Definition \ref{Admissible class2}-(i) ) are not essential and it is for convenience to
		apply implicit function theorem . Also, the analytic conditions on $F$ and $G$ can be replayed by weaker regularity conditions in the proof of the local well-posedness  \cite{LasryLions1} , but these conditions will be utilized in our 
		inverse problem study. 
		
		\item  In order to apply the higher order linearization method, we need the infinite differentiability of the equation with respect to the given input $m_0(x)$, it is shown by the fact that the solution map $S$ is holomorphic.

		\item In the proof of Theorem $\ref{local_wellpose}$, we show the solution map $S$ is holomorphic. As a corollary,  the measurement map $\mathcal{M}=\left.\pi_1\circ S\right|_{t=0}$ is also holomorphic, where $\pi_1$ is the projection map with respect to the first variable.
		
	\end{enumerate}	 	
\end{rmk}

 \section{Inverse Problems in $\mathbb{T}^n$}
 From this section, we begin to study inverse problems for MFG game system and through this section, we assume $\Omega:=\mathbb{T}^n.$
\begin{equation}\label{eq:mfg-periodic'}
	\left\{
	\begin{array}{ll}
		-\partial_t u(x,t) -\Delta u(t,x)+ H\big(x,\nabla u(x,t)\big)-F(x,m(x,t))=0,&  {\rm{in}}\ \mathbb T^n\times (0,T),\medskip\\
		\partial_tm(x,t)-\Delta m(x,t)-{\rm div} \big(m(x,t) \nabla_pH(x, \nabla u(x,t)\big)=0, & {\rm{in}}\ \mathbb T^n\times(0,T),\medskip\\
		u(x,T)=G(x,m(x,T)),\ m(x,0)=m_0(x), & {\rm{in}}\ \mathbb T^n,
	\end{array}
	\right.
\end{equation}

We recall that $m$ is a density function for some distribution. Hence,
\begin{equation}\label{probability measure constraint}
	m\in \mathcal{O}_a:=\{  m:  \int_{\mathbb{T}^n} m \, dx=a\leq  1 . \}
\end{equation}
This is called the probability measure constraint  .

We introduce the following measurement map $\mathcal{M}_{F,G}$:
\begin{align}
	\begin{split}
		\mathcal{M}_{F,G}: m_0\in C^{2+\alpha}(\mathbb{T}^n)\cap\mathcal{O}_a & \rightarrow  L^2(\mathbb{T}^n) , \\  
		m_0&\mapsto  \left(x\in\mathbb{T}^n \mapsto \left.u(x,t) \right|_{t=0}\right), 
	\end{split}
\end{align}
where $u(x,t)$ is the solution of \eqref{eq:mfg-periodic'} with initial data $m(x,0)=m_0(x).$

\subsection{Main Unique Identifiability Results}

Now, we state our main result in this chapter. Here and also in what follows, we sometimes drop the dependence on $F, G$ of $\mathcal{M}$, and in particular in the case that one quantity is a-priori known, say  $\mathcal{M}_F$ or $\mathcal{M}_G$, which should be clear from the context. 

\begin{thm}\label{der g}
	Assume $F \in\mathcal{A}$, $G_j\in\mathcal{B}$ ($j=1,2$). Let $\mathcal{M}_{G_j}$ be the measurement map associated to
	the following system:
	\begin{equation}
		\begin{cases}
			-\p_tu(x,t)-\Delta u(x,t)+\frac 1 2 {|\nabla u(x,t)|^2}= F(x,t,m(x,t)),& \text{ in } \mathbb{T}^n\times (0,T),\medskip\\
			\p_t m(x,t)-\Delta m(x,t)-{\rm div}(m(x,t)\nabla u(x,t))=0,&\text{ in }\mathbb{T}^n\times(0,T),\medskip\\
			u(x,T)=G_j(x,m(x,T)), & \text{ in } \mathbb{T}^n,\medskip\\
			m(x,0)=m_0(x), & \text{ in } \mathbb{T}^n.\\
		\end{cases}  		
	\end{equation}	
	If for any $m_0\in C^{2+\alpha}(\mathbb{T}^n)\cap\mathcal{O}_a$, one has 
	$$\mathcal{M}_{G_1}(m_0)=\mathcal{M}_{G_2}(m_0),$$    then it holds that		
	$$G_1(x,z)=G_2(x,z)\ \text{  in  } \mathbb{T}^n\times \mathbb{R}.$$ 
\end{thm}
Notice that in Theorems \ref{der g} we allow $F$ to depend on time. If we assume  $F$ depends only on $x$ and $m(x,t)$, we can determine $F$ and $G$ simultaneously. 

\begin{thm}\label{der F,g}
	Assume $F_j,G_j \in\mathcal{B}$ ($j=1,2$) . Let $\mathcal{M}_{F_j,G_j}$ be the measurement map associated to
	the following system:
	\begin{equation}\label{eq:mfg3}
		\begin{cases}
			-\p_tu(x,t)-\Delta u(x,t)+\frac 1 2 {|\nabla u(x,t)|^2}= F_j(x,m(x,t)),& \text{ in }\mathbb{T}^n\times(0,T),\medskip\\
			\p_t m(x,t)-\Delta m(x,t)-{\rm div}(m(x,t)\nabla u(x,t))=0,&\text{ in }\mathbb{T}^n\times (0,T),\medskip\\
			u(x,T)=G_j(x,m(x,T)), & \text{ in } \mathbb{T}^n,\medskip\\
			m(x,0)=m_0(x), & \text{ in } \mathbb{T}^n.\\
		\end{cases}  		
	\end{equation}	
	If for any $m_0\in C^{2+\alpha}(\mathbb{T}^n)\cap\mathcal{O}_a$, one has 
	$$\mathcal{M}_{F_1,G_1}(m_0)=\mathcal{M}_{F_2,G_2}(m_0),$$    then it holds that	
	$$(G_1(x,z),F_1(x,z))=(G_2(x,z),F_2(x,z)) \ \text{  in  } \mathbb{T}^n\times \mathbb{R}.$$ 
\end{thm}
\subsection{The Proofs of Inverse Coefficient Problems in $\mathbb{T}^n$}

We first present the proof of Theorem $\ref{der g}$, the unique determination result of single unknown function.

\begin{proof}[Proof of Theorem $\ref{der g}$]
	Consider the following systems for $j=1,2$:
	\begin{equation}\label{j=1,2for g}
		\begin{cases}
			-\p_tu_j(x,t)-\Delta u_j(x,t)+\frac 1 2 {|\nabla u_j|^2}= F(x,t,m_j(x,t)),& \text{ in }\mathbb{T}^n\times (0,T),\medskip\\
			\p_t m_j(x,t)-\Delta m_j(x,t)-{\rm div}(m_j(x,t)\nabla u_j(x,t))=0,&\text{ in }\mathbb{T}^n\times(0,T),\medskip\\
			u_j(x,T)=G_j(x,m_j(x,T)), & \text{ in } \mathbb{T}^n,\medskip\\
			m_j(x,0)=m_0(x), & \text{ in } \mathbb{T}^n.\\
		\end{cases}  		
	\end{equation}
	By the successive linearization procedure,  we first consider the case $N=1.$ Let
	$$u_{j}^{(1)}:=\p_{\varepsilon_1}u_{j}|_{\varepsilon=0},\quad m_{j}^{(1)}:=\p_{\varepsilon_1}m_{j}|_{\varepsilon=0}.$$
	Direct computations show that $(u_{j}^{(1)},v_{j}^{(1)} )$ satisfies the following system
	\begin{equation}\label{linear l=1for g}
		\begin{cases}
			-\p_tu_j^{(1)}(x,t)-\Delta u^{(1)}_j(x,t)= F^{(1)}(x,t)m_j^{(1)}(x,t),& \text{ in }\mathbb{T}^n\times(0,T),\medskip\\
			\p_t m^{(1)}_j(x,t)-\Delta m^{(1)}_j(x,t)=0,&\text{ in }\mathbb{T}^n\times(0,T),\medskip\\
			u^{(1)}_j(x,T)=G_j^{(1)}(x)m^{(1)}_j(x,T), & \text{ in } \mathbb{T}^n,\medskip\\
			m^{(1)}_j(x,0)=f_1(x), & \text{ in } \mathbb{T}^n.
		\end{cases}  	
	\end{equation}	
	We can solve the system \eqref{linear l=1for g} by first deriving $m^{(1)}_j$ and then obtaining $u^{(1)}_j.$ In doing so, we can obtain that the solution is
	$$ m_j^{(1)}(x,t)= \int_{\mathbb{R}^n}\Phi(x-y,t)f_{1}(y)\, dy,$$
	\begin{equation*}
		\begin{aligned}
			u_j^{(1)}(x,t)&=  \int_{\mathbb{R}^n}\Phi(x-y,T-t)G_j^{(1)}(y)m^{(1)}_j(y,T) )\, dy\\
			&+\int_{0}^{T-t}\int_{\mathbb{R}^n}\Phi(x-y,T-t-s)F^{(1)}(y,T-s)\overline{m}_j^{(1)}(y,s)\, dyds,
		\end{aligned}
	\end{equation*}
	where $\overline{m}_j^{(1)}(x,t)= m_j^{(1)}(x,T-t)$ and $\Phi$ is the  fundamental solution of the heat equation:
	\begin{equation}\label{eq:fund1}
		\Phi(x,t)= \frac{1}{(4\pi t)^{n/2}}e^{-\frac{|x|^2}{4t}}.
	\end{equation}
	
	Since  $\mathcal{M}_{G_1}=\mathcal{M}_{G_2}$, we have $$ u_1^{(1)}(x,0)=u_2^{(1)}(x,0),$$ for all $f_1\in C_+^{2+\alpha}(\mathbb{T}^n).$ This implies that
	$$ \int_{\mathbb{R}^n}\Phi(x-y,T)[G_1^{(1)}(y)m_1^{(1)}(y,T))-G_2^{(1)}(y)m_2^{(1)}(y,T)) ]\, dy=0.$$
	Noticing that $m_1^{(1)}(x,t)=m_2^{(1)}(x,t)$,  we  choose 
	$$m_1^{(1)}(x,T)=m_2^{(1)}(x,T)=\exp(-4\pi^2|\mathbf{\zeta}|^2T-2\pi \mathrm{i} \mathbf{\zeta}\cdot x)+M,$$ 
	where $\mathbf{\zeta}\in\mathbb{Z}^n, M\in\mathbb{N}.$ (In this case, $f_1(x)\in C_+^{2+\alpha}(\mathbb{T}^n)$.)
	
	By taking $M=1$ and $M=2$, respectively and then subtracting the resulting equations from one another, one can readily show that
	\begin{equation}
		\int_{\mathbb{R}^n}\Phi(x-y,T)[(G_1^{(1)}(y)-G_2^{(1)}(y))\exp(-2\pi \mathrm{i} \mathbf{\zeta}\cdot y)         ]\, dy=0,
	\end{equation}
	for all $\mathbf{\zeta}\in\mathbb{Z}^n.$
	Therefore $G_1^{(1)}(x)=G_2^{(1)}(x).$

	We proceed to consider the case $N=2.$ Let
	$$u_{j}^{(1,2)}:=\p_{\varepsilon_1}\p_{\varepsilon_2}u_{j}|_{\varepsilon=0},\quad
	m_{j}^{(1,2)}:=\p_{\varepsilon_1}\p_{\varepsilon_2}m_{j}|_{\varepsilon=0},$$
	and
	$$u_{j}^{(2)}:=\p_{\varepsilon_2}u_{j}|_{\varepsilon=0},\quad m_{j}^{(2)}:=\p_{\varepsilon_2}m_{j}|_{\varepsilon=0}.$$
	Then we can deal with the second-order linearization: 
	\begin{equation}\label{linear l=1,2}
		\begin{cases}
			-\p_tu_j^{(1,2)}(x,t)-\Delta u^{(1,2)}_j(x,t)+\nabla u_{j}^{(1)}\cdot \nabla u_{j}^{(2)}\\
			\hspace*{3cm} = F^{(1)}m_j^{(1,2)}+F^{(2)}(x,t)m_j^{(1)}m_j^{(2)},& \text{ in }\mathbb{T}^n\times (0,T),\\
			\p_t m^{(1,2)}_j(x,t)-\Delta m^{(1,2)}_j(x,t)\medskip \\
			\hspace*{3cm} = {\rm div} (m_{j}^{(1)}\nabla u_j^{(2)})+{\rm div}(m_j^{(2)}\nabla u_j^{(1)}) ,&\text{ in }\mathbb{T}^n\times(0,T),\medskip\\
			u^{(1,2)}_j(x,T)=G_j^{(1)}(x)m_j^{(1,2)}(x,T)+G_j^{(2)}(x)m_j^{(1)}m_j^{(2)}(x,T), & \text{ in } \mathbb{T}^n,\medskip\\
			m^{(1,2)}_j(x,0)=0, & \text{ in } \mathbb{T}^n.\\
		\end{cases}  	
	\end{equation}
	Since we have shown that $G_1^{(1)}(x)=G_2^{(1)}(x)$, we have 
	$$ u^{(1)}_1(x,t)= u^{(1)}_2(x,t), \, m^{(1)}_1(x,t)= m^{(1)}_2(x,t)$$
	by solving equation \eqref{linear l=1for g}.
	
	Then by the same argument in the case $N=1$ (considering $m_0=\varepsilon_2f_2$ ), we have
	
	$$u^{(2)}_1(x,t)=u^{(2)}_2(x,t), \, m^{(2)}_1(x,t)= m^{(2)}_2(x,t).$$
	
	Denote 
	\[
	p(x,t)={\rm div} (m_{j}^{(1)}\nabla u_j^{(2)})+{\rm div}(m_j^{(2)}\nabla u_j^{(1)}),\ \ q(x,t)= -\nabla u_{j}^{(1)}\cdot \nabla u_{j}^{(2)}.
	\] 
	Then we can also solve  system \eqref{linear l=1,2} as follows:
	\begin{equation*}
		m^{(1,2)}_j(x,t)=\int_{0}^{t} \int_{\mathbb{R}^n} \Phi(x-y,t-s)p(y,s)\, dyds,
	\end{equation*}
	\begin{equation*}
		\begin{aligned}
			u_j^{(1,2)}(x,t)= &\int_{\mathbb{R}^n}\Phi(x-y,T-t) [G_j^{(1)}(x)m_j^{(1,2)}(x,T)+G_j^{(2)}(x)m_j^{(1)}m_j^{(2)}(x,T) ]\, dy\\
			+&\int_{0}^{T-t}\int_{\mathbb{R}^n}\Phi(x-y,T-t-s)(F^{(2)}(y,T-s)m_j^{(1)}m_j^{(2)}(y,T-s) -\overline{q}(y,s))\, dyds,
		\end{aligned}
	\end{equation*}
	where $\overline{q}(y,s)=q(y,T-s).$
	Since  $$u_1^{(1,2)}(x,0)= u_2^{(1,2)}(x,0),$$ we have	$$ \int_{\mathbb{R}^n}\Phi(x-y,T)[G_1^{(2)}(y)m_1^{(1)}(y,T)-G_2^{(1)}(y)m_2^{(1)}(y,T) ]\, dy=0.$$

	Next, by a similar argument in the case $N=1$, we can prove that $G^{(2)}_1(x)=G^{(2)}_2(x). $ 
	Finally, by mathematical induction, we can show the same result for $N\geq 3$. That is, for any $k\in\mathbb{N},$ we have $G^{(k)}_1(x)=G^{(k)}_2(x).$ 
	The proof is complete. 
\end{proof}
Next, we aim to determinate $F$ and $G$ simultaneously. To that end, we first derive an auxiliary lemma as follows. 

\begin{lem}\label{dense}
	Let $u$ be a solution of the heat equation 
	\begin{equation}\label{per heat}
		\begin{cases}
			\p_t u(x,t)-\Delta u(x,t)=0    &\text{ in } \mathbb{T}^n\times(0,T),\\
			u(x,0)=u_0(x)                  &\text{ in } \mathbb{T}^n.
		\end{cases}
	\end{equation}
	Let $f(x)\in C^{2+\alpha}(\mathbb{T}^n)$ for some $\alpha\in(0,1)$. Suppose 
	\begin{equation}\label{fuv=0}
		\int_{\mathbb{T}^n\times(0,T)} f(x)u(x,t)\, dxdt=0,	
	\end{equation}
	for all $u_0\in C_+^{\infty}(\mathbb{T}^n)$. Then one has $f=0.$
\end{lem}
\begin{proof}
	Let  $\mathbf {\xi}\in\mathbb{Z}^n$ and $M\in\mathbb{N}$. It is directly verified that
	$$
	u(x,t)=\exp(- 2\pi\mathrm{i}\mathbf {\xi}\cdot x-4\pi^2|\mathbf {\xi}|^2t )+M, \quad  \mathrm{i}:=\sqrt{-1},
	$$
	is a solution of \eqref{per heat} with initial value 
	$$u_0(x)= \exp(- 2\pi\mathrm{i}\mathbf {\xi}\cdot x)+M\geq0.$$
	Then \eqref{fuv=0} implies that
	$$\int_{\mathbb{T}^n}  \frac{1-\exp(-4\pi^2|\mathbf {\xi}|^2T)}{4\pi^2|\mathbf {\xi}|^2}f(x)e^{-2\pi \mathrm{i}\mathbf {\xi}\cdot x } dx+MT\int_{\mathbb{T}^n}f(x)dx=0.$$
	By taking $M=1$ and $M=2$, respectively, we have 
	$$ \int_{\mathbb{T}^n}f(x)e^{-2\pi \mathrm{i}\mathbf {\xi}\cdot x } dx=0.$$
	Hence, the Fourier series of $f(x)$ is $0$. Since $f(x)\in C^{2+\alpha}(\mathbb{T}^n)$, its Fourier series converges to $f(x)$ uniformly.
	Therefore, $f(x)=0.$
\end{proof}

We are now in a position to present the proof of Theorem $\ref{der F,g}$.
\begin{proof}[Proof of Thoerem $\ref{der F,g}$]
	Consider the following systems 
	\begin{equation}\label{j=1,2for Fg}
		\begin{cases}
			-\p_tu_j(x,t)-\Delta u_j(x,t)+\frac 1 2 {|\nabla u_j|^2}= F_j(x,m_j(x,t)),& \text{ in }\mathbb{T}^n\times(0,T),\medskip\\
			\p_t m_j(x,t)-\Delta m_j(x,t)-{\rm div}(m_j(x,t)\nabla u_j(x,t))=0,&\text{ in }\mathbb{T}^n\times(0,T),\medskip\\
			u_j(x,T)=G_j(x,m_j(x,T)), & \text{ in } \mathbb{T}^n,\medskip\\
			m_j(x,0)=m_0(x), & \text{ in } \mathbb{T}^n.
		\end{cases}  		
	\end{equation}
	Following a similar method we used in the proof of Theorem $\ref{der g}$, we let 
	$$m_0(x;\varepsilon)=\sum_{l=1}^{N}\varepsilon_lf_l,$$
	where $f_l\in C_+^{2+\alpha}(\mathbb{T}^n)$ and $\varepsilon=(\varepsilon_1,\varepsilon_2,...,\varepsilon_N)\in\mathbb{R}_+^N$ with 
	$|\varepsilon|=\sum_{l=1}^{N}|\varepsilon_l|$ small enough. 
	
	Consider the case $N=1.$ Let
	$$u_{j}^{(1)}:=\p_{\varepsilon_1}u_{j}|_{\varepsilon=0},$$
	$$m_{j}^{(1)}:=\p_{\varepsilon_1}m_{j}|_{\varepsilon=0}.$$
	Then direct computations imply that $(u_{j}^{(1)},v_{j}^{(1)} )$ satisfies the following system:
	\begin{equation}\label{linear l=1for F g}
		\begin{cases}
			-\p_tu_j^{(1)}(x,t)-\Delta u^{(1)}_j(x,t)= F_j^{(1)}(x)m_j^{(1)}(x,t),& \text{ in }\mathbb{T}^n\times(0,T),\medskip \\
			\p_t m^{(1)}_j(x,t)-\Delta m^{(1)}_j(x,t)=0,&\text{ in }\mathbb{T}^n\times (0,T),\medskip \\
			u^{(1)}_j(x,T)=G_j^{(1)}(x)m^{(1)}_j(x,T), & \text{ in } \mathbb{T}^n,\medskip\\
			m^{(1)}_j(x,0)=f_1(x), & \text{ in } \mathbb{T}^n.
		\end{cases}  	
	\end{equation}
	Then we have $ m_1^{(1)}=m_2^{(1)}:=m^{(1)}(x,t)$	.
	Let $ \overline{u}=u_1^{(1)}-u_2^{(1)}$, \eqref{linear l=1for F g} implies that 
	\begin{equation}\label{u1-u2 }
		\begin{cases}
			&-\p_t\overline{u}-\Delta\overline{u}= (F_1^{(1)}-F_2^{(1)})m^{(1)}(x,t),\medskip\\
			&\overline{u}(x,T)=(G_1^{(1)}-G_2^{(1)})m^{(1)}(x,T).
		\end{cases}
	\end{equation}
	Now let $w$ be a solution to the heat equation $\p_t w(x,t)-\Delta w(x,t)=0$ in $\mathbb{T}^n$. Then
	\begin{equation}
		\begin{aligned}
			&\int_Q (F_1^{(1)}-F_2^{(1)})m^{(1)}(x,t)w\, dxdt\medskip\\
			=&\int_Q (-\p_t\overline{u}-\Delta\overline{u})w\, dxdt\medskip\\
			=&\int_{\mathbb{T}^n} (\overline{u}w)\big|_0^T\, dx +\int_Q \overline{u}\p_tw- \overline{u}\Delta w\medskip\\
			=&  \int_{\mathbb{T}^n} (\overline{u}w)\big|_0^T\, dx.
		\end{aligned}	
	\end{equation}
	
	Since $\mathcal{M}_{F_1,G_1}=\mathcal{M}_{F_2,G_2}$, we have $$\overline{u}(x,0)=0.$$ It follows that
	\begin{equation}\label{integral by part}
		\int_Q (F_1^{(1)}-F_2^{(1)})m^{(1)}(x,t)w(x,t)\, dxdt= \int_{\mathbb{T}^n} w(x,T)(G_1^{(1)}-G_2^{(1)})m^{(1)}(x,T)\, dx,
	\end{equation}
	for all solutions $w(x,t),m^{(1)}(x,t)$ of the heat equation in $\mathbb{T}^n$. 
	
	Here, we cannot apply Lemma $\ref{dense}$ directly. Nevertheless, we use the same construction.
	Let $\mathbf {\xi_1},\mathbf {\xi_2}\in\mathbb{Z}^n\backslash\{0\}$, $M\in\mathbb{N}^*$ and $\mathbf {\xi}=\mathbf {\xi_1}+\mathbf {\xi_2}$ .
	Let
	$$w(x,t)=\exp(- 2\pi\mathrm{i}\mathbf {\xi_1}\cdot x-4\pi^2|\mathbf {\xi_1}|^2t ),$$
	and
	
	$$m(x,t)=\exp(- 2\pi\mathrm{i}\mathbf {\xi_2}\cdot x-4\pi^2|\mathbf {\xi_2}|^2t )+M.$$
	Then the left hand side of \eqref{integral by part} is
	\begin{equation}\label{Fourier1}
		\begin{aligned}
			&\int_{\mathbb{T}^n} \frac{1-\exp(-4\pi^2T(|\mathbf {\xi_1}|^2+|\mathbf {\xi_2}|^2))}{4\pi^2( |\mathbf {\xi_1}|^2+|\mathbf {\xi_2}|^2)}(F_1^{(1)}-F_2^{(1)} )e^{-2\pi \mathrm{i}\mathbf {\xi}\cdot x }\, dx\\
			+&M \frac{1-\exp(-4\pi^2T|\mathbf {\xi_1}|^2)}{4\pi^2|\mathbf {\xi_1}|^2}   \int_{\mathbb{T}^n} (F_1^{(1)}-F_2^{(1)} )e^{-2\pi \mathrm{i}\mathbf {\xi_1}\cdot x }\,dx.
		\end{aligned}
	\end{equation}
	And the right hand side is 
	\begin{equation}\label{Fourier2}
		\begin{aligned}
			&\int_{\mathbb{T}^n}\exp(-4\pi^2T(|\mathbf {\xi_1}|^2+|\mathbf {\xi_2}|^2)) (G_1^{(1)}-G_2^{(1)})e^{-2\pi \mathrm{i}\mathbf {\xi}\cdot x }\, dx\\
			+&M\exp(-4\pi^2T|\mathbf {\xi_1}|^2)\int_{\mathbb{T}^n} (G_1^{(1)}-G_2^{(1)})e^{-2\pi \mathrm{i}\mathbf {\xi_1}\cdot x }\, dx.
		\end{aligned}	
	\end{equation}

	By taking $M=1$ and $M=2$, respectively, and then subtracting the resulting equations from one another, one can readily show that
	\begin{equation}
		\begin{aligned}
				&\frac{1-\exp(-4\pi^2T|\mathbf {\xi_1}|^2)}{4\pi^2|\mathbf {\xi_1}|^2}   \int_{\mathbb{T}^n} (F_1^{(1)}-F_2^{(1)} )e^{-2\pi \mathrm{i}\mathbf {\xi_1}\cdot x }\,dx\\
				=& \exp(-4\pi^2T|\mathbf {\xi_1}|^2)\int_{\mathbb{T}^n} (G_1^{(1)}-G_2^{(1)})e^{-2\pi \mathrm{i}\mathbf {\xi_1}\cdot x }\, dx. 
		\end{aligned}
	\end{equation}

	Then \eqref{integral by part}, \eqref{Fourier1} and \eqref{Fourier2} readily yields that
	$$\frac{1-\exp(-4\pi^2T(|\mathbf {\xi_1}|^2+|\mathbf {\xi_2}|^2))}{4\pi^2( |\mathbf {\xi_1}|^2+|\mathbf {\xi_2}|^2)}a_{\mathbf {\xi}  }+\exp(-4\pi^2T(|\mathbf {\xi_1}|^2+|\mathbf {\xi_2}|^2))b_{\mathbf {\xi}}=0.$$
	
	For a given $\mathbf {\xi}\in\mathbb{Z}^n $, there exist $\mathbf {\xi_1},\mathbf {\xi_2},\mathbf {\xi_1}',\mathbf {\xi_2}' \in\mathbb{Z}^n\backslash\{0\}$ such that $\mathbf {\xi}=\mathbf {\xi_1}+\mathbf {\xi_2}=\mathbf {\xi_1}'+\mathbf {\xi_2}'$ and $|\mathbf {\xi_1}|^2+|\mathbf {\xi_2}|^2\neq |\mathbf {\xi_1}'|^2+|\mathbf {\xi_2}'|^2 .$ Therefore, $a_{\mathbf {\xi}}=b_{\mathbf {\xi}}=0$ for all $\mathbf {\xi}\in\mathbb{Z}^n$. 
	It follows that $F_1^{(1)}-F_2^{(1)}=G_1^{(1)}-G_2^{(1)}=0$.

	Next, we consider the case $N=2.$ Let
	\begin{equation}\label{eq:ss1}
		u_{j}^{(1,2)}:=\p_{\varepsilon_1}\p_{\varepsilon_2}u_{j}|_{\varepsilon=0},\quad
		m_{j}^{(1,2)}:=\p_{\varepsilon_1}\p_{\varepsilon_2}m_{j}|_{\varepsilon=0},
	\end{equation}
	and
	\begin{equation}\label{eq:ss2}
		u_{j}^{(2)}:=\p_{\varepsilon_2}u_{j}|_{\varepsilon=0},\quad m_{j}^{(2)}:=\p_{\varepsilon_2}m_{j}|_{\varepsilon=0}.
	\end{equation}
	By the second-order linearization in \eqref{eq:ss1} and \eqref{eq:ss2},  we can obtain
	\begin{equation}
		\begin{cases}
			-\p_tu_j^{(1,2)}(x,t)-\Delta u^{(1,2)}_j(x,t)+\nabla u_{j}^{(1)}\cdot \nabla u_{j}^{(2)}\\
			\hspace*{3cm} = F_j^{(1)}m_j^{(1,2)}+F^{(2)}_j(x)m_j^{(1)}m_j^{(2)},& \text{ in }\mathbb{T}^n\times(0,T),\medskip\\
			\p_t m^{(1,2)}_j(x,t)-\Delta m^{(1,2)}_j(x,t)\\
			\hspace*{3cm}=  {\rm div} (m_{j}^{(1)}\nabla u_j^{(2)})+{\rm div}(m_j^{(2)}\nabla u_j^{(1)}) ,&\text{ in }\mathbb{T}^n\times (0,T),\medskip\\
			u^{(1,2)}_j(x,T)=G^{(1)}(x)m_j^{(1,2)}(x,T)+G^{(2)}(x)m_j^{(1)}m_j^{(2)}(x,T), & \text{ in } \mathbb{T}^n,\medskip\\
			m^{(1,2)}_j(x,0)=0, & \text{ in } \mathbb{T}^n.
		\end{cases}  	
	\end{equation}
	By following a similar argument in the case $N=1$ , we have 
	$$ u^{(1)}_1(x,t)= u^{(1)}_2(x,t),  u^{(2)}_1(x,t)=u^{(2)}_2(x,t),$$
	and
	$$ m^{(1)}_1(x,t)= m^{(1)}_2(x,t) ,  m^{(2)}_1(x,t)= m^{(2)}_2(x,t).$$
	
	Let $\overline{u}^2(x,t)=u_1^{(1,2)}(x,t)-u_2^{(1,2)}(x,t) $. We have
	\begin{equation}\label{u1-u2,2 }
		\begin{cases}
			&-\p_t\overline{u}^2-\Delta\overline{u}^2= (F_1^{(1)}-F_2^{(1)})m^{(1)}(x,t)m_1^{(2)}(x,t),\medskip\\
			&\overline{u}(x,T)=(G_1^{(1)}-G_2^{(1)})m^{(1)}(x,T)m_1^{(2)}(x,t).
		\end{cases}
	\end{equation}
	Let $w$ be a solution of the heat equation $\p_t w(x,t)-\Delta w(x,t)=0$ in $\mathbb{T}^n$. Then by following a similar argument in the case $N=1$, we can show that 
	\begin{equation}\label{integral by part2}
		\begin{split}
			&	\int_Q (F_1^{(2)}-F_2^{(2)})m^{(1)}m_1^{(2)}w(x,t)\, dxdt\\
			=& \int_{\mathbb{T}^n} w(x,T)(G_1^{(2)}-G_2^{(2)})m^{(1)}(x,T)m_1^{(2)}(x,T)\, dx.
		\end{split} 
	\end{equation}
	To proceed further, by using the construction in Lemma $\ref{dense}$ again, we have from \eqref{integral by part2} that
	\[
	F_1^{(2)}-F_2^{(2)}=G_1^{(2)}-G_2^{(2)}=0.
	\]
	
	Finally, via a mathematical induction, we can derive the same result for $N\geq 3$. That is, for any $k\in\mathbb{N},$ we have $$F^{(k)}_1(x)-F^{(k)}_2(x)=G^{(k)}_1(x)-G^{(k)}_2(x)=0.$$ Hence, 	
	$$(F_1(x,z),F_2(x,z))=(G_1(x,z),G_2(x,z)),\text{  in  } \mathbb{R}^n\times \mathbb{R}.$$ 
	
	The proof is complete. 
\end{proof}
 
\section{Inverse Problems in General Domain without Trivial Solution}
In the previous section, we show that one can recover $F$ and/or $G$ from the knowledge of boundary data. However, the proof depends on the fact that $(u,m)=(0,0)$ is a trivial solution of MFG system and we assume that the running cost $F$ depends on $m$ locally.  
In this chapter, we consider a general bounded domain $\Omega$ with smooth boundary and introduce a more general Hamiltonian $H(x,p)=\frac{1}{2}\kappa(x)p^2$ for some function $\kappa$.

Moreover, we present an  approach to ensure the probability measure constraint on $m$ while effectively tackle the MFG inverse problems, especially when the running cost depends only on the state-space variable $x$ and the measure $m$. We term this approach as high-order variation in combination with successive linearisation. In this part, we will linearize the distribution around the uniform distribution (Remember that $m$ is a density function). This provides a more accurate explanation of the probability measure constraint.

\subsection{The Inverse Problem in Bounded Domain}
Consider the following MFG system in bounded domain $\Omega$
\begin{equation}\label{main_domain}
	\left\{
	\begin{array}{ll}
		\displaystyle{-\partial_t u(x,t) -\Delta u(t,x)+\frac{1}{2}\kappa(x)|\nabla u(x,t)|^2= F(x,m)} & \quad {\rm{in}}\ Q,\\
		\displaystyle{\partial_tm(x,t)-\Delta m(x,t)-{\rm div} \big(m(x,t) \kappa(x)\nabla u(x,t)\big)=0} & \quad{\rm{in}}\ Q,\\
		\p_{\nu} u(x,t)=\p_{\nu} m(x,t)=0 &\quad {\rm{on}}\ \Sigma,\\
		u(x,T)=\psi(x),\ m(x,0)=m_0(x) &\quad {\rm{in}}\  \Omega,
	\end{array}
	\right.
\end{equation}
where $\Sigma=\partial\Omega$. 
Next, we introduce the measurement we use. Define 
\begin{equation}\label{eq:meop0}
	\mathcal{M}_{\kappa, F}(m_0,\psi):=(u(x,t)|_{\Sigma}, u(x,0)), 
\end{equation}
where
$(u, m)$ is the (unique) pair of solutions to the MFG system \eqref{main_domain} associated with the initial population distribution $m(x, 0)=m_0(x)$ and the total cost $\psi(x)$.  The following is a formulation of the inverse problem that we aim to investigate:
\begin{equation}\label{eq:ip1_domain}
	\mathcal{M}_{\kappa,F }(m_0,\psi)\rightarrow (F,\kappa),
\end{equation}
for many pairs of $m_0$ and $\psi$, which shall be more detailed in what follows. We consider two different types of running costs: $F$ belongs to an analytic class (locally dependent case), and $F$ is given in the form of an integration with an unidentified  kernel (non-locally dependent case). 

Now we introduce the admissible classes of the running cost of $F$ in two different cases. The first one is the same as the conditions in the previous chapter (Definition \ref{Admissible class1}). 

\begin{defn}\label{Admissible class 2_domain}
	We say $U(x,z):\mathbb{R}^n\times\mathbb{C}\to\mathbb{C}$ is admissible, denoted by $U\in\mathcal{B}$, if it satisfies the following conditions:
	\begin{enumerate}
		\item[(i)] The map $z\mapsto U(\cdot,z)$ is holomorphic with value in $C^{\alpha}(\mathbb{R}^n)$ for some $\alpha\in(0,1)$.
		\item[(ii)] $U(x,0)=0$ for all $x\in\mathbb{R}^n$. 
	\end{enumerate} 	
	Clearly, if (1) and (2) are fulfilled, then $U$ can be expanded into a power series as follows:
	\begin{equation}\label{eq:G_domain}
		U(x,z)=\sum_{k=1}^{\infty} U^{(k)}(x)\frac{(z)^k}{k!},
	\end{equation}
	where $ U^{(k)}(x)=\frac{\p^kU}{\p z^k}(x,0)\in C^{\alpha}(\mathbb{R}^n).$
\end{defn}
Clearly, if $F(x,m)\in\mathcal{B}$ and $m$ is the density of the measure, then $F$ depends on the measure locally. 

Next, we consider the non-local case.
\begin{defn}\label{admi 2_domain}
	Let $m(x,t)$ be the density of a given distribution. We say $$F(x,m)=\int_{\Omega} K(x,y)m(y,t)dy$$ belongs to $\mathcal{C}$ if 
	\begin{enumerate}
		\item[(i)] $K(x,y)$ is smooth in $\Omega\times\Omega$.
		\item[(ii)] $\int_{\Omega} K(x,y)dy=0 $ for all $x\in\Omega$. 
	\end{enumerate} 	
\end{defn}
\begin{rmk}
	In fact, the condition (ii) in the Definition $\ref{admi 2_domain}$ is quite natural. If $m_0(x)=1$ and $\psi(x)$ is a constant $c$, the solution of MFG game system should be $(u,m)=(c,0)$. This is because  mean field games are non-atomic differential games. In other words, if $m_0(x)$ is the uniform distribution and $\psi(x)$ is constant, this is already a equilibrium of this system, then $m(x,t)$ would keep to $1$. This is a common nature of MFG system that the uniform distribution is a stable state. 
\end{rmk}

\subsection{Main Unique Identifiability Results}
We are able to articulate the primary conclusions for the inverse problems, which show that
one can recover  the running cost and Hamiltonian from the measurement map $\mathcal{M}_{F,\kappa}$. 

\begin{thm}\label{der F}
	Assume that  $F_j(x,m)\in\mathcal{B} $. Let $\mathcal{M}_{F_j,\kappa_j}$ be the measurement map associated to
	the following system:
	\begin{equation}\label{eq:mfg1}
		\begin{cases}
			-\p_tu(x,t)-\Delta u(x,t)+\frac 1 2 \kappa_j{|\nabla u(x,t)|^2}= F_j(x,m),& \text{ in }  Q,\medskip\\
			\p_t m(x,t)-\Delta m(x,t)-{\rm div} (m(x,t) \kappa_j\nabla u(x,t))=0,&\text{ in } Q,\medskip\\
			\p_{\nu} u(x,t)=\p_{\nu}m(x,t)=0      &\text{ on } \Sigma,\medskip\\
			u(x,T)=\psi(x), & \text{ in } \Omega,\medskip\\
			m(x,0)=m_0(x), & \text{ in } \Omega.\\
		\end{cases}  		
	\end{equation}	
	If for any $(m_0,\psi)\in [ C^{2+\alpha}(\Omega) \cap \mathcal{O}_a]\times C^{2+\alpha}(\Omega) $, where $\mathcal{O}_a$ is defined in \eqref{probability measure constraint}, one has 
	$$\mathcal{M}_{F_1,\kappa_1}(m_0,\psi)=\mathcal{M}_{F_2,\kappa_2}(m_0,\psi),$$    then it holds that 
	$$\kappa_1=\kappa_2\ \text{  in  }\ \Omega,$$ and 
	$$F_1(x,z)=F_2(x,z)\ \text{  in  }\ \Omega\times \mathbb{R}.$$ 
\end{thm}
\begin{thm}\label{der F2}
	Assume that  $$F_j(x,m)=\int_{\Omega}K_j(x,y)m(y,t)dy\in\mathcal{C}. $$ Let $\mathcal{M}_{F_j,\kappa_j}$ be the measurement map associated to
	the following system:
	\begin{equation}
		\begin{cases}
			-\p_tu(x,t)-\Delta u(x,t)+\frac 1 2 \kappa_j{|\nabla u(x,t)|^2}= \int_{\Omega}K_j(x,y)m(y,t)dy,& \text{ in }  Q,\medskip\\
			\p_t m(x,t)-\Delta m(x,t)-{\rm div} (m(x,t) \kappa_j\nabla u(x,t))=0,&\text{ in } Q,\medskip\\
			\p_{\nu} u(x,t)=\p_{\nu}m(x,t)=0      &\text{ on } \Sigma,\medskip\\
			u(x,T)=\psi(x), & \text{ in } \Omega,\medskip\\
			m(x,0)=m_0(x), & \text{ in } \Omega.\\
		\end{cases}  		
	\end{equation}	
	If for any $(m_0,\psi)\in [ C^{2+\alpha}(\Omega) \cap \mathcal{O}_a]\times C^{2+\alpha}(\Omega) $, where $\mathcal{O}_a$ is defined in \eqref{probability measure constraint}, one has 
	$$\mathcal{M}_{F_1,\kappa_1}(m_0,\psi)=\mathcal{M}_{F_2,\kappa_2}(m_0,\psi),$$    then it holds that 
	$$\kappa_1=\kappa_2\ \text{  in  }\ \Omega,$$ and 
	$$K_1(x,y)=K_2(x,y)\ \text{  in  }\ \Omega\times \Omega.$$ 
\end{thm}
\subsection{Proof of Theorem ~\ref{der F},~\ref{der F2} }

We start from a useful Lemma. It is the key part of the main proofs. 
\begin{lem}\label{E-F is complete}
	Consider
	\begin{equation}\label{eigenfunction}
		\begin{cases}
			-\p_t u-\Delta u=0  &\text{ in } Q\\
			\p_{\nu} u(x,t)=0   &\text{ in } \Sigma\\
		\end{cases}
	\end{equation}
	There exist a sequence of solution $u(x,t)$ of system \eqref{eigenfunction} such that
	
	(1) $u(x,t)=e^{\lambda t}g(x;\lambda)$ for some $\lambda\in\mathbb{R}$ and $g(x;\lambda)\in C^2(\Omega)$.

	(2) There does not exist an open subset $U$ of $\Omega$ such that $\nabla g(x;\lambda)= 0$ in $U$. 
	
\end{lem}
\begin{proof}
	Let $\lambda$  be an  eigenvalue of Neumann-Laplacian operator and $g(x;\lambda)$ be a corresponding eigenfunction
	\begin{equation}
		\begin{cases}
			-\Delta g(x;\lambda)=\lambda g(x;\lambda) &\text{ in } \Omega\\
			\p_{\nu} g(x;\lambda)=0     &\text{ in } \Sigma.
		\end{cases}
	\end{equation}
	Then it is obviously that $u(x,t)=e^{\lambda t}g(x;\lambda) $ is a solution of \eqref{eigenfunction}. This implies that $\lambda=0.$ It is a contradiction.
	
	Furthermore, suppose there is an open subset $U$ of $\Omega$ such that $\nabla g=0$ in $U$, then $g$ in a constant in $U$. This implies that $\lambda=0.$ It is a contradiction.
	
	This completes the proof.
\end{proof}

Before we start the proof of Theorem \ref{der F2}, we need to show the local well-posedness of MFG system in bounded domain too. Since the proof is quite similar to the proof of Theorem $\ref{local_wellpose}$, we just list the results here. The key idea is that the system is infinity many differentiable with respect to small input. 

\begin{thm}\label{local_wellpose1}
	Given $\psi(x)=0$. Suppose that $F\in\mathcal{B}$ . The following results holds:
	\begin{enumerate}		
		\item[(a)]
		There exist constants $\delta>0$ and $C>0$ such that for any 
		\[
		m_0\in B_{\delta}(C^{2+\alpha}(\Omega)) :=\{m_0\in C^{2+\alpha}(\Omega): \|m_0\|_{C^{2+\alpha}(\Omega)}\leq\delta \},
		\]
		the MFG system \eqref{main_domain} has a solution $(u,m)\in
		[C^{2+\alpha,1+\frac{\alpha}{2}}(Q)]^2$ which satisfies
		\begin{equation}
			\|(u,m)\|_{ C^{2+\alpha,1+\frac{\alpha}{2}}(Q)}:= \|u\|_{C^{2+\alpha,1+\frac{\alpha}{2}}(Q)}+ \|m\|_{C^{2+\alpha,1+\frac{\alpha}{2}}(Q)}\leq C\|m_0\|_{ C^{2+\alpha}(\Omega)}.
		\end{equation}
		Furthermore, the solution $(u,m)$ is unique within the class
		\begin{equation}
			\{ (u,m)\in  C^{2+\alpha,1+\frac{\alpha}{2}}(Q)\times C^{2+\alpha,1+\frac{\alpha}{2}}(Q): \|(u,m)\|_{ C^{2+\alpha,1+\frac{\alpha}{2}}(Q)}\leq C\delta \}.
		\end{equation}		
		\item[(b)] Define a function 
		\[
		S: B_{\delta}(C^{2+\alpha}(\Omega)\to C^{2+\alpha,1+\frac{\alpha}{2}}(Q)\times C^{2+\alpha,1+\frac{\alpha}{2}}(Q)\ \mbox{by $S(m_0):=(u,m)$}. 
		\] 
		where $(u,m)$ is the unique solution to the MFG system \eqref{main_domain}.
		Then for any $m_0\in B_{\delta}(C^{2+\alpha}(\Omega))$, $S$ is holomorphic at $m_0$.
	\end{enumerate}
\end{thm}

\begin{thm}\label{local_wellpose1'}
	Given $m_0(x)=0$. Suppose that $F\in\mathcal{B}$ . The following results holds:
	\begin{enumerate}		
		\item[(a)]
		There exist constants $\delta>0$ and $C>0$ such that for any 
		\[
		\psi(x)\in B_{\delta}(C^{2+\alpha}(\Omega)) :=\{\psi\in C^{2+\alpha}(\Omega): \|\psi(x)\|_{C^{2+\alpha}(\Omega)}\leq\delta \},
		\]
		the MFG system \eqref{main_domain} has a solution $(u,m)\in
		[C^{2+\alpha,1+\frac{\alpha}{2}}(Q)]^2$ which satisfies
		\begin{equation}
			\|(u,m)\|_{ C^{2+\alpha,1+\frac{\alpha}{2}}(Q)}:= \|u\|_{C^{2+\alpha,1+\frac{\alpha}{2}}(Q)}+ \|m\|_{C^{2+\alpha,1+\frac{\alpha}{2}}(Q)}\leq C\|\psi(x)\|_{ C^{2+\alpha}(\Omega)}.
		\end{equation}
		Furthermore, the solution $(u,m)$ is unique within the class
		\begin{equation}
			\{ (u,m)\in  C^{2+\alpha,1+\frac{\alpha}{2}}(Q)\times C^{2+\alpha,1+\frac{\alpha}{2}}(Q): \|(u,m)\|_{ C^{2+\alpha,1+\frac{\alpha}{2}}(Q)}\leq C\delta \}.
		\end{equation}		
		\item[(b)] Define a function 
		\[
		S: B_{\delta}(C^{2+\alpha}(\Omega)\to C^{2+\alpha,1+\frac{\alpha}{2}}(Q)\times C^{2+\alpha,1+\frac{\alpha}{2}}(Q)\ \mbox{by $S(\psi(x)):=(u,m)$}. 
		\] 
		where $(u,m)$ is the unique solution to the MFG system \eqref{main_domain}.
		Then for any $\psi(x)\in B_{\delta}(C^{2+\alpha}(\Omega))$, $S$ is holomorphic at $\psi$.
	\end{enumerate}
\end{thm}

We still need to show the local well-posedness in the case that the running cost $F$ is in the form of an integral. This is stated in the following theorems.

\begin{thm}\label{local_wellpose2}
	Given $\psi(x)=0$. Suppose that $F\in\mathcal{C}$ . The following results hold:
	\begin{enumerate}		
		\item[(a)]
		There exist constants $\delta>0$ and $C>0$ such that for any 
		\[
		m_0\in B_{\delta}(C^{2+\alpha}(\Omega)) :=\{m_0\in C^{2+\alpha}(\Omega): \|m_0\|_{C^{2+\alpha}(\Omega)}\leq\delta \},
		\]
		the MFG system \eqref{main_domain} has a solution $(u,m)\in
		[C^{2+\alpha,1+\frac{\alpha}{2}}(Q)]^2$ which satisfies
		\begin{equation}
			\|(u,m)\|_{ C^{2+\alpha,1+\frac{\alpha}{2}}(Q)}:= \|u\|_{C^{2+\alpha,1+\frac{\alpha}{2}}(Q)}+ \|m\|_{C^{2+\alpha,1+\frac{\alpha}{2}}(Q)}\leq C\|m_0\|_{ C^{2+\alpha}(\Omega)}.
		\end{equation}
		Furthermore, the solution $(u,m)$ is unique within the class
		\begin{equation}
			\{ (u,m)\in  C^{2+\alpha,1+\frac{\alpha}{2}}(Q)\times C^{2+\alpha,1+\frac{\alpha}{2}}(Q): \|(u,m)\|_{ C^{2+\alpha,1+\frac{\alpha}{2}}(Q)}\leq C\delta \}.
		\end{equation}		
		\item[(b)] Define a mapping 
		\[
		S: B_{\delta}(C^{2+\alpha}(\Omega))\to C^{2+\alpha,1+\frac{\alpha}{2}}(Q)\times C^{2+\alpha,1+\frac{\alpha}{2}}(Q)\ \mbox{by $S(m_0):=(u,m)$},
		\] 
		where $(u,m)$ is the unique solution to the MFG system \eqref{main_domain}.
		Then for any $m_0\in B_{\delta}(C^{2+\alpha}(\Omega))$, $S$ is holomorphic at $m_0$.
	\end{enumerate}
\end{thm}
Similarly, we have
\begin{thm}\label{local_wellpose2'}
	Given $m_0(x)=0$. Suppose that $F\in\mathcal{C}$ . The following results hold:
	\begin{enumerate}		
		\item[(a)]
		There exist constants $\delta>0$ and $C>0$ such that for any 
		\[
		\psi(x)\in B_{\delta}(C^{2+\alpha}(\Omega)) :=\{\psi\in C^{2+\alpha}(\Omega): \|\psi(x)\|_{C^{2+\alpha}(\Omega)}\leq\delta \},
		\]
		the MFG system \eqref{main_domain} has a solution $(u,m)\in
		[C^{2+\alpha,1+\frac{\alpha}{2}}(Q)]^2$ which satisfies
		\begin{equation}
			\|(u,m)\|_{ C^{2+\alpha,1+\frac{\alpha}{2}}(Q)}:= \|u\|_{C^{2+\alpha,1+\frac{\alpha}{2}}(Q)}+ \|m\|_{C^{2+\alpha,1+\frac{\alpha}{2}}(Q)}\leq C\|\psi(x)\|_{ C^{2+\alpha}(\Omega)}.
		\end{equation}
		Furthermore, the solution $(u,m)$ is unique within the class
		\begin{equation}
			\{ (u,m)\in  C^{2+\alpha,1+\frac{\alpha}{2}}(Q)\times C^{2+\alpha,1+\frac{\alpha}{2}}(Q): \|(u,m)\|_{ C^{2+\alpha,1+\frac{\alpha}{2}}(Q)}\leq C\delta \}.
		\end{equation}		
		\item[(b)] Define a mapping
		\[
		S: B_{\delta}(C^{2+\alpha}(\Omega))\to C^{2+\alpha,1+\frac{\alpha}{2}}(Q)\times C^{2+\alpha,1+\frac{\alpha}{2}}(Q)\ \mbox{by $S(\psi(x)):=(u,m)$}. 
		\] 
		where $(u,m)$ is the unique solution to the MFG system \eqref{main_domain}.
		Then for any $\psi(x)\in B_{\delta}(C^{2+\alpha}(\Omega))$, $S$ is holomorphic at $\psi$.
	\end{enumerate}
\end{thm}

With all the preparations, we are in a position to present the proof of Theorem~\ref{der F} and Theorem~\ref{der F2}  . 
\begin{proof}[ Proof of Theorem $\ref{der F}$ ]
	For $j=1,2$, let us consider 
	\begin{equation}\label{MFG 1,2'}
		\begin{cases}
			-u_t-\Delta u+\frac{1}{2}\kappa_j|\nabla  u|^2= F_j(x,m) & \text{ in } Q,\medskip\\
			m_t-\Delta m-{\rm div} (m \kappa_j\nabla u)=0  & \text{ in } Q, \medskip\\
			\p_{\nu}u(x,t)=\p_{\nu}m(x,t)=0     & \text{ on } \Sigma, \medskip\\
			u(x,T)=\psi(x)     & \text{ in } \Omega,\medskip\\
			m(x,0)=m_0(x) & \text{ in } \Omega.\\
		\end{cases}
	\end{equation}
	Next, we divide our proof into several steps. 
	
	\bigskip
	\noindent {\bf Step I.} We show $\kappa_1=\kappa_2$ first.

	Let $\psi(x)=\varepsilon_1f_1+\varepsilon_2f_2$, $m_0= 0.$ 
	Let	$$u^{(1)}:=\p_{\varepsilon_1}u|_{\varepsilon=0}=\lim\limits_{\varepsilon\to 0}\frac{u(x,t;\varepsilon)-u(x,t;0) }{\varepsilon_1},$$
	$$m^{(1)}:=\p_{\varepsilon_1}m|_{\varepsilon=0}=\lim\limits_{\varepsilon\to 0}\frac{m(x,t;\varepsilon)-m(x,t;0) }{\varepsilon_1}.$$
	Then we have 
	\begin{equation}\label{first-order-analytic}
		\begin{cases}
			-\p_t u_j^{(1)}-\Delta u_j^{(1)}= F_j^{(1)}(x)m^{(1)}(x,t) ,  & \text{ in }  Q,\\
			\p_t m_j^{(1)}(x,t)-\Delta m_j^{(1)}(x,t)=0, & \text{ in }  Q,\\
			\p_{\nu} u_j^{(1)}(x,t)=\p_{\nu}m_j^{(1)}(x,t)=0      &\text{ on } \Sigma,\medskip\\
			u_j^{(1)}(x,T)=f_1, & \text{ in } \Omega,\medskip\\
			m_j^{(1)}(x,0)=0, & \text{ in } \Omega.\\
		\end{cases}
	\end{equation}
	This implies that $m^{(1)}(x,t)=0.$ We define $m^{(2)}(x,t)$ in the same way (see section \ref{HLM}). Similarly, we have $m^{(2)}(x,t)=0.$ 			Therefore, $u_j^{(1)}(x,t)$ are independent of $j$. Let $u_1^{(1)}(x,t)=u_2^{(1)}(x,t)=u^{(1)}(x,t) $, then it satisfies the following system 
	\begin{equation}\label{first order-analytic'}
		\begin{cases}
			-\p_t u^{(1)}(x,t)-\Delta u^{(1)}(x,t)= 0 ,  & \text{ in }  Q,\\
			\p_{\nu} u^{(1)}(x,t)=0      &\text{ on } \Sigma,\medskip\\
			u^{(1)}(x,T)=f_1, & \text{ in } \Omega,\medskip\\
		\end{cases}
	\end{equation}
	Similarly, $u_1^{(2)}(x,t)=u_2^{(2)}(x,t)=u^{(2)}(x,t) $ also satisfy \eqref{first order-analytic'}.
	Notice that $m^{(1)}(x,t)= m^{(2)}(x,t)=0$,
	we have the second order linearization system is given by
	\begin{equation}\label{second order-analytic}
		\begin{cases}
			-\p_t u_j^{(1,2)}-\Delta u_j^{(1,2)}+\kappa_j(x)\nabla u^{(1)}\cdot\nabla u^{(2)}= F_j^{(1)}m^{(1,2)}(x,t) ,  & \text{ in }  Q,\\
			\p_t m_j^{(1,2)}(x,t)-\Delta m_j^{(1,2)}(x,t)=0, & \text{ in }  Q,\\
			\p_{\nu} u_j^{(1,2)}(x,t)=\p_{\nu}m_j^{(1,2)}(x,t)=0      &\text{ on } \Sigma,\medskip\\
			u^{(1,2)}(x,T)=0, & \text{ in } \Omega,\medskip\\
			m^{(1,2)}(x,0)=0, & \text{ in } \Omega.\\
		\end{cases}
	\end{equation}
	
	Note that $m^{(1,2)}(x,t)$ must be $0$ and hence, 
	\begin{equation}
		-\p_t u^{(1,2)}-\Delta u^{(1,2)}+\kappa_j(x)\nabla u^{(1)}\cdot\nabla u^{(2)}=0,
	\end{equation}
	holds if $u^{(1)}, u^{(2)}$ are solution of \eqref{first order}. Let $\overline{u}(x,t)=u_1^{(1,2)}(x,t)-u_2^{(1,2)}(x,t)$. Since 	$\mathcal{M}_{\kappa_1,F_1}=\mathcal{M}_{\kappa_2,F_2}$,  we have
	\begin{equation}\label{u_1-u_2'}
		\begin{cases}
			-\p_t \overline{u}-\Delta \overline{u}+(\kappa_1(x)-\kappa_2(x))\nabla u^{(1)}\cdot\nabla u^{(2)}=0, & \text{ in }  Q,\\
			\p_{\nu}\overline{u}(x,t)=\overline{u}(x,t)=0&\text{ on } \Sigma,\medskip\\
			\overline{u}(x,T)=\overline{u}(x,0)=0, & \text{ in } \Omega.	
		\end{cases}
	\end{equation}
	Let $\omega$ be a solution of the following system
	\begin{equation}\label{adjoint'}
		\p_t \omega-\Delta \omega=0   \text{ in }  Q,\\	
	\end{equation}
	then we multiply $\omega$ on the both side of \eqref{u_1-u_2'} and then integration by part implies that
	\begin{equation}\label{integral by part 1'}
		\int_{\Omega} (\kappa_1(x)-\kappa_2(x))\nabla u^{(1)}\cdot\nabla u^{(2)}\omega dxdt =0. 
	\end{equation}
	
	By Lemma $\ref{E-F is complete}$, there exists $\lambda\in\mathbb{R}$ and $g(x)\in C^{\infty}(\Omega)$ such that 
	$e^{\lambda t}g(x)$ satisfies \eqref{first order-analytic'}.
	Let $f_1=e^{\lambda T}g(x)$, then by the uniqueness of the solution of heat equation, we have 
	$$  u^{{(1)}}(x,t)=e^{\lambda t}g(x). $$
	Then we have 
	\begin{equation}\label{integral by part 1'' }
		\int_{Q} (\kappa_1(x)-\kappa_2(x))e^{2\lambda t}|\nabla g(x)|^2\omega dxdt =0. 
	\end{equation}
	Consider  $\omega= e^{-|\xi|^2t-ix\cdot\xi}$  for some $\xi\in\mathbb{R}^n.$
	It follows that 
	$$\int_{0}^{T} e^{2\lambda t}e^{-|\xi|^2t }\int_{\Omega}(\kappa_1(x)-\kappa_2(x))|\nabla g(x)|^2 e^{-ix\cdot\xi } =0.$$
	i.e.
	$$\int_{\Omega}(\kappa_1(x)-\kappa_2(x))|\nabla g(x)|^2 e^{-ix\cdot\xi } =0.$$
	Therefore, we have $ (\kappa_1(x)-\kappa_2(x))|\nabla g(x)|^2=0$ in $\Omega.$ By the construction in Lemma $\ref{E-F is complete}$, we have
	$$ \kappa_1(x)-\kappa_2(x)=0. $$
	
	\noindent {\bf Step II.}	Let $\kappa=\kappa_1=\kappa_2$. Next, we aim to show $F_1=F_2$. 
	Consider the following systems 	
	\begin{equation}\label{MFG 1,2-analytic}
		\begin{cases}
			-u_t-\Delta u+\frac{1}{2}\kappa(x)|\nabla  u|^2= F_j(x,m) & \text{ in } Q,\medskip\\
			m_t-\Delta m-{\rm div} (m\kappa(x)\nabla u)=0  & \text{ in } Q, \medskip\\
			\p_{\nu}u(x,t)=\p_{\nu}m(x,t)=0     & \text{ on } \Sigma, \medskip\\
			u(x,T)=\psi(x)     & \text{ in } \Omega,\medskip\\
			m(x,0)=m_0(x) & \text{ in } \Omega.\\
		\end{cases}
	\end{equation} 
	Let $\psi(x)=0$ and 
	$$m_0(x;\varepsilon)=\sum_{l=1}^{N}\varepsilon_lf_l,$$
	where 
	\[
	f_l\in C^{2+\alpha}(\mathbb{R}^n)\quad\mbox{and}\quad f_l\geq 0,
	\]
	and $\varepsilon=(\varepsilon_1,\varepsilon_2,...,\varepsilon_N)\in\mathbb{R}_+^N$ with 
	$|\varepsilon|=\sum_{l=1}^{N}|\varepsilon_l|$ small enough. 
	First, we do the first order linearization to the MFG system \eqref{MFG 1,2-analytic} in $Q$ and can derive: 
	\begin{equation}
		\begin{cases}
			-\p_{t}u^{(1)}_j-\Delta u_j^{(1)}= F_j^{(1)}(x)m_j^{(1)} & \text{ in } Q, \medskip\\
			\p_{t}m^{(1)} _j-\Delta m_j^{(1)} =0  & \text{ in } Q, \medskip\\
			\p_{\nu}u^{(1)}(x,t)=\p_{\nu}m^{(1)}(x,t)=0     & \text{ on } \Sigma, \medskip\\
			u^{(1)}(x,T)=0   & \text{ in } \Omega,\medskip\\
			m^{(1)} _j(x,0)=f_1(x) & \text{ in } \Omega.\\
		\end{cases}
	\end{equation}
	We just choose $f_1(x)=1$, then we have $m_1^{(1)}(x,t)=m_2^{(1)}(x,t)=1.$ Hence, we have $u_j^{(1)}(x,t)$ is the solution of
	\begin{equation}\label{linearization-without-m}
		\begin{cases}
			-\p_{t}u^{(1)}_j-\Delta u_j^{(1)}= F_j^{(1)}(x) & \text{ in } Q, \medskip\\
			\p_{\nu}u^{(1)}(x,t)=0     & \text{ on } \Sigma, \medskip\\
			u^{(1)}(x,T)=0   & \text{ in } \Omega,\medskip\\
		\end{cases}
	\end{equation}
	Let $\overline{u}(x,t)=u_1^{(1)}(x,t)-u_2^{(1)}(x,t)$. Since  $\mathcal{M}_{\kappa_1,F_1}=\mathcal{M}_{\kappa_2,F_2} $, we have
	\begin{equation}\label{u_1-u_2-analytic}
		\begin{cases}
			-\p_t \overline{u}(x,t)-\Delta \overline{u}(x,t) = F_1^{(1)}(x)-F_2^{(1)}(x), & \text{ in }  Q,\\
			\p_{\nu}\overline{u}(x,t)=\overline{u}(x,t)=0&\text{ on } \Sigma,\medskip\\
			\overline{u}(x,T)=\overline{u}(x,0)=0, & \text{ in } \Omega.	
		\end{cases}
	\end{equation}
	Let $\omega$ be a solution of the following system
	\begin{equation}\label{adjoint-anlaytic}
		\p_t \omega-\Delta \omega=0   \text{ in }  Q,\\	
	\end{equation}
	then we multiply $\omega$ on the both side of \eqref{u_1-u_2-analytic} and then integration by part implies that
	\begin{equation}
		\int_{Q}  \,(F_1^{(1)}(x)-F_2^{(1)}(x))\omega(x,t)\,dxdt=0.
	\end{equation}
	Consider  $\omega= e^{-|\xi|^2t-ix\cdot\xi}$  for some $\xi\in\mathbb{R}^n.$
	Similarly to the proof of part (I), we have
	\begin{equation}
		\int_{\Omega}  \,(F_1^{(1)}(x)-F_2^{(1)}(x))e^{i\xi\cdot x}\,dx=0,
	\end{equation}
	for all $\xi\in\mathbb{R}^n.$ Hence, we have $F_1^{(1)}(x)=F_2^{(1)}(x)$.
	
	\noindent {\bf Step III.}
	We proceed to consider the second linearization to the MFG system \eqref{MFG 1,2-analytic} in $Q$ and can obtain for $j=1,2$:
	\begin{equation}
		\begin{cases}
			-\p_tu_j^{(1,2)}-\Delta u_j^{(1,2)}(x,t)+\kappa(x)\nabla u_j^{(1)}\cdot \nabla u_j^{(2)}\medskip\\
			\hspace*{3cm}= F_j^{(1)}(x)m_j^{(1,2)}+F_j^{(2)}(x)m_j^{(1)}m_j^{(2)} & \text{ in } \Omega\times(0,T),\medskip\\
			\p_t m_j^{(1,2)}-\Delta m_j^{(1,2)}= {\rm div} (m_j^{(1)}\kappa(x)\nabla u_j^{(2)})+{\rm div}(m_j^{(2)}\kappa(x)\nabla u_j^{(1)}) ,&\text{ in } \Omega\times (0,T) \medskip\\
			\p_{\nu}u^{(1,2)}(x,t)=\p_{\nu}m^{(1,2)}(x,t)=0     & \text{ on } \Sigma, \medskip\\
			u_j^{(1,2)}(x,T)=0 & \text{ in } \Omega,\medskip\\
			m_j^{(1,2)}(x,0)=0 & \text{ in } \Omega.\\
		\end{cases}  	
	\end{equation}
	Now we may choose $f_1(x)=f_2(x)=1$, then we have $$m_j^{(1)}(x,t)=m_j^{(2)}(x,t)=1.$$
	Notice that we have shown that $F_1^{(1)}(x)=F_2^{(1)}(x)$ in $\Omega$ and this implies that
	$$ m^{(1,2)}(x,t)=m_1^{(1,2)}(x,t)=m_2^{(1,2)}(x,t). $$
	
	Let $\hat{u}(x,t)=u_1^{(1,2)}(x,t)-u_2^{(1,2)}(x,t)$. Since  $\mathcal{M}_{\kappa_1,F_1}=\mathcal{M}_{\kappa_2,F_2} $, we have
	\begin{equation}\label{u_1-u_2-analytic2}
		\begin{cases}
			-\p_t \hat{u}(x,t)-\Delta \hat{u}(x,t) =( F_1^{(2)}(x)-F_2^{(2)}(x)) , & \text{ in }  Q,\\
			\p_{\nu}\hat{u}(x,t)=\hat{u}(x,t)=0&\text{ on } \Sigma,\medskip\\
			\hat{u}(x,T)=\hat{u}(x,0)=0, & \text{ in } \Omega.	
		\end{cases}
	\end{equation}
	Let $\omega$ be a solution of the following system
	\begin{equation}\label{adjoint-anlaytic2}
		\p_t \omega-\Delta \omega=0   \text{ in }  Q,\\	
	\end{equation}
	then we multiply $\omega$ on the both side of \eqref{u_1-u_2-analytic2} and then integration by part implies that
	\begin{equation}
		\int_{Q}  \,(F_1^{(2)}(x)-F_2^{(2)}(x))\omega(x,t)\,dxdt=0.
	\end{equation}
	Similarly to the proof of part (II), we have
	
	$$F_1^{(2)}(x)-F_2^{(2)}(x).$$
	\noindent {\bf Step IV.}
	Finally, using mathematical induction and reiterating similar arguments as in Steps II and III, one can show that
	$$F^{(k)}_1(x)-F^{(k)}_2(x)=0 ,$$
	for all $k\in\mathbb{N}$. Hence, $$F_1(x,z)=F_2(x,z)$$
	in $\Omega\times\mathbb{R}.$
	The proof is now complete.
\end{proof}

Next we show the Theorem $\ref{der F2}$. Recall that in this case, we have 
$$ F(x,m)= \int_{\Omega}K(x,y)m(y,t)dy.$$
By comparing the proof of Theorem $\ref{der F}$ with the proof of Theorem $\ref{der F2}$, we claim that high-frequency method is much better in the non-local case.
\begin{proof}[ Proof of Theorem $\ref{der F2}$ ]
	For $j=1,2$, let us consider 
	\begin{equation}
		\begin{cases}
			-u_t-\Delta u+\frac{1}{2}\kappa_j|\nabla  u|^2= \int_{\Omega}K_j(x,y)m(y,t)dy & \text{ in } Q,\medskip\\
			m_t-\Delta m-{\rm div} (m\kappa_j\nabla u)=0  & \text{ in } Q, \medskip\\
			\p_{\nu}u(x,t)=\p_{\nu}m(x,t)=0     & \text{ on } \Sigma, \medskip\\
			u(x,T)=\psi(x)     & \text{ in } \Omega,\medskip\\
			m(x,0)=m_0(x) & \text{ in } \Omega.\\
		\end{cases}
	\end{equation}
	Next, we divide our proof into several steps. 
	
	\bigskip
	\noindent {\bf Step I.} We show $\kappa_1=\kappa_2$ first.

	Let $\psi(x)=\varepsilon_1f_1+\varepsilon_2f_2$, $m_0= 0.$ 
	Let
	$$u^{(1)}:=\p_{\varepsilon_1}u|_{\varepsilon=0}=\lim\limits_{\varepsilon\to 0}\frac{u(x,t;\varepsilon)-u(x,t;0) }{\varepsilon_1},$$
	$$m^{(1)}:=\p_{\varepsilon_1}m|_{\varepsilon=0}=\lim\limits_{\varepsilon\to 0}\frac{m(x,t;\varepsilon)-m(x,t;0) }{\varepsilon_1}.$$
	Then we have 
	\begin{equation}\label{first order}
		\begin{cases}
			-\p_t u_j^{(1)}-\Delta u_j^{(1)}= \int_{\Omega}K_j(x,y)m^{(1)}(y,t)dy ,  & \text{ in }  Q,\\
			\p_t m_j^{(1)}(x,t)-\Delta m_j^{(1)}(x,t)=0, & \text{ in }  Q,\\
			\p_{\nu} u_j^{(1)}(x,t)=\p_{\nu}m_j^{(1)}(x,t)=0      &\text{ on } \Sigma,\medskip\\
			u_j^{(1)}(x,T)=f_1, & \text{ in } \Omega,\medskip\\
			m_j^{(1)}(x,0)=0, & \text{ in } \Omega.\\
		\end{cases}
	\end{equation}
	This implies that $m^{(1)}(x,t)=0.$ We define $m^{(2)}(x,t)$ in the same way (see section \ref{HLM}). Similarly, we have $m^{(2)}(x,t)=0.$ 			Therefore, $u_j^{(1)}(x,t)$ are independent of $j$. Let $u_1^{(1)}(x,t)=u_2^{(1)}(x,t)=u^{(1)}(x,t) $, then it satisfies the following system
	\begin{equation}\label{first order'}
		\begin{cases}
			-\p_t u^{(1)}(x,t)-\Delta u^{(1)}(x,t)= 0 ,  & \text{ in }  Q,\\
			\p_{\nu} u^{(1)}(x,t)=0      &\text{ on } \Sigma,\medskip\\
			u^{(1)}(x,T)=f_1, & \text{ in } \Omega.\medskip\\
		\end{cases}
	\end{equation}
	Then the second order linearization system is given by
\begin{equation}\label{second order}
		\begin{cases}
			-\p_t u_j^{(1,2)}-\Delta u_j^{(1,2)}+\kappa_j(x)\nabla u^{(1)}\cdot\nabla u^{(2)}= \int_{\Omega}K_j(x,y)m^{(1,2)}(y,t)dy ,  & \text{ in }  Q,\\
			\p_t m_j^{(1,2)}(x,t)-\Delta m_j^{(1,2)}(x,t)=0, & \text{ in }  Q,\\
			\p_{\nu} u_j^{(1,2)}(x,t)=\p_{\nu}m_j^{(1,2)}(x,t)=0      &\text{ on } \Sigma,\medskip\\
			u^{(1,2)}(x,T)=0, & \text{ in } \Omega,\medskip\\
			m^{(1,2)}(x,0)=0, & \text{ in } \Omega.\\
		\end{cases}
	\end{equation}
	
	Note that $m^{(1,2)}(x,t)$ must be $0$ and hence, 
	\begin{equation}
		-\p_t u^{(1,2)}-\Delta u^{(1,2)}+\kappa_j(x)\nabla u^{(1)}\cdot\nabla u^{(2)}=0,
	\end{equation}
	holds if $u^{(1)}, u^{(2)}$ are solution of \eqref{first order}. Let $\overline{u}=u_1^{(1,2)}(x,t)-u_2^{(1,2)}(x,t)$. Since 	$\mathcal{M}_{\kappa_1,F_1}=\mathcal{M}_{\kappa_2,F_2}$,  we have
	\begin{equation}\label{u_1-u_2}
		\begin{cases}
			-\p_t \overline{u}-\Delta \overline{u}+(\kappa_1(x)-\kappa_2(x))\nabla u^{(1)}\cdot\nabla u^{(2)}=0, & \text{ in }  Q,\\
			\p_{\nu}\overline{u}(x,t)=\overline{u}(x,t)=0&\text{ on } \Sigma,\medskip\\
			\overline{u}(x,T)=0, & \text{ in } \Omega.
		\end{cases}
	\end{equation}
	Let $\omega$ be a solution of the following system
	\begin{equation}\label{adjoint}
		\p_t \omega-\Delta \omega=0,  \text{ in }  Q,\\
	\end{equation}
	then we multiply $\omega$ on the both side of \eqref{u_1-u_2} and then integration by part implies that
	\begin{equation}\label{integral by part 1}
		\int_{\Omega} (\kappa_1(x)-\kappa_2(x))\nabla u^{(1)}\cdot\nabla u^{(2)}\omega dxdt =0. 
	\end{equation}
	
	By Lemma $\ref{E-F is complete}$, there exists $\lambda\in\mathbb{R}$ and $g(x)\in C^{\infty}(\Omega)$ such that 
	$e^{\lambda t}g(x)$ satisfies \eqref{first order'}.
	Let $f_1=e^{\lambda T}g(x)$, then by the uniqueness of the solution of heat equation, we have 
	$$  u^{{(1)}}(x,t)=e^{\lambda t}g(x). $$

	Consider  $\omega= e^{-|\xi|^2t-ix\cdot\xi}$  for some $\xi\in\mathbb{R}^n.$
	By the same argument in the proof of Theorem $\ref{der F}$, we have
	$$\int_{0}^{T} e^{2\lambda t}e^{-|\xi|^2t }\int_{\Omega}(\kappa_1(x)-\kappa_2(x))|\nabla g(x)|^2 e^{-ix\cdot\xi } =0,$$
	i.e.
	$$\int_{\Omega}(\kappa_1(x)-\kappa_2(x))|\nabla g(x)|^2 e^{-ix\cdot\xi } =0.$$
	Therefore, we have $ (\kappa_1(x)-\kappa_2(x))|\nabla g(x)|^2=0$ in $\Omega.$ By the construction in Lemma $\ref{E-F is complete}$, we have
	$$ \kappa_1(x)-\kappa_2(x)=0. $$
	
	\noindent {\bf Step II.}	Let $\kappa=\kappa_1=\kappa_2$. Next, we aim to show $K_1(x,y)=K_2(x,y)$. 
	
	Consider another type of linearization method, let 
	$m_0=\varepsilon g_1 +\varepsilon^2 g_2$, where $g_1>0, \varepsilon>0 $. Let $\psi(x)=0.$ Then the first order linearization system is given by
	\begin{equation}\label{first order type2}
		\begin{cases}
			-\p_t u_j^{(I)}-\Delta u_j^{(I)}= \int_{\Omega}K_j(x,y)m^{(I)}(y,t)dy ,  & \text{ in }  Q,\\
			\p_t m_j^{(I)}(x,t)-\Delta m_j^{(I)}(x,t)=0, & \text{ in }  Q,\\
			\p_{\nu} u_j^{(I)}(x,t)=\p_{\nu}m_j^{(I)}(x,t)=0      &\text{ on } \Sigma,\medskip\\
			u_j^{(I)}(x,T)=0, & \text{ in } \Omega,\medskip\\
			m_j^{(I)}(x,0)=g_1, & \text{ in } \Omega.\\
		\end{cases}	
	\end{equation}
	
	Let $g_1(x)=1$. Since $\int_{\Omega} K_j(x,y) dy=0$ for $j=1,2$, we have $u_j^{(I)}(x,t)=0, m_j^{(I)}(x,t)=1.$ Then the second order linearization system is given by
	\begin{equation}\label{second order type2}
		\begin{cases}
			-\p_t u_j^{(II)}-\Delta u_j^{(II)}+\kappa(x)|\nabla u^{(I)}|^2= \int_{\Omega}K_j(x,y)m^{(II)}(y,t)dy ,  & \text{ in }  Q,\\
			\p_t m_j^{(II)}(x,t)-\Delta m_j^{(II)}(x,t)=2{\rm div}(m^{(I)}\kappa\nabla u^{(I)}), & \text{ in }  Q,\\
			\p_{\nu} u_j^{(II)}(x,t)=\p_{\nu}m_j^{(II)}(x,t)=0      &\text{ on } \Sigma,\medskip\\
			u_j^{(II)}(x,T)=0, & \text{ in } \Omega,\medskip\\
			m_j^{(II)}(x,0)=2g_2, & \text{ in } \Omega.\\
		\end{cases}	
	\end{equation}
	Define $\hat{K}=K_1-K_2$. Let $\omega$ be a solution of \eqref{adjoint}, by a similar argument, we have
	
	\begin{equation}
		\int_{Q}\left[ \int_{\Omega}  \hat{K}(x,y) m^{(2)}(y,t)  dy \right]\omega(x,t)dxdt=0,
	\end{equation}
	for all $ m^{(II)}(y,t) $ such that it is a solution of
	\begin{equation}
		\begin{cases}
			\p_t m_j^{(II)}(x,t)-\Delta m_j^{(II)}(x,t)=0, & \text{ in }  Q,\\
			\p_{\nu}m_j^{(II)}(x,t)=0      &\text{ on } \Sigma,\medskip\\
			m_j^{(II)}(x,0)=2g_2, & \text{ in } \Omega.\\
		\end{cases}	
	\end{equation}
	Similarly, by Lemma $\ref{E-F is complete}$, we may choose $m^{(II)}(x,t)=e^{\lambda t}g(x;\lambda)$. Then by the same argument, we have
	\begin{equation}
		\int_{0}^{T}e^{\lambda t}e^{-|\xi|^2t}dt\int_{\Omega}\left[ \int_{\Omega}  \hat{K}(x,y) g(y;\lambda)  dy \right] e^{-ix\cdot\xi}dx=0,
	\end{equation}
	Then we have
	$$ \int_{\Omega}  \hat{K}(x,y) g(y;\lambda)  dy=0,$$
	for all $g(y;\lambda).$
	Note that $g(y;\lambda)$ can be all Neumann eigenfunctions of $-\Delta$ and these functions are complete in $L^2(\Omega)$. Therefore, we have
	$$K_1(x,y)=K_2(x,y)$$
	in $\Omega\times\Omega$. The proof is now complete.
\end{proof}
\begin{rmk}
	By the proof of these two Theorems, we show that we do not need the full information of $\mathcal{N}_{\kappa,F}(m_0,\psi)$. One can only use $\mathcal{N}_{\kappa,F}(m_0,0)$ and $\mathcal{N}_{\kappa,F}(0,\psi)$ to recover $F$ and $\kappa(x).$
\end{rmk}
 \section{The Local-dependent Case with Non-trivial Solution in a Bounded Domain}

We already consider the inverse problems in bounded domain $\Omega$  in the local-dependent case when the MFG admits a trivial solution and the case that $F$ depends on $m$ non-locally without assuming trivial solution exist. In this chapter, we aim to consider the case that $F$ depends on $m$ locally  without assuming trivial solution exist. It is quite different from the previous cases because the structure of the linearization system becomes totally different. The two unknown functions still coupled with each other and we need to introduce different techniques to handle this case.
\subsection{The Inverse Problem and The Main Result}

The MFG system for our study in this section is introduced as follows:
\begin{equation}\label{main_constant} 
	\left\{
	\begin{array}{ll}
		\displaystyle{-\partial_t u(x,t) -\Delta u(t,x)+\frac{1}{2}|\nabla u(x,t)|^2-F(x, m(x,t))=0} &  {\rm{in}}\ Q,\medskip\\
		\displaystyle{\partial_tm(x,t)-\Delta m(x,t)-{\rm div} \big(m(x,t) \nabla u(x,t)\big)=0} & {\rm{in}}\ Q,\medskip\\
		\p_{\nu} u(x,t)=\p_{\nu} m(x,t)=0 & {\rm{on}}\ \Sigma,\medskip\\
		u(x,T)=G,\ m(x,0)=m_0(x) & {\rm{in}}\  \Omega,\medskip
	\end{array}
	\right.
\end{equation}
where  $G$ is a constant and it signifies the terminal cost. We can define
\begin{equation}\label{eq:distr1}
	\mathcal{O}:=\{ m:\Omega\to [0,\infty) \ \ |\ \ \int_{\Omega} m\, dx =1 \}.
\end{equation}
In other words, if $m\in \mathcal{O}$, then it is the density of a distribution in $\Omega$. We next introduce the measurement data set:
\begin{equation}\label{eq:measure_constant}
	\mathcal{M}_F(m_0):=\left(u(x,0),m(x,T)\right), 
\end{equation}
where
$(u, m)$ is the (unique) pair of solutions to the MFG system \eqref{main_constant} associated with the initial population distribution $m(x, 0)=m_0(x)$. 
Similarly, we introduce the admissible class as before but we need monotonicity assumption here.

\begin{defn}\label{Admissible class_constant}
	We say $U(x,z):\mathbb{R}^n\times\mathbb{C}\to\mathbb{C}$ is admissible, denoted by $U\in\mathcal{A'}$, if it satisfies the following conditions:
	\begin{enumerate}
		\item[(i)] The map $z\mapsto U(x,z)$ is holomorphic with the value in $C^{2+\alpha}(\Omega)$ for $0<\alpha<1$.
		\item[(ii)] $U(x,1)=0$ for all $x\in\mathbb{R}^n$. Here we recall that we assume $|\Omega|=1.$
		\item[(iii)]  $U^{(1)}$ is a positive real number. 
	\end{enumerate} 	
	
	Clearly, if (1) and (2) are fulfilled, then $U$ can be expanded into a power series as follows:
	\begin{equation}
		U(x,z)=\sum_{k=1}^{\infty} U^{(k)}(x)\frac{(z-1)^k}{k!},
	\end{equation}
	where $ U^{(k)}(x)=\frac{\p^kU}{\p z^k}(1).$
\end{defn}
\begin{rmk}
	First, we also need to assume the constant term is zero ($U(x,1)=0$) which is a necessary condition for our inverse problem.  The condition (iii) is in fact just the assumption $\ref{hypo-monoton}$ in this analytic case.
\end{rmk}
Now we claim the following Theorem:
\begin{thm}\label{der F_constant}
	Assume that $F_j \in\mathcal{A'}$ ($j=1,2$) and $F_j^{(1)}$ is constant . Let $\mathcal{M}_{F_j}$ be the measurement map associated to
	the following system:
	\begin{equation}\label{eq:mfg1_constant}
		\begin{cases}
			-\p_tu(x,t)-\Delta u(x,t)+\frac 1 2 {|\nabla u(x,t)|^2}= F_j(x,m(x,t)),& \text{ in }  Q,\medskip\\
			\p_t m(x,t)-\Delta m(x,t)-{\rm div} (m(x,t) \nabla u(x,t))=0,&\text{ in } Q,\medskip\\
			\p_{\nu} u(x,t)=\p_{\nu}m(x,t)=0      &\text{ on } \Sigma,\medskip\\
			u(x,T)=G, & \text{ in } \Omega,\medskip\\
			m(x,0)=m_0(x), & \text{ in } \Omega.\\
		\end{cases}  		
	\end{equation}
	If 	
	$$\mathcal{M}_{R_1}(m_0)=\mathcal{M}_{R_2}(m_0),$$  
	for all  $m_0\in  C^{2+\alpha}(\Omega) \cap \mathcal{O}$, where $\mathcal{O}$ is defined in \eqref{eq:distr1},
	then it holds that 
	$$F_1(x,z)=F_2(x,z)\ \text{  in  } \mathbb{R}.$$ 
\end{thm}
\subsection{Construction of Probing Modes}
\begin{thm}\label{con_of_pb}
	Consider the system
	\begin{equation}\label{F is constant}
		\begin{cases}
			-u_t-\Delta u= Fm & \text{ in } Q,\medskip\\
			m_t-\Delta m-\Delta u=0  & \text{ in } Q,\medskip\\
			\p_{\nu}m(x,t)=\p_{\nu}u(x,t)=0  & \text{ on } \Sigma,\medskip\\
			u(x,T)=0& \text{ on } \Omega.\\
		\end{cases}
	\end{equation}
	Suppose $F(x)=c$ is constant, then there exist a sequence of $\lambda_i\geq 0$ and $D_i\in\mathbb{R}$ such that 
	$$(u,m)=(u, -\lambda_ie^{-\lambda_i t}\overline{m}_i(x)+D_i e^{\lambda_i t}\overline{m}_i(x))$$
	are solutions of system \eqref{F is constant}, where $\overline{m}_i(x)$ are normalized Neumann eigenfunctions of $-\Delta$ in $\Omega$. 
\end{thm}
\begin{proof}
	Let $\beta_i$ be a positive Neumann eigenvalue of $-\Delta$ in $\Omega$ and $\overline{m}_i(x)$ is the corresponding normalized eigenfunction.
	In other words, we have $\overline{m}_i(x)$ are not $0$ functions and
	\begin{equation}
		\begin{cases}
			-\Delta \overline{m}_i(x)=\beta_i\overline{m}_i(x)   & \text{ in }\Omega\\
			\p_{\nu}\overline{m}_i(x)=0 & \text{ in }\Sigma. 
		\end{cases}
	\end{equation} 
	
	Now we choose $\lambda_i=\sqrt{\beta_i^2+c\beta_i}$ and $k_i=\beta_i-\lambda_i\leq 0$, then we can check  $$(u,m)=\left(\frac{c+k_i}{\lambda_i}e^{-\lambda_it}\overline{m}_i(x), e^{-\lambda_i t}\overline{m}_i(x)\right)$$ are solutions of the following system
	\begin{equation}\label{F is constant'}
		\begin{cases}
			-u_t-\Delta u= F(x)m & \text{ in } Q,\medskip\\
			m_t-\Delta m-\Delta u=0  & \text{ in } Q,\medskip\\
			\p_{\nu}m(x,t)=\p_{\nu}u(x,t)=0  & \text{ on } \Sigma.\medskip\\
		\end{cases}
	\end{equation}
	Notice that
	\begin{equation*}
		\begin{aligned}
			-u_t-\Delta u&=(c+k_i)e^{-\lambda_it}\overline{m}(x)-\frac{c+k_i}{\lambda_i}e^{-\lambda_it}\Delta\overline{m}(x)\\
			&=e^{-\lambda_it}\overline{m}(x)\left[ c+k_i+\frac{(c+k_i)\beta_i}{\lambda_i}\right]\\
			&=cm(x,t)+ e^{-\lambda_it}\overline{m}(x)\left[k_i+\frac{(c+k_i)\beta_i}{\lambda_i}\right],
		\end{aligned} 
	\end{equation*}
	and
	\begin{equation*}
		\begin{aligned}
			m_t-\Delta m-\Delta u&=-\lambda_i e^{-\lambda_it}\overline{m}(x)-e^{-\lambda_it}\Delta\overline{m}(x)-\frac{c+k_i}{\lambda_i}e^{-\lambda_it}\Delta\overline{m}(x)\\
			&=-e^{-\lambda_it}\overline{m}(x)\left[\lambda_i-(\lambda_i+k_i)-\frac{c+k_i}{\lambda_i}\beta_i\right]\\
			&= e^{-\lambda_it}\overline{m}(x)\left[k_i+ \frac{(c+k_i)\beta_i}{\lambda_i}  \right].             
		\end{aligned}
	\end{equation*}
	Hence, we only need to check $k_i+ \frac{(c+k_i)\beta_i}{\lambda_i}=0$, and it can be shown by the choice of $\lambda_i$ and $k_i$.

	Similarly, we have 
	$$(u,m)=\left(\frac{c}{\beta_i-\lambda_i}e^{\lambda_it}\overline{m}_i(x), e^{\lambda_i t}\overline{m}_i(x)\right)$$
	are solutions of \eqref{F is constant'}.
	
	Then we may choose $D_i=\dfrac{c}{k_i(c+k_i)}\leq 0$ (since $k\leq0$ and $c+k_i\geq 0$). By the linearity of this system, we have there is a solution of \eqref{F is constant} in the form
	
	$$(u,m)=\left(u, -\lambda_ie^{-\lambda_i t}\overline{m}_i(x)+D_i e^{\lambda_i t}\overline{m}_i(x)\right).$$
\end{proof}

\begin{thm}\label{con_of_pb2}
	Consider the system
	\begin{equation}\label{F is constant2}
		\begin{cases}
			u_t-\Delta u= Fm & \text{ in } Q,\medskip\\
			-m_t-\Delta m-\Delta u=0  & \text{ in } Q,\medskip\\
			\p_{\nu}m(x,t)=\p_{\nu}u(x,t)=0  & \text{ on } \Sigma,\medskip\\
		\end{cases}
	\end{equation}
	Suppose $F(x)=c$ is constant, then there exist a sequence of $\lambda_i\geq 0$ and $D_i\in\mathbb{R}$ such that 
	$$(u,m)=(u, -\lambda_ie^{-\lambda_i(T-t)}\overline{m}_i(x)+D_i e^{\lambda_i (T-t)}\overline{m}_i(x))$$
    $$(u,m)=(u,e^{-\lambda_i t}\overline{m}_i(x))$$
	are solutions of system \eqref{F is constant}, where $\overline{m}_i(x)$ are normalized Neumann eigenfunctions of $-\Delta$ in $\Omega$. 
\end{thm}
\begin{proof}
	The first part is obtained by Theorem $\ref{con_of_pb}$ by a simple change of variable. The second part is true because there is no initial boundary condition in \eqref{con_of_pb2}.
\end{proof}

\subsection{The Proof of Theorem $\ref{der F_constant}$}
Before the main proof, we need the following lemma which plays as a key step of the proof. 
\begin{lem}\label{key}
	Let $(v,\rho)$ be a solution of the following system:
	\begin{equation}
		\begin{cases}
			v_t-\Delta v= F_1(x)\rho & \text{ in } Q,\medskip\\
			-\rho_t-\Delta \rho-\Delta v=0  & \text{ in } Q, \medskip\\
			\p_{\nu}v(x,t)=\p_{\nu}\rho(x,t)=0      & \text{ on } \Sigma, \medskip\\
			v(x,0)=  0& \text{ in } \Omega.
		\end{cases}
	\end{equation}
	Let $(\overline{u},\overline{m})$ satisfy
	\begin{equation}
		\begin{cases}
			-\overline{u}_t-\Delta \overline{u}- F_1(x)\overline{m}=(F_1-F_2)m_2  & \text{ in } Q, \medskip\\
			\overline{m}_t-\Delta \overline{m}-\Delta \overline{u}=0  & \text{ in } Q, \medskip\\
			\p_{\nu}\overline{u}(x,t)=	\p_{\nu}\overline{m}(x,t)=0    & \text{ on } \Sigma, \medskip\\
			\overline{u}(x,T)=0	   &\text{ in } \Omega, \medskip\\
			\overline{u}(x,0)=0    & \text{ in } \Omega, \medskip\\
			\overline{m}(x,0)=\overline{m}(x,T)=0& \text{ in } \Omega\\
		\end{cases}
	\end{equation}
	Then we have 
	\begin{equation}\label{implies F_1=F_2}
		\int_Q \quad (F_1-F_2)m_2\rho \ dxdt=0.
	\end{equation}
\end{lem}
\begin{proof}
	Notice that
	\begin{equation}\label{IBP1}
		\begin{aligned}
			0&=\int_Q (\overline{m}_t-\Delta \overline{m}-\Delta \overline{u} )\rho\ dxdt\\
			&=\int_{\Omega}\left.\overline{m}\rho\right|_0^T dx-\int_Q\overline{m}\rho_t dxdt-\int_Q\rho(\Delta\overline{m}+\Delta\overline{u})\ dxdt\\
			&=-\int_Q\overline{m}( -\Delta\rho-\Delta v)dxdt-\int_Q\rho( \Delta\overline{m}+\Delta\overline{u})\ dxdt\\
			&=\int_Q (\overline{m}\Delta v-\rho\Delta\overline{u})\ dxdt.
		\end{aligned}
	\end{equation}
	
	Similarly, one can deduce that
	\begin{equation}\label{IBP2}
		\begin{aligned}
			0&=\int_Q (\overline{m}_t-\Delta \overline{m}-\Delta \overline{u} )v\ dxdt\\
			&=-\int_Q\overline{m}(\Delta v+F_1\rho)dxdt-\int_Q v(\Delta\overline{m}+\Delta\overline{u} )\ dxdt\\
			&=-\int_Q(2\overline{m}\Delta v+F_1\rho\overline{m}+\overline{u} \Delta v)\ dxdt.
		\end{aligned}
	\end{equation}
	Then we have 
	\begin{align*}
		\int_Q (F_1-F_2)m_2\rho dxdt&=\int_Q\rho(-\overline{u}_t-\Delta\overline{u}-F_1\overline{m})\ dxdt\\
		&=\int_{\Omega}\left.\overline{u}\rho\right|_0^T dx+\int_Q(\rho_t\overline{u}-\rho\Delta\overline{u}-F_1\rho\overline{m})\  dxdt\\
		&=\int_{\Omega}  \overline{u}(x,T)\rho(x,T) dx+ \int_Q(-\Delta\rho-\Delta v)\overline{u}-\rho\Delta\overline{u}-F_1\rho\overline{m}\ dxdt\\
		&= \int_Q(-2\rho\Delta\overline{u}-\overline{u}\Delta v-F_1\rho\overline{m})\ dxdt\\
		&=\int_Q(-2\rho\Delta\overline{u}-\overline{u}\Delta v-F_1\rho\overline{m})\ dxdt, 
	\end{align*}
	which in combination with \eqref{IBP1} and \eqref{IBP2} readily yields that 
	$$	\int_Q (F_1-F_2)m_2\rho \ dxdt=0.$$
	
	The proof is complete. 
\end{proof}

\begin{proof}
	For $j=1,2$, let us consider 
	\begin{equation}\label{MFG 1,2}
		\begin{cases}
			-u_t-\Delta u+\frac{1}{2}|\nabla  u|^2= F_j(x,m) & \text{ in } Q,\medskip\\
			m_t-\Delta m-\text{div} (m\nabla u)=0  & \text{ in } Q, \medskip\\
			\p_{\nu}u(x,t)=\p_{\nu}m(x,t)=0     & \text{ on } \Sigma, \medskip\\
			u(x,T)=G    & \text{ in } \Omega,\medskip\\
			m(x,0)=m_0(x) & \text{ in } \Omega.\\
		\end{cases}
	\end{equation}
	Next, we divide our proof into three steps. 
	
	\bigskip
	\noindent {\bf Step I.}~First, we do the first order linearization to the MFG system \eqref{MFG 1,2} in $Q$ and can derive: 
	\begin{equation}\label{linearization}
		\begin{cases}
			-\p_{t}u^{(1)}_j-\Delta u_j^{(1)}= F_j^{(1)}(x)m_j^{(1)} & \text{ in } Q, \medskip\\
			\p_{t}m^{(1)} _j-\Delta m_j^{(1)} -\Delta u_j^{(1)}=0  & \text{ in } Q, \medskip\\
			u^{(1)}_j(x,T)=0     & \text{ in } \Omega,\medskip\\
			m^{(1)} _j(x,0)=f_1(x) & \text{ in } \Omega.\\
		\end{cases}
	\end{equation}
	Let $\overline{u}^{(1)}=u^{(1)}_1-u^{(1)}_2$ and $ \overline{m}^{(1)}=m^{(1)} _1-m^{(1)} _2. $ Let 
	$(v,\rho)$ be a solution to the following system
	\begin{equation}
		\begin{cases}
			v_t-\Delta v= F^{(1)}_1(x)\rho & \text{ in } Q, \medskip\\
			-\rho_t-\Delta \rho-\Delta v=0  & \text{ in } Q, \medskip\\
			\partial_{\nu}v(x,t)=\partial_{\nu}\rho(x,t)=0      & \text{ on } \Sigma, \medskip\\
		\end{cases}
	\end{equation}
	Since $\mathcal{M}_{F_1}=\mathcal{M}_{F_2}$, 
	by Lemma $\ref{key}$, we have 
	\begin{equation}\label{implies to 0;1}
		\int_Q \quad( F^{(1)}_1-F^{(1)}_2)m^{(1)} _2\rho \ dxdt =0,
	\end{equation}
	for all $ m^{(1)} _2\in C^{1+\frac{\alpha}{2},2+\alpha}(Q)$ with $m^{(1)} _2 $ being a solution 
	to \eqref{linearization}.
	By Theorems $\ref{con_of_pb}$ and $\ref{con_of_pb2}$, we may choose 
	$$m_2^{(1)}(x,t)=-\lambda_ie^{-\lambda_i t}\overline{m}_i(x)+D_i e^{\lambda_i t}\overline{m}_i(x),$$
	$$\rho(x,t)= -\lambda_ie^{-\lambda_i (T-t)}\overline{m}_i(x)+D_i e^{\lambda_i (T-t)}\overline{m}_i(x),$$
	where the notations are the same as the notations in the proof of Theorem $\ref{con_of_pb}.$
	
	In fact, we use the $f_1(x)=-\lambda_i\overline{m}_i(x)+D_i\overline{m}_i(x)$ here. Since $\overline{m}_i(x)$ are Neumann eigenfunction of $-\Delta$, the condition $\int_{\Omega}f_1(x)dx=0$ is satisfied. Furthermore, we can determine $m_2^{(1)}(x,t)$ by $f_1(x)$ and equation \eqref{linearization} by Lemma $\ref{linear app unique}.$

	Then \eqref{implies to 0;1} implies that
	\begin{equation}
		(F_1-F_2)\int_{0}^{T}(-\lambda_ie^{-\lambda_i t}+D_i e^{\lambda_i t})(-\lambda_ie^{-\lambda_i (T-t)}+D_i e^{\lambda_i (T-t)})dt\int_{\Omega}\overline{m}_i(x)^2dx=0. 
	\end{equation}
	
	Notice that $\int_{\Omega}\overline{m}_i(x)^2dx=1$ and $(-\lambda_ie^{-\lambda_i t}+D_i e^{\lambda_i t})(-\lambda_ie^{-\lambda_i (T-t)}+D_i e^{\lambda_i (T-t)})$ is positive, we have 
	$$F_1=F_2.$$
	\medskip
	
	\noindent{\bf Step II.}~We proceed to consider the second linearization to the MFG system \eqref{MFG 1,2} in $Q$ and can obtain for $j=1,2$:
	\begin{equation}
		\begin{cases}
			-\p_tu_j^{(1,2)}-\Delta u_j^{(1,2)}(x,t)+\nabla u_j^{(1)}\cdot \nabla u_j^{(2)}\medskip\\
			\hspace*{3cm}= F^{(1)}m_j^{(1,2)}+F^{(2)}(x)m_j^{(1)}m_j^{(2)} & \text{ in } \Omega\times(0,T),\medskip\\
			\p_t m_j^{(1,2)}-\Delta m_j^{(1,2)}-\Delta u_j^{(1,2)}= {\rm div} (m_j^{(1)}\nabla u_j^{(2)})+{\rm div}(m_j^{(2)}\nabla u_j^{(1)}) ,&\text{ in } \Omega\times (0,T) \medskip\\
			u_j^{(1,2)}(x,T)=0 & \text{ in } \Omega,\medskip\\
			m_j^{(1,2)}(x,0)=0 & \text{ in } \Omega.\\
		\end{cases}  	
	\end{equation}
	By the proof in Step~I, we have $ (u_1^{(1)},m_1^{(1)})=( u_2^{(1)},m_2^{(1)})$.

	Define $\overline{u}^{(1,2)}=u_1^{(1,2)}-u_2^{(1,2)} $ and $\overline{m}^{(1,2)}=m_1^{(1,2)}-m_2^{(1,2)} $. Since  
	$\mathcal{M}_{F_1}=\mathcal{M}_{F_2}$,  we have
	\begin{equation}
		\begin{cases}
			-\overline{u}^{(1,2)}_t-\Delta \overline{u}^{(1,2)}- F_1\overline{m}^{(1,2)}=(F^{(2)}_1-F^{(2)}_2)m_1^{(1)}m_2^{(1)}   & \text{ in } Q\\
			\overline{m}^{(1,2)}_t-\Delta \overline{m}^{(1,2)}-\Delta \overline{u}^{(1,2)}=0  & \text{ in } Q, \medskip\\
			\p_{\nu}\overline{u}^{(1,2)}(x,t)=	\p_{\nu}\overline{m}^{(1,2)}(x,t)=0    & \text{ on } \Sigma, \medskip\\
			\overline{u}^{(1,2)}(x,T)=0	   &\text{ in } \Omega, \medskip\\
			\overline{u}^{(1,2)}(x,0)=0    & \text{ in } \Omega, \medskip\\
			\overline{m}^{(1,2)}(x,0)=\overline{m}^{(1,2)}(x,T)=0& \text{ in } \Omega.\\
		\end{cases}
	\end{equation}
	Let 
	$(v,\rho)$ be a solution to the following system
	\begin{equation}
		\begin{cases}
			v_t-\Delta v= F^{(1)}_1\rho & \text{ in } Q, \medskip\\
			-\rho_t-\Delta \rho-\Delta v=0  & \text{ in } Q, \medskip\\
			\p_{\nu}v(x,t)=\p_{\nu}\rho(x,t)=0      & \text{ on } \Sigma, \medskip\\
		\end{cases}
	\end{equation}
	By Lemma $\ref{key}$, we have
	\begin{equation}
		\int_Q ( F^{(2)}_1-F^{(2)}_2)m^{(1)} _1m_2^{(1)} \rho \ dxdt =0.
	\end{equation}
	Now we may choose $\rho=e^{-\lambda_it}\overline{m}_i(x)$ and 
	$$m_1^{(1)}(x,t)=m_2^{(1)}(x,t)=-\lambda_je^{-\lambda_j t}\overline{m}_j(x)+D_j e^{\lambda_j t}\overline{m}_j(x),$$
	for a fixed  $j\in\mathbb{N}.$
	Next, by a similar argument in the proof of Step~I, we can derive that 
	\begin{equation}
		\int_{\Omega} ( F^{(2)}_1-F^{(2)}_2)\overline{m}_j^2\overline{m}_i(x) \ dxdt =0,
	\end{equation}
	for all $i\in\mathbb{N}$.
	
	Since $\overline{m}_i(x) $ form a complete $L^2$-basis, we have $( F^{(2)}_1-F^{(2)}_2)\overline{m}_j^2=0.$ Notice that the zero set of Neumann eigenfunction must be measure zero, we have
	$$F^{(2)}_1(x)=F^{(2)}_2(x).$$
	\medskip
	
	\noindent{\bf Step~III.}~Finally, by mathematical induction and repeating similar arguments as those in the proofs of Steps I and II, one can show that
	$$F^{(k)}_1-F^{(k)}_2=0 ,$$
	for all $k\in\mathbb{N}$. Hence, $F_1(x,m)=F_2(x,m).$
	
	The proof is complete. 
\end{proof}

\section{Multipopulation mean field game (MFG) system}\label{sec:multipop}
In this section, we extend the results to a multipopulation mean field game (MFG) model. Let $x\in \mathbb{R}^d$ denote the state variable, $t\in[0,\infty)$ denote the time variable, and $\mathbb{T}^d:=\mathbb{R}^d/\mathbb{Z}^d$ be the standard $d$-dimensional torus. The $n$-population mean field game model we are interested in takes the form, 
\begin{equation}\label{EQ:MFG MultiPop}
	\begin{cases}
		-\partial_t u_i(x,t)-\Delta u_i(x,t)+\frac{1}{2}|\nabla u_i(x,t)|^2=F_i(x,\mathbf{m}(x,t)), &(x,t)\in\mathbb{T}^d\times(0,T),\\[1ex]
		\partial_t m_i(x,t)-\Delta m_i(x,t)-{\rm div}(m_i(x,t)\nabla u_i)=0, &(x,t)\in\mathbb{T}^d\times(0,T),\\[1ex]
		u_i(x,T)=G_i(x,\mathbf{m}(x,T)),\ m_i(x,0)=m_{i,0}(x), &x\in \mathbb{T}^d,
	\end{cases}
\end{equation}
$i=1,2,\dots,n$, where $(u_i,m_i)$ is the value function-density pair of the $i$-th population, and $\Delta$, $\rm div$ are respectively the Laplacian and divergence operators with respect to $x$-variable. For simplicity we denote $\mathbf{u}=(u_1,\dots,u_n)$, $\mathbf{m}=(m_1,\dots,m_n)$. 
The functions $F_i$ and $G_i$ are, respectively, the running cost and terminal cost of the $i$-th population. The initial density of the $i$-th population is denoted by $m_{i,0}$.
Comparing with the single-population case, the key difference is that the running cost $F_i$ and total cost $G_i$ depend on the whole measure $\mathbf{m}$ but not $m_i$.

In this section, we are interested in reconstructing the cost functions 
\[
\mathbf{F}:=(F_1, \dots, F_n),\quad\mbox{and},\quad \mathbf{G}:=(G_1, \dots, G_n)\,.
\]
of the MFG system~\eqref{EQ:MFG MultiPop} from observed data on the value function $\mathbf{u}(x,t)$. We assume that we can measure  $\mathbf{u}(x,0)$ for all possible initial conditions $\mathbf{m}_0:=(m_{1,0}, \dots, m_{n,0})$, that is, we have data from the map:
\begin{equation}\label{EQ:Data 1}
	\mathcal{M}_{(\mathbf{F},\mathbf{G})}: \mathbf{m}_0 \mapsto \mathbf{u}(x,t)|_{t=0}\,.
\end{equation}
Let $\mathcal{M}_{(\mathbf{F}^1, \mathbf{G}^1)}$ and $\mathcal{M}_{(\mathbf{F}^2, \mathbf{G}^2)}$ be the data corresponding to the cost function pairs $(\mathbf{F}^1, \mathbf{G}^1)$ and $(\mathbf{F}^2, \mathbf{G}^2)$, respectively. We establish the following uniqueness result:
\begin{equation}\label{EQ:Unique Full}
	\mathcal{M}_{(\mathbf{F}^1,\mathbf{G}^1)}=\mathcal{M}_{(\mathbf{F}^2,\mathbf{G}^2)}\quad  \text{implies}\quad (\mathbf{F}^1,\mathbf{G}^1)=(\mathbf{F}^2,\mathbf{G}^2),
\end{equation}
under appropriate assumptions that we will specify later.

\subsection{Assumptions and main result}
We assume that the cost functions $F_i$ and $G_i$ ($1\le i\le n$) admit a generic power series representation of the form, $\beta$ being an $n$-dimensional multi-index,
\begin{equation}\label{EQ:power expansion}        F_i(x,\mathbf{z})=\sum_{\beta\in\mathbb{N}^n,|\beta|\ge 1} F_i^{(\beta)}(x)\frac{\mathbf{z}^\mathbf{\beta}}{|\beta|!}\,.
\end{equation}
Clearly, functions satisfy the conditions in the following definition admits this power series presentation.
\begin{defn}
	We say a function $U(x,\mathbf{z}):\mathbb{R}^d\times\mathbb{C}^n\to\mathbb{C}$ is admissible or in class $\mathcal{A}$, that is, $U\in\mathcal{A}$, if it satisfies the following conditions:\\[1ex]
	(i) The map $\mathbf{z}\mapsto U(\cdot,\mathbf{z})$ is holomorphic with value in $C^\alpha(\mathbb{T}^d)$ for some $\alpha\in(0,1)$;\\[1ex]
	(ii) $U(x,\mathbf{0})=0$ for all $x\in\mathbb{T}^d$.
\end{defn}
We use the standard notation
\begin{equation}\label{EQ:A}
	\mathbf{F}=(F_1, \cdots, F_n) \in\mathcal{A}^n \iff F_i\in\mathcal{A}, \forall i\ge 1\,.
\end{equation}
Throughout the rest of this section, we assume the cost functions $\mathbf{F}\in\mathcal{A}^n$ and $\mathbf{G}\in\mathcal{A}^n$.

\begin{thm}
		Let $\mathbf{F}^j, \mathbf{G}^j\in\mathcal{A}^n$ ($j=1,2$) and $\mathcal{M}_{\mathbf{F}^j,\mathbf{G}^j}$ be the measurement map associated to the following system:
		\begin{equation}\label{EQ:MFG MultiPop j}
			\begin{cases}
				-\partial_t u_i^j(x,t)-\Delta u_i^j(x,t)+\frac{1}{2}|\nabla u_i^j(x,t)|^2=F_i^j(x,\mathbf{m}^j(x,t)), &(x,t)\in\mathbb{T}^d\times(0,T),\\
				\partial_t m_i^j(x,t)-\Delta m_i^j(x,t)-{\rm div}(m_i^j(x,t)\nabla u_i^j)=0, &(x,t)\in\mathbb{T}^d\times(0,T)\\
				u_i^j(x,T)=G_i^j(x,\mathbf{m}(x,T)),\ m_i^j(x,0)=m_{i,0}(x), &x\in \mathbb{T}^d.
			\end{cases}
		\end{equation}
		If there exists $\delta>0$, such that for any $\mathbf{m}_0\in B_\delta\left([C^{2+\alpha}(\mathbb{T}^d)]^n\right)$, one has
		\[
		\mathcal{M}_{\mathbf{F}^1,\mathbf{G}^1}(\mathbf{m}_0)=\mathcal{M}_{\mathbf{F}^2,\mathbf{G}^2}(\mathbf{m}_0),
		\]
		then it holds that
		\[
		(\mathbf{F}^1(x,\mathbf{z}),\mathbf{G}^1(x,\mathbf{z}))=(\mathbf{F}^2(x,\mathbf{z}),\mathbf{G}^2(x,\mathbf{z}))\ \text{in}\ \mathbb{T}^d\times\mathbb{R}^n.
		\]
\end{thm}

We may also consider single-population data. We assume that we can measure the value function at $t=0$ for only one given population, say, the $i$-th population. That is, we have data given by the map:
\begin{equation}\label{EQ:Data 2}
	\mathcal{L}_{(\mathbf{F},\mathbf{G})}: \mathbf{m}_0\mapsto u_i(x,t)|_{t=0}, \qquad \mbox{for a given}\ i\in\{1, \dots, n\}\,.
\end{equation}

\begin{cor}\label{THM:part recon}
	Let $F_i(x, \mathbf{z})=F(x,\mathbf{z})$ and $G_i(x, \mathbf{z})=G(x, \mathbf{z})$ for all $1\le i\le n$, and $ \mathcal{L}_{(\mathbf{F}^j,\mathbf{G}^j)}$ be the data associated to the following system:
	\begin{equation}
		\begin{cases}
			-\partial_t u_i^j(x,t)-\Delta u_i^j(x,t)+\frac{1}{2}|\nabla u_i^j(x,t)|^2=F^j(x,m^j(x,t)), &(x,t)\in\mathbb{T}^d\times(0,T),\\
			\partial_t m_i^j(x,t)-\Delta m_i^j(x,t)-{\rm div}(m_i^j(x,t)\nabla u_i^j)=0, &(x,t)\in\mathbb{T}^d\times(0,T)\\
			u_i^j(x,T)=G^j(x,\mathbf{m}(x,T)),\ m_i^j(x,0)=m_{i,0}(x), &x\in \mathbb{T}^d.
		\end{cases}
	\end{equation}
	If there exists $\delta>0$, such that for any $\mathbf{m}_0\in B_\delta\left([C^{2+\alpha}(\mathbb{T}^d)]^n\right)$, one has
	\[
	 \mathcal{L}_{(\mathbf{F}^1,\mathbf{G}^1)}=\mathcal{L}_{(\mathbf{F}^2,\mathbf{G}^2)},
	\]
	then it holds that
	\[
	(\mathbf{F}^1(x,\mathbf{z}), \mathbf{G}^1(x,\mathbf{z}))=(\mathbf{F}^2(x,\mathbf{z}), \mathbf{G}^2(x,\mathbf{z}))\ \text{in}\ \mathbb{T}^d\times\mathbb{R}^n.
	\]
\end{cor}

Since the proofs of these two theorems rely on the same method used in the previous section, we omit them here. Some other related results can be found in \cite{ren2023unique,ren2024reconstructing}. A challenging question to consider is whether these inverse problems can be extended to the multi-population case under more general conditions.

 \chapter{Inverse Coefficient Problems Using Cauchy Data}\label{chap:Cauchy}

The proofs in the previous chapter rely on the linearization process around the trivial solution $(u,m)=(0,0)$ or $(u,m)=(C,1)$ for some constant. In this chapter, we establish that the running cost $F$ can be recovered  and the proof does not rely on the existence of a special solution in the stationary case.

In particular, this chapter presents an alternative approach to treating the probability density constraint on $m$ and the Neumann boundary conditions, by bypassing them and not imposing any boundary conditions on the MFG system, but instead consider the set of all possible solutions $(u,m)$ to the system. This can be understood in practical terms: Given a stock market where players can trade from various locations and through various avenues, for instance through local stock exchanges or online broker accounts, we only consider those tradings that are conducted online from a single country. In this case, it is not necessary to obey the probability density constraint and Neumann boundary conditions. The same holds if we consider the investments, sales and profits of multi-national technology companies in a single country. This can also be applied to airport management, where certain airlines are based at a particular airport, but there are also many other airlines that land or take-off from that airport. In these scenarios, the probability density constraint and Neumann boundary conditions are not required.

\section{The Inverse Problem for Static MFG system}
Let us first consider the stationary MFG system :
\begin{equation}\label{eq:zs_main}
	\begin{cases}
		-\Delta u(x)+\frac{1}{2}|\nabla u(x)|^2=F(x,m)  &  \text{ in } \Omega\\
		-\Delta m(x)-\text{div}(m(x)\nabla u(x))=0  & \text{ in } \Omega\\
		u(x)=f(x),\quad m(x)=g(x)  & \text{ in } \Sigma.
	\end{cases}
\end{equation}

We define the Cauchy set of \eqref{eq:zs_main} by
\begin{equation}\label{Cauchy_set} 
	\mathcal{M}_{F,m_0}:=\{m\big|_{\Sigma},\nabla m\big|_{\Sigma}:(u,m) \text{ satisfies equation } \eqref{eq:zs_main} \}.
\end{equation}

Now we introduce the admissible classes of $F$, as in the previous chapter.
\begin{defn}\label{AdmissClass2}
		We say $U(x,z):\mathbb{R}^n\times\mathbb{C}\to\mathbb{C}$ is admissible, denoted by $U\in\mathcal{A}$, if it satisfies the following conditions:
		\begin{enumerate}
			\item[(i)] The map $z\mapsto U(\cdot,z)$ is holomorphic with value in $C^{2+\alpha}(\mathbb{R}^n)$ for some $\alpha\in(0,1)$;
			\item[(ii)] $U(x,m_0)=0$ for all $x\in\mathbb{R}^n$ and given $m_0(x)$ depending only on $x$.
		\end{enumerate}

	\end{defn}

Then we are able to fully recover the running costs $F$ and unknown stable state $m_0$ from Cauchy data, and the following uniqueness result holds:
\begin{thm}\label{zs_main_result}
	Let $n\geq3$ and $F_i(x,m)\in \mathcal{A}$ $ (i=1,2)$ such that $F_i(x,m_{0,i})=0$, where $m_{0,i}$ are solutions of system \eqref{eq:zs_main} with $F(x,m)=F_i(x,m)$ respectively. Assume that $\mathcal{M}_{F_1,m_{0,1}}=\mathcal{M}_{F_2,m_{0,2}}$, then $F_1=F_2$ and $m_{0,1}=m_{0,2}$.
\end{thm}
\begin{remark}\label{zeroCond}
    For $F\in\mathcal{A}$, $F^{(0)}(x)=0$ is a necessary condition in order to have uniqueness of $F(x,m)$ from the knowledge of $\mathcal{M}_F$ (cf. the discussion in the previous section). Otherwise, we may consider the following counter-example: $F_1(x,m)=x,F_2(x,m)=x^2$. Then $u_i$ is independent of $\mathcal{M}_{F_i}$. This is a relatively well-known result in inverse boundary problems, that one cannot recover the ``source term" from the knowledge of the Cauchy data.
\end{remark}

\subsection{Exponentially Growing solutions of Linearization System}
As we discussed in Section \ref{HLM}, we consider the first-order linearization system of \eqref{static-mfg} in $\Omega$:
\begin{equation}\label{eq:zs_linear}
	\begin{cases}
		-\Delta u^{(1)}+\nabla u_0\cdot\nabla u^{(1)}=F^{(1)}m^{(1)}  & \text{ in } \Omega\\
		-\Delta m^{(1)}-\text{div}(m_0\nabla u^{(1)})-\text{div}(m^{(1)}\nabla u_0)=0  & \text{ in } \Omega,\\
	\end{cases}
\end{equation}
where $(u_0,m_0)$ is a solution of \eqref{static-mfg}. One classical result \cite{MFG-stationary,MFG-book} shows that one may choose $(u_0,m_0)$ such that
\begin{equation}\label{m=e^{-v}}
	\begin{cases}
		m_0=\dfrac{e^{-u_0}}{\int_{\Omega}e^{-u_0}dx}                 &\text{ in } \Omega\\
		\partial_{\nu}m_0=\partial_{\nu} u_0=0  &\text{ in } \Sigma
	\end{cases}
\end{equation}
for some constant $C>0.$ It can be shown by classical fixed point argument, see also chapter \ref{reviw}.
Notice that \eqref{m=e^{-v}} implies that 
\begin{equation}\label{eq:StatRelation}
	\nabla (m_0)=-m_0\nabla u_0.
\end{equation}
Expanding the second equation in \eqref{eq:zs_linear} and substituting in the first equation in \eqref{eq:zs_linear}, we have
\begin{equation}\label{calc}
	\begin{aligned}
		0&=-\Delta m^{(1)}-(\nabla u_0\cdot\nabla m^{(1)}+ m^{(1)} \Delta u_0)-(m_0\Delta u^{(1)}+\nabla m_0\cdot\nabla u^{(1)})\\
		&=-\Delta m^{(1)}-(\nabla u_0\cdot\nabla m^{(1)}+ m^{(1)} \Delta u_0) -[m_0(\nabla u_0\cdot\nabla u^{(1)}- F^{(1)}m^{(1)}) +\nabla m_0\cdot\nabla u^{(1)}]\\
		&=-\Delta m^{(1)}-(\nabla u_0\cdot\nabla m^{(1)}+(\Delta u_0-F^{(1)}m_0) m^{(1)})-[(m_0\nabla u_0+\nabla m_0)\cdot\nabla u^{(1)}]\\
		&=-\Delta m^{(1)}-\nabla u_0\cdot\nabla m^{(1)}-\left(\Delta u_0-F^{(1)}m_0\right) m^{(1)},
	\end{aligned}
\end{equation}
where the last step makes use of \eqref{eq:StatRelation}.

From this last step, the next natural step is to focus on PDEs in the form
\begin{equation}\label{semi-linear}
	\Delta m\pm\nabla u_0\cdot\nabla m+ qm=0,
\end{equation}
where $\nabla u_0\in [C^1(\Omega)]^n$ and $q\in L^{\infty}(\Omega).$

\begin{thm}\label{thm:cgo}
	There exist solutions of \eqref{semi-linear} in the form
	\begin{equation}\label{cgo}
		m=e^{\xi\cdot x\mp\frac{1}{2}u_0}(1+\omega(x;\xi)),
	\end{equation}
	where $\xi\in \mathbb{C}^n$, $\xi\cdot \xi=0$ is a complex vector and $\omega(x;\xi)\in H^1(\Omega)$ satisfies the following decay condition:
	\begin{equation}\label{decay}
		\|\omega(x;\xi)\|_{L^2(\Omega)}\leq C|\xi|^{-1}
	\end{equation}
	for some constant $C>0.$
\end{thm}
The proof of this theorem is based on the following lemma, which has been proved in \cite{sun_cgo}.
\begin{lem}\label{zs_main_lemma}
	Let $L_{\xi}:=\Delta+ 2\xi\cdot\nabla$. Then the operator $L_{\xi}$ admits a bounded
	inverse $L_{\xi}^{-1}: L^2(\Omega)\to H^1(\Omega)$. If $f\in L^{\infty}(\Omega)$ and 
	$v=L^{-1}_{\xi}(f)\in H^1(\Omega)$, then
	$$\|v\|_{L^2(\Omega)}\leq C|\xi|^{-1}\|f\|_{L^2(\Omega)},$$
	$$\|\nabla v\|_{L^2(\Omega)}\leq C\|f\|_{L^2(\Omega)},$$
	where $C$ is constant only depends on $\Omega.$
\end{lem}
Now we give the proof of Theorem \ref{thm:cgo}.  
\begin{proof}
	We will prove this theorem for the case
	\begin{equation}\label{semi-linear'}
		\Delta m+\nabla u_0\cdot\nabla m+ qm=0,
	\end{equation}
	and a similar argument can be employed to obtain the result for the other cases.
	
	Let $m=e^{\xi\cdot x-\frac{1}{2}u_0}(1+\omega(x;\xi))$. By direct computation,  \eqref{semi-linear'} is equivalent to
	\begin{equation}\label{for_omega}
		\Delta \omega+2\xi\cdot\nabla \omega+\left(q-\frac{1}{4}|\nabla u_0|^2-\frac{1}{4}\Delta u_0\right)(1+\omega)=0.
	\end{equation}
	Let $H:=q-\frac{1}{4}|\nabla u_0|^2-\frac{1}{4}\Delta u_0$ and $L_{\xi}:=\Delta+ 2\xi\cdot\nabla$. Then \eqref{for_omega} can be rewritten as
	\begin{equation}\label{operator_form}
		L_{\xi}\omega+H\omega=-H.
	\end{equation}
	By Lemma \ref{zs_main_lemma}, we take the action of $L^{-1}_{\xi}$ on both sides of \eqref{operator_form}, to obtain
	\begin{equation}\label{operator_form_2}
		\left( I+L^{-1}_{\xi}\circ H\right)\omega=-L^{-1}_{\xi}(H).
	\end{equation}
	Lemma \ref{zs_main_lemma} also implies that $L^{-1}_{\xi}\circ H$ is a bounded operator from 
	$L^2(\Omega)$ to $L^2(\Omega)$ and 
	$$\| L^{-1}_{\xi}\circ H  \|\leq C \|H\|_{L^{\infty}}|\xi|^{-1}$$
	for some constant $C$. Therefore, $(I+L^{-1}_{\xi}\circ H)^{-1} $ exists for large enough $|\xi|$. In other words, there exists a unique $L^2$-solution $\omega$ of \eqref{operator_form}, and the first part of Theorem \ref{thm:cgo} is proved.
	
	Notice that $\omega$ also satisfies 
	\begin{equation}\label{operator_form_3}
		\omega=L^{-1}_{\xi}(H(1+\omega)).
	\end{equation}
	Since $H$ is bounded and by Lemma \ref{zs_main_lemma} again, there exists a constant $C>0$ such that
	$$\|\omega(x;\xi)\|_{L^2(\Omega)}\leq C|\xi|^{-1},$$
	$$ \|\omega(x;\xi)\|_{H^1(\Omega)}\leq C $$
	when $|\xi|$ is large enough.
\end{proof}
Since we mainly work in H\"{o}lder space and we only get exponentially growing solutions in $H^1(\Omega)$, we still need the following denseness property.
\begin{lem}\label{denseness}
	Let $u_0\in C^1(\Omega)$ and $q\in L^{\infty}(\Omega)$. For any solution $ m\in H^1(\Omega)$ to 
	\begin{equation}\label{Runge 1}
		\Delta m+\nabla u_0\cdot\nabla m+ qm=0,
	\end{equation}
	and any $\eta>0$, there exists a solution  $\hat{m}\in C^{2+\alpha}(\Omega)$ to \eqref{Runge 1} such that
	$$\|m-\hat{m}\|_{L^2(\Omega)}\leq \eta. $$
\end{lem}
\begin{proof}
	Since $m\in H^1(\Omega)$, we have $m|_{\Sigma}\in L^2(\Sigma)$. Then by using the fact that $C^{\infty}(\Sigma)$ is dense in $L^2(\Sigma)$, there exist $M\in C^{\infty}(\Sigma)$ such that $\|M- m|_{\Sigma}\|_{L^2(\Omega)}\leq \eta$ for any $\eta>0.$

	Let $\hat{m}$ be the solution of
	\begin{equation*}
		\begin{cases}
			\Delta \hat{m}+\nabla u_0\cdot\nabla \hat{m}+ q\hat{m}=0 &\text{ in } \Omega\\
			\hat{m}=M &\text{ in }\Sigma.
		\end{cases}
	\end{equation*}
	Then $\hat{m}\in C^{2+\alpha}(\Omega)$ and $m-\hat{m}$ satisfies
	\begin{equation*}
		\begin{cases}
			\Delta (m-\hat{m})+\nabla u_0\cdot\nabla (m-\hat{m})+ q(m-\hat{m})=0 &\text{ in } \Omega\\
			(m-\hat{m})=m|_{\Sigma}-M &\text{ in }\Sigma.
		\end{cases}
	\end{equation*}
	Therefore,  for any $\eta>0$,
	$$\|m-\hat{m}\|_{L^2(\Omega)}\leq \|m|_{\Sigma}-M\|_{L^2(\Sigma)}\leq\eta. $$
\end{proof}   
\subsection{Proof of Theorem \ref{zs_main_result}}

Now we are ready to give the proof of Theorem \ref{zs_main_result}.
\begin{proof}
	Let $(u_{0,i},m_{0,i})$ be the solutions of \eqref{eq:zs_main} which satisfy \eqref{m=e^{-v}} with corresponding running cost $F_i(x,m)$, for $i=1,2$.
	
	Since $F_i(x,m_{0,i})=0$ , we have 
	\begin{equation}\label{v0eq}
		\begin{cases}
			-\Delta u_{0,i}+\frac{1}{2}|\nabla u_{0,i}|^2=0 & \text{ in }\Omega\\
			\partial_{\nu}u_{o,i}=0 &\text{ in }\Sigma.
		\end{cases}
	\end{equation}
	This shows that $\nabla u_{0,1}=\nabla u_{0,2}=:\nabla u_0$, and then $\Delta u_{0,1}=\Delta u_{0,2}=:\Delta u_0$. By \eqref{m=e^{-v}}, we also have 
    \[m_{0,1}=m_{0,2}=:m_0.\]
	
	As discussed in Section \ref{HLM}, the first-order linearization system of \eqref{eq:zs_main} in $\Omega$ is given, for $i=1,2$, by:
	\begin{equation}\label{eq:zs_linear_12}
		\begin{cases}
			-\Delta u^{(1)}_i+\nabla u_{0}\cdot\nabla u^{(1)}_i=F_i^{(1)}m^{(1)}_i  & \text{ in } \Omega\\
			-\Delta m^{(1)}_i-\text{div}(m_{0,i}\nabla u^{(1)}_i)-\text{div}(m^{(1)}_i\nabla u_{0})=0  & \text{ in } \Omega.\\
		\end{cases}
	\end{equation}
	We have shown in \eqref{calc} that \eqref{eq:zs_linear_12} implies that
	\begin{equation}
		\Delta m^{(1)}_i+\nabla u_{0}\cdot\nabla m^{(1)}_i+(\Delta u_{0}-F_i^{(1)}m_{0}) m^{(1)}_i=0.
	\end{equation}
	Let $\overline{m}:=m^{(1)}_1-m^{(1)}_2$. Since $\mathcal{M}_{F_1}=\mathcal{M}_{F_2}$, we have
	\begin{equation}\label{m1-m2}
		\begin{cases}
			\Delta \overline{m} +\nabla u_{0}\cdot \nabla \overline{m}+\overline{m}\Delta u_{0}- F_1^{(1)}m_{0}\overline{m}=(F_1^{(1)}m_{0}-F_2^{(1)}m_{0})m_2 & \text{ in } \Omega\\ 
			\overline{m}=\partial_{\nu}\overline{m}=0,& \text{ in } \Sigma. 
		\end{cases}
	\end{equation}
	Let $u$ be a solution of the following equation:
	\begin{equation*}
		\Delta u-\nabla u_{0}\cdot\nabla u-F_1^{(1)}m_{0}u=0.
	\end{equation*}
	Then we have
	\begin{equation}\label{integral_by_part}
		\begin{aligned}
			\int_{\Omega}(F_1^{(1)}m_{0}-F_2^{(1)}m_{0})m_2u \, dx 
			&=\int_{\Omega} (\Delta \overline{m}+\nabla u_{0}\cdot\nabla \overline{m} +\overline{m}\Delta u_{0}- F_1^{(1)}\overline{m}m_0)u\, dx\\
			&=\int_{\Omega} \overline{m}\Delta u-\overline{m}\nabla u_{0}\cdot\nabla u-\overline{m}u\Delta u_{0}+\overline{m}u\Delta u_{0}-F_1^{(1)}\overline{m}m_{0}u\, dx\\
			&\quad +
			\int_{\Sigma}u\partial_{\nu}\overline{m}-\overline{m}\partial_{\nu}u      +\overline{m}u\partial_{\nu}u_{0}\,ds\\
			&=\int_{\Omega}\overline{m}(\Delta u-\nabla u_{0}\cdot\nabla u-F_1^{(1)}m_{0}u)\,dx\\
			&=0,
		\end{aligned}
	\end{equation}
	since $\mathcal{M}_{F_1}=\mathcal{M}_{F_2}$.
	
	Now by Theorem \ref{thm:cgo}, there exist $\omega_i(x;\xi)$ and $\xi_i\in \mathbb{C}^n$ with $\xi_i\cdot\xi_i=0$ such that 
	\begin{equation}
		\begin{aligned}
			&\hat{m}_2=e^{\xi_1\cdot x-\frac{1}{2}u_0}(1+\omega_1(x;\xi_1))\\
			&u=e^{\xi_2\cdot x+\frac{1}{2}u_0}(1+\omega_2(x;\xi_2)),
		\end{aligned}
	\end{equation}
	and $\omega_j(x;\xi)$, $j=1,2$, satisfy the decay conditions
	$$\lim\limits_{|\xi_i|\to\infty}\|\omega_j \|_{L^2(\Omega)}=0.$$
	Let $k\in\mathbb{R}^n\backslash\{0\}$, then there exists $a,b\in\mathbb{R}^n,|a|=|b|=|k|$ such that $\{k,a,b\}$ is an orthogonal basis of $\mathbb{R}^n$. Now we choose
	\begin{equation}\label{choose_xi}
		\begin{aligned}
			&\xi_1=\frac{i}{2}k+\left(\sqrt{R^2+\frac{1}{16}}+i\sqrt{R^2-\frac{1}{16}}\right)a+\left(\sqrt{R^2+\frac{1}{16}}-i\sqrt{R^2-\frac{1}{16}}\right)b\\
			&\xi_2=\frac{i}{2}k-\left(\sqrt{R^2+\frac{1}{16}}+i\sqrt{R^2-\frac{1}{16}}\right)a-\left(\sqrt{R^2+\frac{1}{16}}-i\sqrt{R^2-\frac{1}{16}}\right)b,
		\end{aligned}
	\end{equation}
	where $R\in\mathbb{R}.$
	Then we have, for $j=1,2$,
	$$\xi_j\cdot\xi_j=0,$$
	$$\xi_j\cdot\overline{\xi_j}=\frac{1}{4}|k|^2+4R^2|k|^2.$$
	By using this construction, \eqref{integral_by_part} and Lemma \ref{denseness} implies that
	\begin{equation}
		\int_{\Omega} (F_1^{(1)}m_0-F_2^{(1)}m_0)e^{ik\cdot x}(1+\omega_1(x;\xi_1))(1+\omega_2(x;\xi_2))\,dx=0.
	\end{equation}
	Letting $R\to\infty$, we have $$ 	\int_{\Omega} (F_1^{(1)}m_0-F_2^{(1)}m_0)e^{ik\cdot x}\,dx=0.$$
	Since $k\in\mathbb{R}^n\backslash\{0\}$ is an arbitrary vector, we have
	$F_1^{(1)}(x)m_0(x)-F_2^{(1)}(x)m_0(x)=0.$ and since $m_0(x)>0$ for all $x\in\Omega$ by construction in \eqref{m=e^{-v}}, we have
	$$F_1^{(1)}(x)=F^{(1)}_2(x),$$
	and we obtain the uniqueness of the first-order Taylor coefficients of $F$. 
	
	For the higher-order Taylor coefficients, we apply higher-order linearization and use the same method as in the proof of Theorem \ref{der F}. Consequently, we obtain $F_1(x,m)=F_2(x,m).$
\end{proof}

\section{Time-dependent MFG}\label{sec:TimeMFGProof}

Next, we return to the time-dependent case. We assume that the $\mathcal{H}$ is a quadratic form, i.e. for a smooth Riemannian metric $A=(g_{kj}(x))_{k,j=1}^n$, the Hamiltonian $\mathcal{H}$ defined in the phase space is of the form
\begin{equation}\label{eq:ham1}
\mathcal{H}(x,p)=\sum_{kj}g_{kj}p_kp_j,
\end{equation}
where $p=\nabla u(x, t)$ is the momentum variable. This signifies that the energy Lagrangian of the MFG system is kinetic, which is often found in practical situations (cf. \cite{ding2022mean} ). 
\begin{defn}\label{def:A}
    A Riemannian metric $A$ is called admissible if it belongs to some fixed known conformal class $C_g$, i.e. $A(x)=\kappa(x)g(x)$ for $\kappa(x)\in C^\infty(\Omega)$ and $g(x)$ is a known Riemannian metric.
\end{defn}
Next, we introduce the admissible class of $F$.
\begin{defn}\label{AdmissClass1}
	We say that  $U(x,z):\mathbb{R}^n\times\mathbb{C}\to\mathbb{C}$ is admissible, denoted by $U\in\mathcal{B}$, if the following two conditions are satisfied:
	\begin{enumerate}
		\item[(i)] The map $z\mapsto U(\cdot,z)$ is holomorphic with value in $C^{2+\alpha}(\mathbb{R}^n)$ for some $\alpha\in(0,1)$;
			\item[(ii)] $U(x,m_0(x))=0$ for all $x\in\mathbb{R}^n$ and given $m_0(x)$ depending only on $x$;
            \item[(iii)] $U^{(1)}(x)=0$ for all $x\in \Omega$.
		\end{enumerate} 	
\end{defn}

The quadratic MFG system is then given by
\begin{equation}\label{eq:MFG2}
    \begin{cases}
        -\partial_t u(x,t) -\Delta u(x,t) + [\nabla u(x,t)]^TA(x)\nabla u(x,t) = F(x,m) &\quad \text{in }Q,\\
        \partial_t m(x,t) -\Delta m(x,t) - 2\text{div}\left(m(x,t)[\nabla u(x,t)]^TA(x)\right) = 0  &\quad \text{in }Q,\\
        u(x,T)=u_T,\quad m(x,0)=f(x) &\quad \text{in }\Omega,
    \end{cases}
\end{equation} 
for bounded positive real function $(x,t)>0$, $(x,t)\in C^{1,0}(\tilde{Q})$.
Consider the measurement map 
$\mathcal{M}_{F,U,A}$ given by \[\mathcal{M}_{F,U,A}=\left.\left(u,\nabla u,m,\nabla m\right)\right|_{\Gamma}\to F,U,A \]
associated to \eqref{eq:MFG2}, where $U=(u_0,m_0)$ and $\Gamma:=\Sigma\times(0,T)$. 
First, we define
\begin{defn}\label{stable}
Let $F\in\mathcal{A}$. Let $g,h\in C^{2+\alpha,1+\frac{\alpha}{2}}(\Gamma)$ satisfy the following  compatibility conditions:
\begin{equation}\label{compatibility conditions }
\begin{cases}
     g(x,T)=0&\quad \text{ in }\Sigma,\\
      -\partial_t g -\Delta g + [\nabla u_0]^TA\nabla g+ [\nabla g]^TA\nabla u_0=0 &\quad \text{ in }\Sigma,\\
     h(x,0)=0&\quad \text{ in }\Sigma,\\
     \partial_t h -\Delta h - 2\text{div}\left( m_0[\nabla g]^TA\right) -2\text{div} \left( h[\nabla u_0]^TA\right)= 0 &\quad \text{ in }\Sigma.
\end{cases}
\end{equation}

    A solution $U=(u_0,m_0)$ is called a \emph{stable solution} of \eqref{eq:MFG2} if for any
    $g,h$ that satisfy the compatibility conditions, the system
  \begin{equation}\label{MFG2:stable}
    \begin{cases}
        -\partial_t u^{(1)}(x,t) -\Delta u^{(1)}(x,t) + [\nabla u_0(x)]^TA(x) \nabla u^{(1)}(x,t)\\\quad\quad\quad\quad\quad\quad\quad\quad\quad\quad\quad\quad + [\nabla u^{(1)}(x,t)]^TA(x) \nabla u_0(x)= 0 &\quad \text{in }Q,\\
        \partial_t m^{(1)}(x,t) -\Delta m^{(1)}(x,t) - 2\text{div}\left( m_0(x)[\nabla u^{(1)}(x,t)]^TA(x)\right) \\\quad\quad\quad\quad\quad\quad\quad\quad\quad\quad\quad\quad -2\text{div} \left( m^{(1)}(x,t)[\nabla u_0(x)]^TA(x)\right)= 0  &\quad \text{in }Q,\\
        u^{(1)}(x,t)=g,\quad m^{(1)}(x,t)=h&\quad \text{in }\Gamma,\\
        u^{(1)}(x,T)=0, \quad m^{(1)}(x,0) = 0  &\quad \text{in }\Omega.
    \end{cases}
\end{equation}
admits a unique solution $(u,m)\in [ C^{2+\alpha,1+\frac{\alpha}{2}}(Q)]^2.$
\end{defn}

\begin{defn}
    A solution $U=(u_0,m_0)$ is a \emph{stationary solution} of \eqref{eq:MFG2} if it satisfies the following system
    \begin{equation}\label{MFG2Linear0}
    \begin{cases}
        -\Delta u_0(x) + [\nabla u_0(x)]^TA(x) \nabla u_0(x)= 0&\quad \text{in }Q,\\
        -\Delta m_0(x) - 2\text{div}\left( m_0(x)[\nabla u_0(x)]^TA(x)\right)= 0  &\quad \text{in }Q.
    \end{cases}
\end{equation}
\end{defn}

    Then, we are able to reconstruct the running costs, Hamiltonian and stable stationary state solutions, through measurements for general analytic cost functions, and the following uniqueness result holds:

\begin{thm}\label{mainthm}
For $i=1,2$, let $A_i$ be an admissible Riemannian metric and $F_i\in\mathcal{A}$ such that $U_i=(u_{0,i},m_{0,i})\in [C^{2+\alpha}(\Omega)]^2$ is a stable stationary solution of the following system:
   \begin{equation}\label{eq:MFG2i}
    \begin{cases}
        -\partial_t u(x,t) -\Delta u(x,t) + [\nabla u(x,t)]^TA_i(x)\nabla u(x,t) = F_i(x,m) &\quad \text{in }Q,\\
        \partial_t m(x,t) -\Delta m(x,t) - 2\text{div}\left(m(x,t)[\nabla u(x,t)]^TA_i(x)\right) = 0  &\quad \text{in }Q,\\
        u(x,T)=u_{T,i},\quad m(x,0)=f_i(x) &\quad \text{in }\Omega,
    \end{cases}
\end{equation}
such that the Hamiltonian $\mathcal{H}_i=[\nabla u]^TA_i\nabla u$ is non-degenerate at the stationary state $U_i$. 
 Let $\mathcal{M}_{F_i,A_i,U_i}$ be the associated measurements. Suppose \[\mathcal{M}_{F_1,A_1,U_1}=\mathcal{M}_{F_2,A_2,U_2},\] 
     Then, one has \[F_1=F_2,\quad \quad U_1=U_2\]
     and 
     \[A_1(x)=A_2(x)\quad \text{ in }Q,\]
     up to the conformal class $C_g$.
\end{thm}

\subsection{Recovery of Stationary Solution and Hamiltonian}
Before we begin the proof of the main theorem, we first give an essential auxiliary lemma regarding the construction of complex geometric optics solutions, which is a result of \cite{ParabolicConvectionVectorCGO}.

\begin{thm}\label{thm:cgotime} Let $\psi:=\rho^2t+\rho\zeta\cdot x$ for some fixed $\zeta\in\mathbb{S}^{n-1}$. Suppose $\phi\in [C^1(Q)]^n$ and $q\in L^{\infty}(Q)$. For $\xi\in\mathbb{R}^n$ and $\tau\in\mathbb{R}$, and any arbitrary $\chi\in C_c^\infty((0,T))$ and $\rho>0$,
	\begin{enumerate}
	    \item there exist solutions $w\in H^1(0,T;H^{-1}(\Omega))\cap L^2(0,T;H^1(\Omega))$ to the forward parabolic equation \begin{equation}\label{forwardcgoeq}\begin{cases}
	\partial_t w - \Delta w -\phi\cdot\nabla w+ qw=0\quad &\text{ in }Q,\\
 w(x,t)=0 &\text{ on }\Sigma,\\
 w(x,0)=0 &\text{ in }\Omega,\end{cases}
\end{equation}  of the form
	\begin{equation}\label{cgo1}
		w=e^\psi\left(\chi(t)e^{-i(x,t)\cdot(\xi,\tau)}\exp\left(\frac{1}{2}\int_0^\infty \zeta\cdot\phi(x+s\xi)\,ds\right) +z_+(x,t)\right)
	\end{equation}
	where $z_+(x,t;\phi,q)\in H^1(0,T;H^{-1}(\Omega))\cap L^2(0,T;H^1(\Omega))$ satisfies the following condition:
	\begin{equation}\label{decay1}
		\lim_{\rho\to\infty}\norm{z_+(x,t;\phi,q)}_{L^2(Q)}=0.
	\end{equation}
 
 \item There exist solutions $v\in H^1(0,T;H^{-1}(\Omega))\cap L^2(0,T;H^1(\Omega))$ to the backward parabolic equation \begin{equation}\label{backwardcgoeq}\begin{cases}
	-\partial_t v - \Delta v +\phi\cdot\nabla v+ qv=0\quad &\text{ in }Q,\\
 v(x,T)=0 &\text{ in }\Omega,\end{cases}
\end{equation} 
of the form 
\begin{equation}\label{cgo2}
		v=e^{-\psi}\left(\chi(t)\exp\left(-\frac{1}{2}\int_0^\infty \zeta\cdot\phi(x+s\xi)\,ds\right) +z_-(x,t)\right)
	\end{equation}
	where $z_-(x,t;\phi,q)\in H^1(0,T;H^{-1}(\Omega))\cap L^2(0,T;H^1(\Omega))$ satisfies the following condition:
	\begin{equation}\label{decay2}
		\lim_{\rho\to\infty}\norm{z_-(x,t;\phi,q)}_{L^2(Q)}=0.
	\end{equation}
	\end{enumerate}
\end{thm}

Since we mainly work in H\"{o}lder space and we only get exponentially growing solutions in $H^1(\Omega)$, we still need the following denseness property.

\begin{lemma}[Runge approximation with full data]\label{RungeDenseApprox}
   Suppose $\phi\in [C^1(Q)]^n$ and $q\in L^{\infty}(Q)$. Then for any solutions $w,v\in L^2(0,T;H^1(\Omega))\cap H^1(0,T;H^{-1}(\Omega))$	to \begin{equation}\label{forwardcgoeqapprox}\begin{cases}
	\partial_t w - \Delta w -\phi\cdot\nabla w+ qw=0\quad &\text{ in }Q,\\
 w(x,0)=0 &\text{ in }\Omega,\end{cases}
\end{equation} 
and 
\begin{equation}\label{backwardcgoeqapprox}\begin{cases}
	-\partial_t v - \Delta v +\phi\cdot\nabla v+ qv=0\quad &\text{ in }Q,\\
 v(x,T)=0 &\text{ in }\Omega,\end{cases}
\end{equation} respectively, 
and  any $\eta,\eta'>0$, there exist solutions $W,V\in C^{2+\alpha,1+\alpha/2}(\overline{Q})$ to \begin{equation}\label{forwardcgoeqapproxV}\begin{cases}
	\partial_t W - \Delta W -\phi\cdot\nabla W+ qW=0\quad &\text{ in }Q,\\
 W(x,0)=0 &\text{ in }\Omega,\end{cases}
\end{equation} 
and 
\begin{equation}\label{backwardcgoeqapproxW}\begin{cases}
	-\partial_t V - \Delta V +\phi\cdot\nabla V+ qV=0\quad &\text{ in }Q,\\
 V(x,T)=0 &\text{ in }\Omega,\end{cases}
\end{equation}  respectively
	such that
	\[\norm{W-w}_{L^2(Q)}<\eta\quad \text{ and }\quad \norm{V-v}_{L^2(Q)}<\eta.\]
\end{lemma}
\begin{proof}
    We will show the case for the forward parabolic equation, and the backward one can be proved similarly. Define
	\[X=\left\{W\in C^{2+\alpha,1+\alpha/2}(\overline{Q})\,\Big|\, 
	W \text{ is a solution to } \eqref{forwardcgoeq} \right\}\] and
	\[Y=\left\{w\in L^2(0,T;H^1(\Omega))\cap H^1(0,T;H^{-1}(\Omega)) \,\Big|\, w \text{ is a solution to } \eqref{forwardcgoeq} \right\}.\] 
	We aim to show that $X$ is dense in $Y$. By the Hahn-Banach theorem, it suffices to prove the following statement: If $f\in L^2(Q)$ satisfies
	$$\int_{Q}fW\,dxdt=0, \quad \text{ for any } W\in X,$$
	then
	$$\int_{Q}fw\,dxdt=0, \quad \text{ for any }w\in Y.$$
	To this end, 
	let $f\in L^2(Q)$ and suppose $\int_{Q}fW\,dxdt=0$,  for any  $W\in X$. Consider     
    \begin{equation}\begin{cases}
	-\partial_t \tilde{W} - \Delta \tilde{W} +\phi\cdot\nabla \tilde{W}+\tilde{W}\text{div}\phi+ q\tilde{W}=f\quad &\text{ in }Q,\\
    \tilde{W}(x,t)=0\quad&\text{ on }\Sigma,\\
 \tilde{V}(x,T)=0 &\text{ in }\Omega,\end{cases}
\end{equation} 
with solution in $H^{2, 1}(Q)$. For any $W\in X$, one has 
\begin{align*}
	0=&\int_{Q}fW\,dx\,dt=\int_Q (-\partial_t \tilde{W} - \Delta \tilde{W} +\phi\cdot\nabla \tilde{W}+\tilde{W}\text{div}\phi+ q\tilde{W}) W\,dx\,dt=\int_{\Sigma} \p_\nu \overline{W}W\, dS\,dt.
\end{align*}
Since $W|_{\Sigma}$ can be arbitrary function, which  is compactly supported on $\Sigma$, we must have $\partial_\nu \overline{W}=0$ on $\Sigma$.
Thus, for any $ w\in Y$,
\begin{align*}
	\int_{Q}fw \,dx\,dt&=\int_Q (-\partial_t \tilde{W} - \Delta \tilde{W} +\phi\cdot\nabla \tilde{W}+\tilde{W}\text{div}\phi+ q\tilde{W}) w\,dxdt=\int_{\Sigma}\p_\nu\overline{W}w\,dSdt=  0,
\end{align*}
which verifies the assertion.
\end{proof}

Next, we give the proof of Theorem \ref{mainthm}, beginning with the recovery of the stationary solution and the Hamiltonian.

Let $(u_{0,i},m_{0,i})$ be the solutions of \eqref{eq:MFG2i} which satisfy \eqref{MFG2Linear0} with corresponding Hamiltonian $A_i(x)$ and running cost $F_i(x,m)$, for $i=1,2$, i.e. 
\begin{equation}\label{MFG2Linear0i}
    \begin{cases}
        -\Delta u_{0,i}(x) + [\nabla u_{0,i}(x)]^TA_i(x) \nabla u_{0,i}(x)= 0&\quad \text{in }\Omega,\\
        -\Delta m_{0,i}(x) - 2\text{div}\left( m_{0,i}(x)[\nabla u_{0,i}(x)]^TA_i(x)\right)= 0  &\quad \text{in }\Omega.
    \end{cases}
\end{equation}
At the same time, the corresponding first order linearized solutions $(u^{(1)}_i,m^{(1)}_i)$ with the Hamiltonian $A_i(x)$ and running cost $F_i(x,m)$, for $i=1,2$, satisfy 
\begin{equation}\label{MFG2Linear1i}
    \begin{cases}
        -\partial_t u^{(1)}_i(x,t) -\Delta u^{(1)}_i(x,t) + [\nabla u_{0,i}(x)]^TA_i(x) \nabla u^{(1)}_i(x,t)
        \\\quad\quad\quad\quad\quad\quad\quad\quad\quad\quad\quad\quad + [\nabla u^{(1)}_i(x,t)]^TA_i(x) \nabla u_{0,i}(x)= 0 &\quad \text{in }Q,\\
        \partial_t m^{(1)}_i(x,t) -\Delta m^{(1)}_i(x,t) - 2\text{div}\left( m_{0,i}(x)[\nabla u^{(1)}_i(x,t)]^TA_i(x)\right) \\\quad\quad\quad\quad\quad\quad\quad\quad\quad\quad\quad\quad -2\text{div} \left(m^{(1)}_i(x,t)[\nabla u_{0,i}(x)]^TA_i(x)\right) = 0  &\quad \text{in }Q,\\
        u^{(1)}_i(x,t)=g_1,\quad m^{(1)}_i(x,t)=h_1 &\quad \text{in }\Gamma,\\
        u^{(1)}_i(x,T)=0, \quad m^{(1)}(x,0) = 0  &\quad \text{in }\Omega.
    \end{cases}
\end{equation}

Observe that we have the identity
\[[\nabla u_{0,i}]^TA_i \nabla u^{(1)}_i+[\nabla u^{(1)}_i]^TA_i \nabla u_{0,i}=\left(\left[[\nabla u_{0,i}]^TA_i\right]^T+A_i\nabla u_{0,i}\right)\cdot\nabla u^{(1)}_i = 2A_i\nabla u_{0,i}\cdot \nabla u^{(1)}_i. \]
Let $\tilde{u}=u^{(1)}_1-u^{(1)}_2$, $\tilde{m}=m^{(1)}_1-m^{(1)}_2$ and $q_i=2A_i\nabla u_{0,i}$. Then, taking the difference of the two equations for $i=1,2$ for the first equation in \eqref{MFG2Linear1i}, we have 
\begin{equation}\label{MFG2Linear1U}
    \begin{cases}
        -\partial_t \tilde{u} -\Delta \tilde{u} + q_1\cdot \nabla \tilde{u} + (q_1-q_2)\cdot\nabla u^{(1)}_2= 0 &\quad \text{in }Q,\\
        \tilde{u}(x,t)=0,&\quad \text{in }\Gamma,\\
        \tilde{u}(x,T)=0&\quad \text{in }\Omega.
    \end{cases}
\end{equation}
Furthermore, by Theorem \ref{thm:cgotime}, we have a solution of $u_2^{(1)}$ to \eqref{MFG2Linear1i} of the form \eqref{cgo2}.

Consider the solution $w$ of the form \eqref{cgo1} to 
\begin{equation}\label{TestFunction}
    \begin{cases}
        \partial_t w - \Delta w - q_1\cdot \nabla w+w\text{div} q_1= 0 &\quad \text{in }Q,\\
        w(x,0)=0&\quad \text{in }\Omega.
    \end{cases}
\end{equation}
Multiplying \eqref{MFG2Linear1U} by $w$ and using integration by parts, we have
\[ \int_Q\tilde{u}\partial_t w -\int_Q \tilde{u}\Delta w - \int_Q \tilde{u}(q_1\cdot\nabla w + w\text{div} q_1) + \int_Q w(q_1-q_2)\cdot\nabla u^{(1)}_2(x,t)= 0.\] 
When $\mathcal{M}_{F_1,A_1,U_1}=\mathcal{M}_{F_2,A_2,U_2}$, this gives 
\begin{equation}\label{IntegralIdentity}\int_Q w(q_1-q_2)\cdot\nabla u^{(1)}_2(x,t)\leq C\rho^{\frac{1}{2}},\end{equation}  by following the argument of Section 5 of \cite{ParabolicConvectionVectorCGO}. 
This means that when $\rho\to\infty$, 
\begin{equation}\label{IntegralIdentity2}\frac{1}{\rho}\int_Q w(q_1-q_2)\cdot\nabla u^{(1)}_2(x,t)\to0.\end{equation}

Substituting into \eqref{IntegralIdentity2} the expressions for $u_2^{(1)}$ and $w$ from \eqref{cgo2} and \eqref{cgo1} respectively, in the limit $\rho\to\infty$, we have 
\[
\int_Q \chi^2(t)e^{-i(x,t)\cdot(\xi,\tau)} (q_1-q_2)\cdot\zeta\,dx\,dt= 0\quad\text{ for all }\zeta\in\mathbb{S}^{n-1}
\]
Since the above identity holds for all $\chi\in C_{c}^{\infty}(0,T)$, therefore we get 
\[
\int_{\mathbb{R}^n} e^{-ix\cdot\xi} (q_1-q_2)\cdot\zeta\,dx= 0\quad\text{ for all }\zeta\in\mathbb{S}^{n-1}.
\]
Identifying this as the Fourier transform with respect to $\xi$, and since this holds for all unit vectors $\zeta$, this gives that 
\begin{equation}\label{qEq}q_1=q_2\quad \text{ in }\Omega,\end{equation} and we denote $q=q_1=q_2$.

Next, we consider \eqref{MFG2Linear0i}. Then the first equation can be rewritten as 
\begin{equation}
        -\Delta u_{0,i}(x) + \frac{1}{2}q(x)\cdot \nabla u_{0,i}(x)= 0\quad \text{in }\Omega.
\end{equation}
Observe that this is a second order elliptic equation in $u_{0,i}$. Since $\mathcal{M}_{F_1,A_1,U_1}=\mathcal{M}_{F_2,A_2,U_2}$, in particular $u_{0,1}=u_{0,2}$ and $\partial_\nu u_{0,1}=\partial_\nu u_{0,2}$ on $\partial\Omega$, by the unique continuation principle for elliptic operators \cite{KochTataru2001UCPElliptic}, it must hold that 
\begin{equation}\label{uEq}u_{0,1}=u_{0,2} \quad \text{ in }\Omega,\end{equation} and we denote $u_0=u_{0,1}=u_{0,2}$.

In particular, it holds that $\nabla u_{0,1}=\nabla u_{0,2}$. By Definition \ref{def:A}, for $A_1=\kappa_1g$ and $A_2=\kappa_2g$, \eqref{qEq} can be rewritten as \[\kappa_1g\nabla u_{0,1}=\kappa_2g\nabla u_{0,2},\] from which we obtain \[\kappa_1=\kappa_2.\] This means that up to the conformal class $C_g$, 
\[A(x):=A_1(x)=A_2(x)\quad\text{ in }Q.\]

Therefore, the second equation in \eqref{MFG2Linear0i} can now be rewritten as 
\begin{equation}\label{MFG2Linear0M}
    -\Delta m_{0,i}(x) - \text{div}\left(q(x) m_{0,i}(x)\right)= 0  \quad \text{in }\Omega.
\end{equation}
Once again, this is a second order elliptic equation, this time in $m_{0,i}$. By the equality of the measurement maps $\mathcal{M}_{F_1,A_1,U_1}=\mathcal{M}_{F_2,A_2,U_2}$, in particular $m_{0,1}=m_{0,2}$ and $\partial_\nu m_{0,1}=\partial_\nu m_{0,2}$ on $\partial\Omega$, once again applying the unique continuation principle for elliptic operators, it must hold that 
\begin{equation}\label{mEq}m_0:=m_{0,1}=m_{0,2} \quad \text{ in }\Omega.\end{equation}

Substituting these results into \eqref{MFG2Linear1i}, we have that $\left(u_i^{(1)},m_i^{(1)}\right)$ satisfy a heat-type equation for $i=1,2$. When $\mathcal{M}_{F_1,A_1,U_1}=\mathcal{M}_{F_2,A_2,U_2}$, by the uniqueness of solutions for heat equations, it must hold that 
\begin{equation}\label{UM1Eq}u_1^{(1)}=u_2^{(1)}\quad \text{ and }m_1^{(1)}=m_2^{(1)}\quad \text{ in }Q.\end{equation}

\subsection{Recovery of Running Cost}

Now that we have obtained the equality of the stable stationary solutions $U=(u_0,m_0)$ and the uniqueness of Hamiltonian $A$, we can proceed with the unique identifiability results for $F$. The main ingredient for the recovery of $F$ is the following unique continuation principle:
\begin{thm}\label{UCP}

Let $\Sigma'\subset \Sigma$ be an arbitrarily chosen
non-empty relatively open sub-boundary. For $h\in L^\infty(Q)$, we assume that $(u_i,m_i) \in H^{2,1}(Q)\times H^{2,1}(Q)$ 
satisfy, for $i=1,2$, 
\begin{equation}\label{eq:UCPeq}
    \begin{cases}
        -\partial_t u_i(x,t) -\mathcal{L}_1 u_i(x,t) = E_i(x,t)+h(x,t)m_i(x,t) &\quad \text{in }Q,\\
        \partial_t m_i(x,t) -\mathcal{L}_2 m_i(x,t) + \mathcal{L}_3 u_i(x,t) = G_i(x,t)  &\quad \text{in }Q
    \end{cases}
\end{equation}
for some regular second-order elliptic operators $\mathcal{L}_1$, $\mathcal{L}_2$ and $\mathcal{L}_3$, such that the coefficients (depending on $x$ and $t$) of $\mathcal{L}_3$ are bounded,
and
\[\begin{cases}
u_i, \nabla u_i, \Delta u_i \in L^{\infty}(Q),\\
m_i, \nabla m_i \in L^{\infty}(Q), \quad 
\partial_t(u_1-u_2),\partial_t(m_1-m_2) \in L^2((\Sigma\backslash\Sigma')\times(0,T)).
\end{cases}\]
Then $u_1=u_2$, $\nabla u_1 = \nabla u_2$, $m_1=m_2$ and $\nabla m_1
= \nabla m_2$ on $\Sigma' \times (0,T)$ implies $u_1=u_2$ and $m_1=m_2$ 
in $Q$.

\end{thm}

\begin{proof}
    The proof follows similarly as in \cite{imanuvilov2023unique}. Indeed, for arbitrarily fixed $t_0\in (0,T)$ and $\delta>0$ such that 
$0 < t_0 - \delta \le t_0 + \delta < T$, consider $I = (t_0-\delta, t_0+\delta)$ and $Q_I = \Omega \times I.$ Set  
\[
P_k w(x,t):= \partial_tw + (-1)^k a(x,t)\Delta w + R(x,t,w),
\quad k=1,2,
\]
where $a \in C^{2}(\overline{Q_I})$, $>0$ on $\overline{Q_I}$, and 
\begin{equation}\label{UCPeq1}
| R(x,t,w)| \leq C_0(| w(x,t)|
+ |\nabla w(x,t)|), \quad (x,t)\in Q_I.
\end{equation}
Moreover, let
\begin{equation}\label{UCPeq2}
    \varphi(x,t) = e^{\lambda(d(x) - \beta (t-t_0)^2)},
\end{equation}
where $\lambda>0$ is a sufficiently large parameter, $\beta>0$ is arbitrarily given, and
$d \in C^2(\overline{\Omega})$ is such that 
\begin{equation}
d>0 \quad \text{in }\Omega, \quad | \nabla d| > 0 \quad 
\text{on $\overline{\Omega}$}, \quad d=0 \quad\text{on }\partial\Omega\setminus
\Gamma, \quad \nabla d\cdot \nu \le 0 \quad \text{on }\partial\Omega\setminus 
\Gamma,
\end{equation}  which is known to exist. Then the following Carleman estimate (Lemma 1 of \cite{imanuvilov2023unique}) is known to hold: 
For $k=1,2$, there exist constants $s_0>0$ and $C_1, C_2>0$ such that 
\begin{equation}\label{UCPeq4}
    \int_{Q_I} \left( \frac{1}{s}(| \partial_tw|^2
+ | \Delta w|^2) + s |\nabla w|^2 
+ s^3| w|^2 \right)e^{2s\varphi} dxdt
\le Cs^4\int_{Q_I} | P_kw|^2 e^{2s\varphi} dxdt
+ C\mathcal{B}(w),
\end{equation}
for all $s > s_0$ and $w\in H^{2,1}(Q_I)$ satisfying 
$w\in H^1(\partial\Omega\times I)$, where
\begin{multline*}
 \mathcal{B}(w) := e^{Cs}\Vert w\Vert^2_{H^1(\Gamma\times I)}
+ s^3\int_{(\partial\Omega\setminus \Gamma)\times I}
(|w|^2 + | \nabla_{x,t}w|^2) e^{2s} dSdt\\
+ s^2\int_{\Omega} (| w(x,t_0-\delta)|^2
+ | \nabla w(x,t_0-\delta)|^2
+ | w(x,t_0+\delta)|^2
+ | \nabla w(x,t_0+\delta)|^2) e^{2s\varphi(x,t_0-\delta)} dx.
\end{multline*}

Next, we expand \eqref{eq:UCPeq} in the form 
    \begin{equation}\label{UCPeq5}
    \begin{cases}
        -\partial_t u_i(x,t) -\alpha\Delta u_i(x,t) + R_1(x,t,u_i) = E_i(x,t)+h(x,t)m_i(x,t) &\quad \text{in }Q,\\
        \partial_t m_i(x,t) - \beta \Delta m_i(x,t) +R_3(x,t,m_i) = \gamma\Delta u_i(x,t) + R_2(x,t,u_i) + G_i(x,t)  &\quad \text{in }Q,
    \end{cases}
\end{equation} for some bounded coefficients $\alpha(x,t)$, $\beta(x,t)$ and $\gamma(x,t)$, so that \[|R_1(x,t,u_i)|,|R_2(x,t,u_i)|\leq C\sum_{k=0}^1|\nabla^k u_i(x,t)|, \quad |R_3(x,t,m_i)|\leq C\sum_{k=0}^1|\nabla^k m_i(x,t)|\] 
for some different constants $C$. Then, taking the difference of the two systems for $i=1,2$, we can apply the Carleman estimate \eqref{UCPeq4}, and choose $s>0$ large enough to absorb terms on the right hand side of \eqref{UCPeq4}, to obtain the following Carleman estimate: 
\begin{multline}\label{UCPeq6}
 \int_{Q_I} \biggl( | \partial_t(u_1-u_2)|^2 
+ | (\Delta (u_1-u_2)|^2 
+ s^2|\nabla (u_1-u_2)|^2 + s^4| u_1-u_2|^2 
\\+ \frac{1}{s}(| \partial_t(m_1-m_2)|^2 
+ | \Delta (m_1-m_2)|^2) 
+ s|\nabla (m_1-m_2)|^2 + s^3|m_1-m_2|^2 \biggr)e^{2s\varphi} dxdt
\\\leq C_3\int_{Q_I} (s| F-\widetilde{F}|^2 + | G-\widetilde{G}|^2)e^{2s\varphi} dxdt\\+  C_4s(\mathcal{B}(u_1-u_2) + \mathcal{B}(m_1-m_2)) \quad \text{for all $s>s_0$}
\end{multline}
for some constants $s_0>0$ and $C_3,C_4>0$.

Next, we arbitrarily choose $t_0\in (0,T)$ and $\delta>0$ such that 
$0<t_0-\delta<t_0+\delta<T$.
We define 
\begin{equation}\label{UCPeq7}
d_0:= \min_{x\in \overline{\Omega}} d(x),\quad
d_1:= \max_{x\in \overline{\Omega}} d(x),\quad
0 < r < \left( \frac{d_0}{d_1}\right)^{\frac{1}{2}} < 1.
\end{equation}
We observe that to prove Theorem \ref{UCP}, it suffices to measure $u,m,\nabla u,\nabla m$ on 
$\Gamma \times (t_0-\delta,\, t_0+\delta)$ and prove the result in $\Omega \times (t_0-r\delta,\, t_0+r\delta)$. 
Indeed, since $t_0 \in (0,T)$ and $\delta>0$ can be arbitrarily chosen and the Carleman estimate is invariant with respect to $t_0$ provided that $0<t_0-\delta<t_0+\delta<T$, we can change $t_0$ over $(\delta, T-\delta)$ to obtain
$u_1=u_2$ and $m_1=m_2$ in $\Omega \times ((1-r)\delta,\, T-(1-r)\delta)$.
Since $\delta>0$ can be arbitrary, this means that 
$u_1=u_2$ and $m_1=m_2$ in $\Omega\times (0,T)$.

Consequently, we choose $\beta > 0$ in the weight of the Carleman estimate such that 
\[
\frac{d_1-d_0}{\delta^2 - r^2\delta^2} < \beta < \frac{d_0}{r^2\delta^2}.\]
Considering only the dominating terms in \eqref{UCPeq7} for large $s>0$, and observing that 
\[\varphi(x,t) = e^{\lambda(d(x) - \beta(t-t_0)^2)}
\geq e^{\lambda(d_0-\beta r^2\delta^2)}=:\mu_1\quad \text{in }
\Omega \times (t_0-r\delta,\, t_0+r\delta),\] 
we shrink the region of integration of the left hand side of \eqref{UCPeq7} to $\Omega \times (t_0-r\delta,\, t_0+r\delta)$ to obtain that 
\begin{multline}\label{UCPeq8}\norm{u_1-u_2}^2_{L^2(\Omega \times (t_0-r\delta,\, t_0+r\delta))}
+ \norm{m_1-m_2}^2_{L^2(\Omega \times (t_0-r\delta,\, t_0+r\delta))}
\\\leq C_5s^2M_1e^{-2s(\mu_1-1)} + C_6s^2M_2e^{-2s(\mu_1-\mu_2)}\end{multline}
for all large $s>0$ for some constants $C_5,C_6>0$, where $\mu_2:= e^{\lambda(d_1-\beta\delta^2)}$ and
\begin{align*}
& M_1:= \sum_{k=0}^1 
(\Vert\nabla^k_{x,t}(u_1-u_2)\Vert^2
_{L^2((\partial\Omega\setminus \Gamma)\times I)}
+ \Vert\nabla^k_{x,t}(m_1-m_2)\Vert^2
_{L^2((\partial\Omega\setminus \Gamma)\times I)}), \\
& M_2:= \sum_{k=0}^1 (\Vert (u_1-u_2)(\cdot,t_0 + (-1)^k\delta) \Vert^2
_{H^1(\Omega)} 
+ \Vert (m_1-m_2)(\cdot,t_0 + (-1)^k\delta) \Vert^2_{H^1(\Omega)}.
\end{align*}
Observing that 
$\mu_1 > \max\{ 1, \, \mu_2\}$, we take $s \to \infty$ in \eqref{UCPeq8} to obtain 
$u_1=u_2$ and $m_1=m_2$ in $\Omega \times (t_0-r\delta,\, t_0+r\delta)$.
Thus the proof is complete.
\end{proof}

With this result in hand, we can proceed to recover the running cost $F$, by making use of the higher orders of linearization. 

    Consider the second order linearization.
    Let $\bar{u}=u^{(1,2)}_1-u^{(1,2)}_2$ and $\bar{m}=m^{(1,2)}_1-m^{(1,2)}_2$. Then, since we have obtained the equality of the first order linearized solutions in \eqref{UM1Eq}, by taking the difference of the two second order linearized systems for $i=1,2$, we have that $(\bar{u},\bar{m})\in [H^{2,1}(Q)]^2$ is the solution to the system 
    \begin{equation}\label{eq:MFG2Diff2}
    \begin{cases}
        -\partial_t \bar{u}(x,t) -\Delta \bar{u}(x,t) + 2q(x)\cdot\nabla \bar{u}(x,t) =  [F^{(2)}_1(x) - F^{(2)}_2(x)]m^{(1)}(x,t)m^{(2)}(x,t) & \quad \text{in }Q,\\
        \partial_t \bar{m}(x,t) -\Delta\bar{m}(x,t) -2\text{div}\left( m_0(x)[\nabla \bar{u}(x,t)]^TA(x)\right) - \text{div}\left( \bar{m}(x,t) q(x)\right) = 0  &\quad \text{in }Q,\\
        \bar{u}(x,T)=0&\quad \text{in }\Omega,\\
        \bar{m}(x,0)=0 &\quad \text{in }\Omega,\\
        \bar{u}=\nabla \bar{u} = \bar{m} = \nabla \bar{m} = 0 &\quad \text{on }\Sigma,
    \end{cases}
    \end{equation}
    when $\mathcal{M}_{F_1,A_1,U_1}=\mathcal{M}_{F_2,A_2,U_2}$.

    Since $m^{(1)},m^{(2)}$ have been uniquely obtained, expanding $\text{div}\left( m_0[\nabla \bar{u}]^TA\right)$ and $\text{div}\left( \bar{m} q\right)$, we can view $D_i:=F^{(2)}_i(x)m^{(1)}(x,t)m^{(2)}(x,t)$ in \eqref{eq:UCPeq}, and apply the unique continuation principle given by Theorem \ref{UCP} to obtain \begin{equation}
        \bar{u}=\bar{m}\equiv0. 
    \end{equation}
    Substituting this into the first equation of \eqref{eq:MFG2Diff2}, we have that $D_1=D_2$. Choosing $m^{(1)},m^{(2)}\not\equiv0$, we obtain the result
    \[F^{(2)}_1(x)=F^{(2)}_2(x).\] 

    For the higher order Taylor coefficients of $F$, we make use of the high order linearized parabolic systems. The most important ingredient is that the non-linear terms in higher order systems only depend on the solutions of lower order terms. Therefore, we can apply mathematical induction and repeat similar argument to show that $F^{(k)}_1(x)=F^{(k)}_2(x)$ for all $k\in\mathbb{N}$. Hence we have the unique identifiability for the source functions, i.e. \[F_1=F_2.\]
    
    The proof is complete.

\chapter{Determining Internal Topological Structures Using Boundary Measurements}\label{chap:internaltop}

In this chapter, we are concerned with the inverse problem of determining anomalies in the state space associated with the stationary mean field game (MFG) system. 
We recover such anomalies from the behavior of the solution at the boundary of the domain. This anomaly is portrayed as a discontinuity in the Hamiltonian $H$ and the running cost $F$, which had always been assumed to be continuous in the previous chapters. Such an inverse problem for MFG has a wide range of applications. For instance, when MFGs are applied in the study of traffic flow or other forms of congestion management \cite{MFGCar2, MFGCrowd, MFGCar, MFGCrowd+Econs}, we can make use of boundary measurements of the traffic to detect an obstruction \cite{MFGStateConstraintsTrafficObstruct, ghattassi2023nonseparable}. An example is where we can measure the flow of cars in and out of a city to detect any bridge collapse or car accident \cite{MFGCarAccident}. Such applications can also be extended to other forms of flow optimization in engineering, including short circuits in smart grids \cite{MFGSmartGrid2} and breaks in supply chains \cite{MFGSupplyChain}. Moreover, MFGs can also be applied in the development of autonomous driving \cite{MFGAutoCar1,MFGAutoCar2,MFGAutoCar3}, and considering the inverse problem can help in the detection of corners and walls or other types of anomalous inhomogeneities. On the other hand, when MFGs are applied in the context of economics to model market dynamics, price formation or resource allocation \cite{MFGCrowd+Econs, CarmonaDelarue2018_1}, one can recover the external influences limiting the price, including governmental influences, limit ups/downs or irreversible investments \cite{MFGIrreversibleInvestJump}. In the context of social sciences, MFGs can also model epidemic spreading \cite{MFGDisease2} and disease propagation \cite{MFGDisease1}, as well as social network information spread \cite{MFGSocialNetwork}. In these cases, there might be gaps in the state space due to deserted areas of the country, vaccinated parts of the population, or internet censorship, which can be detected by considering the inverse problem.

\section{Smooth Internal Topological Structures}
We first begin with the situation where such anomalies have smooth boundary. 
We focus on the following MFG system
\begin{equation}
\label{um0}
\begin{cases}
 -\partial_tu(x,t)-\Delta u(x,t)+\frac{1}{2}|\nabla u(x,t)|^2=F(x,m)
\ \  &\mathrm{in}\ Q,\\
\partial_tm(x,t)-\Delta m(x,t)-\mathrm{div}(m(x,t)\nabla u(x,t))=0
 &\mathrm{in}\ Q,\\
u=0, \quad m=g_0 &\mathrm{on}\  \Sigma_1,\\
Bu=0,  \quad Bm=f_0 &\mathrm{on}\  \Sigma_2,\\
u(x,T)=0,\ m(x,0)=m_0(x) &\mathrm{in}\  \Omega\setminus\overline{D},
\end{cases}
\end{equation}
where $g_0$, $f_0$ are nonnegative constants on the boundary. 

The operator $B$ is defined as follows:
$$Bu=u\  \mathrm{or}\  Bu=\partial_v u.$$
In this MFG model, particular attention is paid to the internal topological boundary $\partial D$. In the case of the Dirichlet boundary condition $Bu=u$, this indicates that agents entering or exiting the game at a certain cost via $\partial D$ \cite{FerreiraGomesTada2019StationaryMFGDirichlet}. For instance, in the context of traffic management \cite{Cardaliaguet2018}, consider a large car park, denoted as $D$, situated within a specific urban area $\Omega$. In this case, vehicles can enter and exit the parking lot for a fee. On the other hand, the Neumann boundary $Bu=\partial_\nu u$ signifies that the agent cannot leave the domain $\Omega$, and will be reflected back to the domain when meeting the boundary. 

\begin{defn}\label{adm1}
($\mathbf{Admissible\ class \ \mathcal{A}}$). We define $F(x, z): \Omega\setminus\overline{D}\times\mathbb{C}\rightarrow \mathbb{C}$ to be an element of the admissible set $\mathcal{A}$ if it satisfies the following equations\\
(i)\ The map $z\rightarrow F(\cdot, z)$ is holomorphic with value in $C^{\alpha}(\Omega\setminus\overline{ D})$;\\
(ii)\ $F(x,0)=0$.
\end{defn}
Based on Definition \ref{adm1}, it can be inferred that $F(x,z)$ can be expressed as a power series expansion in the following form:
\[
F(x,z)=\sum_{i=1}^{\infty}F^{(i)}(x)\frac{z^i}{i!},
\]
where $F^{(i)}(x)=\partial_z^iF(x,0)\in C^{\alpha}(\Omega\setminus\overline{ D})$.

\begin{defn}\label{adm}
($\mathbf{Admissible\ class \ \mathcal{B}}$). We say $F(x,z):(\Omega\setminus\overline{D})\times \mathbb{C}\rightarrow \mathbb{C}\in \mathcal{B}$, if it satisfies the following conditions:\\
(i)\ The map $z\rightarrow F(\cdot,z)$ is holomorphic with value in $C^{\alpha}(\Omega\setminus\overline{ D})$;\\
(ii)\ $F(x,g_0)=0$ and $F^{(1)}(x)=\partial_zF(x,g_0)=0$ for all $x\in\Omega\setminus\overline{ D}$, where $g_0$ is a positive constant.
\end{defn}
According to Definition \ref{adm}, it can be inferred that $F(x,z)$ can be represented by a power series expansion in the following manner:
\[
F(x,z)=\sum_{i=2}^{\infty}F^{(i)}(x)\frac{(z-g_0)^i}{i!},
\]
where $F^{(i)}(x)=\partial_z^iF(x,g_0)\in C^{\alpha}(\Omega\setminus\overline{ D})$.

\subsection{Uniqueness associated with an external Dirichlet boundary MFG systems}
We commence by establishing the measurement mapping for the system \eqref{um0} when we have Dirichlet boundary conditions, i.e. when $$Bu=u\ \  \mathrm{and} \ \ Bm=m.$$ This mapping is defined as follows:
\begin{align}
\mathcal{M}^{\mathcal{D}}_{F,D}(m_0)=\left(\int_{\Sigma}\partial_\nu u(x,t) h(x,t)\,dx\,dt, \partial_\nu m(x,t)|_{\Sigma}\right),
\end{align}
where $\Sigma=\Gamma\times(0,T]$, $\Gamma\subseteq\partial\Omega$ and $h(x,t)$ is the weight function. 

It is crucial to emphasize that the measurement data comprises two distinct components. The first component involves the measurement of the value function $u(x,t)$. In the case of a crowd-dynamics model \cite{chow2022numerical}, $u(x,t)$ represents the travel cost incurred by a traveler located at $x$ at time $t$. The integral over the partial boundary $\int_{\Sigma}\partial_\nu u(x,t) h(x,t)dxdt$ quantifies the mean travel cost flux of a traveler passing through parts of the border from $0$ to $T$ moments. The second component pertains to the Neumann data measurement of the probability density function $m(x,t)$.

We consider the simultaneous inversion of the running cost function $F(x,m)$ and  the boundary $\partial D$, i.e. 
\begin{align}
\mathcal{M}^{\mathcal{D}}_{F,D}(m_0)\rightarrow F, D.
\end{align}
In particular, we consider the unique identifiability issue
\begin{align*}
\mathcal{M}^{\mathcal{D}}_{F_1,D_1}(m_0)=\mathcal{M}^{\mathcal{D}}_{F_2,D_2}(m_0)
\quad \text{if and only if}\quad  (F_1,D_1)=(F_2,D_2),
\end{align*}
where $D_j$ is a non-empty open subsets with $C^{2+\alpha}$ boundary such that $\Omega\setminus \overline{D}$ is connected, and $F_j\in\mathcal{A}$ with $j=1,2$. 

Next, we give the following auxiliary lemma, which is a variant of Lemma \ref{E-F is complete}. We omit the proof here.

\begin{lem}\label{lem1}
Consider the following linear parabolic equation
\begin{equation}
\label{eigD}
\begin{cases}
\partial_tu(x,t)-\Delta i(x,t)=0
 &\mathrm{in}\ Q,\\
W=0  &\mathrm{on}\ \Sigma_1,\\
BW=0 &\mathrm{on}\ \Sigma_2.
\end{cases}
\end{equation}
Then, there exists a sequence of solutions $u(x,t)$ for the equation \eqref{eigD} such that
$u(x,t)=e^{\lambda t}y(x;\lambda)$, where $\lambda\in \mathbb{R}$ and $y(x;\lambda)\in C^2(\Omega\setminus \overline{D})$.
\end{lem}

\begin{thm}\label{DirichletSmoothAnomaly}
Assume $F_j$ in the admissible class $\mathcal{A}$, $D_j$ $\Subset\Omega$ is a non-empty open subset with $C^{2+\alpha}$ boundaries such that $\Omega\backslash \overline{D_j}$ is connected for $j=1,2$, the weight function $h\in C_0^{2+\alpha}(\Sigma)$ is a nonnegative and nonzero function.  Let $\mathcal{M}^{\mathcal{\mathcal{D}}}_{F_j, D_j}$ be the map associated the MFG system \eqref{um0} with  $B=Id$ and  $ f_0=g_0>0$. If for any $m_0\in C^{2+\alpha}(\Omega\backslash \overline{D_j} )$, one has
\begin{align}\label{cod1}
\mathcal{M}^{\mathcal{D}}_{F_1, D_1}(m_0)=\mathcal{M}^{\mathcal{D}}_{F_2, D_2}(m_0),
\end{align}
 then we have
 \[
 F_1=F_2\ and\ D_1=D_2.
 \]
\end{thm}

\begin{proof}
\textbf{Step I}. Let us introduce the following initial value
\begin{align}\label{fg2}
 m_0(x)=g_0+\sum_{l=1}^N\epsilon_lg_l \ \ \mathrm{for}\ x\in\Omega\setminus \overline{D},\ 
\end{align}
where $g_l\in C^{2+\alpha}(\Omega\setminus \overline{D})$ for $l=1,2,...,N$, and $\epsilon=(\epsilon_1,\epsilon_2,..., \epsilon_N )\in \mathbb{R}^{N}$ with $|\epsilon|=|\epsilon_l|+|\epsilon_2|+...+|\epsilon_N|$ small enough, such that $\| g_0\|_{C^{2+\alpha,1+\alpha/2}(\Sigma_1\cup\Sigma_2)}+\| \sum_{l=1}^N\epsilon_lg_l\|_{C^{2+\alpha}(\Omega\setminus D)}$ is sufficiently small, and ensure that $m_0>0$ for $x\in\Omega\setminus \overline{D}$. Given that the input $m_0$ is in the form of \eqref{fg2}, $f_0=g_0$ and $B=Id$ for the MFG system \eqref{um0}, then the system \eqref{um0} becomes the following form
\begin{equation}
\label{um0h}
\begin{cases}
-\partial_tu(x,t)-\Delta u(x,t)+\frac{1}{2}|\nabla u(x,t)|^2=F(x,m)
&\mathrm{in}\ Q,\\
\partial_tm(x,t)-\Delta m(x,t)-\mathrm{div}(m(x,t)\nabla u(x,t))=0
&\mathrm{in}\ Q,\\
u=0, \quad m=g_0&\mathrm{on}\  \Sigma_1,\\
u=0, \quad m=g_0&\mathrm{on}\  \Sigma_2,\\
u(x,T)=0,\ m(x,0)=g_0+\sum_{l=1}^N\epsilon_lg_l&\mathrm{in}\  \Omega\setminus\overline{D}.
\end{cases}
\end{equation}
Based on the well-posedness of the forward problem, it follows that there exists a unique solution $(u,m)\in C^{2+\alpha, 1+\alpha/2}(\overline{Q})\times C^{2+\alpha, 1+\alpha/2}(\overline{Q})$. 
Denoting the solution of \eqref{um0h} as $(u(x,t;\epsilon),$ $ m(x,t;\epsilon))$, when $\epsilon=0$, we obtain the following
\[
(u(x,t;0),m(x,t;0))=(0,g_0)\ \ \mathrm{for}\ (x,t)\in Q.
\]
We choose a $g_1$ such that $g_1> 0$ in $\Omega\setminus\overline{D}$. 

Apply the first-order linearization $\partial_{\epsilon_1}u|_{\epsilon=0}, \partial_{\epsilon_1}m|_{\epsilon=0}$ to the system \eqref{um0h} around $(u_0, m_0)=(0,g_0)$.  For different domain $D_j$ with $j=1,2$, $(u_j^{(1)}, m_j^{(1)})$ for $j=1,2$ satisfies the following equations:
\begin{equation}
\label{mD11}
\begin{cases}
-\partial_tu_j^{(1)}-\Delta u_j^{(1)}=0
&\mathrm{in}\  (\Omega\setminus \overline{D_j})\times(0,T],\\
\partial_tm_j^{(1)}-\Delta m_j^{(1)}-g_0\Delta u_j^{(1)}=0
&\mathrm{in}\ (\Omega\setminus \overline{D_j})\times(0,T],\\
u_j^{(1)}= m_j^{(1)}=0 &\mathrm{on}\ (\partial D_j\cup\partial\Omega)\times(0,T],\\
u_j^{(1)}(x,T)=0, \quad m_j^{(1)}(x,0)=g_1>0&\mathrm{in}\ \Omega\setminus \overline{D_j}.
\end{cases}
\end{equation}
From the first equation of \eqref{mD11} and zero boundary condtion, we deduce that $u_j^{(1)}=0$ in $\overline{Q}$, which implies $\Delta u_j^{(1)}=0$ in $\overline{Q}$. Consequently, the equation \eqref{mD11} can be simplified to
\begin{equation}
\label{mD1}
\begin{cases}
\partial_tm_j^{(1)}-\Delta m_j^{(1)}=0
&\mathrm{in}\ (\Omega\setminus \overline{D_j})\times(0,T],\\
 \ m_j^{(1)}=0 \hspace*{1.0cm} &\mathrm{on}\ (\partial D_j\cup\partial\Omega)\times(0,T],\\
m_j^{(1)}(x,0)=g_1>0&\mathrm{in}\ \Omega\setminus \overline{D_j}.
\end{cases}
\end{equation}
Let $G$ represent the connected component of $\Omega\setminus(\overline{D_1}\cup \overline{D_2})$ whose boundary includes $\partial\Omega$. Consider $\tilde{m}=m^{(1)}_1-m^{(1)}_2$ and from \eqref{cod1}, it follows that $\tilde{m}$ satisfies  equation \begin{align}
\label{mD21}
\begin{cases}
\partial_t\tilde{m}-\Delta \tilde{m}=0
 &\mathrm{in}\ G\times(0,T],\\
\partial_\nu \tilde{m}=0  &\mathrm{on}\ \Sigma,\\
\tilde{m}=0 &\mathrm{on}\ \Sigma_2,\\
\tilde{m}(x,0)=0&\mathrm{in}\ G.
\end{cases}
\end{align} 
By applying the unique continuation principle for the linear parabolic equation \eqref{mD21}, we conclude that $\tilde{m}=0$ in $\overline{G}\times (0,T]$. Consequently, we have $m^{(1)}_1=m^{(1)}_2$ in $\overline{G}\times (0,T]$. Considering \eqref{mD1}, we find that $m^{(1)}_1=m^{(1)}_2=0$ on $\partial(D_2\setminus D_1)\times (0,T]$. As $g_1>0$, the maximum principle implies that $m^{(1)}_1>0$ in $\overline{Q}$. This contradicts the fact that $m_1^{(1)}=0$ on $\partial(D_2\setminus D_1)\times (0,T]$. Therefore, we have $D_1=D_2$.

\textbf{Step II}. Fix $D_1=D_2$ we proceed with the second-order linearization. For $j=1,2$, $u_j^{(1,2)}$ and $m^{(1,2)}$ satisfy the following coupled equations
\begin{equation}
\label{umDc2}
\begin{cases}
-\partial_t u_j^{(1,2)}-\Delta u_j^{(1,2)}+\nabla u^{(1)}\cdot \nabla u^{(2)}=F_j^{(2)}m^{(1)}m^{(2)}
 &\mathrm{in}\  Q,\\
\partial_t m^{(1,2)}-\Delta m^{(1,2)}-\mathrm{div}(m^{(1)}\nabla u^{(2)})-\mathrm{div}(m^{(2)}\nabla u^{(1)})-g_0\Delta u_j^{(1,2)}=0
&\mathrm{in}\ Q,\\
 u_j^{(1,2)}=0,\ m^{(1,2)}=0&\mathrm{on}\  \Sigma_1\cup\Sigma_2,\\
 u^{(1,2)}(x,T)=m^{(1,2)}(x,0)=0&\mathrm{in}\ \Omega\setminus\overline{D}.
\end{cases}
\end{equation}
It is important to emphasize that $u^{(1)}, u^{(2)}, m^{(1)}, m^{(2)}$ are independent of $F_j^{(2)}$. Based on the previously derived result $u^{(1)}=0$, $u^{(2)}=0$, we have $\nabla u^{(1)}=0$ and $\nabla u^{(2)}=0$. As a result, the system \eqref{umDc2} can be simplified to
 \begin{equation}
\label{umDcsimple}
\begin{cases}
-\partial_t u_j^{(1,2)}-\Delta u_j^{(1,2)}=F_j^{(2)}m^{(1)}m^{(2)}
&\mathrm{in}\  Q,\\
\partial_t m_j^{(1,2)}-\Delta m_j^{(1,2)}-g_0\Delta u_j^{(1,2)}=0
&\mathrm{in}\ Q,\\
 u_j^{(1,2)}=0,\ m_j^{(1,2)}=0\hspace*{5.3cm} &\mathrm{on}\  \Sigma_1\cup\Sigma_2,\\
 u_j^{(1,2)}(x,T)=m_j^{(1,2)}(x,0)=0 &\mathrm{in}\ \Omega\setminus\overline{D}.
\end{cases}
\end{equation}
 Note that the system \eqref{umDcsimple} is a single coupled system, which is advantageous for solving inverse problems. Taking the difference of the two equations $j=1,2$ for the first equation of \eqref{umDcsimple}, we arrive at the following expression
\begin{equation}
\label{FD2}
\begin{cases}
-\partial_t\bar{u}-\Delta \bar{u}=(F_1^{(2)}-F_2^{(2)})m^{(1)}m^{(2)}
 &\mathrm{in}\  Q,\\
 \bar{u}=0&\mathrm{on}\  \Sigma_1\cup\Sigma_2,\\
 \bar{u}(x,T)=0&\mathrm{in}\ \Omega\setminus\overline{D},
\end{cases}
\end{equation}
where $\bar{u}=u_1^{(1,2)}-u_2^{(1,2)}$. Let $b(x,t)$ be a solution to the equation 
\begin{equation}
\label{z}
\begin{cases}
\partial_tb-\Delta b=0
&\mathrm{in}\  Q,\\
 b=0&\mathrm{on}\  \Sigma_1,\\
 b=h_0&\mathrm{on}\  \Sigma_2,\\
 b(x,0)=0&\mathrm{in}\ \Omega\setminus\overline{D}.
\end{cases}
\end{equation}
Multiplying both sides of \eqref{FD2} by $b$ and integrating over $Q$, according to Green's formula and  the fact \[\int_{\Sigma}\partial_\nu u_1^{(1,2)}h(x,t)dxdt=\int_{\Sigma}\partial_\nu u_2^{(1,2)}h(x,t)dxdt,\]
we derive that
\[
\int_0^T\int_{\Omega\setminus\overline{D}}(F_1^{(2)}(x)-F_2^{(2)}(x))m^{(1)}(x,t)m^{(2)}(x,t)b(x,t) \,dx\,dt=0.
\]
According to Lemma \ref{lem1}, we take
$m_2^{(2)}\in C^{2+\alpha,1+\alpha/2}(Q)$ and of the following form
\begin{equation*}
\begin{aligned}
m_2^{(2)}:=e^{\lambda t}\beta(x,\lambda),
\end{aligned}
\end{equation*}
where $\beta(x,\lambda)$ is any Dirichlet eigenfunction of $-\Delta$ in $\Omega\setminus \overline{D}$. Then, we have
\begin{align*}
\int_{0}^T \int_{\Omega\setminus\overline{D}}(F^{(2)}_1(x)-F^{(2)}_2(x))e^{\lambda t}\beta(x,\lambda)m^{(1)}(x,t)b(x,t)dxdt=0.
\end{align*}
Since the Dirichlet-Laplacian eigenfunction form a complete set in $L^2(\Omega\setminus\overline{D})$,  we have
 \begin{align*}
(F^{(2)}_1(x)-F^{(2)}_2(x))\int_{0}^Te^{\lambda t}m^{(1)}(x,t)b(x,t)dt=0.
\end{align*}
By applying the maximum principle, we can assert that $b(x,t)>0$ and $m^{(1)}(x,t)>0$ in $Q$. Consequently, we conclude that $F^{(2)}_1(x)=F^{(2)}_2(x)$ in $\Omega\setminus\overline{D}$.

\textbf{Step III}. To validate the result for $N\geq 3$, it is essential to compute the higher-order Taylor coefficients. By employing the method of mathematical induction, we have
\begin{align*}
\int_{0}^T \int_{\Omega\setminus\overline{D}}(F^{(N)}_1(x)-F^{(N)}_2(x))m^{(1)}m^{(2)}...m^{(N)}b(x,t)\,dx\,dt=0.
\end{align*}
Similar to Step II, let us choose $m^{(2)}=e^{\lambda t}\beta(x,\lambda)$, and $m^{(1)},m^{(3)},...m^{(N)}$ are positive solutions in $\overline{Q}$, we have $F^{(N)}_1(x)=F^{(N)}_1(x)$ in $\Omega\setminus\overline{D}$. Then we derive that $F_1=F_2$.
\end{proof}

\subsection{Uniqueness associated with the external Neumann boundary MFG systems}
Next, we introduce the measurement mapping for the system \eqref{um0} when
\[Bu=\partial_\nu u\quad \mathrm{and}\quad Bm=\partial_\nu m.\] 
This mapping is defined as follows:
\[\mathcal{M}^{\mathcal{N}}_{F,D}(m_0)=\left(\int_\Sigma u(x,t) g(x,t)dxdt, m(x,t)|_\Sigma\right),
\]
where $g$ is a weight function. 
Once again, we emphasize that the measurement data comprises two distinct components. The first component involves the measurement of the value function $u(x,t)$, with integral over the partial boundary $\int_\Sigma u(x,t) g(x,t)dxdt$ represents the mean travel cost of the traveler parts of the border within the time interval from $0$ to $T$.  The second component pertains to the Dirichlet data measurement of the probability density function $m(x,t)$.

Once again, we intend to simultaneously determine $F$ and $D$, i.e. 
\begin{align}\label{IP2}
\mathcal{M}^{\mathcal{N}}_{F,D}(m_0)\rightarrow F, D
\end{align}
and we prove the unique identifiability result
\begin{align*}
\mathcal{M}^{\mathcal{N}}_{F_1,D_1}(m_0)=\mathcal{M}^{\mathcal{N}}_{F_2,D_2}(m_0)
\quad\text{if and only if}\quad  (F_1,D_1)=(F_2,D_2),
\end{align*}
where,  $D_j$ is a non-empty open subsets with $C^{2+\alpha}$ boundary such that $\Omega\setminus \overline{D}$ is connected, and $F_j\in\mathcal{B}$ with $j=1,2$, as defined in Definition \ref{adm}. 

\begin{thm}
Assume $F_j$ lies in the admissible class $\mathcal{B}$, $D_j$ $\Subset\Omega$ is a non-empty open subset with $C^{2+\alpha}$ boundaries such that $\Omega\backslash \overline{D_j}$ is connected for $j=1,2$, the weight function $g\in C_0^{2+\alpha}(\Sigma)$ is a nonnegative and nonzero function. Let $\mathcal{M}^{\mathcal{N}}_{F_j, D_j}$ be the map associated the MFG \eqref{um0} with  $B=\partial_\nu$ and $f_0=0,g_0>0$. If for any $m_0\in C^{2+\alpha}(\Omega)$, one has
\begin{align}\label{cod2}
\mathcal{M}^{\mathcal{N}}_{F_1, D_1}(m_0)=\mathcal{M}^{\mathcal{N}}_{F_2, D_2}(m_0),
\end{align}
 then we have
 \[
 F_1=F_2\quad \text{ and }\quad D_1=D_2.
 \]
\end{thm}
\begin{proof}
\textbf{Step I}. When $m_0$ is in the form of \eqref{fg2}, $f_0=0$, $g_0>0$ and $B=\partial_\nu$ for the system \eqref{um0}, we denote the solution as $(u(x,t;\epsilon), m(x,t;\epsilon))$. When $\epsilon=0$, we obtain the following result
\[
(u(x,t;0),m(x,t;0))=(0,g_0)\ \ \mathrm{for}\ (x,t)\in Q.
\]
We first choose a positive $g_1$ of \eqref{fg2}. By applying the first-order linearization $\partial_{\epsilon_1}u|_{\epsilon=0}, \partial_{\epsilon_1}m|_{\epsilon=0}$ around $(u_0, m_0)=(0,g_0)$, we find that $(u_j^{(1)}, m_j^{(1)})$ satisfies the following equations for $j=1,2$:
\begin{equation}
\label{mN12}
\begin{cases}
-\partial_tu_j^{(1)}-\Delta u_j^{(1)}=0 &\mathrm{in}\ (\Omega\setminus \overline{D_j})\times(0,T],\\
\partial_tm_j^{(1)}-\Delta m_j^{(1)}-g_0\Delta u_j^{(1)}=0
&\mathrm{in}\ (\Omega\setminus \overline{D_j})\times(0,T],\\
u_j^{(1)}=0,\quad m_j^{(1)}=0 &\mathrm{in}\ \partial D_j\times(0,T],\\
\partial_\nu u_j^{(1)}=0, \quad \partial_\nu m_j^{(1)}=0 &\mathrm{on}\ \Sigma_2 ,\\
  u_j^{(1)}(x,T)=0,\quad m_j^{(1)}(x,0)=g_1>0&\mathrm{on}\ \Omega\setminus \overline{D_j}.
\end{cases}
\end{equation}
We observe that $u_j^{(1)}$ satisfies the heat equation which has the unique trivial solution, so $u_j^{(1)}=0$ in $Q$, which implies $\Delta u_j^{(1)}=0$. Consequently, the equation \eqref{mN12} can be simplified as
\begin{equation*}
\begin{cases}
\partial_tm_j^{(1)}-\Delta m_j^{(1)}=0
&\mathrm{in}\ (\Omega\setminus \overline{D_j})\times(0,T],\\
m_j^{(1)}=0 &\mathrm{in}\ \partial D_j\times(0,T],\\
\partial_\nu m_j^{(1)}=0 &\mathrm{on}\ \Sigma_2 ,\\
   m_j^{(1)}(x,0)=g_1>0&\mathrm{on}\ \Omega\setminus \overline{D_j}.
\end{cases}
\end{equation*}
Similar to the proof in Theorem \ref{DirichletSmoothAnomaly}, by \eqref{cod2}, utilizing the unique continuation principle for the linear parabolic equation and maximum principle, we can conclude that $D_1 = D_2$.

\textbf{Step II}. Fix $D_1=D_2$. We consider the second-order linearization. For $j=1,2$, we have $u_j^{(1,2)}$ and  $m_j^{(1,2)}$ satisfy the following equations:
\begin{equation}
\label{umN2}
\begin{cases}
-\partial_t u_j^{(1,2)}-\Delta u_j^{(1,2)}+\nabla u^{(1)}\cdot \nabla u^{(2)}=F_j^{(2)}m^{(1)}m^{(2)}
&\mathrm{in}\  Q,\\
\partial_t m^{(1,2)}-\Delta m^{(1,2)}-\mathrm{div}(m^{(1)}\nabla u^{(2)})-\mathrm{div}(m^{(2)}\nabla u^{(1)})-g_0\Delta u_j^{(1,2)}=0
&\mathrm{in}\ Q,\\
 u_j^{(1,2)}=0,\quad m^{(1,2)}=0&\mathrm{on}\  \Sigma_1,\\
 \partial_{\nu}u_j^{(1,2)}=\partial_{\nu}m^{(1,2)}=0&\mathrm{on}\  \Sigma_2,\\
 u^{(1,2)}(x,T)=m^{(1,2)}(x,0)=0 &\mathrm{in}\ \Omega\setminus\overline{D}.
\end{cases}
\end{equation}
Notice that $u^{(1)},u^{(2)}, m^{(1)}, m^{(2)}$ are not dependent on $F_j^{(2)}$. Substituting the solutions $u_j^{(1)}=0$ and  $\nabla u_j^{(1)}=0$ for $j=1,2$, then equation \eqref{umN2} can be simplified as
\begin{equation}
\label{umNs2}
\begin{cases}
-\partial_t u_j^{(1,2)}-\Delta u_j^{(1,2)}=F_j^{(2)}m^{(1)}m^{(2)}
&\mathrm{in}\  Q,\\
\partial_t m_j^{(1,2)}-\Delta m_j^{(1,2)}-g_0\Delta u_j^{(1,2)}=0
&\mathrm{in}\ Q,\\
 u_j^{(1,2)}=0,\quad m_j^{(1,2)}=0&\mathrm{on}\  \Sigma_1,\\
 \partial_{\nu}u_j^{(1,2)}=\partial_{\nu}m^{(1,2)}=0&\mathrm{on}\  \Sigma_2,\\
 u_j^{(1,2)}(x,T)=m_j^{(1,2)}(x,0)=0 &\mathrm{in}\ \Omega\setminus\overline{D}.
\end{cases}
\end{equation}
Subtracting  the first equation of \eqref{umNs2} with $j = 1,2,$ we find
\begin{equation}
\label{F2}
\begin{cases}
-\partial_t\check{u}-\Delta \check{u}=(F_1^{(2)}-F_2^{(2)})m^{(1)}m^{(2)}
&\mathrm{in}\  Q,\\
 \check{u}=0&\mathrm{on}\  \Sigma_1.\\
 \partial_{\nu}\check{u}=0&\mathrm{on}\  \Sigma_2.\\
 \check{u}(x,T)=0 &\mathrm{in}\ \Omega\setminus\overline{D}.
\end{cases}
\end{equation}
where $\check{u}=u_1^{(1,2)}-u_2^{(1,2)}$. Let $\tilde{b}$ be a solution to the equation \begin{equation}
\label{zN}
\begin{cases}
\partial_t\tilde{b}-\Delta \tilde{b}=0
&\mathrm{in}\  Q,\\
 \tilde{b}=0&\mathrm{on}\  \Sigma_1,\\
 \partial_\nu \tilde{b}=q_0&\mathrm{on}\  \Sigma_2,\\
 \tilde{b}(x,0)=0&\mathrm{in}\ \Omega\setminus\overline{D}.
\end{cases}
\end{equation}
By multiplying both sides of \eqref{F2} by $\tilde{b}$ and integrating over $Q$, applying Green's formula and
from \eqref{cod2}, we can deduce that
\[
 \int_{Q}(F_1^{(2)}(x)-F_2^{(2)}(x))m^{(1)}(x,t)m^{(2)}(x,t)\tilde{b}(x,t) \,dx\,dt=0.
\]
Similarly, based on Lemma \ref{lem1}, we take $m^{(2)}=e^{\mu t}\alpha(x,\mu)$ for $\mu\in\mathbb{R}$. Then, we have
\[
\int_{0}^T \int_{\Omega\setminus \overline{D}}(F^{(2)}_1(x)-F^{(2)}_2(x))e^{\mu t}\alpha(x,\mu)m^{(1)}(x,t)\tilde{b}(x,t)\,dx\,dt=0.
\]
Since the mixed-Laplacian eigenfunction form a complete set in $L^2(\Omega\setminus \overline{D})$,  we have
 \[
(F^{(2)}_1(x)-F^{(2)}_2(x))\int_{0}^Te^{-\rho t}m^{(1)}(x,t)\tilde{b}(x,t)\,dt=0.
\]
By applying the maximum principle, we can conclude that $\tilde{b}(x,t)>0$  and $m^{(1)}(x,t)>0$ in $Q$. Consequently, we have $F^{(2)}_2(x)=F^{(2)}_1(x)$ in $\Omega\backslash\overline{D}$. 

Similarly, to calculate higher-order Taylor coefficients, we utilize mathematical induction and derive $F_1^{(N)}(x)=F_2^{(N)}(x)$ for $N\geq 3$. Hence, we can conclude that $F_1=F_2$ holds.
\end{proof}


\section{Non-Smooth Internal Anomalies for the Static MFG System}

Next, we consider internal anomalies that have non-smooth boundaries. Let $\mathcal{P}$ stand for the set of Borel probability measures on $\mathbb{R}^n$, and $\mathcal{P}(\Omega')$ stand for the set of Borel probability measures on $\Omega'$. Let $m\in\mathcal{P}(\Omega')$ denote the population distribution of the agents and $u(x, t):\Omega'\times [0, T]\to \mathbb{R}$ denote the value function of each player. We restrict our study to stationary MFGs with quadratic Hamiltonians given by the following form:

\begin{equation}\label{MFGQuadraticStat}
    \begin{cases}
        -\Delta u(x) + \frac{1}{2}\kappa(x) |\nabla u(x)|^2 + \lambda - F(x,m(x)) = 0 &\quad \text{in }\Omega',\\
        -\Delta m(x) - \text{div}(\kappa(x) m(x)\nabla u(x)) = 0  &\quad \text{in }\Omega',\\
        \partial_\nu u(x)=\partial_\nu m(x)=0 &\quad \text{on }\partial\Omega', \\
        u(x) = \psi(x), \quad m(x)=\varphi(x) & \quad \text{on } \partial\Omega,
    \end{cases}
\end{equation} 
for a Lipschitz domain $\Omega'\subset\mathbb{R}^n$, and $F:\Omega'\times\mathcal{P}(\Omega')\to\mathbb{R}$, where $\lambda$ is a constant that can be determined through the normalisation of $m$.. Then, it is well-known (see, for instance, Theorem 2.1 of \cite{LasryLions1}) that
\begin{theorem}\label{ForwardPbMainThm}
    Suppose $F$ is Lipschitz and bounded in $\Omega'$. Then there exists $\lambda\in\mathbb{R}$, $u\in C^2(\Omega')$ and $m\in W^{1,p}(\Omega')$ for any $p<\infty$ solving \eqref{MFGQuadraticStat}.
\end{theorem}

Let $\Omega$ be a closed proper subdomain of $\Omega'$ with a smooth boundary $\partial\Omega$. In this subdomain, there are no Neumann boundary conditions. 
\begin{equation}\label{MFGStat}
    \begin{cases}
		- \Delta u(x) + \frac{1}{2}\kappa(x) |\nabla u(x)|^2 + \lambda - F(x,m(x)) = 0 &\quad \text{in }\Omega\\
        - \Delta m(x) - \text{div}(\kappa(x) m(x)\nabla u(x)) = 0  &\quad \text{in }\Omega,\\
        u(x) = \psi(x), \quad m(x)=\varphi(x) &\quad \text{on } \partial\Omega,
    \end{cases}
\end{equation}
We will consider this system such that there is some topological structure in the system, given by an anomalous inhomogeneity $ D\Subset\Omega$, which is a bounded Lipschitz domain such that $\Omega\backslash\bar{D}$ is connected. Here, by an inhomogeneity, we mean that $\kappa$, $\lambda$ and $F$ has a discontinuity across the boundary of $D$. In particular, $\kappa$ has a jump of the form 
\[\kappa(x) = \begin{cases}
    \kappa_1(x) & \text{ if }x\in D,\\
    \kappa_0(x) & \text{ otherwise },
\end{cases}
\] such that
\[\kappa_1(x)\neq \kappa_0(x) \quad \text{ for }x\in\partial D.\] A similar form holds for $\lambda$ and $F$. See Definitions \ref{kappaAdmis}--\ref{FAdmis1} and \ref{FDef} in Section \ref{sect:MainResults} for more precise descriptions of these functions.

In order to characterize the anomaly $D$, we introduce the inverse boundary problem given by the following measurement map of a single pair of Cauchy data:
\begin{equation}\label{MeasureMap}
    \mathcal{M}_{ D,H,F}  := \left((\psi,\varphi),(\partial_\nu u, \partial_\nu m)_{\partial\Omega}\right),\, \psi,\varphi \text{ fixed} \to  D,H,F,
\end{equation}
where $\nu$ is the exterior unit normal vector to $\partial\Omega$. In the physical context, $ D$ signifies the support of the anomalous polyhedral inclusion in the state space $\Omega$. Hence, the inverse problem \eqref{MeasureMap} is concerned with recovering the location and shape of this anomaly, as well as its parameter configuration. It is also referred to as the inverse inclusion problem in the theory of inverse problems.

For the inverse inclusion problem \eqref{MeasureMap}, we mainly consider its unique identifiability issue. That is, we aim at establishing the sufficient conditions under which $ D$ can be uniquely determined by the measurement map $\mathcal{M}_{ D,H,F}$ in the sense that if two admissible inclusions $( D_j, H_j, F_j)$, $j = 1, 2$, produce the same boundary measurement, i.e. $\mathcal{M}_{ D_1,H_1,F_1}(u,m)=\mathcal{M}_{ D_2,H_2,F_2}(u,m)$ associated with a fixed $(u,m)$, then one has $( D_1, H_1, F_1)=( D_2, H_2, F_2)$.

\subsection{Geometrical Setup}\label{sect:geometry}

Consider the convex conic cone $\mathcal{S}\subset\Omega$ with apex $x_c$ and axis $v_c$ and opening angle $2\theta_c\in(0,\pi)$, defined by
\[\mathcal{S}:=\{y\in\Omega:0\leq\angle (y-x_c,v_c)\leq\theta_c,\theta_c\in(0,\pi/2)\}.\] Define the truncated conic cone by
\[\mathcal{S}_h:=\mathcal{S}\cap B_h,\] where $B_h:=B_h(x_c)$ is an open ball contained in $\Omega$ centred at $x_c$ with radius $h>0$. Observe that both $\mathcal{S}$ and $\mathcal{S}_h$ are Lipschitz domains. 

We first consider some asymptotics for a CGO solution we will be using. 
\begin{lemma}\label{wCGOlem}
Let $w$ be the solution to 
\begin{equation}\label{CGOEq}
- \Delta w(x) = 0\text{ in }\Omega',
\end{equation}
of the form 
\begin{equation}\label{CGO}
w = e^{\tau (\xi + i\xi^\perp)\cdot (x-x_c)}
\end{equation} 
such that $\xi\cdot\xi^\perp=0$, $\xi,\xi^\perp\in\mathbb{S}^{n-1}$. 
Then there exists a positive number $\rho$ depending on $\mathcal{S}_h$ satisfying 
\begin{equation}\label{CGOCond}
-1 < \xi\cdot\widehat{(x-x_c)}\leq -\rho < 0\quad\text{ for all } x\in \mathcal{S}_h,
\end{equation}
where $\hat{x} = \frac{x}{|x|}$. Moreover, for sufficiently large $\tau$, \cite{DiaoFeiLiuWang-2022-Semilinear-Shape-Corners}
\begin{equation}\label{CGOEst1}
\left|\int_{\mathcal{S}_h}w\right|\geq C_{\mathcal{S}_h}\tau^{-n} +\mathcal{O}\left(\frac{1}{\tau}e^{-\frac{1}{2}\rho h \tau}\right),
\end{equation}
\begin{equation}\label{CGOEst2}
\left|\int_{\mathcal{S}_h}|x-x_c|^\alpha w\right|\lesssim \tau^{-(\alpha+n)}+\frac{1}{\tau}e^{-\frac{1}{2}\rho h \tau} \quad\forall \alpha>0,
\end{equation}
\begin{equation}\label{CGOEst3}
\norm{w}_{H^1(\partial\mathcal{S}_h)}\lesssim (2\tau^2+1)^{\frac12}e^{-\rho h\tau},
\end{equation}
\begin{equation}\label{CGOEst4}
\norm{\partial_\nu w}_{L^2(\partial\mathcal{S}_h)}\lesssim \tau e^{-\rho h \tau}.
\end{equation}
Here, we use the symbol `` $\lesssim$" to denote that the inequality holds up to a constant which is independent of $\tau$.
\end{lemma}

\begin{proof}
    We first remark that since $\xi\perp\xi^\perp$, clearly $w$ satisfies $\Delta w=0$ in $\mathbb{R}^n$, and in particular, in $\Omega'$.

    Furthermore, \eqref{CGOCond} is easily satisfied with an appropriate choice of $\xi$, since $\mathcal{S}_h$.

    Next, we recall from Lemma 2.2 of \cite{DiaoFeiLiuYang-2022-EM-Shape-Corners} that, for some fixed $\alpha>0$ and $0<\delta<e$,
    \begin{equation}\label{LaplaceTransformIdentity}
        \int_0^\delta r^\alpha e^{-\mu r}\,dr = \frac{\Gamma(\alpha+1)}{\mu^{\alpha+1}} + \int_\delta^\infty r^\alpha e^{-\mu r}\,dr,
    \end{equation} 
    where $\mu\in\mathbb{C}$ and $\Gamma$ is the Gamma function. Moreover, if the real part of $\mu$, denoted by $\mathscr{R}\mu$, is such that $\mathscr{R}\mu \geq \frac{2\alpha}{e}$, then $r^\alpha\leq e^{\mathscr{R}\mu r/2}$, and hence 
    \begin{equation}\label{LaplaceTransformIneq}
        \left|\int_\delta^\infty r^\alpha e^{-\mu r}\,dr\right| \leq \frac{2}{\mathscr{R}\mu}e^{-\mathscr{R}\mu\delta/2}.
    \end{equation}
    With this, we write $x-x_c=(x_1,\dots,x_n)\in\mathbb{R}^n$ in polar coordinates:
    \[x-x_c= \left( \begin{array}{l}
    r\cos\theta_1 \\
    r\sin\theta_1\cos\theta_2 \\
    r\sin\theta_1\sin\theta_2\cos\theta_3 \\
    \vdots \\
    r\sin\theta_1\cdots\sin\theta_{n-2}\cos\varphi \\
    r\sin\theta_1\cdots\sin\theta_{n-2}\sin\varphi
    \end{array} \right)^T,\]
    where $\theta_i\in[0,\pi]$ for $i=1,\dots,n-2$ and $\varphi\in[0,2\pi)$, with Jacobian
    \[r^{n-1}\sin^{n-2}(\theta_1)\sin^{n-3}(\theta_2)\cdots\sin(\theta_{n-2})\,dr\,d\theta_1 \,d\theta_2 \cdots\,d\theta_{n-2}\,d\varphi.\]

    Since $\mathcal{S}_h$ has an opening angle of $2\theta_c$, by \eqref{LaplaceTransformIdentity}, we have 
    \begin{multline}\label{CGOEstProofEq}
    \int_{\mathcal{S}_h}e^{\tau (\xi + i\xi^\perp)\cdot (x-x_c)} \\ = I_1 + \int_0^{2\pi} \left(\int_0^{\theta_c} \cdots \int_0^{\theta_c} I_2 \sin^{n-2}(\theta_1)\sin^{n-3}(\theta_2)\cdots\sin(\theta_{n-2})\,d\theta_1\,d\theta_2\cdots\,d\theta_{n-2} \right) \,d\varphi, 
    \end{multline}
    where 
    \begin{equation}\label{CGOEstProofEqI1}I_1 := \int_0^{2\pi} \int_0^{\theta_c} \cdots \int_0^{\theta_c}\frac{\Gamma(n)}{\tau^n \left((\xi + i\xi^\perp) \cdot \widehat{(x-x_c)}\right)^n} \sin^{n-2}(\theta_1)\cdots\sin(\theta_{n-2})\,d\theta_1\cdots\,d\theta_{n-2}\,d\varphi\end{equation} 
    and 
    \begin{equation}\label{CGOEstProofEqI2}I_2 := \int_h^\infty r^{n-1} e^{\tau r (\xi + i\xi^\perp)\cdot \widehat{(x-x_c)}} \,dr.\end{equation}
    By the integral mean value theorem, we have that 
    \begin{align}\label{CGOEstProofEq1}
        I_1 & = \frac{\Gamma(n)}{\tau^n} \int_0^{2\pi} \frac{1}{\left((\xi + i\xi^\perp) \cdot \widehat{(x-x_c)}(\varphi,\theta_\zeta)\right)^n} \,d\varphi\int_0^{\theta_c} \sin^{n-2}(\theta_1) \,d\theta_1 \cdots \int_0^{\theta_c}\sin(\theta_{n-2})\,d\theta_{n-2} \nonumber \\
        & = \frac{2\pi\Gamma(n)C_{\theta_c}}{\tau^n} \frac{1}{\left((\xi + i\xi^\perp) \cdot \widehat{(x-x_c)}(\varphi_\zeta,\theta_\zeta)\right)^n} 
    \end{align}
    for some constant $C_{\theta_c}$ such that $0<C_{\theta_c}<1$. Since $\xi \cdot \widehat{(x-x_c)} > -1$ by \eqref{CGOCond}, this gives that 
    \begin{equation}\label{CGOEstProofEq2}
    |I_1| \geq \frac{2\pi\Gamma(n)C_{\theta_c}}{\tau^n} \frac{1}{2^{n/2}}. 
    \end{equation}

    On the other hand, for sufficiently large $\tau$, by \eqref{LaplaceTransformIneq},
    \begin{equation}|I_2| = \left| \int_h^\infty r^{n-1} e^{r \tau (\xi + i\xi^\perp)\cdot \widehat{(x-x_c)}} \,dr \right| \leq \frac{2}{\rho\tau} e^{-\frac{1}{2}\rho h \tau}.\end{equation} 
    Consequently, we have that 
    \begin{multline}
    \left| \int_0^{2\pi} \left(\int_0^{\theta_c} \cdots \int_0^{\theta_c} I_2 \sin^{n-2}(\theta_1)\sin^{n-3}(\theta_2)\cdots\sin(\theta_{n-2})\,d\theta_1\,d\theta_2\cdots\,d\theta_{n-2} \right) \,d\varphi \right| 
    \\
    \leq \int_0^{2\pi} \left(\int_0^{\theta_c} \cdots \int_0^{\theta_c} |I_2| \,d\theta_1\,d\theta_2\cdots\,d\theta_{n-2} \right) \,d\varphi  \leq \frac{4 \pi \theta_c^{n-2}}{\rho \tau } e^{-\frac{1}{2}\rho h \tau}.
    \end{multline}
    Combining this with \eqref{CGOEstProofEq2}, we obtain \eqref{CGOEst1} with $C_{\mathcal{S}_h} = \frac{\pi\Gamma(n)C_{\theta_c}}{2^{n/2-1}}$.

    Next, for \eqref{CGOEst2}, utilising \eqref{LaplaceTransformIdentity}, we obtain higher exponents for $r$ in \eqref{CGOEstProofEq}, i.e. 
    \begin{multline*}
    \int_{\mathcal{S}_h}(x-x_c)^\alpha e^{\tau (\xi + i\xi^\perp)\cdot (x-x_c)} \\ = I_{3} + \int_0^{2\pi} \left(\int_0^{\theta_c} \cdots \int_0^{\theta_c} I_{4} \sin^{n-2}(\theta_1)\sin^{n-3}(\theta_2)\cdots\sin(\theta_{n-2})\,d\theta_1\,d\theta_2\cdots\,d\theta_{n-2} \right) \,d\varphi, 
    \end{multline*}
    where 
    \begin{align*}& |I_{3}| \\
    & := \left|\int_0^{2\pi} \int_0^{\theta_c} \cdots \int_0^{\theta_c}\frac{\Gamma(n+\alpha)}{\tau^{n+\alpha} \left((\xi + i\xi^\perp) \cdot \widehat{(x-x_c)}\right)^{n+\alpha}} \sin^{n-2}(\theta_1)\cdots\sin(\theta_{n-2})\,d\theta_1 \cdots\,d\theta_{n-2}\,d\varphi \right|\\
    & = \left| \frac{2\pi\Gamma(n+\alpha)C_{\theta_c}}{\tau^{n+\alpha}} \frac{1}{\left((\xi + i\xi^\perp) \cdot \widehat{(x-x_c)}(\varphi_\zeta,\theta_\zeta)\right)^{n+\alpha}} \right| \\
    & \lesssim\frac{1}{\tau^{n+\alpha}} 
    \end{align*} 
    as in the procedure to obtain \eqref{CGOEstProofEq1}--\eqref{CGOEstProofEq2}, 
    and 
    \begin{equation*}|I_{4}| := \left| \int_h^\infty r^{n+\alpha-1} e^{r \tau (\xi + i\xi^\perp)\cdot \widehat{(x-x_c)}} \,dr \right| \leq \frac{2}{\rho\tau} e^{-\frac{1}{2}\rho h \tau},\end{equation*} 
    which directly leads to \eqref{CGOEst2}.

    At the same time, 
    \begin{equation}\norm{w}_{H^1(\partial\mathcal{S}_h)} = \left( \norm{w}_{L^2(\partial\mathcal{S}_h)}^2 + \norm{\tau (\xi + i\xi^\perp)w}_{L^2(\partial\mathcal{S}_h)}^2 \right)^{\frac{1}{2}} \leq (2\tau^2+1)^{\frac{1}{2}} \norm{w}_{L^2(\partial\mathcal{S}_h)} \end{equation} 
    by the Cauchy-Schwarz inequality, since $\xi\in\mathbb{S}^{n-1}$. 
    But, applying the polar coordinate transformation on $w$, we have that 
    \begin{equation}\label{CGOEstProofEq3}\norm{w}_{L^2(\partial\mathcal{S}_h)} = \left( \int_0^{2\pi} \int_0^{\theta_c} \cdots \int_0^{\theta_c} e^{2\tau h \xi \cdot \widehat{(x-x_c)}} \,d\theta_1 \cdots \, d\theta_{n-2} \, d\varphi \right)^{\frac{1}{2}} \leq (2\pi \theta_c^{n-2})^{\frac{1}{2}} e^{-\rho h\tau}\end{equation} 
    by \eqref{CGOCond}. This gives \eqref{CGOEst3}.

    Finally, 
    \begin{equation}\norm{\partial_\nu w}_{L^2(\partial\mathcal{S}_h)} \leq \norm{\nabla w}_{L^2(\partial\mathcal{S}_h)} = \norm{\tau (\xi + i\xi^\perp)w}_{L^2(\partial\mathcal{S}_h)} \leq \sqrt{2}\tau \norm{w}_{L^2(\partial\mathcal{S}_h)} \end{equation} once again by the Cauchy-Schwarz inequality. Using \eqref{CGOEstProofEq3} then gives the desired result \eqref{CGOEst4}.
\end{proof}

\subsection{Main Results}\label{sect:MainResults}

With these definitions in place, we can now state the main results of this section. Let $\Omega$ be a closed proper subdomain of $\Omega'$ with smooth boundary $\partial\Omega$. Consider a bounded Lipschitz polyhedral domain $ D\Subset\Omega$ such that $\Omega\backslash\bar{ D}$ is connected and each corner of $ D$ is a convex cone $\mathcal{S}$, as defined in Subsection \ref{sect:geometry}. We define the following admissibility conditions for $\kappa$, $\lambda$ and $F$.
\begin{definition}\label{kappaAdmis}
    We say that $\kappa(x):\Omega'\to\mathbb{R}$ is admissible, denoted by $\kappa\in\mathcal{B}$, if $\kappa(x)$ is of the form 
    \[\kappa(x) = \kappa_0(x) + (\kappa_1(x)-\kappa_0(x))\chi_ D, \quad x\in\Omega'\]
    such that $\kappa_1$ and $\kappa_0$ are $C^\gamma$ H\"older-continuous for some $\gamma\in(0,1)$ with respect to $x\in  D$ and $x\in \Omega\backslash  D$ respectively, and in an open neighbourhood $U$ of $\partial  D$, $\kappa_1(U\cap  D) \neq \kappa_0(U\cap(\Omega\backslash  D))$.
\end{definition}

\begin{definition}\label{lambdaAdmis}
    We say that the constant $\lambda$ is admissible, denoted by $\lambda\in\mathcal{C}$, if $\lambda$ is of the form 
    \[\lambda = \lambda_0 + (\lambda_1-\lambda_0)\chi_ D\]
    such that $\lambda_0$ and $\lambda_1$ are (possibly different) constants.
\end{definition}

Next, we introduce the admissible class of locally analytic functions.
\begin{definition}\label{FAdmis1}
    Let $E$ be a compact subset of $\Omega\subset\mathbb{R}^n$. We say that $U(x,z):E\times\mathbb{C}\to\mathbb{C}$ is admissible, denoted by $U\in\mathcal{A}_E$, if:
    \begin{enumerate}[(a)]
        \item The map $z\mapsto U(\cdot,z)$ is holomorphic with value in $C^{2+\alpha}(E)$ for some $\alpha\in(0,1)$, 
        \item $U(x,\cdot)$ is $C^\sigma$-continuous with respect to $x\in E$ for some $\sigma\in(0,1)$. 
        \item $U(x,0)=\lambda_U$ for all $x\in \mathbb{R}^n$.
    \end{enumerate}
    It is clear that if $U$ satisfies these two conditions, $U$ can be expanded into a power series
    \[U(x,z)=\lambda_U+\sum_{m=1}^\infty U^{m}(x)\frac{z^m}{m!},\]
    where $U^{m}(x)=\frac{\partial^m }{\partial z^m}U(x,0)\in C^{2+\alpha}(E)$.
\end{definition}

This admissibility condition is imposed a priori on $U$, by extending these functions of real variables to the complex plane with respect to the $z$-variable, given by $\tilde{U}$, and assume that they are holomorphic as functions of the complex variables $z$. Then, $U$ is simply the restriction of $\tilde{U}$ to the real line. Furthermore, since the image of $U$ is in $\mathbb{R}$, we can assume that the series expansion of $\tilde{U}$ is real-valued. This is more general than the cases considered in the previous chapters, and allows for singularities at the boundary $\partial E$ of compact subsets $E$ of $\Omega$.

Observe that if $F\in\mathcal{A}$ with the appropriate $\lambda$, then $(u,m)=(0,0)$ is a solution to \eqref{MFGQuadraticStat}. In fact, it is easy to check that $\left(u,\exp\left(\int \kappa(x)\nabla u(x)\,dx\right)\right)$ is also a solution to \eqref{MFGQuadraticStat}, for some $u$.  
With this admissibility condition on $F$, we are able to improve the result of Theorem \ref{ForwardPbMainThm} and obtain a stronger local regularity of the solution $(u,m)$ to the MFG system \eqref{MFGQuadraticStat}, using the implicit functions theorem for Banach spaces, as in Theorem \ref{local_wellpose}.

Next, we define the admissibility class of the running cost $F$.

\begin{definition}\label{FDef}
    We say that $F(x,m):\Omega'\times\mathcal{P}(\Omega')\to\mathbb{R}$ is admissible, denoted by $F\in\mathcal{D}$, if 
    $F$ is of the form 
    \[F(x) = F_0(x) + (F_1(x)-F_0(x))\chi_ D, \quad x\in\Omega',\]
    such that $F_1\in\mathcal{A}_ D$ with $F_1(x,0)=\lambda_1$, $F_0\in\mathcal{A}_{\Omega\backslash D}$ with $F_0(x,0)=\lambda_0$, and for some $\ell\in\mathbb{N}$, the $\ell$-th Taylor coefficient of $F_1$ and $F_0$, denoted by $F^{(\ell)}_1(x)$ and $F^{(\ell)}_0(x)$ respectively, are $C^\sigma$ H\"older-continuous for some $\sigma\in(0,1)$ with respect to $x\in  D$ and $x\in \Omega\backslash  D$ respectively, and in an open neighbourhood $U$ of $\partial  D$, $F^{(\ell)}_1(U\cap  D) \neq F^{(\ell)}_0(U\cap(\Omega\backslash  D))$.
\end{definition}

With the above admissibility conditions, then, microlocally in a corner $\mathcal{S}_h$ of $ D$, we have that $(u_1,m_1)$ and $(u_0,m_0)$ satisfies \begin{equation}\label{MFGQuadraticStat1}
    \begin{cases}
        - \Delta u_1(x) + \frac{1}{2}\kappa_1(x) |\nabla u_1(x)|^2 + \lambda_1 - F_1(x,m_1(x)) = 0 &\quad \text{in }\mathcal{S}_h,\\
        - \Delta m_1(x) - \text{div}(\kappa_1(x) m_1(x)\nabla u_1(x)) = 0  &\quad \text{in }\mathcal{S}_h,\\
        \partial_\nu u_1(x)=\partial_\nu m_1(x)=0 &\quad \text{on }\partial\Omega',\\
        u_1(x) = \psi(x), \quad m_1(x)=\varphi(x) & \quad\text{on } \partial\Omega.
    \end{cases}
\end{equation} and 
\begin{equation}\label{MFGQuadraticStat0}
    \begin{cases}
        - \Delta u_0(x) + \frac{1}{2}\kappa_0(x) |\nabla u_0(x)|^2 + \lambda_0 - F_0(x,m_0(x)) = 0 &\quad \text{in }B_h\backslash\mathcal{S}_h,\\
        - \Delta m_0(x) - \text{div}(\kappa_0(x) m_0(x)\nabla u_0(x)) = 0  &\quad \text{in }B_h\backslash\mathcal{S}_h,\\
        \partial_\nu u_0(x)=\partial_\nu m_0(x)=0 &\quad \text{on }\partial\Omega',\\
        u_0(x) = \psi(x), \quad m_0(x)=\varphi(x) & \quad\text{on } \partial\Omega
    \end{cases}
\end{equation}
respectively, such that 
\begin{equation}
    \begin{cases}
        \partial_\nu u_0(x)=\partial_\nu u_1(x), \quad \partial_\nu m_0(x) = \partial_\nu m_1(x) &\quad \text{on }\partial\mathcal{S}_h\backslash\partial B_h,\\
        u_0(x) = u_1(x), \quad m_0(x)=m_1(x) & \quad\text{on } \partial\mathcal{S}_h\backslash\partial B_h.
    \end{cases}
\end{equation}

We next state our main results.

\begin{theorem}\label{thm:Statkappa}
    Suppose that $\kappa\in\mathcal{B}$, $\lambda\in\mathcal{C}$ and $F\in\mathcal{A}_\Omega$. Then $ D$ and $\kappa$ are uniquely determined by the boundary measurement $\mathcal{M}_{ D,H,F}$.
\end{theorem}

\begin{theorem}\label{thm:StatF}
    Suppose that $\lambda\in\mathcal{C}$, $F\in\mathcal{D}$ . Then $ D$ and $F$ are uniquely determined by the boundary measurement $\mathcal{M}_{ D,H,F}$. 
\end{theorem}

\subsection{Proof of Main Results}\label{sect:MainProof}

We next proceed to prove the main Theorems \ref{thm:Statkappa} and \ref{thm:StatF}. Before that, we begin first with a main auxiliary theorem. 

\begin{theorem}\label{AuxThm}
    Let $ D\Subset\Omega$ be the bounded Lipschitz domain such that $\Omega\backslash\bar{ D}$ is connected, with conic corner $\mathcal{S}_h$. For $P\in L^2(\Omega)$, suppose that $q_1,q_0$ are $C^\gamma$ H\"older-continuous for some $\gamma\in(0,1)$ with respect to $x\in  D$, such that $q_1\neq q_0$ in $\mathcal{S}_h$. For any function $h\in C^1(\Omega)$, consider the following system of equations for $v_j\in C^{2+\alpha}(\mathcal{S}_h)$, $j=1,0$:
    \begin{equation}\label{AuxThmEq}
    \begin{cases}
		- \Delta v_j(x)+ P + q_j(x)h(x)= 0& \quad \text{ in } \mathcal{S}_h,\\
		\partial_\nu v_1(x)=\partial_\nu v_0(x) &\quad \text{ on }\partial\mathcal{S}_h\backslash\partial B_h,\\
        v_1(x)=v_0(x) &\quad \text{ on }\partial\mathcal{S}_h\backslash\partial B_h.
    \end{cases}  	
\end{equation}
Then, it must hold that \[h(x_c) =0,\] where $x_c$ is the apex of $\mathcal{S}_h$.
\end{theorem}

\begin{proof}

Taking the difference between the $j$-th equation, we have, denoting $\tilde{v}=v_1-v_0$
\begin{equation}\label{Linear2jUDiff}
    \begin{cases}
		- \Delta \tilde{v}(x)+(q_1(x)-q_0(x))h(x)= 0& \quad \text{ in } \mathcal{S}_h\\
		\partial_\nu \tilde{v}(x)=0 & \quad \text{ on }\partial\mathcal{S}_h\backslash\partial B_h,\\
		\tilde{v}(x)=0 & \quad \text{ on } \partial\mathcal{S}_h\backslash\partial B_h.
    \end{cases}  	
\end{equation}
Multiplying this by the solution $w$ to \eqref{CGOEq} given by \eqref{CGO} and integrating in the truncated cone $\mathcal{S}_h$, we have, by Green's formula, 
\begin{equation}\label{k1DiffEq}
    \int_{\mathcal{S}_h}(q_1(x) -q_0(x))h(x)  w(x) \,dx = - \int_{\partial \mathcal{S}_h} w(x) \partial_\nu \tilde{v}(x) \, d\sigma + \int_{\partial \mathcal{S}_h} \tilde{v}(x) \partial_\nu w(x) \, d\sigma.
\end{equation}

Since $q_j(x)\in C^\gamma(\mathcal{S}_h)$, $\gamma\in(0,1)$ and $h \in C^1(\Omega)$ for $j=1,0$, the function 
\begin{equation}\label{barkappa}\bar{q}:=(q_1(x) -q_0(x))h(x)\in C^\gamma(\mathcal{S}_h).\end{equation} 
Therefore, we can expand $\bar{q}$ as follows:
\[\bar{q}=\bar{q}(x_c)+\delta\bar{q},\quad |\delta\bar{q}|\leq \norm{\bar{q}}_{C^\gamma(\mathcal{S}_h)}|x-x_c|^\gamma,\] 
where \[\bar{q}(x_c)=(q_1(x_c) -q_0(x_c))h(x_c).\] Thus, the left-hand-side of \eqref{k1DiffEq} can be expanded as
\begin{align*}
    \int_{\mathcal{S}_h}\bar{q}  w(x) \,dx = \bar{q}(x_c)\int_{\mathcal{S}_h} w(x) \,dx + \int_{\mathcal{S}_h}\delta\bar{q}w \,dx,
\end{align*}
where
\begin{align*}
    \left|\int_{\mathcal{S}_h}\delta\bar{q}w \,dx\right| \leq \norm{\bar{q}}_{C^\gamma(\mathcal{S}_h)} \int_{\mathcal{S}_h}|x-x_c|^\gamma |w| \,dx    
\end{align*}
for 
\begin{equation}\label{PfEq5}\norm{\bar{q}}_{C^\gamma(\mathcal{S}_h)}\leq \norm{h}_{C^\gamma(\mathcal{S}_h)}\sup_{\mathcal{S}_h}\left|q_1 - q_0\right| + \norm{q_1 -q_0}_{C^\gamma(\mathcal{S}_h)}\sup_{\mathcal{S}_h} h\leq C\end{equation} for some constant $C$ depending on the norm of $h$. 

On the other hand, the right-hand-side of \eqref{k1DiffEq} can be analysed as follows: By the Cauchy-Schwarz inequality and the trace theorem, 
\begin{align*}
    \left|\int_{\partial \mathcal{S}_h} w(x) \partial_\nu \tilde{v}(x) \, d\sigma\right| & \leq \norm{w}_{H^{\frac12}(\partial \mathcal{S}_h)}\norm{\partial_\nu \tilde{v}}_{H^{-\frac12}(\partial \mathcal{S}_h)} \\
    & \leq C \norm{w}_{H^1(\partial \mathcal{S}_h)}\norm{\tilde{v}}_{H^1(\mathcal{S}_h)}
    \\
    & \lesssim (2\tau^2+1)^{\frac12}e^{-\rho h\tau},
\end{align*}
by \eqref{CGOEst3}, while
\begin{align*}
    \left|\int_{\partial \mathcal{S}_h} \tilde{v}(x) \partial_\nu w(x) \, d\sigma\right| & \leq \norm{\tilde{v}}_{L^2(\partial \mathcal{S}_h)}\norm{\partial_\nu w}_{L^2(\partial \mathcal{S}_h)} \\
    & \leq C \norm{\tilde{v}}_{H^1(\partial \mathcal{S}_h)}\norm{\partial_\nu w}_{L^2(\partial\mathcal{S}_h)}
    \\
    & \lesssim \sqrt{2}\tau e^{-\rho h\tau}.
\end{align*}
by \eqref{CGOEst4}. Note that the (trace or Sobolev) constants $C>0$ here may be different, and different from that in \eqref{PfEq5}.

Combining these estimates and \eqref{CGOEst2}, we have
\begin{align*}
    \bar{q}(x_c)\int_{\mathcal{S}_h} w(x) \,dx & \lesssim \int_{\mathcal{S}_h}|x-x_c|^\gamma |w| \,dx + (2\tau^2+1)^{\frac12}e^{-\zeta h\tau} + \sqrt{2}\tau e^{-\zeta h\tau} \\
    & \lesssim \tau^{-(\gamma+n)}+\frac{1}{\tau}e^{-\frac{1}{2}\rho h \tau} + (2\tau^2+1)^{\frac12}e^{-\rho h\tau} + \sqrt{2}\tau e^{-\rho h\tau}.
\end{align*}
At the same time, applying the estimate \eqref{CGOEst1}, we have 
\begin{align*}
    \bar{q}(x_c) \left[C_{\mathcal{S}_h}\tau^{-n} +\mathcal{O}\left(\frac{1}{\tau}e^{-\frac{1}{2}\rho h \tau}\right)\right] \lesssim \tau^{-(\gamma+n)}+\frac{1}{\tau}e^{-\frac{1}{2}\rho h \tau} + (1+\tau)e^{-\rho h\tau}
\end{align*}
when $\tau\to\infty$.

Multiplying by $\tau^n $ on both sides and letting $\tau\to\infty$, we have that $\bar{q}(x_c)=0$. Since $q_1(x_c) \neq q_0(x_c)$, by the definition of $\bar{q}$ in \eqref{barkappa}, it must be that $h(x_c)=0$.

\end{proof}

\subsubsection{A Corner Singularity due to Different Hamiltonians}\label{sect:ThmStatKappaPf}

We next proceed to prove Theorem \ref{thm:Statkappa}.

\begin{proof}[Proof of Theorem \ref{thm:Statkappa}]
Since $\psi,\varphi$ is fixed in \eqref{MeasureMap}, we can take $\varphi(x)=\sum_{\ell=1}^N\varepsilon^\ell g_\ell$ such that $g_1=g_2=0$ and $g_3>0$, and $\psi(x)=\sum_{\ell=1}^2\varepsilon^\ell f_\ell$ such that the solution $u^{(1)}$ of the first order linearized system corresponding to $f_1$ has non-zero gradient on $\partial D$. This ensures that the results are physically meaningful with $m>0$.

Now, suppose on the contrary that there exists $ D_1$ and $ D_2$, such that $ D_1\neq D_2$. Since $ D_1, D_2$ are polyhedrons, it must be that there exists a corner $\mathcal{S}$ of $ D_1$ such that $\mathcal{S}\Subset\Omega\backslash\overline{ D_2}$. We also assume that for $ D_i$, $\kappa\in\mathcal{B}$ with the functions $\kappa_0$ and $\kappa_i$ for $i=1,2$ respectively. Then, in the neighbourhood $B_h$ of $\mathcal{S}_h$, it holds that 
\begin{equation}\label{MFGQuadraticStatExtCornerkappa}
    \begin{cases}
        - \Delta u_1(x) + \frac{1}{2}\kappa_1(x) |\nabla u_1(x)|^2 + \lambda -  F(x,m_1(x)) = 0 &\quad \text{in }\mathcal{S}_h,\\
        - \Delta m_1(x) - \text{div}(\kappa_1(x) m_1(x)\nabla u_1(x)) = 0  &\quad \text{in }\mathcal{S}_h,\\
        - \Delta u_0(x) + \frac{1}{2}\kappa_0(x) |\nabla u_0(x)|^2 + \lambda - F(x,m_0(x)) = 0 &\quad \text{in }B_h\backslash\mathcal{S}_h,\\
        - \Delta m_0(x) - \text{div}(\kappa_0(x) m_0(x)\nabla u_0(x)) = 0  &\quad \text{in }B_h\backslash\mathcal{S}_h,\\
        \partial_\nu u_1(x)=\partial_\nu u_0(x), \quad \partial_\nu m_1(x)=\partial_\nu m_0(x) &\quad \text{on }\partial\mathcal{S}_h\backslash\partial B_h,\\
        u_1(x)=u_0(x), \quad m_1(x)=m_0(x) &\quad \text{on }\partial\mathcal{S}_h\backslash\partial B_h.
    \end{cases}
\end{equation}

For $F\in\mathcal{A}_\Omega$ with $F(x,0)=\lambda$, we conduct higher order linearisation separately for $(u_1,m_1)$ and $(u_0,m_0)$ in $\mathcal{S}_h$ and $B_h\backslash\mathcal{S}_h$ respectively. By the choice of $\varphi$ with $g_1=0$, it is straightforward to obtain from the first order linearisation  that $m^{(1)}_1=m^{(1)}_0\equiv0$, and subsequently $m^{(2)}_1=m^{(2)}_0\equiv0$ from the second order linearisation with $g_2=0$. It can also be easily seen that as a result, when $\mathcal{M}_{ D_1,H_1,F}=\mathcal{M}_{ D_2,H_2,F}$, $u^{(1)} := u^{(1)}_0 + (u^{(1)}_1 - u^{(1)}_0)\chi_ D \in C^{2+\alpha}(\Omega)$ is such that $u^{(1)}_0 = u^{(1)}_1$ on $\partial D$, since they satisfy 
\begin{equation}
    \begin{cases}
		- \Delta u^{(1)}_1= 0  & \quad\text{ in }  \mathcal{S}_h,\\
        - \Delta u^{(1)}_0= 0  & \quad\text{ in }  B_h\backslash\mathcal{S}_h,\\
        \partial_\nu u_1^{(1)}(x)=\partial_\nu u_0^{(1)}(x) &\quad \text{ on }\partial\mathcal{S}_h\backslash\partial B_h,\\
        u_1^{(1)}(x)=u_0^{(1)}(x) &\quad \text{ on }\partial\mathcal{S}_h\backslash\partial B_h,\\
        u_1^{(1)}(x)=u_0^{(1)}(x) = f_1 &\quad \text{ on }\partial\Omega,
    \end{cases}
\end{equation}
from the first order linearisation. 

On the other hand, $u^{(2)}_j$ satisfies
\begin{equation}\label{Linear2jU}
    \begin{cases}
		- \Delta u^{(2)}_1(x)+\kappa_1(x)|\nabla u^{(1)}|^2= 0& \quad \text{ in } \mathcal{S}_h,\\
        - \Delta u^{(2)}_0(x)+\kappa_0(x)|\nabla u^{(1)}|^2= 0& \quad \text{ in } B_h\backslash\mathcal{S}_h,\\
		\partial_\nu u_1^{(2)}(x)=\partial_\nu u_0^{(2)}(x) &\quad \text{ on }\partial\mathcal{S}_h\backslash\partial B_h,\\
        u_1^{(2)}(x)=u_0^{(2)}(x) &\quad \text{ on }\partial\mathcal{S}_h\backslash\partial B_h.
    \end{cases}  	
\end{equation}
Since $u^{(1)}$ is fixed and smooth, this is simply an elliptic equation. Therefore, we can apply the unique continuation principle for elliptic equations \cite{KochTataru2001UCPElliptic} to obtain 
\[
    \begin{cases}
		- \Delta u^{(2)}_1(x)+\kappa_1(x)|\nabla u^{(1)}|^2 = - \Delta u^{(2)}_0(x)+\kappa_0(x)|\nabla u^{(1)}|^2 = 0& \quad \text{ in } \mathcal{S}_h,\\
		\partial_\nu u_1^{(2)}(x)=\partial_\nu u_0^{(2)}(x) &\quad \text{ on }\partial\mathcal{S}_h\backslash\partial B_h,\\
        u_1^{(2)}(x)=u_0^{(2)}(x) &\quad \text{ on }\partial\mathcal{S}_h\backslash\partial B_h,
    \end{cases} 
\]
which is simply \eqref{AuxThmEq} with $v_j=u^{(2)}_j$, $P=0$, $q_j=\kappa_j$ and $h=|\nabla u^{(1)}|^2$. Since $\kappa\in\mathcal{B}$ and $u\in C^2(\Omega')$ by Theorem \ref{ForwardPbMainThm}, the assumptions of Theorem \ref{AuxThm} are satisfied, and we have that $h(x_c)=|\nabla u^{(1)}(x_c)|^2=0$, i.e. $\nabla u^{(1)}(x_c)=0$, for the apex $x_c$ of $\mathcal{S}$. 

Yet, we have fixed $\psi(x)$ such that $u^{(1)}$ has non-zero gradient on $\partial D_1\supset\partial\mathcal{S}$. Thus, we arrive at a contradiction, and it must be that $ D_1= D_2$ and we have uniquely determined $ D$.

Moreover, when $\mathcal{M}_{ D_1,H_1,F}=\mathcal{M}_{ D_2,H_2,F}$, we have that the solution sets $(u_1,m_1),(u_2,m_2)$ satisfy 
\begin{equation}
    \begin{cases}
        - \Delta u_1(x) + \frac{1}{2}\kappa_1(x) |\nabla u_1(x)|^2 + \lambda - F(x,m_1(x)) = 0 &\quad \text{in }\mathcal{S}_h,\\
        - \Delta m_1(x) - \text{div}(\kappa_1(x) m_1(x)\nabla u_1(x)) = 0  &\quad \text{in }\mathcal{S}_h,\\
        - \Delta u_2(x) + \frac{1}{2}\kappa_2(x) |\nabla u_2(x)|^2 + \lambda - F(x,m_2(x)) = 0 &\quad \text{in }\mathcal{S}_h,\\
        - \Delta m_2(x) - \text{div}(\kappa_2(x) m_2(x)\nabla u_2(x)) = 0  &\quad \text{in }\mathcal{S}_h,\\
        \partial_\nu u_1(x)=\partial_\nu u_2(x), \quad \partial_\nu m_1(x)=\partial_\nu m_2(x) &\quad \text{on }\partial\mathcal{S}_h\backslash\partial B_h,\\
        u_1(x)=u_2(x), \quad m_1(x)=m_2(x) &\quad \text{on }\partial\mathcal{S}_h\backslash\partial B_h.
    \end{cases}
\end{equation}
Repeating the argument above, we have that $m^{(1)}_1=m^{(1)}_2=m^{(2)}_1=m^{(2)}_2\equiv0$ in $\Omega$, $u^{(1)}_1 = u^{(1)}_2$ in $\mathcal{S}_h$, and $u^{(2)}_j$ satisfies
\begin{equation}\label{PfkappaEq5}
    \begin{cases}
		- \Delta u^{(2)}_1(x)+\kappa_1(x)|\nabla u^{(1)}|^2 = - \Delta u^{(2)}_2(x)+\kappa_2(x)|\nabla u^{(1)}|^2 = 0& \quad \text{ in } \mathcal{S}_h,\\
		\partial_\nu u_1^{(2)}(x)=\partial_\nu u_2^{(2)}(x) &\quad \text{ on }\partial\mathcal{S}_h\backslash\partial B_h,\\
        u_1^{(2)}(x)=u_2^{(2)}(x) &\quad \text{ on }\partial\mathcal{S}_h\backslash\partial B_h.
    \end{cases} 
\end{equation}
Then, once again by Theorem \ref{AuxThm}, if $\kappa_1(x_c)\neq\kappa_2(x_c)$, we have that $\nabla u^{(1)}(x_c)=0$. This once again contradicts the choice of $\psi$, so it must be that $\kappa_1(x_c)=\kappa_2(x_c)$ and we have uniquely determined $\kappa$.

Observe that in this second part of the proof for the uniqueness of $\kappa$, it is not necessary to apply the unique continuation principle, since the equations in \eqref{PfkappaEq5} for $u^{(2)}_1$ and $u^{(2)}_2$ are both in $\mathcal{S}_h$.

\end{proof}

Note that in the proof, we only require that the first two Taylor coefficients of $m$ are zero on the boundary of $\Omega$. This allows us to ensure the physical relevance of the problem, by picking positive higher order Taylor coefficients, to ensure that $m$ is positive.

\subsubsection {A Corner Singularity due to Different Running Costs $F$}\label{sect:ThmStatFPf}

Finally, we prove Theorem \ref{thm:StatF}.

\begin{proof}[Proof of Theorem \ref{thm:StatF}]
Since $\psi,\varphi$ is fixed in \eqref{MeasureMap}, we can take $\psi(x)=0$ and $\varphi(x)=\sum_{\ell=1}^N\varepsilon^\ell g_\ell$, such that the corresponding solution $m^{(1)}$ to $g_1>0$ of the first order linearized system satisfies $m^{(1)}\neq0$ on $\partial D$. Once again, this guarantees the positivity of $m$, making the problem physically meaningful.

As in the proof of Theorem \ref{thm:Statkappa}, we suppose on the contrary that there exists $ D_1$ and $ D_2$, such that $ D_1\neq D_2$. Since $ D_1, D_2$ are polyhedrons, it must be that there exists a corner $\mathcal{S}$ of $ D_1$ such that $\mathcal{S}\Subset\Omega\backslash\overline{ D_2}$. We also assume that for $ D_i$, $F\in\mathcal{D}$ with the functions $F_0$ and $F_i$ for $i=1,2$ respectively. Then, in the neighbourhood $B_h$ of $\mathcal{S}_h$, it holds that 
\begin{equation}\label{MFGQuadraticStatExtCornerF}
    \begin{cases}
        - \Delta u_1(x) + \frac{1}{2}\kappa(x) |\nabla u_1(x)|^2 + \lambda_1 - F_1(x,m_1(x)) = 0 &\quad \text{in }\mathcal{S}_h,\\
        - \Delta m_1(x) - \text{div}(\kappa(x) m_1(x)\nabla u_1(x)) = 0  &\quad \text{in }\mathcal{S}_h,\\
        - \Delta u_0(x) + \frac{1}{2}\kappa(x) |\nabla u_0(x)|^2 + \lambda_0 - F_0(x,m_0(x)) = 0 &\quad \text{in }B_h\backslash\mathcal{S}_h,\\
        - \Delta m_0(x) - \text{div}(\kappa(x) m_0(x)\nabla u_0(x)) = 0  &\quad \text{in }B_h\backslash\mathcal{S}_h,\\
        \partial_\nu u_1(x)=\partial_\nu u_0(x), \quad \partial_\nu m_1(x)=\partial_\nu m_0(x) &\quad \text{on }\partial\mathcal{S}_h\backslash\partial B_h,\\
        u_1(x)=u_0(x), \quad m_1(x)=m_0(x) &\quad \text{on }\partial\mathcal{S}_h\backslash\partial B_h.
    \end{cases}
\end{equation}

For $F_1\in\mathcal{A}_ D$ and $F_0\in\mathcal{A}_{\Omega\backslash D}$ with $F_1(x,0)=\lambda_1$ and $F_0(x,0)=\lambda_0$, we can conduct higher order linearisation separately for $(u_1,m_1)$ and $(u_0,m_0)$ in $\mathcal{S}_h$ and $B_h\backslash\mathcal{S}_h$ respectively. 

We first show the case when $\ell=1$. When $\mathcal{M}_{ D_1,H,F_1}=\mathcal{M}_{ D_2,H,F_2}$, we have that $m^{(1)} := m^{(1)}_0 + (m^{(1)}_1 - m^{(1)}_0)\chi_ D \in C^{2+\alpha}(\Omega)$ with $m^{(1)}_0 = m^{(1)}_1$ on $\partial D$, since they satisfy 
\begin{equation}
    \begin{cases}
		- \Delta m^{(1)}_1= 0  & \quad\text{ in }  \mathcal{S}_h,\\
        - \Delta m^{(1)}_0 = 0  & \quad\text{ in }  B_h\backslash\mathcal{S}_h,\\
        \partial_\nu m_1^{(1)}(x)=\partial_\nu m_0^{(1)}(x) &\quad \text{ on }\partial\mathcal{S}_h\backslash\partial B_h,\\
        m_1^{(1)}(x)=m_0^{(1)}(x) &\quad \text{ on }\partial\mathcal{S}_h\backslash\partial B_h,\\
        m_1^{(1)}(x)=m_0^{(1)}(x) = g_1 &\quad \text{ on }\partial\Omega,
    \end{cases}
\end{equation}
from the first order linearisation. 

On the other hand, $u^{(1)}_j$ satisfies
\begin{equation}\label{Linear1m}
    \begin{cases}
		- \Delta u^{(1)}_1= F^{(1)}_1(x)m^{(1)}(x)  & \quad\text{ in }  \mathcal{S}_h,\\
        - \Delta u^{(1)}_0= F^{(1)}_0(x)m^{(1)}(x)  & \quad\text{ in }  B_h\backslash\mathcal{S}_h,\\
        \partial_\nu u_1^{(1)}(x)=\partial_\nu u_0^{(1)}(x) &\quad \text{ on }\partial\mathcal{S}_h\backslash\partial B_h,\\
        u_1^{(1)}(x)=u_0^{(1)}(x) &\quad \text{ on }\partial\mathcal{S}_h\backslash\partial B_h.
    \end{cases}
\end{equation}
Once again, observe that this is simply an elliptic equation, so we can apply the unique continuation principle for elliptic equations \cite{KochTataru2001UCPElliptic} to obtain 
\[
    \begin{cases}
		- \Delta u^{(1)}_1 - F^{(1)}_1(x)m^{(1)}(x) = - \Delta u^{(1)}_0 - F^{(1)}_0(x)m^{(1)}(x) = 0  & \quad\text{ in }  \mathcal{S}_h,\\
        \partial_\nu u_1^{(1)}(x)=\partial_\nu u_0^{(1)}(x) &\quad \text{ on }\partial\mathcal{S}_h\backslash\partial B_h,\\
        u_1^{(1)}(x)=u_0^{(1)}(x) &\quad \text{ on }\partial\mathcal{S}_h\backslash\partial B_h,
    \end{cases}
\]
which is simply \eqref{AuxThmEq} with $v_j=u^{(1)}_j$, $P=0$, $q_j=F^{(1)}_j$ and $h=m^{(1)}$. Since $F\in\mathcal{D}$ and $m^{(1)}\in C^{2+\alpha}(\mathcal{S}_h)$ by Theorem \ref{local_wellpose}, the assumptions of Theorem \ref{AuxThm} are satisfied with $F^{(1)}_1(x_c)\neq F^{(1)}_0(x_c)$, and we have that $h(x_c)=m^{(1)}(x_c)=0$ for the apex $x_c$ of $\mathcal{S}$. 

Yet, we have fixed $\varphi(x)$ such that $m^{(1)}\neq0$ on $\partial D_1\supset\partial\mathcal{S}$. Thus, we arrive at a contradiction, and it must be that $ D_1= D_2$ and we have uniquely determined $ D$.

Moreover, when $\mathcal{M}_{ D_1,H,F_1}=\mathcal{M}_{ D_2,H,F_2}$, we have that the solution sets $(u_1,m_1),(u_2,m_2)$ satisfy 
\begin{equation}\label{PfF12}
    \begin{cases}
        - \Delta u_1(x) + \frac{1}{2}\kappa(x) |\nabla u_1(x)|^2 + \lambda_1 - F_1(x,m_1(x)) = 0 &\quad \text{in }\mathcal{S}_h,\\
        - \Delta m_1(x) - \text{div}(\kappa(x) m_1(x)\nabla u_1(x)) = 0  &\quad \text{in }\mathcal{S}_h,\\
        - \Delta u_2(x) + \frac{1}{2}\kappa(x) |\nabla u_2(x)|^2 + \lambda_2 - F_2(x,m_2(x)) = 0 &\quad \text{in }\mathcal{S}_h,\\
        - \Delta m_2(x) - \text{div}(\kappa(x) m_2(x)\nabla u_2(x)) = 0  &\quad \text{in }\mathcal{S}_h,\\
        \partial_\nu u_1(x)=\partial_\nu u_2(x), \quad \partial_\nu m_1(x)=\partial_\nu m_2(x) &\quad \text{on }\partial\mathcal{S}_h\backslash\partial B_h,\\
        u_1(x)=u_2(x), \quad m_1(x)=m_2(x) &\quad \text{on }\partial\mathcal{S}_h\backslash\partial B_h.
    \end{cases}
\end{equation}
Repeating the argument above, we have that $m^{(1)}_1 = m^{(1)}_2$ in $\mathcal{S}_h$, and $u^{(1)}_j$ satisfies
\begin{equation}
    \begin{cases}
		- \Delta u^{(1)}_1 - F^{(1)}_1(x)m^{(1)}(x) = - \Delta u^{(1)}_2 - F^{(1)}_2(x)m^{(1)}(x) = 0  & \quad\text{ in }  \mathcal{S}_h,\\
        \partial_\nu u_1^{(1)}(x)=\partial_\nu u_2^{(1)}(x) &\quad \text{ on }\partial\mathcal{S}_h\backslash\partial B_h,\\
        u_1^{(1)}(x)=u_2^{(1)}(x) &\quad \text{ on }\partial\mathcal{S}_h\backslash\partial B_h.
    \end{cases} 
\end{equation}
Then, once again by Theorem \ref{AuxThm}, if $F^{(1)}_1(x_c)\neq F^{(1)}_2(x_c)$, we have that $m^{(1)}(x_c)=0$. This once again contradicts the choice of $\varphi$, so it must be that $F^{(1)}_1(x_c)=F^{(1)}_2(x_c)$ and we have uniquely determined $F^{(1)}$.

Next, we assume that $F^{(1)}_1= F^{(1)}_0$ in $\mathcal{S}_h$ but $F^{(2)}_1\neq F^{(2)}_0$ in $\mathcal{S}_h$, i.e. when $\ell=2$ in Definition \ref{FDef}. In this case, from \eqref{Linear1m}, when $\mathcal{M}_{ D_1,H,F_1}=\mathcal{M}_{ D_2,H,F_2}$, in addition to the result on $m^{(1)}$ in the case of $\ell=1$, we have $u^{(1)} := u^{(1)}_0 + (u^{(1)}_1 - u^{(1)}_0)\chi_ D \in C^{2+\alpha}(\Omega)$ with $u^{(1)}_0 = u^{(1)}_1$ on $\partial D$. Then, from the second order linearized system, we can see that $m^{(2)} := m^{(2)}_0 + (m^{(2)}_1 - m^{(2)}_0)\chi_ D \in C^{2+\alpha}(\Omega)$ with $m^{(2)}_0 = m^{(2)}_1$ on $\partial D$ when $\mathcal{M}_{ D_1,H,F_1}=\mathcal{M}_{ D_2,H,F_2}$, while $u^{(2)}_j$ satisfies
\begin{equation}
    \begin{cases}
		- \Delta u^{(2)}_1(x)+\kappa(x)|\nabla u^{(1)}(x)|^2= F^{(1)}(x)m^{(2)}(x)+F^{(2)}_1(x)[m^{(1)}(x)]^2  & \quad\text{ in }  \mathcal{S}_h,\\
        - \Delta u^{(2)}_0(x)+\kappa(x)|\nabla u^{(1)}(x)|^2= F^{(1)}(x)m^{(2)}(x)+F^{(2)}_0(x)[m^{(1)}(x)]^2  & \quad\text{ in }  B_h\backslash\mathcal{S}_h,\\
        \partial_\nu u_1^{(2)}(x)=\partial_\nu u_0^{(2)}(x) &\quad \text{ on }\partial\mathcal{S}_h\backslash\partial B_h,\\
        u_1^{(2)}(x)=u_0^{(2)}(x) &\quad \text{ on }\partial\mathcal{S}_h\backslash\partial B_h.
    \end{cases}
\end{equation}
Since $u^{(1)}, m^{(1)}, m^{(2)}$ are fixed and smooth, this is simply an elliptic equation. Therefore, we can once again apply the unique continuation principle for elliptic equations \cite{KochTataru2001UCPElliptic} to obtain 
\begin{equation}\label{Linear2m}
    \begin{cases}
		- \Delta u^{(2)}_1(x)+\kappa(x)|\nabla u^{(1)}(x)|^2 - F^{(1)}(x)m^{(2)}(x) - F^{(2)}_1(x)[m^{(1)}(x)]^2 = 0  & \quad\text{ in }  \mathcal{S}_h,\\
        - \Delta u^{(2)}_0(x)+\kappa(x)|\nabla u^{(1)}(x)|^2 - F^{(1)}(x)m^{(2)}(x) - F^{(2)}_0(x)[m^{(1)}(x)]^2 =0 & \quad\text{ in }  \mathcal{S}_h,\\
        \partial_\nu u_1^{(2)}(x)=\partial_\nu u_0^{(2)}(x) &\quad \text{ on }\partial\mathcal{S}_h\backslash\partial B_h,\\
        u_1^{(2)}(x)=u_0^{(2)}(x) &\quad \text{ on }\partial\mathcal{S}_h\backslash\partial B_h.
    \end{cases}
\end{equation}
This is simply \eqref{AuxThmEq} with $v_j=u^{(2)}_j$, $P=\kappa|\nabla u^{(1)}|^2 - F^{(1)}m^{(2)}$, $q_j=F^{(2)}_j$ and $h=[m^{(1)}]^2$. Since $F\in\mathcal{D}$ and $m^{(1)}\in C^{2+\alpha}(\mathcal{S}_h)$ by Theorem \ref{local_wellpose}, the assumptions of Theorem \ref{AuxThm} are satisfied with $F^{(2)}_1(x_c)\neq F^{(2)}_0(x_c)$, and we have that $h(x_c)=[m^{(1)}(x_c)]^2=0$, i.e. $m^{(1)}(x_c)=0$, for the apex $x_c$ of $\mathcal{S}$. But $m^{(1)}\neq0$ on $\partial D_1\supset\partial\mathcal{S}$ by our choice of $\varphi(x)$. Thus, we arrive at a contradiction, and it must be that $ D_1= D_2$ and we have uniquely determined $ D$.

Moreover, when $\mathcal{M}_{ D_1,H,F_1}=\mathcal{M}_{ D_2,H,F_2}$, we have that the solution sets $(u_1,m_1),(u_2,m_2)$ satisfy \eqref{PfF12}.
Repeating the argument above, we have that $m^{(1)}_1 = m^{(1)}_2$, $u^{(1)}_1 = u^{(1)}_2$ and $m^{(2)}_1 = m^{(2)}_2$ in $\mathcal{S}_h$, and $u^{(2)}_j$ satisfies
\begin{equation}
    \begin{cases}
		- \Delta u^{(2)}_1(x)+\kappa(x)|\nabla u^{(1)}(x)|^2 - F^{(1)}(x)m^{(2)}(x) - F^{(2)}_1(x)[m^{(1)}(x)]^2 = 0  & \quad\text{ in }  \mathcal{S}_h,\\
        - \Delta u^{(2)}_2(x)+\kappa(x)|\nabla u^{(1)}(x)|^2 - F^{(1)}(x)m^{(2)}(x) - F^{(2)}_2(x)[m^{(1)}(x)]^2 =0 & \quad\text{ in }  \mathcal{S}_h,\\
        \partial_\nu u_1^{(2)}(x)=\partial_\nu u_0^{(2)}(x) &\quad \text{ on }\partial\mathcal{S}_h\backslash\partial B_h,\\
        u_1^{(2)}(x)=u_0^{(2)}(x) &\quad \text{ on }\partial\mathcal{S}_h\backslash\partial B_h.
    \end{cases}
\end{equation}
Then, once again by Theorem \ref{AuxThm}, if $F^{(2)}_1(x_c)\neq F^{(2)}_2(x_c)$, we have that $m^{(1)}(x_c)=0$. This once again contradicts the choice of $\varphi$, so it must be that $F^{(2)}_1(x_c)=F^{(2)}_2(x_c)$ and we have uniquely determined $F^{(2)}$.

The case for higher order $\ell$ follows similarly, by first deducing the equality of $u^{(\ell')}_1=u^{(\ell')}_0$ for $\ell'<\ell$ and $m^{(\ell'')}_1=m^{(\ell'')}_0$ for $\ell''\leq \ell$ on $\partial D$, which then gives an elliptic equation in the $\ell$-th order linearisation for $u^{(\ell)}_j$. Then, applying the unique continuation principle for elliptic equations, we obtain \eqref{AuxThmEq} with sufficient regularity of the coefficients. Applying Theorem \ref{AuxThm} then gives $h(x_c)=[m^{(1)}(x_c)]^\ell=0$, which contradicts our choice of $\varphi$. Therefore we have uniquely determined $ D$. 

Then, the solution sets $(u_1,m_1),(u_2,m_2)$ once again satisfy \eqref{PfF12}. Repeating the argument, we have that $u^{(\ell')}_1=u^{(\ell')}_2$ for $\ell'<\ell$ and $m^{(\ell'')}_1=m^{(\ell'')}_2$ for $\ell''\leq \ell$ in $\mathcal{S}_h$, while $u^{(\ell)}_j$ satisfies an equation of the form \eqref{AuxThmEq}. Applying Theorem \ref{AuxThm} then gives $m^{(1)}(x_c)=0$, which again contradicts our choice of $\varphi$. Thus we have uniquely determined $F^{(\ell)}$. 

\end{proof}

\chapter{Single Measurement Using Carleman Estimate}\label{chap:Carleman}

In this chapter, we discuss stability results related to the MFG system \eqref{eq:MFG1}. We will make use of Carleman estimates. Notably, this method enables the determination of the unknowns in the system from a single observation, i.e. the measurement of one solution $(u,m)$, rather than a full set of boundary or interior measurements. A Carleman estimate for a partial differential operator is always proven only for the principal part of this operator since it is independent on its lower order terms. This makes Carleman estimates instrumental in establishing uniqueness and stability results in diverse inverse problems.

\section{First Stability Result}
Let $\Omega \subset \mathbb{R}^n$ be a smooth bounded domain and $Q:= \Omega \times (0,T).$ 
For our first result, we consider the simple case of the linearized equation \eqref{MFG2Linear1} of \eqref{main_linear example} about a known state $(u_0,m_0)$
as follows:
\begin{equation}\label{(1.3)}
\begin{cases}
\partial_tu + A(t)u = c_0(x,t)m + F(x,t),&\text{in }Q \\
\partial_tm - B(t)m = A_0(t)u + G(x,t), &\text{in }Q
\end{cases}
\end{equation}
with the Robin boundary condition
\begin{equation}\label{(1.4)}
\begin{cases}
\partial_{\nu_A}u(x,t) - p(x,t)u(x,t) = g(x,t),&\text{on }\Sigma \\
\partial_{\nu_B}m(x,t) - q(x,t)m(x,t) = h(x,t), &\text{on }\Sigma,
\end{cases}
\end{equation}
for
\begin{equation}\label{1.7}
p, q \in C^1(\partial\Omega\times [0,T]) \quad \mbox{and}\quad c_0\in L^\infty(Q).
\end{equation}
Here, the second order partial differential
operators are defined
by
\begin{equation}\label{eq:coef1}
\begin{cases}
A(t)u:= \sum\limits_{i,j=1}^n a_{ij}(x,t)\partial_i\partial_j u + \sum\limits_{j=1}^n a_j(x,t)\partial_ju
+ a_0(x,t)u, \\
B(t)m:= \sum\limits_{i,j=1}^n b_{ij}(x,t)\partial_i\partial_j m + \sum\limits_{j=1}^n b_j(x,t)\partial_jv
+ b_0(x,t)m,\\
A_0(t)u := \sum\limits_{| \gamma| \le 2}\widetilde  a_{\gamma}(x,t)
\partial_x^{\gamma}u,
\end{cases}
\end{equation}
where we assume that
\begin{equation}\label{lopukh1}
\begin{cases}
a_{ij}, b_{ij} \in C^1(\bar{Q}), \quad a_{ij} = a_{ji},\,\,
b_{ij} = b_{ji}\quad &\mbox{for } 1\le i,j\le n, \\
a_k, b_k \in L^{\infty}(Q)\quad &\mbox{for }  0\le k \le n,\\
\widetilde a_{\gamma}\in L^{\infty}(Q)\quad  &\mbox{for } 
|\gamma | \le 2,
\end{cases}
\end{equation}
and there exists a constant
$\chi > 0$ such that
\begin{equation}\label{lopukh}
\sum_{i,j=1}^n a_{ij}(x,t)\xi_i\xi_j \ge \chi \sum_{j=1}^n \xi_j^2 \quad\mbox{and}\quad
\sum_{i,j=1}^n b_{ij}(x,t)\xi_i\xi_j \ge \chi \sum_{j=1}^n \xi_j^2
\end{equation}
for all $(x,t) \in Q$ and $\xi_1, ..., \xi_d \in \mathbb{R}$.
We also define
$$
\partial_{\nu_A}u := \sum_{i,j=1}^n a_{ij}(\partial_ju)\nu_i, \quad
\partial_{\nu_B}m := \sum_{i,j=1}^n b_{ij}(\partial_jv)\nu_i  \quad
\mbox{on }\partial\Omega\times (0,T).
$$
These Robin boundary conditions \eqref{(1.4)}
are more general than the Neumann boundary conditions we considered in the previous chapters.

Let $\Gamma$ be an arbitrarily chosen
non-empty subboundary of $\partial\Omega$, $t_0 \in (0,T)$ be arbitrarily given. 
Assume
$$
F(x,t) = q_1(x,t)f_1(x), \quad G(x,t) = q_2(x,t)f_2(x), \quad (x,t) \in Q,
$$
where
\begin{equation}\label{lopukh5}
q_1, q_2 \in W^{1,\infty}(0,T;L^{\infty}(\Omega))
\end{equation}
are given functions.

In addition to \eqref{lopukh1} and \eqref{1.7}, we assume
\begin{equation}\label{lopukh2}\begin{split}
&\partial_ta_{ij}, \, \partial_tb_{ij} \in C^1(\bar{Q}), \\
&\partial_tc_0, \partial_ta_k, \partial_tb_k \in L^{\infty}(Q)
\quad \mbox{for $1\le i,j\le n$ and $0\le k \le n$},\\
&\partial_t\widetilde a_{\gamma} \in L^{\infty}(Q) \quad
\mbox{for $| \gamma | \le 2$}, \\& \partial_tp, \partial_tq
\in C^1(\partial\Omega\times [0,T]).\end{split}
\end{equation}
We arbitrarily fix $t_0 \in (0,T)$ and a non-empty open interval
$I \in (0,T)$ such that $t_0 \in I$.  Furthermore we write
\begin{equation}\label{lopukh4}
u_0(x) := u(x,t_0), \quad m_0(x):= m(x,t_0), \quad x\in \Omega.
\end{equation}

Then, our main result is as follows:
\begin{thm}[Global unconditional Lipschitz stability for an
inverse source problem] \label{t2} 
Assume that \eqref{1.7}, \eqref{lopukh1}, \eqref{lopukh},  \eqref{lopukh5}
and \eqref{lopukh2} hold true.
Let $\partial_t c_0\in L^\infty(Q), g,h, \partial_t g,\partial_t h
\in L^2(I; H^{\frac{1}{2}}(\partial\Omega))\cap H^1(\Gamma \times I)$ and  $u, m
\in H^{2,1}(Q)$ satisfy \eqref{(1.3)}, \eqref{(1.4)}, \eqref{lopukh4},
$\partial_tu, \partial_tm \in H^{2,1}(Q)$.
We assume
\begin{equation}\label{(1.5)}
| q_1(x,t_0)|>0, \quad \mbox{and}\quad | q_2(x,t_0)| > 0 \quad
\mbox{for all $x \in \overline{\Omega}$}.
\end{equation}
Then there exists a constant $C>0$ such that
\begin{align*}
& \Vert f_1\Vert_{L^2(\Omega)} + \Vert f_2\Vert_{L^2(\Omega)}
\\\le& C\biggl(\Vert u(\cdot,t_0)\Vert_{H^2(\Omega)}
+ \Vert m(\cdot,t_0)\Vert_{H^2(\Omega)}+ \sum_{k=0}^1 \left(\Vert \partial_t^ku\Vert_{H^1(\Gamma \times I)}
+ \Vert \partial_t^kv \Vert_{H^1(\Gamma \times I)}\right)\\&+\sum_{k=0}^1 (\Vert \partial_t^kg \Vert
_{H^1(I; L^2(\partial\Omega\setminus \Gamma))} +
 \Vert \partial_t^kg \Vert_{L^2(I; H^{\frac{1}{2}}(\partial\Omega))}
+ \Vert \partial_t^kh \Vert_{H^1(I; L^2(\partial\Omega\setminus \Gamma))}
+ \Vert \partial_t^kh \Vert_{L^2(I; H^{\frac{1}{2}}(\partial\Omega))}) \biggr).
\end{align*}
\end{thm}
In particular, in the case of the homogeneous Robin boundary condition, we have
\begin{cor}
Under the conditions of Theorem \ref{t2}, we assume that $g=h=0$
on $\partial\Omega\times (0,T)$.
Then there exists a constant $C>0$ such that
\begin{align*}
& \Vert f_1\Vert_{L^2(\Omega)} + \Vert f_2\Vert_{L^2(\Omega)}
\\&\le C\biggl(\Vert u(\cdot,t_0)\Vert_{H^2(\Omega)}
+ \Vert m(\cdot,t_0)\Vert_{H^2(\Omega)}
+ \sum_{k=0}^1 (\Vert \partial_t^ku\Vert_{H^1(\Gamma \times I)}
+ \Vert \partial_t^kv \Vert_{H^1(\Gamma \times I)})\biggr).
\end{align*}
\end{cor}
We emphasize that in our stability estimates, we do not use neither data
$u(\cdot,T)$ nor $m(\cdot,0)$ in $\Omega$,
nor any a priori bounds on $u,m,f_1, f_2$, but we require data $u(\cdot,t_0)$
and $m(\cdot,t_0)$ over $\Omega$ at an intermediate time $t_0 \in (0,T)$.
Our stability can be understood unconditional in the sense that
we do not need to assume any a priori boundedness conditions.

Here, by a single-event measurement, we mean that the measurement data are
associated with a single and fixed input pair $(u_T, m_0)$ and boundary data,
otherwise the data are referred to as being collected from multiple-event
measurements.

The mixed forward and backward equations in \eqref{(1.3)},
makes the forward problem difficult, but thanks to the symmetry of the time
variable in the weight function of the key Carleman estimates, this does not affect the applicability of
Carleman estimates, our main tool.

Systems with coupled principal parts usually cause
difficulty for establishing relevant Carleman estimates.
However, in our case, although the second equation in \eqref{(1.3)} is
coupled with the second-order terms $A_0(t)u$ of $u$, the first one in
\eqref{(1.3)} is coupled only
with zeroth order term of $m$, which enables us to execute an argument
of absorbing the second-order terms of $u$ by taking large parameters
of the Carleman estimate, as we will see in the following sections.

\subsection{Forward Stability Result}

We first describe a stability result for the state $(u,m)$.

\begin{thm} \label{t1}
We assume \eqref{1.7}, \eqref{lopukh1} and \eqref{lopukh}.
Moreover let $u,m \in H^{2,1}(Q)$ satisfy \eqref{(1.3)} and \eqref{(1.4)}.
For arbitrarily given $\varepsilon > 0$, we can find a constant $C_{\varepsilon} >0$ such that
\begin{align*}
&\quad\Vert u\Vert_{H^{2,1}(\Omega\times (\varepsilon,T-\varepsilon))}
+ \Vert m\Vert_{H^{2,1}(\Omega\times (\varepsilon,T-\varepsilon))}
\\& \le C_{\varepsilon}\Big(\Vert F\Vert_{L^2(Q)} + \Vert G\Vert_{L^2(Q)}+\Vert u\Vert_{H^1(\Gamma \times (0,T))}
+ \Vert m \Vert_{H^1(\Gamma \times (0,T))}
+ \Vert \partial_tg\Vert_{L^2((\partial\Omega\setminus \Gamma)\times (0,T))}
\\
& \quad + \Vert \partial_th\Vert_{L^2((\partial\Omega\setminus \Gamma)\times (0,T))}
+ \Vert g\Vert_{L^2(0,T;H^{\frac{1}{2}}(\partial\Omega))}
   + \Vert h\Vert_{L^2(0,T;H^{\frac{1}{2}}(\partial\Omega))}\Big).
\end{align*}
\end{thm}

In particular, we directly see
\begin{align*}
& \quad\Vert u(\cdot,t)\Vert_{L^2(\Omega)}
+ \Vert m(\cdot,t)\Vert_{L^2(\Omega)}
\\&\le C_{\varepsilon}\Big(\Vert F\Vert_{L^2(Q)} + \Vert G\Vert_{L^2(Q)}
+ \Vert u\Vert_{H^1(\Gamma \times (0,T))}
+ \Vert m \Vert_{H^1(\Gamma \times (0,T))}
+ \Vert \partial_tg\Vert_{L^2((\partial\Omega\setminus \Gamma)\times (0,T))}
\\&\quad+ \Vert \partial_th\Vert_{L^2((\partial\Omega\setminus \Gamma)\times (0,T))}
+ \Vert g\Vert_{L^2(0,T;H^{\frac{1}{2}}(\partial\Omega))}
 + \Vert h\Vert_{L^2(0,T;H^{\frac{1}{2}}(\partial\Omega))}\Big)
\end{align*}
for $\varepsilon \le t \le T-\varepsilon$.

We emphasize that Theorem \ref{t1} asserts an unconditional stability estimate
in the case of a non-homogeneous Robin boundary condition. The
unconditional stability means
that we do not need to impose any boundedness assumptions for $u$ and $m$.

This result is also shown using a Carleman estimate as follows.

Let $\Gamma \subset \partial\Omega$ be an arbitrarily given subboundary.
We choose a subboundary $\Gamma_1 \subset \partial\Omega$ such that
$\overline{\partial\Omega \setminus \Gamma}\subset \Gamma_1$.  Then it is known
(e.g., Fursikov and Imanuvilov \cite{FI}, Imanuvilov \cite{Ima}) that
there exists a function $\eta \in C^1(\overline{\Omega})$ such that
$$
\eta(x)>0\quad\mbox{in}\quad \Omega, \quad \eta\vert_{\Gamma_1} = 0, \quad
\nabla \eta \ne 0\quad \mbox{on}\quad \overline{\Omega}.
$$
Let a function $\mu = \mu(t)$ satisfy 
\begin{equation}\label{(2.1)}
\begin{cases}
\mu \in C^{\infty}[0,T], \quad \mu(t) = t^2 \quad \mbox{for $0\le
t \le \frac{T}{4}$}, \\
\mbox{$\mu(t)$ is monotone increasing in $\left[0,\, \frac{T}{2}\right]$},\\
\mu(t) = \mu(T-t) \quad \mbox{for $0\le t\le T$}.
\end{cases}
\end{equation}
For an arbitrarily chosen sufficiently large constant $\lambda > 0$, we set
$$
\varphi(x,t) = \frac{e^{\lambda\eta(x)}}{\mu(t)}, \quad
\alpha(x,t) = \frac{e^{\lambda\eta(x)}
- e^{2\lambda ||\eta||_{C(\overline{\Omega})}} }{\mu(t)}, \quad (x,t)\in
Q:= \Omega \times (0,T).
$$
We recall the assumptions \eqref{1.7} and \eqref{lopukh1}, and
further set
\begin{equation}\label{(2.2)}
M:= \sum_{i,j=1}^n ||a_{ij}||_{C^1(\bar{Q})}
+ \sum_{j=1}^n ||a_j||_{L^{\infty}(Q)}
+ ||c_0||_{L^\infty(Q)}, \quad
M_0:= \sum_{i,j=1}^n ||a_{ij}||_{C^1(\bar{Q})}.
\end{equation}

We consider a
boundary value problem
\begin{equation}\label{(2.3)}
\partial_tu + A(t)u = F \quad \mbox{in $Q$} \quad \mbox{or}
\quad \partial_tu - B(t)u = F \quad \mbox{in $Q$},
\end{equation}
and
\begin{equation}\label{(2.4)}
\partial_{\nu_A}u - p(x,t)u = g \quad \mbox{on $\partial\Omega\times (0,T)$}.
\end{equation}
\\

Now we state the key Carleman estimate for a parabolic equation.
\begin{lem} \label{1}
Assume that \eqref{lopukh1}, \eqref{lopukh} and
$p \in C^1(\partial\Omega\times [0,T])$.
We choose a sufficiently large constant $\lambda>0$.
Let $u\in H^{2,1}(Q)$ and $g\in L^2(0,T,H^\frac 12 (\partial\Omega))$,
$\partial_tg\in L^2((\partial\Omega\setminus \Gamma) \times (0,T)), F\in L^2(Q)$.
Then there exist constants $s_0 > 0$ and $C>0$, independent of $u$,
such that
\begin{equation}\label{(2.5)}\begin{split}
&\int_Q \left(
\frac{1}{{s}\varphi} \left(|\partial_t u|^2
+ \sum_{i,j=1}^n| \partial_i\partial_ju|^2\right)
+ s\varphi |\nabla u|^2 + s^3\varphi^3 | u|^2
\right) e^{2s\alpha}  \,dx\,\,dt
\\
\le&  C\biggl(
\int_Q | F|^2 e^{2s\varphi} \,dx\,dt
+ \int_{(\partial\Omega\setminus \Gamma)\times (0,T)}
\left(\frac{| \partial_t g|^2}{s^2\varphi^2}
+ \frac{1}{\root\of{s\varphi}}| g|^2\right) e^{2{s}\alpha} dS\,dt\\
&+ \Vert ge^{s\alpha}\Vert^2_{L^2(0,T;H^\frac 12(\partial\Omega))}
\biggr)
+  C\int_{\Gamma \times (0,T)} \left(s\varphi| \nabla u|^2
+ s^3\varphi^3| u|^2 + \frac{| \partial_tu|^2}{s\varphi}\right)
e^{2s\varphi} \,dS\,dt \end{split}
\end{equation}
for each $s \ge s_0$ and $u\in H^{2,1}(Q)$
satisfying \eqref{(2.3)} and \eqref{(2.4)}.
Here the constant $C>0$ depends continuously on
$\Vert p\Vert_{C^1(\partial\Omega\times [0,T])}$,
$M$, $\lambda$.
\end{lem}

Observe that when we set $u(x,t) = w(x,T-t)$ for $(x,t) \in Q$, since $\alpha(x,t) = \alpha(x,T-t)$ and
$\varphi(x,t) = \varphi(x,T-t)$ for $(x,t) \in Q$, the change of variables
$t \mapsto T-t$ transfers the Carleman estimate for $\partial_t + A(t)$ to
$\partial_t - A(t)$.
Thus it is sufficient to prove Lemma \ref{1} for the parabolic operator
$\partial_t - \sum_{i,j=1}^n a_{ij}\partial_i\partial_j$, which is forward in time.

Indeed, we will prove a sharper estimate than Lemma \ref{1}.
In order to formulate our estimate,
let us introduce a quadratic form:\smallskip
$$
\mbox{\bf a}(V,W) : =\sum_{i,j=1}^n a_{ij}(x,t)v_iw_j \quad \mbox{for
$V:= (v_1,..., v_d)$ and $W:= (w_1, ..., w_d)$}
$$ and  operators
\begin{equation}\label{(5.9')}
\begin{cases}
L^2(x,t,D,s) w  = - \sum\limits_{i,j=1}^n a_{ij} \partial_i\partial_jw - s^2\lambda^2\varphi^2
\mbox{\bf a}(\nabla\psi, \nabla\psi) w -  s(\partial_t\alpha) w,  \\
L_1(x,t,D,s) w  = \partial_t w + 2s\lambda \varphi
\sum\limits_{i,j=1}^n a_{ij} (\partial_i\psi)\partial_jw + 2s\lambda^2 \varphi
\mbox{\bf a}(\nabla\psi, \nabla\psi) w.
\end{cases}
\end{equation}
We have
\begin{lem}\label{6}
Let $F\in L^2(Q)$, $g\in L^2(0,T;H^\frac 12(\partial\Omega))$,
$\partial_t g\in L^2((\partial\Omega\setminus\Gamma)\times (0,T))$ and
$p\in C^1(\partial\Omega\times [0,T])$.
There exists a constant $\lambda_0 > 0$ such that for an arbitrary
$\lambda \ge \la_0$, we can choose a constant $s_0(\lambda)$
satisfying: there exists a constant $C>0$ such that
\begin{align*}
& \int_Q \left( \frac{1}{s\varphi} \left(|\partial_t u|^2
+ \sum_{i,j=1}^n | \partial_i\partial_ju|^2\right)
+ s\lambda^2 \varphi |\nabla u|^2 + s^3 \lambda^4\varphi^3| u|^2\right)
e^{2s\alpha}  \,dx\,\,dt \\
&+  \sum_{k=1}^2\Vert L_k(x,t,D,s) (ue^{{s}\alpha})\Vert^2_{L^2(Q)} \le C\biggl(
\int_Q |F|^2 e^{2s\alpha}  \,dx\,dt  \\
&+ \int_{(\partial\Omega\setminus \Gamma)\times (0,T)}
\left(\frac{| \partial_t g|^2}{s^2\lambda^2\varphi^2}
+ \frac{1}{\root\of{s\varphi}} | g|^2\right) e^{2s\alpha} \,dS\,dt
+ \Vert ge^{s\alpha}\Vert^2_{L^2(0,T;H^\frac 12(\partial\Omega))} \\
&+ \int_{\Gamma \times (0,T)}
\left( s\lambda \varphi | \nabla u|^2
+ s^3\lambda^3\varphi^3 | u|^2
+ \frac{| \partial_t u|^2}{s\varphi} \right) e^{2s\alpha} \,dS \,dt
\biggr)
\end{align*}
for all $s> s_0(\lambda)$ and all $u\in H^{2,1}(Q)$ satisfying
$$
\begin{cases}
\partial_tu(x,t) - \sum\limits_{i,j=1}^n a_{ij}(x,t)\partial_i\partial_ju(x,t) = F(x,t), &\text{in }Q \\
\partial_{\nu_A} u(x,t) - p(x,t)u(x,t) = g(x,t), &\text{in }\partial\Omega\times
(0,T).
\end{cases}
$$
\end{lem}

Fixing $\lambda>0$ sufficiently large, we obtain Lemma \ref{1}
from Lemma \ref{6}. We recall that the bounds $M, M_0$ are defined by
\eqref{(2.2)}.

\begin{proof}[Proof of Lemma \ref{6}]

{\bf First Step}

We recall that $\Gamma_1$ is a relatively open subboundary of $\partial\Omega$ and
$\overline{\partial\Omega \setminus \Gamma} \subset \Gamma_1$.
Let $\mathcal U$ be a subdomain of $\Omega$ such that
$\mathcal{U}\cap \partial\Omega \subset \Gamma_1$.
Without loss of generality, we can assume that
$$
\mbox{supp} \,u \subset \mathcal U\times [0,T].
$$
Indeed, let $\mathcal U_1$ be an open set such that
$\Omega\subset \mathcal U\cup\mathcal U_1$ and $\overline{\mathcal U_1}\cap
\overline\Gamma =\emptyset$.  Let $e_1, e_2\in C^\infty_0(\mathbb{R}^n)$
be a partition of unity subject to the covering $\mathcal U_1,\mathcal U_2$.
Similarly to \cite{Hor}, it suffices to prove the Carleman estimate
Lemma \ref{6} for the functions $ue_1$ and $ue_2$.
The proof for the function $ue_2$ is simpler, since it does not require
consideration of the function $\widetilde w$, which is introduced below in
\eqref{im}, and follows directly from the inequality \eqref{(5.24)}
derived below.

We consider the operators
$$
 \widehat{L}(x,t,D)u = \partial_t u - \sum_{i,j=1}^n a_{ij}(x,t)\partial_i\partial_ju
$$
and
\begin{equation}\label{(5.1)}
\widehat{L}(x,t,D) u = \widetilde F \quad \mbox{in $Q$},
\end{equation}
where
\begin{equation}\label{(5.2)}
\widetilde{F}(x,t) = F(x,t) +  \sum_{i,j=1}^n (\partial_ia_{ij})(x,t) \partial_ju.
                      \end{equation}

We set $\widetilde \psi(x) = -\psi(x)$ in a neighborhood $\mathcal{U}$.
Using the function $\widetilde \psi$, we introduce the functions
$\widetilde\alpha$ and $\widetilde\varphi:$
\begin{equation}\label{(5.3)}
  \widetilde\alpha(x,t) = \frac{e^{\lambda\widetilde \psi(x)}
- e^{2\lambda ||\psi||_{C(\overline{\Omega})}}}{\mu(t)},
\quad \widetilde \varphi(x,t) = \frac{e^{\lambda\widetilde \psi(x)}}{\mu(t)}.
\end{equation}

We denote
\begin{equation}\label{im} w(x,t)=e^{s\alpha}u(x,t)\quad \mbox{and}
\quad \widetilde w(x,t) = e^{s\widetilde \alpha}u(x,t).\end{equation}

Then, we have
\begin{equation}\label{(5.4)}
w(\cdot,0)=w(\cdot,T) = \widetilde w(\cdot,0) = \widetilde w(\cdot,T)= 0
\quad \mbox{in } \Omega.
\end{equation}
We define the operators $P(x,t,D,s)$ and $\widetilde P(x,t,D,s)$ by
\begin{equation}\label{(5.5)}
P(x,t,D,s)w = e^{s\alpha} \widehat{L}(x,t,D) e^{-s\alpha} w,\quad
\widetilde{P}(x,t,D,s) w = e^{s\widetilde{\alpha}} \widehat{L}(x,t,D)
e^{-s\widetilde{\alpha}} w.             \end{equation}

It follows from \eqref{(5.1)} and \eqref{(5.2)} that
\begin{equation}\label{(5.6)}
P(x,t,D,s)w  = e^{s\alpha} \widehat{L}(x,t,D)(e^{-s\alpha} w)
= e^{s\alpha} \widetilde{F} \quad \mbox{in }  Q
\end{equation}
and
\begin{equation}\label{(5.7)}
\widetilde{P}(x,t,D,s)\widetilde{w}
= e^{s\widetilde{\alpha}} \widehat{L}(x,t,D)
e^{-s\widetilde{\alpha}} \widetilde{w}
= e^{s\widetilde{\alpha}} \widetilde{F} \quad \mbox{in }  Q.
\end{equation}

The operator $P$ can be written explicitly as follows
\begin{equation}\label{(5.8)}\begin{split}
P(x,t,D,s)w &= \partial_t w - \sum_{i,j=1}^n a_{ij}\partial_i\partial_jw
+ 2\lambda\varphi \sum_{i,j=1}^n a_{ij} (\partial_i\psi)\partial_jw
  + s\lambda^2 \varphi \mbox{\bf a}(\nabla\psi, \nabla\psi) w
 \\&\quad- s^2\lambda^2\varphi^2  \mbox{\bf a}( \nabla\psi, \nabla\psi) w
 + s\lambda \varphi w \sum_{i,j=1}^n a_{ij} \partial_i\partial_j\psi - s(\partial_t\alpha) w.
\end{split}     
\end{equation}
We introduce the operators $\widetilde{L}_k(x,t,D,s)$, $k=1,2$
as follows
\begin{equation}\label{(5.9)}
\begin{cases}
\widetilde L^2(x,t,D,s) w  = - \sum\limits_{i,j=1}^n a_{ij} \partial_i\partial_jw
- s^2\lambda^2\widetilde \varphi^2 \mbox{\bf a}(\nabla\widetilde\psi, \nabla\widetilde\psi) w
- s(\partial_t\widetilde \alpha) w, \\
\widetilde L_1(x,t,D,s) w  = \partial_t w + 2{{s}}\lambda\widetilde\varphi
\sum\limits_{i,j=1}^n a_{ij} (\partial_i\widetilde\psi) \partial_jw + 2\lambda^2 \widetilde \varphi
\mbox{\bf a}(\nabla\widetilde \psi, \nabla\widetilde \psi) w.
\end{cases}
\end{equation}
Then,
\begin{equation}\label{(5.10)}
L_1(x,t,D,s) w + L^2(x,t,D,s)w= H(x,t,\lambda,{s})\quad\mbox{in}\,\,Q,
\end{equation}
where
\begin{equation}\label{(5.11)}
H(x,t,\lambda,s) := \widetilde ge^{s\alpha}+2s\lambda^2 \varphi
\mbox{\bf a}( \nabla\psi, \nabla\psi) w - s\lambda \varphi w \sum_{i,j=1}^n a_{ij}
(\partial_i\partial_j\psi) w
\end{equation}
and
\begin{equation}\label{(5.12)}
\widetilde L_1(x,t,D,s) \widetilde w
+ \widetilde L^2(x,t,D,s)\widetilde w
= \widetilde{H}(x,t,\lambda,s) \quad\mbox{in}\,\,Q,
\end{equation}
with
\begin{equation}\label{(5.13)}
\widetilde{H}(x,t,\lambda,s) = \widetilde ge^{s\alpha}
+ 2s\lambda^2 \widetilde\varphi
\mbox{\bf a}(\nabla\widetilde\psi, \nabla\widetilde \psi) w
- s\lambda \widetilde\varphi\widetilde w
\sum_{i,j=1}^n a_{ij}(\partial_i\partial_j\widetilde  \psi)\widetilde w.
\end{equation}

Henceforth we set
$$
\Sigma:= \partial\Omega \times (0,T), \quad d\Sigma := dS \,dt.
$$

{\bf Second Step}

We will verify the following equality:
\begin{equation}\label{(5.14)}\begin{split}
(L^2 w, L_1 w)_{L^2 (Q)}
=& \int_Q L_1(x,t,D,s)w \sum_{i,j=1}^n (\partial_ja_{ij})(\partial_iw) \,dx\,dt \\
&+ \int_Q \biggl\{
 -\sum_{i,j=1}^n \frac 12 (\partial_t a_{ij})(\partial_iw)(\partial_jw)
+ \partial_t\left( \frac{s^2\lambda^2 \varphi^2}{2}
\mbox{\bf a}( \nabla\psi, \nabla\psi)\right) w^2\\
 & + \frac{s\partial_t^2\alpha}{2} w^2
+  s^3 \lambda^4\varphi^3 \mbox{\bf a}( \nabla\psi, \nabla\psi)^2 w^2
- 2s^2\lambda^2\varphi (\partial_t\alpha)+ \mbox{\bf a}( \nabla\psi, \nabla\psi) w^2\\
&+  s\lambda^2 \varphi \mbox{\bf a}( \nabla\psi, \nabla\psi)\mbox{\bf a}(\nabla w,\nabla w)
+  2s\lambda^2 w \sum_{i,j=1}^n a_{ij} (\partial_jw)
\partial_i(\varphi \mbox{\bf a}( \nabla\psi, \nabla\psi)) \\
&+ 2s\lambda^2 \varphi \mbox{\bf a}( \nabla\psi, \nabla w)^2
+ 2s\lambda \varphi \sum_{i,j=1}^n  a_{ij}\partial_iw
\left(\sum_{k,\ell=1}^n \partial_j(a_{k\ell}\partial_k\psi) \partial_\ell w \right)\\
&- s\lambda\varphi \sum_{k,\ell=1}^n a_{k\ell} (\partial_k\psi)
\left( \sum_{i,j=1}^n (\partial_\ell a_{ij})(\partial_i w) \partial_jw\right)
\\
&-  s\lambda\varphi \sum_{k,\ell=1}^n a_{k\ell} (\partial_k\psi)
 \left( \sum_{i,j=1}^n (\partial_\ell a_{ij})(\partial_iw)\partial_jw\right) \\
 &- \mbox{\bf a}( \nabla w, \nabla w) s\lambda\varphi \sum_{k,\ell=1}^n
\partial_\ell(a_{k\ell}\partial_k\psi) \biggr\} \,dx\,dt   \\
&+  \int_\Sigma (2s\lambda\varphi \mbox{\bf a}(\nu,\nabla w)\mbox{\bf a}(\nabla\psi,\nabla w)
 - s\lambda\varphi \mbox{\bf a}( \nabla w, \nabla w)\mbox{\bf a}( \nu,\nabla \psi))
 \,d\Sigma                                             \\
&- \int_\Sigma (2s^3\lambda^3\varphi^3 \mbox{\bf a}( \nabla\psi, \nabla\psi)
\mbox{\bf a}( \nabla\psi, \nu)
  + 2s^2 \lambda\varphi (\partial_t\alpha)
\mbox{\bf a}( \nabla\psi, \nu))w^2 \,d\Sigma                           \\
&- \int_\Sigma \mbox{\bf a}( \nu,\nabla w)(\partial_t w
+ 2s\lambda^2\varphi \mbox{\bf a}( \nabla\psi, \nabla\psi) w)\,d\Sigma.
\end{split}
\end{equation}

\begin{proof}[Proof of \eqref{(5.14)}]
By \eqref{(5.9)}, we have the following equality:
\begin{equation}\label{(5.15)}\begin{split}
&(L^2 w, L_1 w)_{L^2 (Q)}
\\=
&- \int_Q \left( \sum_{i,j=1}^n a_{ij} \partial_i\partial_jw - s^2\lambda^2\varphi^2
\mbox{\bf a}( \nabla\psi, \nabla\psi) w - s(\partial_t\alpha) w\right)\\
&\times \left( \partial_tw + 2s\lambda^2
\varphi \mbox{\bf a}( \nabla\psi, \nabla\psi) w\right) \,dx\,dt\\
&-  \int_Q(2s^3\lambda^3\varphi^3  \mbox{\bf a}( \nabla\psi, \nabla\psi)w
+ 2s^2 \lambda\varphi (\partial_t\alpha) w) \mbox{\bf a}( \nabla\psi, \nabla w) \,dx\,dt \\
 &-  \int_Q \left( \sum_{i,j=1}^n a_{ij}\partial_i\partial_j w\right)
2s\lambda\varphi \mbox{\bf a}( \nabla\psi, \nabla w) \,dx\,dt\\
=: & A_1 + A_2 + A_3.
\end{split}
\end{equation}

Now we calculate $A_1, A_2, A_3$.

{\bf Calculations of $A_1$}

Integrating by parts the first term on the right-hand-side, we obtain
\begin{equation}\label{(5.16)}\begin{split}
A_1  =&  \int_Q \left( -\sum_{i,j=1}^n a_{ij} \partial_i\partial_jw -
s^2\lambda^2\varphi^2 \mbox{\bf a}( \nabla\psi, \nabla\psi) w
- s(\partial_t\alpha) w\right)\\
&\times  (\partial_tw + 2s\lambda^2\varphi a(\nabla\psi,
\nabla\psi) w) \,dx\,dt \\
= &\int_Q \biggl\{ \partial_t w \sum_{i,j=1}^n (\partial_ja_{ij})\partial_iw
+ \sum_{i,j=1}^n a_{ij} (\partial_i w)(\partial_j\partial_t w) \\
&-  \frac{s^2\lambda^2\varphi^2}{2}\mbox{\bf a}( \nabla\psi, \nabla\psi) \partial_t (w^2)
- \frac{s\partial_t\alpha}{2}  \partial_t(w^2)
 - 2s^3\lambda^4 \varphi^3 \mbox{\bf a}( \nabla\psi, \nabla\psi)^2 w^2  \\
&- 2s^2\lambda^2 \varphi (\partial_t\alpha) \mbox{\bf a}( \nabla\psi, \nabla\psi) w^2
 + 2s\lambda^2\varphi \mbox{\bf a}( \nabla\psi, \nabla\psi) w
\sum_{i,j=1}^n (\partial_ja_{ij}) \partial_iw  \\
&+  2s\lambda^2 \varphi \mbox{\bf a}( \nabla\psi, \nabla\psi)
\mbox{\bf a}( \nabla w, \nabla w)
 + 2s\lambda^2 w \sum_{i,j=1}^n a_{ij} (\partial_jw)\partial_i
(\varphi \mbox{\bf a}( \nabla\psi, \nabla\psi)) \biggr\} \,dx \,dt \\
&- \int_\Sigma \mbox{\bf a}(\nu,\nabla w)(\partial_t w + 2s\lambda^2
\varphi \mbox{\bf a}( \nabla\psi, \nabla\psi) w)\,d\Sigma\\\
 =& \int_Q \biggl\{ \partial_t w \sum_{i,j=1}^n (\partial_ja_{ij})\partial_iw \\
&+ \frac 12\partial_t( \sum_{i,j=1}^n a_{ij} (\partial_iw)\partial_jw)
- \sum_{i,j=1}^n \frac{1}{2}(\partial_t a_{ij})(\partial_iw)\partial_jw
- \frac{s^2\lambda^2\varphi^2}{2}
\mbox{\bf a}( \nabla\psi, \nabla\psi) \partial_t (| w|^2)\\
&- \frac{s\partial_t\alpha}{2}  \partial_t | w|^2
 - 2s^3 \lambda^4 \varphi^3 \mbox{\bf a}( \nabla\psi, \nabla\psi)^2 w^2  \\
&-  2s^2\lambda^2 \varphi (\partial_t\alpha) \mbox{\bf a}( \nabla\psi, \nabla\psi) w^2
 + 2s\lambda^2\varphi \mbox{\bf a}( \nabla\psi, \nabla\psi) w
\sum_{i,j=1}^n (\partial_ja_{ij})\partial_jw  \\
&+ 2s\lambda^2 \varphi \mbox{\bf a}( \nabla\psi, \nabla\psi)\mbox{\bf a}(\nabla w, \nabla w)
+ 2s\lambda^2 w \sum_{i,j=1}^n a_{ij} (\partial_jw)
\partial_i(\varphi \mbox{\bf a}( \nabla\psi, \nabla\psi))\biggr\} \,dx \,dt\\
&-  \int_\Sigma \mbox{\bf a}(\nu,\nabla w)(\partial_t w + 2s\lambda^2\varphi
\mbox{\bf a}( \nabla\psi, \nabla\psi) w)\,d\Sigma.
\end{split}
\end{equation}
Integrating this equality by parts with respect to $t$, we obtain \smallskip
\begin{equation}\label{(5.17)}\begin{split}
 A_1 =&
\int_Q \biggl\{ \partial_tw\sum_{i,j=1}^n \partial_ja_{ij}
- \frac{s^2\lambda^2\varphi^2}{2}\mbox{\bf a}( \nabla\psi, \nabla\psi)
\partial_t (| w|^2)\\
  &+ \frac{s\partial^2_t\alpha}{2}  w^2
 - 2s^3\lambda^4 \varphi^3 \mbox{\bf a}( \nabla\psi, \nabla\psi)^2 w^2\\
&-  2s^2\lambda^2 \varphi (\partial_t\alpha) \mbox{\bf a}( \nabla\psi, \nabla\psi) w^2
 + 2s\lambda^2\varphi \mbox{\bf a}( \nabla\psi, \nabla\psi) w
\sum_{i,j=1}^n (\partial_ja_{ij}) \partial_iw                      \\
&+  2s\lambda^2 \varphi \mbox{\bf a}( \nabla\psi, \nabla\psi)
\mbox{\bf a}( \nabla w, \nabla w)\\
& + 2s\lambda^2 w \sum_{i,j=1}^n a_{ij} (\partial_jw)\partial_i
(\varphi \mbox{\bf a}( \nabla\psi, \nabla\psi))\biggr\} \,dx \,dt
\\
&-  \int_\Sigma \mbox{\bf a}(\nu,\nabla w)(\partial_t w
+ 2s\lambda^2\varphi \mbox{\bf a}( \nabla\psi, \nabla\psi) w)\,d\Sigma.
\end{split}               \end{equation}

{\bf Calculation of $A_2$}
\begin{equation}\label{(5.18)}\begin{split}
A_2 = &- \int_Q (2s^3\lambda^3\varphi^3 w \mbox{\bf a}( \nabla\psi, \nabla\psi)
\mbox{\bf a}( \nabla\psi, \nabla w)
+ 2s^2 \lambda\varphi (\partial_t\alpha) w
\mbox{\bf a}( \nabla\psi, \nabla w)) \,dx\,dt
\\
= &  -\int_Q (s^3\lambda^3\varphi^3 \mbox{\bf a}(\nabla\psi,\nabla\psi)
\mbox{\bf a}( \nabla\psi, \nabla (w^2) )
+ s^2\lambda\varphi (\partial_t \alpha)\mbox{\bf a}( \nabla\psi, \nabla (w^2))) \,dx\,dt
\\
=&\int_Q \biggl\{ 3s^3\lambda^4 \varphi^3 \mbox{\bf a}( \nabla\psi, \nabla\psi)^2 w^2
 + s^3\lambda^3\varphi^3 w^2 \sum_{i,j=1}^n \partial_i
(a_{ij} (\partial_j\psi) \mbox{\bf a}( \nabla\psi, \nabla\psi))
\\
 &+ \sum_{i,j=1}^n \partial_j \left(\frac{s^2 \lambda^2\varphi (\partial_t\alpha)}{2}
a_{ij}\partial_i\psi \right) w^2 \biggr\} \,dx\,dt
\\
&- \int_\Sigma(2s^3\lambda^3 \varphi^3 \mbox{\bf a}( \nabla\psi, \nabla\psi)
\mbox{\bf a}( \nabla\psi, \nu)
  + 2s^2 \lambda\varphi (\partial_t\alpha) \mbox{\bf a}( \nabla\psi, \nu))w^2
\,d\Sigma.
\end{split}
\end{equation}

{\bf Calculation of $A_3$}
\begin{equation}\begin{split}
 A_3  =& \int_Q - \left(\sum_{i,j=1}^n a_{ij}\partial_i\partial_j w\right)
\left( 2s\lambda\varphi \sum_{k,\ell=1}^n a_{k\ell}(\partial_k\psi)
(\partial_\ell w)\right) \,dx\,dt\\
 =& \int_Q \biggl\{
\sum_{i,j=1}^n (\partial_ja_{ij})(\partial_iw) 2s\lambda\varphi \sum_{k,\ell=1}^n
a_{k\ell} (\partial_k\psi)(\partial_\ell w)
 + 2s\lambda^2 \varphi \mbox{\bf a}( \nabla\psi, \nabla w)^2  \\
 &+  2s\lambda\varphi (\sum_{i,j=1}^n a_{ij} (\partial_iw))
\left( \sum_{k,\ell=1}^n \partial_j(a_{k\ell}\partial_k\psi )\partial_\ell w\right) \\
 &+ 2s\lambda\varphi \sum_{i,j=1}^n a_{ij} \partial_{i} w
\sum_{k,\ell=1}^n a_{k\ell}(\partial_k\psi)(\partial_j\partial_\ell w) \biggr\} \,dx\,dt
\\&+  \int_\Sigma 2s\lambda\varphi
\mbox{\bf a}(\nu,\nabla w)\mbox{\bf a}(\nabla\psi,\nabla w) \,d\Sigma     \\
 = & \int_Q \biggl\{
 \sum_{i,j=1}^n (\partial_j a_{ij})(\partial_iw) 2s\lambda\varphi
\sum_{k,\ell=1}^n a_{k\ell} (\partial_k\psi)\partial_\ell w
+ 2s\lambda^2 \varphi \mbox{\bf a}( \nabla\psi, \nabla w)^2  \\
 &+  2s\lambda\varphi \sum_{i,j=1}^n a_{ij} (\partial_i w)
\left(  \sum_{k, \ell =1}^n \partial_j (a_{k\ell} (\partial_k\psi)
\partial_\ell w \right)\\
&- s\lambda\varphi
\sum_{k, \ell = 1}^n a_{k\ell} (\partial_k\psi)
\left( \sum_{i,j=1}^n (\partial_\ell a_{ij})(\partial_i w)(\partial_j w)\right)
\\
&+ s\lambda \varphi \sum_{k,\ell = 1}^n a_{k\ell}(\partial_k\psi)
\partial_{\ell} ( \sum_{i,j=1}^n a_{ij}(\partial_iw)\partial_jw )\biggr\} \,dx\,dt
\\&+  \int_\Sigma  2s\lambda\varphi
\mbox{\bf a}(\nu,\nabla w)\mbox{\bf a}(\nabla\psi,\nabla w) \,d\Sigma.  \nonumber
\end{split}
\end{equation}
Integrating by parts once again, we obtain
\begin{equation}\label{(5.19)}\begin{split}
A_3 =& \int_Q \biggl\{ \sum_{i,j=1}^n (\partial_j a_{ij})(\partial_iw)
2s\lambda\varphi \sum_{k,\ell=1}^n a_{k\ell} (\partial_k\psi)
\partial_{\ell}w + 2s\lambda^2 \varphi
 \mbox{\bf a}( \nabla\psi, \nabla w)^2\\
&+ 2s\lambda \varphi \sum_{i,j=1}^n a_{ij} (\partial_iw)
\left( \sum_{k,\ell=1}^n \partial_j (a_{k\ell}\partial_k\psi)
\partial_\ell w\right)\\
&- s\lambda\varphi \sum_{k,\ell=1}^n a_{k\ell} (\partial_k\psi)
\left( \sum_{i,j=1}^n (\partial_\ell a_{ij})(\partial_i w)\partial_jw\right)
                          \\
&-  s\lambda^2 \varphi \mbox{\bf a}( \nabla\psi, \nabla\psi)\mbox{\bf a}( \nabla w, \nabla w)
 - s\lambda\varphi \sum_{k,\ell=1}^n a_{k\ell}(\partial_k\psi)
 \left( \sum_{i,j=1}^n (\partial_\ell a_{ij})(\partial_i w) \partial_jw \right)
\\
 &-  \mbox{\bf a}( \nabla w, \nabla w) s\lambda\varphi \sum_{k,\ell=1}^n
\partial_\ell (a_{k\ell}\partial_k\psi)\biggr\} \,dx\,dt             \\
&+  \int_\Sigma(2s\lambda\varphi \mbox{\bf a}(\nu,\nabla w)\mbox{\bf a}(\nabla\psi,\nabla w)
- s\lambda\varphi \mbox{\bf a}( \nabla w, \nabla w)
\mbox{\bf a}( \nu, \nabla\psi)) \,d\Sigma.
\end{split}
\end{equation}

Taking the sum of \eqref{(5.17)} - \eqref{(5.19)}, we obtain
\begin{equation}\label{(5.20)}\begin{split}
&(L^2 w, L_1 w)_{L^2 (Q)}
\\=& \int_Q \biggl\{ \partial_t w \sum_{i,j=1}^n (\partial_ja_{ij})\partial_jw
- \sum_{j=1}^n \frac12(\partial_t a_{ij})(\partial_iw)\partial_jw    \\
&+ \partial_t\left( \frac{s^2\lambda^2\varphi^2}{2}
\mbox{\bf a}( \nabla\psi, \nabla\psi)\right) w^2
  +  \frac{s\partial^2_t\alpha}{2}  w^2
- 2s^3\lambda^4 \varphi^3 \mbox{\bf a}( \nabla\psi, \nabla\psi)^2 w^2   \\
&-  2s^2\lambda^2\varphi (\partial_t\alpha) \mbox{\bf a}( \nabla\psi, \nabla\psi) w^2
 + 2s\lambda^2 \varphi \mbox{\bf a}( \nabla\psi, \nabla\psi) w
\sum_{i,j=1}^n (\partial_ja_{ij})(\partial_i w)      \\
&+  2s\lambda^2 \varphi \mbox{\bf a}( \nabla\psi, \nabla\psi)\mbox{\bf a}(\nabla w, \nabla w)
+ 2s\lambda^2 w \sum_{i,j=1}^n a_{ij} (\partial_j w)
\partial_i( \varphi \mbox{\bf a}( \nabla\psi, \nabla\psi))   \\
&+  \sum_{i,j=1}^n (\partial_j a_{ij})(\partial_iw) 2s\lambda\varphi
\sum_{k,\ell=1}^n a_{k\ell}(\partial_k \psi) \partial_{\ell} w
 + 2s\lambda^2 \varphi \mbox{\bf a}( \nabla\psi, \nabla (w^2))       \\
&+ 2s\lambda \varphi \sum_{i,j=1}^n  a_{ij} (\partial_i w)
\left( \sum_{k,\ell=1}^n \partial_j (a_{k\ell} \partial_k\psi)\partial_\ell w
\right)\\
&- s\lambda\varphi \sum_{k,\ell=1}^n
a_{k\ell} (\partial_k\psi) \left(\sum_{i,j=1}^n (\partial_\ell a_{ij})(\partial_i w)
\partial_j w \right)
 -  s\lambda^2 \varphi \mbox{\bf a}( \nabla\psi, \nabla\psi)
\mbox{\bf a}( \nabla w, \nabla w)\\
&- s\lambda\varphi \sum_{k,\ell=1}^n a_{k\ell} (\partial_k\psi)
\left( \sum_{i,j=1}^n (\partial_\ell a_{ij})(\partial_i w)
\partial_j w \right)                  \\
&- a( \nabla w, \nabla w) s\lambda\varphi \sum_{k,\ell=1}^n
\partial_\ell(a_{k\ell}\partial_k\psi) \biggr\} \,dx\,dt        \\
&+  \int_\Sigma (2s\lambda\varphi \mbox{\bf a}(\nu,\nabla w)
\mbox{\bf a}(\nabla\psi,\nabla w)
- s\lambda\varphi \mbox{\bf a}( \nabla w, \nabla w)\mbox{\bf a}( \nu, \nabla\psi)) \,d\Sigma
                                                               \\
&- \int_\Sigma(2s^3\lambda^3 \varphi^3 \mbox{\bf a}( \nabla\psi, \nabla\psi)
\mbox{\bf a}( \nabla\psi, \nu)
+ 2s^2 \lambda\varphi (\partial_t\alpha) \mbox{\bf a}( \nabla\psi, \nu))w^2
\,d\Sigma                                                      \\
&- \int_\Sigma \mbox{\bf a}(\nu,\nabla w)(\partial_t w
+ 2s\lambda^2\varphi \mbox{\bf a}( \nabla\psi, \nabla\psi) w)\,d\Sigma.
\end{split}
\end{equation}
Finally we observe that
\begin{multline}
 \partial_t w \left(\sum_{i,j=1}^n (\partial_j a_{ij})\partial_i w\right)
+ 2s\lambda^2 \varphi \mbox{\bf a}( \nabla\psi, \nabla\psi) w
\sum_{i,j=1}^n (\partial_j a_{ij}) \partial_i w \nonumber           \\
+  2s\lambda^2 \varphi \mbox{\bf a}( \nabla\psi, \nabla\psi) w
\sum_{i,j=1}^n (\partial_j a_{ij}) \partial_i w
= L_1(x,t,D) w\, \sum_{i,j=1}^n (\partial_j a_{ij}) \partial_iw.
\end{multline}
Thus the proof of \eqref{(5.14)} is complete.
\end{proof}

Similarly to \eqref{(5.14)}, we can readily verify
\begin{equation}\label{(5.21)}\begin{split}
 &(\widetilde L^2\widetilde  w,\, \widetilde  L_1\widetilde  w)_{L^2 (Q)}
\\= &\int_Q \widetilde L_1(x,t,D,s)\widetilde w
\sum_{i,j=1}^n (\partial_j a_{ij}) \partial_i\widetilde w\, \,dx\,dt \\
&+ \int_Q \biggl\{ -\sum_{i,j=1}^n \frac 12 (\partial_t a_{ij})(\partial_i\widetilde w)
\partial_j\widetilde w
+ \partial_t\left( \frac{s^2\lambda^2\varphi^2}{2}
\mbox{\bf a}( \nabla\psi, \nabla\psi)\right) \widetilde w^2
  + \frac{1}{2} s(\partial_t^2\alpha) \widetilde  w^2\\
&+ s^3\lambda^4 \varphi^3 \mbox{\bf a}( \nabla\psi, \nabla\psi)^2\widetilde w^2
- 2s^2\lambda^2\varphi (\partial_t \alpha) \mbox{\bf a}( \nabla\psi, \nabla\psi)
\widetilde w^2\\
&+ s\lambda^2 \varphi \mbox{\bf a}( \nabla\psi, \nabla\psi)
\mbox{\bf a}( \nabla\widetilde w, \nabla\widetilde  w)
 +  2s\lambda^2\widetilde w \left(\sum_{i,j=1}^n a_{ij} (\partial_j\widetilde w)
\partial_i (\varphi \mbox{\bf a}( \nabla\psi, \nabla\psi))\right)\\
&+ 2s\lambda^2 \varphi \mbox{\bf a}( \nabla\psi, \nabla \widetilde w)^2
+  2s\lambda \varphi \sum_{i,j=1}^n a_{ij} (\partial_i\widetilde w)
\left(\sum_{k,\ell=1}^n \partial_j (a_{k\ell} (\partial_k\psi))\partial_\ell
\widetilde w \right)\\
&- s\lambda\varphi \sum_{k,\ell=1}^n
a_{k\ell} (\partial_k\psi)\left( \sum_{i,j=1}^n (\partial_\ell a_{ij})
(\partial_i\widetilde  w) \partial_j\widetilde w \right)
 - s\lambda\varphi \sum_{k,\ell=1}^n a_{k\ell} (\partial_k\psi)\times \\
&\times
\left( \sum_{i,j=1}^n (\partial_\ell a_{ij})(\partial_i\widetilde w)
\partial_j\widetilde w \right)
- s\lambda\varphi \mbox{\bf a}( \nabla\widetilde w, \nabla\widetilde w)
\sum_{k,\ell=1}^n \partial_\ell (a_{k\ell}\partial_k\psi) \biggr\} \,dx\,dt \\
 &- \int_\Sigma ( 2s\lambda\varphi
\mbox{\bf a}(\nu,\nabla \widetilde w)\mbox{\bf a}(\nabla\psi,\nabla \widetilde w)
- s\varphi\lambda\varphi  \mbox{\bf a}( \nabla\widetilde  w, \nabla \widetilde w)
\mbox{\bf a}( \nu,\nabla \psi)) \,d\Sigma            \\
&+  \int_\Sigma(2s^3\lambda^3 \varphi^3 \mbox{\bf a}( \nabla\psi, \nabla\psi)
\mbox{\bf a}( \nabla\psi, \nu) + 2s^2 \lambda\varphi(\partial_t\alpha)
\mbox{\bf a}( \nabla\psi, \nu))\widetilde w^2 \,d\Sigma
\\
&- \int_\Sigma \mbox{\bf a}(\nu,\nabla \widetilde w)
(\partial_t \widetilde w + 2s\lambda^2 \varphi \mbox{\bf a}(\nabla\psi, \nabla\psi)
\widetilde w)\,d\Sigma.\end{split}
\end{equation}

{\bf Third Step: completion of the proof of Lemma \ref{6}}

Taking the $L^2$ norms of both sides of the equations \eqref{(5.10)} and
\eqref{(5.12)},
we have
$$
\Vert L_1w\Vert^2_{L^2(Q)} + 2(L_1w,\, L^2w)_{L^2(Q)}
+ \Vert L^2w\Vert^2_{L^2(Q)} = \Vert H\Vert^2_{L^2(Q)}
$$
and
$$
\Vert \widetilde L_1\widetilde w\Vert^2_{L^2(Q)}
+ 2(\widetilde L_1\widetilde w,\, \widetilde L^2\widetilde w)_{L^2(Q)}
+ \Vert \widetilde L^2\widetilde w\Vert^2_{L^2(Q)}
= \Vert \widetilde{H}\Vert^2_{L^2(Q)}.
$$
We take the parameter $\lambda$ sufficiently large, so that
\begin{equation}\label{(5.22)}\begin{split}
 &(L_1 w, \, L^2 w)_{L^2 (Q)}
\\\ge& \int_Q L_1(x,t,D,s)w \sum_{i,j=1}^n (\partial_j a_{ij}) \partial_i w\, \,dx\,dt
                                                      \\
&+  \frac 14 \int_Q( s\lambda^2 \varphi
\mbox{\bf a}( \nabla\psi, \nabla\psi) \mbox{\bf a}( \nabla w, \nabla w)
+ s^3 \lambda^4\varphi^3 \mbox{\bf a}( \nabla\psi, \nabla\psi)^2 w^2) \,dx\,dt
                                                 \\
&+   \int_\Sigma ( 2s\lambda\varphi \mbox{\bf a}(\nu,\nabla w)
\mbox{\bf a}(\nabla\psi,\nabla w)
- s\lambda\varphi \mbox{\bf a}( \nabla w, \nabla w) \mbox{\bf a}( \nu,\nabla \psi)) \,d\Sigma
\\
&-  \int_\Sigma (2s^3\lambda^3 \varphi^3 \mbox{\bf a}( \nabla\psi, \nabla\psi)
\mbox{\bf a}( \nabla\psi, \nu)
+ 2s^2 \lambda\varphi (\partial_t\alpha) \mbox{\bf a}( \nabla\psi, \nu))w^2 \,d\Sigma
\\
&- \int_\Sigma \mbox{\bf a}(\nu,\nabla w)(\partial_t w
+ 2s\lambda^2\varphi \mbox{\bf a}( \nabla\psi, \nabla\psi) w)\,d\Sigma.
\end{split}
\end{equation}
Using \eqref{(5.14)}, we have
\begin{equation}\label{(5.23)}\begin{split}
 &\sum_{k=1}^2\Vert L_kw\Vert^2_{L^2(Q)}
+ 2\int_Q L_1(x,t,D,s)w \sum_{i,j=1}^n (\partial_j a_{ij})\partial_iw\, \,dx\,dt
                                                         \\
&+ \frac 12 \int_Q( s\lambda^2 \varphi
\mbox{\bf a}( \nabla\psi, \nabla\psi) \mbox{\bf a}( \nabla w, \nabla w)
+ s\lambda^4\varphi^3 \mbox{\bf a}( \nabla\psi, \nabla\psi)^2 w^2)\,dx\,dt   \\
&+  \int_\Sigma (2s\lambda\varphi \mbox{\bf a}(\nu,\nabla w)\mbox{\bf a}(\nabla\psi,\nabla w)
- s\lambda\varphi \mbox{\bf a}( \nabla w, \nabla w)
\mbox{\bf a}( \nu,\nabla \psi)) \,d\Sigma     \\
&-  \int_\Sigma(2s^3\lambda^3 \varphi^3 \mbox{\bf a}( \nabla\psi, \nabla\psi)
\mbox{\bf a}( \nabla\psi, \nu) + 2s^2 \lambda \varphi (\partial_t\alpha)
\mbox{\bf a}( \nabla\psi, \nu))w^2 \,d\Sigma                      \\
&- \int_\Sigma \mbox{\bf a}(\nu,\nabla w)(\partial_t w
+ 2s\lambda^2\varphi \mbox{\bf a}( \nabla\psi, \nabla\psi) w)\,d\Sigma
\le \Vert H \Vert^2_{L^2(Q)}.
\end{split}
\end{equation}
This inequality implies
\begin{equation}\label{(5.24)}\begin{split}
 &\frac 14 \sum_{k=1}^2\Vert L_kw\Vert^2_{L^2(Q)}\\
&+ \frac 14 \int_Q ( s\lambda^2 \varphi \mbox{\bf a}( \nabla\psi, \nabla\psi)
\mbox{\bf a}( \nabla w, \nabla w)
+ s^3 \lambda^4\varphi^3 \mbox{\bf a}( \nabla\psi, \nabla\psi)^2 w^2)\,dx\,dt      \\
&+  \int_\Sigma(2s\lambda\varphi \mbox{\bf a}(\nu,\nabla w)\mbox{\bf a}(\nabla\psi,\nabla w)
 - s\lambda\varphi \mbox{\bf a}( \nabla w, \nabla w)\mbox{\bf a}( \nu,\nabla \psi)) \,d\Sigma
\\
&-  \int_\Sigma(2s^3\lambda^3\varphi^3 \mbox{\bf a}( \nabla\psi, \nabla\psi)
\mbox{\bf a}( \nabla\psi,\nu)
+ 2s^2 \lambda\varphi (\partial_t\alpha) \mbox{\bf a}( \nabla\psi, \nu))w^2 \,d\Sigma
\\
&- \int_\Sigma \mbox{\bf a}(\nu,\nabla w)(\partial_t w
+ 2s\lambda^2 \varphi \mbox{\bf a}( \nabla\psi, \nabla\psi) w)\,d\Sigma
\le \Vert H \Vert^2_{L^2(Q)}.
\end{split}
\end{equation}
Similarly, using \eqref{(5.21)}, we obtain
\begin{equation}\label{(5.25)}\begin{split}
 &\frac 14 \sum_{k=1}^2\Vert \widetilde L_k\widetilde w\Vert^2_{L^2(Q)}
                                                             \\
&+  \frac 14 \int_Q( s\lambda^2 \varphi \mbox{\bf a}( \nabla\psi, \nabla\psi)
\mbox{\bf a}( \nabla \widetilde w, \nabla\widetilde  w)
+ s^3 \lambda^4\varphi^3 \mbox{\bf a}( \nabla\psi, \nabla\psi)^2\widetilde w^2)\,dx\,dt
                                                               \\
&-  \int_\Sigma (2s\lambda\varphi \mbox{\bf a}(\nu,\nabla\widetilde w)
\mbox{\bf a}(\nabla\psi,\nabla\widetilde  w)
- s\lambda\varphi \mbox{\bf a}( \nabla\widetilde w, \nabla\widetilde  w)
\mbox{\bf a}( \nu,\nabla \psi)) \,d\Sigma                            \\
&+  \int_\Sigma(2s^3\lambda^3\varphi^3 \mbox{\bf a}(\nabla\psi, \nabla\psi)
\mbox{\bf a}( \nabla\psi, \nu)
+ 2s^2 \lambda\varphi(\partial_t\alpha) \mbox{\bf a}( \nabla\psi, \nu))\widetilde w^2
\,d\Sigma
\\
&-\int_\Sigma \mbox{\bf a}(\nu,\nabla\widetilde  w)(\partial_t\widetilde  w
+ 2s\lambda^2 \varphi \mbox{\bf a}( \nabla\psi, \nabla\psi) \widetilde w)\,d\Sigma
\le \Vert \widetilde{H}\Vert^2_{L^2(Q)}.
\end{split}
\end{equation}
We set
\begin{equation}\label{(5.26)}\begin{split}
I :=& \int_{(\partial\Omega\setminus \Gamma)\times (0,T)}
(2s\lambda\varphi \mbox{\bf a}(\nu,\nabla w)
\mbox{\bf a}(\nabla\psi,\nabla w)
- s\lambda\varphi \mbox{\bf a}(\nabla w, \nabla w)
\mbox{\bf a}( \nu,\nabla \psi)) \,d\Sigma        \\
&- \int_{(\partial\Omega\setminus \Gamma)\times (0,T)}
(2s^3\lambda^3\varphi^3 \mbox{\bf a}( \nabla\psi, \nabla\psi)
\mbox{\bf a}(\nabla\psi, \nu) + 2s^2 \lambda\varphi (\partial_t\alpha)
\mbox{\bf a}(\nabla\psi, \nu))w^2 \,d\Sigma                      \\
&-  \int_{(\partial\Omega\setminus \Gamma)\times (0,T)}
\mbox{\bf a}(\nu,\nabla w)(\partial_t w
+ 2s\lambda^2\varphi \mbox{\bf a}(\nabla\psi, \nabla\psi) w)\,d\Sigma   \\
&- \int_{\Sigma}
(2s\lambda\varphi \mbox{\bf a}(\nu,\nabla\widetilde w)
\mbox{\bf a}(\nabla\psi,\nabla\widetilde w)
- s\lambda\varphi\mbox{\bf a}( \nabla \widetilde w, \nabla\widetilde w)
\mbox{\bf a}(\nu,\nabla \psi)) \,d\Sigma                \\
&+  \int_{\Sigma}
(2s^3\lambda^3\varphi^3 \mbox{\bf a}( \nabla\psi, \nabla\psi)
\mbox{\bf a}( \nabla\psi, \nu) + 2s^2 \lambda\varphi (\partial_t\alpha)
\mbox{\bf a}(\nabla\psi, \nu))\widetilde w^2 \,d\Sigma
\\
&- \int_{\Sigma} \mbox{\bf a}( \nu,\nabla\widetilde w)
(\partial_t\widetilde w + 2s\lambda^2\varphi \mbox{\bf a}( \nabla\psi, \nabla\psi)
\widetilde w)\,d\Sigma.
\end{split}
\end{equation}
We observe that on  $(\partial\Omega \setminus \Gamma) \times (0,T)$ we have
$$
\nabla w= (\nabla u)e^{s\alpha} + s\lambda \varphi (\nabla \psi) ue^{s\alpha},
\quad
\nabla \widetilde w= (\nabla u)e^{s\alpha} - s\lambda \varphi (\nabla \psi)
u e^{s\alpha}.
$$
Therefore, we rewrite \eqref{(5.26)} as
\begin{equation}\label{(5.27)}\begin{split}
 I = &\int_{(\partial\Omega\setminus \Gamma)\times (0,T)}
\{ s^2\lambda^2\varphi^2
(2\mbox{\bf a}(\nu,\nabla u)\mbox{\bf a}(\nabla\psi, \nabla \psi)
+ 2\mbox{\bf a}(\nu,\nabla \psi)\mbox{\bf a}(\nabla\psi,\nabla u) \\
&- 2s \mbox{\bf a}(\nabla\psi,\nabla u) \mbox{\bf a}( \nu,\nabla \psi)u\} e^{2s\varphi} \,d\Sigma
\\
&- \int_{(\partial\Omega\setminus \Gamma)\times (0,T)}
\mbox{\bf a}(\nu,\,(\nabla u)e^{s\alpha}+s\lambda\varphi(\nabla \psi)ue^{s\alpha})
(\partial_t w + 2s\lambda^2\varphi \mbox{\bf a}( \nabla\psi, \nabla\psi) w)\,d\Sigma
\\
&- \int_{(\partial\Omega\setminus \Gamma)\times (0,T)}
\mbox{\bf a}(\nu,(\nabla u)e^{s\alpha}-s\lambda\varphi(\nabla \psi)ue^{s\alpha})(\partial_t w + 2s\lambda^2\varphi \mbox{\bf a}( \nabla\psi, \nabla\psi) w)\,d\Sigma
\\
= & \int_{(\partial\Omega\setminus \Gamma)\times (0,T)}
2s^2\lambda^2\varphi^2 \mbox{\bf a}(\nu,\nabla u)\mbox{\bf a}(\nabla\psi, \nabla \psi)
ue^{2s\alpha} \,d\Sigma
\\
&- 2\int_{(\partial\Omega\setminus \Gamma)\times (0,T)}
\mbox{\bf a}(\nu,(\nabla u)e^{s\alpha})(\partial_t w
+ 2s\lambda^2\varphi \mbox{\bf a}( \nabla\psi, \nabla\psi) w)\,d\Sigma.
\end{split}
\end{equation}
Since the Robin boundary condition implies
\[
\partial_{\nu_A}u-p(x,t)u = g \quad \mbox{on $(\partial\Omega\setminus \Gamma)
\times (0,T)$},
\]
by \eqref{(5.27)}, we obtain
\begin{equation}\begin{split}\nonumber
 I =& \int_{(\partial\Omega\setminus \Gamma)\times (0,T)}
s^2\lambda^2\varphi^2
2\mbox{\bf a}(\nabla \psi,\nabla \psi)(pu+g)ue^{2s\alpha} \,d\Sigma \\
&- 2\int_{(\partial\Omega\setminus \Gamma)\times (0,T)} (pu+g)e^{s\alpha}
(\partial_t w + 2s\lambda^2 \varphi \mbox{\bf a}( \nabla\psi,\nabla\psi)w)\,d\Sigma
\\
= &\int_{(\partial\Omega\setminus \Gamma)\times (0,T)}
 2s^2\lambda^2\varphi^2 \mbox{\bf a}(\nabla \psi,\nabla \psi)(pu+g)ue^{2s\alpha}
\,d\Sigma  \\
&+ 2\int_{(\partial\Omega\setminus \Gamma)\times (0,T)}
\left( (\partial_tp)\frac{w^2}{2}+\partial_t (ge^{s\alpha})w\right.
\left.
- 2s\lambda^2 \varphi \mbox{\bf a}(\nabla\psi, \nabla\psi)
(pw+ge^{s\alpha})w \right)\,d\Sigma.
\end{split}
\end{equation}
Hence,  for any constant $\epsilon>0$, we have
\begin{equation}\label{(5.28)}
| I|\le \int_{(\partial\Omega\setminus \Gamma)\times (0,T)}
\left(\epsilon s^\frac 52\lambda^2\varphi^\frac 52w^2 +  C(\epsilon)\left(
\frac{| \partial_t g|^2e^{2s\alpha}}
{s^2\lambda^2\varphi^2} + \frac{1}{\root\of{s\varphi}} g^2e^{2s\alpha}\right)
\right)\,d\Sigma.
\end{equation}
By \eqref{(5.24)}, \eqref{(5.25)} and \eqref{(5.28)}, we have
\begin{equation}\label{(5.29)}\begin{split}
 &\frac 14 \sum_{k=1}^2\Vert L_kw\Vert^2_{L^2(Q)}
+ \frac 14 \sum_{k=1}^2\Vert \widetilde L_k\widetilde w\Vert^2_{L^2(Q)}
\\
&+ \frac 14 \int_Q( s\lambda^2 \varphi
\mbox{\bf a}( \nabla\psi, \nabla\psi) \mbox{\bf a}(\nabla w, \nabla w)
+ s^3 \lambda^4\varphi^3 \mbox{\bf a}( \nabla\psi, \nabla\psi)^2 w^2)\,dx\,dt \\
&+  \int_{\Gamma \times (0,T)}(2s\lambda\varphi \mbox{\bf a}(\nu,\nabla w)
\mbox{\bf a}(\nabla\psi,\nabla w)
- s\lambda\varphi \mbox{\bf a}( \nabla w, \nabla w)\mbox{\bf a}(\nu,\nabla \psi))\,d\Sigma\\
&-  \int_{\Gamma \times (0,T)}(2s^3\lambda^3\varphi^3
\mbox{\bf a}(\nabla\psi, \nabla\psi)\mbox{\bf a}( \nabla\psi, \nu)
+ 2s^2 \lambda\varphi (\partial_t\alpha) \mbox{\bf a}(\nabla\psi, \nu))w^2 \,d\Sigma
\\
&-  \int_{\Gamma \times (0,T)} \mbox{\bf a}(\nu,\nabla w)(\partial_t w
+ 2s\lambda^2\varphi \mbox{\bf a}(\nabla\psi, \nabla\psi) w)\,d\Sigma     \\
 &-   \int_{\Gamma \times (0,T)}
(2s\lambda\varphi \mbox{\bf a}(\nu,\nabla\widetilde w)
\mbox{\bf a}(\nabla\psi,\nabla\widetilde  w)
 -s\lambda\varphi \mbox{\bf a}( \nabla\widetilde w, \nabla\widetilde  w)
\mbox{\bf a}(\nu,\nabla \psi)) \,d\Sigma                           \\
&+  \int_{\Gamma \times (0,T)}(2s^3\lambda^3\varphi^3
\mbox{\bf a}(\nabla\psi, \nabla\psi)\mbox{\bf a}(\nabla\psi, \nu)
+ 2s^2 \lambda\varphi(\partial_t\alpha)\mbox{\bf a}(\nabla\psi, \nu))\widetilde w^2
\,d\Sigma                         \\
&-  \int_{\Gamma \times (0,T)} \mbox{\bf a}(\nu,\nabla\widetilde w)
(\partial_t\widetilde w + 2s\lambda^2 \varphi \mbox{\bf a}(\nabla\psi, \nabla\psi)
\widetilde w)\,d\Sigma
\\\le& \Vert\widetilde{H}\Vert^2_{L^2(Q)}
+ \Vert H \Vert^2_{L^2(Q)}+ \int_{(\partial\Omega\setminus \Gamma)\times (0,T)}
\left(\epsilon s^\frac 52\lambda^2\varphi^\frac 52w^2 + C(\epsilon)\left(
\frac{| \partial_t g|^2e^{2s\alpha}}
{s^2\lambda^2\varphi^2} + \frac{1}{\root\of{s\varphi}} g^2e^{2s\alpha}\right)
\right)\,d\Sigma.
\end{split}
\end{equation}
We note
\begin{equation}\label{(5.30)}
\Vert\widetilde{H} \Vert^2_{L^2(Q)}
\le C\Vert H \Vert^2_{L^2(Q)}.
\end{equation}

Taking the scalar product of the functions $L_1(x,t,D,s) w$ and
$s^\frac 32\lambda^2 \varphi^\frac 32 w$ in $L^2(Q)$ and integrating by parts,
we obtain
\begin{align*}
& (L_1(x,t,D,s) w,s^\frac 32\lambda^2 \varphi^\frac 32 w)_{L^2(Q)}\\
= & \left(\partial_t w + 2{{s}}\lambda \varphi
\sum_{i,j=1}^n a_{ij} \psi_{x_i} \partial_{j} w + 2{{s}}\lambda^2 \varphi
\mbox{\bf a}(\nabla\psi, \nabla\psi) w,s^\frac 32\lambda^2 \varphi^\frac 32 w
\right)_{L^2(Q)}                               \\
=& \int_Q \left(-\frac 12 \partial_t (s^\frac 32 \lambda^2\varphi^\frac 32)
w^2 - \sum_{i,j=1}^n s\lambda\partial_{x_j}(\varphi a_{ij}\psi_{i})w^2
+ 2s^\frac 52 \lambda^4 \varphi^\frac 52 w^2\right)\,dx \\
&+ \int_\Sigma s^\frac 52\lambda^2 \varphi^\frac 52 \mbox{\bf a}(\nu,\nabla \psi)
w^2\,d\Sigma.
\end{align*}

The above equality implies
\begin{multline}\label{(5.31)}
\int_{(\partial\Omega\setminus \Gamma)\times (0,T)}
s^\frac 52\lambda^2 \varphi^\frac 52 w^2\,d\Sigma\le C\int_Q(s\lambda^2 \varphi \mbox{\bf a}( \nabla\psi, \nabla\psi)
\mbox{\bf a}( \nabla w, \nabla w)
\\
+ s^3 \varphi^3 \lambda^4 \mbox{\bf a}( \nabla\psi, \nabla\psi)^2 w^2)\,dx\,dt
+ C\int_{\Gamma \times (0,T)}s^\frac 52\lambda^2 \varphi^\frac 52  w^2\,d\Sigma.
\end{multline}
By \eqref{(5.29)} - \eqref{(5.31)}, we obtain
\begin{equation}\label{(5.32)}\begin{split}
&\int_Q ( s\varphi | \nabla u|^2 + s^3 \varphi^3 | u|^2)
e^{2s\alpha} \,dx\,\,dt
+ \sum_{k=1}^2\Vert L_k(x,t,D,s) (ue^{s\alpha})\Vert^2_{L^2(Q)}
\\
\le& C\int_Q |F|^2 e^{2s\alpha} \,dx\,dt
+ C\int_{(\partial\Omega\setminus \Gamma)\times (0,T)}
\left( \frac{| \partial_t g|^2e^{2s\alpha}}{s^2\lambda^2\varphi^2}
+ \frac{1}{\root\of{s\varphi}} | g|^2e^{2s\alpha} \right)\,d\Sigma
\\
&+ C\int_{\Gamma\times (0,T)} \left(
s\lambda\varphi | \nabla u|^2 + s^3\lambda^3\varphi^3 | u|^2
+ \frac{| \partial_t u|^2}{s\varphi}\right) e^{2s\alpha}  \,d\Sigma.
\end{split}
\end{equation}
From the definition of the operator $L^2$ and \eqref{(5.32)}, we have
\begin{equation}\begin{split}\label{(5.33)}
\int_Q \frac{1}{s\varphi} |\partial_t u|^2e^{2s\alpha}\,dx\,dt 
\le& C\left(\Vert L_1w\Vert^2_{L^2(Q)} +
\int_Q (s\varphi |\nabla u|^2 + s^3 \varphi^3 u^2) e^{2s\varphi} \,dx\,dt\right)
\\
\le& C\int_Q |F|^2 e^{2s\varphi} \,dx\,dt
+ C\int_{(\partial\Omega\setminus \Gamma)\times (0,T)}
\left( \frac{| \partial_t g|^2e^{2s\varphi}}{s^2\lambda^2\varphi^2}
+ \frac{1}{\root\of{s\varphi}} | g|^2e^{2s\alpha}\right)\,d\Sigma
\\
&+ C\int_{\Gamma \times (0,T)} \left(s\lambda\varphi | \nabla u|^2
+ s^3\lambda^3\varphi^3 | u|^2
+ \frac{| \partial_t u|^2}{s\varphi} \right)e^{2s\varphi} \,d\Sigma.
\end{split}
\end{equation}
On the other hand
\begin{equation}\begin{split}\label{(5.34)}
\int_Q \frac{1}{s\varphi} \sum_{i,j=1}^n |\partial_i\partial_j u|^2e^{2s\alpha}\,dx\,dt
\le& C\biggl( \Vert L_1w\Vert^2_{L^2(Q)}
+ \int_Q (s\lambda^2\varphi |\nabla u|^2 + s^3 \lambda^3\varphi^3 u^2)
e^{2s\varphi} \,dx\,dt           \\
&+  \Vert ge^{s\alpha}\Vert^2_{L^2(0,T;H^\frac 12(\partial\Omega))}
+ \Vert s\varphi w\Vert^2_{L^2((\partial\Omega\setminus \Gamma)\setminus (0,T))}
\biggr).
\end{split}
\end{equation}
From \eqref{(5.32)} - \eqref{(5.34)}, we complete the proof of Lemma \ref{6}.
\end{proof}

Here and henceforth, by $C>0$ we denote generic constants which
are independent of the parameter $s>0$, and we write $\tilde{C}(s)$
when we need to specify the dependency.

Moreover, we set
\begin{equation}\label{(2.6)} \Vert g\Vert_*:= \Vert g\Vert_{H^1(0,T;L^2(\partial\Omega\setminus \Gamma))}
+ \Vert g\Vert_{L^2(0,T;H^{\frac{1}{2}}(\partial\Omega))}.
\end{equation}

We note
$$
\partial_i\varphi = \lambda(\partial_i\eta)\varphi, \quad
\partial_i\partial_j\varphi = (\lambda\partial_i\partial_j\eta + \lambda^2(\partial_i\eta)(\partial_j\eta))\varphi
$$
for $1\le i,j \le n$ and
$$
\left| \frac{n}{\,dt}\left(\frac{1}{\mu(t)}\right) \right|
\le \frac{C}{\mu^2(t)}, \quad 0<t<T.
$$
Hence
\begin{equation}\label{(2.10)}
\begin{cases}
| \partial_t\varphi| \le C\varphi^2, \quad
| \nabla\varphi| \le C\varphi, \quad
| \partial_i\partial_j\varphi| \le C\varphi \quad \mbox{in $Q$ for $1\le i,j\le n$},
\\
| \nabla \alpha | \le C\varphi, \quad | \partial_t\alpha|
\le C\varphi^2 \quad \mbox{in $Q$}. 
\end{cases}
\end{equation}

In order to rewrite the norms appearing in \eqref{(2.5)}, we show

\begin{lem}\label{3}
(i) For each $\rho\in \mathbb{R}$, we have
\begin{equation}\label{(2.7)}
\sup_{s\ge 1}
\sup_{(x,t)\in Q} | \varphi(x)^{\rho}e^{2s\alpha(x,t)}| < \infty.
\end{equation}
\\
(ii) Let $\rho\in \mathbb{R}$ and $\psi \in C([0,T]; C^1(\partial\Omega))$
be an arbitrarily given function and let $s\ge 1$ be arbitrary.
Then there exists a constant $C_{\rho},$ independent of $s$, such that
\begin{equation}\label{(2.8)}
\Vert \varphi^{\rho}\psi ge^{s\alpha}\Vert_* \le C_{\rho}s\Vert g\Vert_*.
\end{equation}
\end{lem}

\begin{proof}
(i) First we have
$$
\varphi(x,t)^{\rho}e^{2s\alpha} \le \varphi(x,t)^{\rho}
\le \mu(t)^{-\rho} e^{\lambda\rho\eta(x)},
\quad\text{in } Q
$$
if $\rho\le 0$, which readily verifies \eqref{(2.7)}.
On the other hand, for $\rho>0$, we have
$$
\varphi(x,t)^{\rho}e^{2s\alpha}
\le \frac{C}{\mu(t)^{\rho}}
\exp\left(
2 \frac{e^{\lambda\eta(x)} - e^{2\lambda\Vert\eta\Vert_{C(\overline{\Omega})}} }{\mu(t)}
\right)
\le \frac{C}{\mu(t)^{\rho}}e^{-\frac{C_1}{\mu(t)}}, \quad\text{in } Q,
$$
where
$$
C_1:=  2(e^{2\lambda\Vert\eta\Vert_{C(\overline{\Omega})}}
- e^{\lambda\Vert\eta\Vert_{C(\overline{\Omega})}}) > 0.
$$
Noting that $\xi:= \frac{1}{\mu(t)}$, $t\in (0,T)$ varies $[c_0,\infty)$
with some constant $c_0>0$ and 
$\sup_{\xi\ge c_0} \xi^{\rho}e^{-C_1\xi} < \infty$,
we see \eqref{(2.7)} for all $\rho \in \mathbb{R}$.
\\
(ii) By \eqref{(2.10)}, we have
\begin{align*}
& \Vert \varphi^{\rho}\psi ge^{s\alpha}\Vert_{H^1(0,T;L^2(\partial\Omega\setminus
\Gamma))}
\le \Vert \rho\varphi^{\rho-1}(\partial_t\varphi) \psi ge^{s\alpha}\Vert
_{L^2(0,T;L^2(\partial\Omega\setminus \Gamma))}\\
&+ \Vert \varphi^{\rho}(\partial_t\psi)ge^{s\alpha}\Vert
_{L^2(0,T;L^2(\partial\Omega\setminus\Gamma))}
+ \Vert \varphi^{\rho}\psi (\partial_tg)e^{s\alpha}\Vert
_{L^2(0,T;L^2(\partial\Omega\setminus \Gamma))}\\
&+ s\Vert \varphi^{\rho}\psi g(\partial_t\alpha)e^{s\alpha}\Vert
_{L^2(0,T;L^2(\partial\Omega\setminus\Gamma))}
+ \Vert \varphi^{\rho}\psi ge^{s\alpha}\Vert_{L^2(0,T;L^2(\partial\Omega\setminus
\Gamma))}\\
\le& Cs\Vert \varphi^{\rho+2}ge^{s\alpha}\Vert_{L^2(0,T;L^2(\partial\Omega \setminus
\Gamma))}
+ C\Vert \varphi^{\rho}(\partial_tg)e^{s\alpha}\Vert
_{L^2(0,T;L^2(\partial\Omega \setminus\Gamma))}\\ &
+ C\Vert \varphi^{\rho+1}ge^{s\alpha}\Vert_{L^2(0,T;L^2(\partial\Omega \setminus
\Gamma))}.
\end{align*}
Therefore, \eqref{(2.7)} yields
\begin{equation*}
\Vert \varphi^{\rho}\psi ge^{s\alpha}\Vert_{H^1(0,T;L^2(\partial\Omega
\setminus \Gamma))}
\le C(s\Vert g\Vert_{L^2(0,T;L^2(\partial\Omega\setminus \Gamma))}
+ \Vert \partial_tg\Vert_{L^2(0,T;L^2(\partial\Omega\setminus \Gamma))})
\le Cs\Vert g\Vert_*
\end{equation*}
for all $s\ge 1$.

Next, in view of the Sobolev-Slobodecki norm in $H^{\frac{1}{2}}(\partial\Omega)$
(e.g., Adams \cite{Ad}), we can directly verify  that there exists
a constant $C$ such that
\begin{equation}\label{(2.9)}
\Vert \psi a\Vert_{H^{\frac{1}{2}}(\partial\Omega)}
\le C\Vert \psi\Vert_{C^1(\partial\Omega)} \Vert a\Vert_{H^{\frac{1}{2}}(\partial\Omega)}
                                           \end{equation}
for $a \in H^{\frac{1}{2}}(\partial\Omega)$ and $\psi \in C^1(\partial\Omega)$.

By \eqref{(2.7)}, we have
$$
\Vert \varphi^{\rho}\psi e^{s\alpha}(\cdot,t)\Vert_{C^1(\overline{\Omega})}
\le \Vert \varphi^{\rho}\psi e^{s\alpha}(\cdot,t)\Vert_{C(\overline{\Omega})}
+ \Vert \nabla(\varphi^{\rho}\psi e^{s\alpha})(\cdot,t)\Vert_{C^1(\overline{\Omega})}
\le C_{\rho}s.
$$
Hence, \eqref{(2.10)} and \eqref{(2.9)} yield
\begin{equation*}
\Vert \varphi^{\rho}\psi g e^{s\alpha}(\cdot,t)\Vert_{H^{\frac{1}{2}}(\partial\Omega)}
\le C\Vert \varphi^{\rho}\psi e^{s\alpha}(\cdot,t)\Vert_{C^1(\overline{\partial\Omega})}
\Vert g(\cdot,t)\Vert_{H^{\frac{1}{2}}(\partial\Omega)}
\le C_{\rho}s\Vert g(\cdot,t)\Vert_{H^{\frac{1}{2}}(\partial\Omega)},
\end{equation*}
and so
$$
\Vert \varphi^{\rho}\psi g e^{s\alpha}\Vert_{L^2(0,T;H^{\frac{1}{2}}(\partial\Omega))}
\le C_{\rho}s\Vert g \Vert_{L^2(0,T;H^{\frac{1}{2}}(\partial\Omega))}.
$$
Thus the proof of \eqref{(2.8)} is complete.
\end{proof}

In view of Lemma \ref{3}, we can rewrite \eqref{(2.5)} as
\begin{multline}\label{2.10}
\int_Q \left(
\frac{1}{s\varphi} \left(|\partial_t u|^2
+ \sum_{i,j=1}^n| \partial_i\partial_ju|^2\right)
+ s\varphi |\nabla u|^2 + s^3\varphi^3 | u|^2
\right) e^{2s\alpha}  \,dx\,\,dt
\\
\le C\int_Q | F|^2 e^{2s\varphi} \,dx\,dt
+ \tilde{C}(s)(\Vert g\Vert^2_* + \Vert u\Vert_{H^1(\Gamma \times (0,T))}^2)
\end{multline}
for each $s \ge s_0$.

\subsection{Carleman Estimate for the Mean Field Game Equations}

Next we derive a Carleman estimate, which will allow us to proof the main Theorem \ref{t2}.
\begin{lem}\label{2}
Assume \eqref{lopukh1}, \eqref{lopukh} and
$p\in C^1(\partial\Omega\times [0,T]).$
Let $g\in L^2(0,T,H^\frac 12 (\partial\Omega))$,
$\partial_tg\in L^2((\partial\Omega\setminus \Gamma) \times (0,T))$
and $F\in L^2(Q)$.
We fix a sufficiently large constant $\lambda>0$.
Then, for each $m\in \mathbb{R}$,  there exist constants $s_0 > 0$ and $C>0$ such that
\begin{equation}\begin{split}
 &\quad\int_Q \left((s\varphi)^{m-1}\left(| \partial_tu|^2
+ \sum_{i,j=1}^n | \partial_i\partial_ju|^2\right)
+ (s\varphi)^{m+1}| \nabla u|^2 + (s\varphi)^{m+3}| u|^2\right)
e^{2s\varphi} \,dx\,dt\\
&\le C\int_Q (s\varphi)^m| F|^2 e^{2s\varphi} \,dx\,dt+  C\int_{(\partial\Omega\setminus \Gamma)\times (0,T)}
s^m\left( \frac{| \partial_t(\varphi^{\frac{m}{2}}g)|^2}{s^2\varphi^2}
+\frac{1}{\root\of{s\varphi}} | \varphi^{\frac{m}{2}} g|^2 \right) e^{2s\alpha} \,dS\,dt\\
&\quad+  Cs^m \Vert \varphi^{\frac{m}{2}} ge^{s\alpha}\Vert^2_{L^2(0,T;
H^{\frac{1}{2}}(\partial\Omega))} + \tilde{C}(s,m)\Vert u\Vert_{H^1(\Gamma \times (0,T))}^2\nonumber
\end{split}
\end{equation}
for all $s > s_0$ and $u\in H^{2,1}(Q)$ satisfying \eqref{(2.3)} and
\eqref{(2.4)}. 

Here $\tilde{C}(s,m)$ is a positive constant depending on
$s$ and $m$.
\end{lem}

\begin{proof}
We will derive Lemma \ref{2} from Lemma \ref{1} in the case where
$\partial_tu + A(t)u = F$.  The derivation for the case of
$\partial_tu - B(t)u = F$ is quite similar.

Moreover, choosing $s>0$ sufficiently large, by
$a_j, b_j, a_0, b_0 \in L^{\infty}(Q)$, we can absorb the lower-order terms
$\sum_{j=1}^n a_j\partial_ju$, $a_0u$, $\sum_{j=1}^n b_j\partial_ju$, $b_0u$
into the left-hand side of the Carleman estimate.
Thus it suffices to prove the lemma for
$A(t) = \sum_{i,j=1}^n a_{ij}\partial_i\partial_j$ and $B(t) = \sum_{i,j=1}^n b_{ij}\partial_i\partial_j$.

We set
$$
w:= \varphi^{\frac{m}{2}}u.
$$
Then we can directly calculate:
\begin{equation}\label{(2.11)}
\begin{cases}
\partial_tw = \frac{m}{2}\varphi^{\frac{m}{2}-1}(\partial_t\varphi)u
+ \varphi^{\frac{m}{2}}\partial_tu,\\
\partial_iw = \frac{m}{2}\varphi^{\frac{m}{2}-1}(\partial_i\varphi)u
+ \varphi^{\frac{m}{2}}\partial_iu, \\
\partial_i\partial_jw = \frac{m}{2}\left(\frac{m}{2}-1\right)\varphi^{\frac{m}{2}-2}
(\partial_i\varphi)(\partial_j\varphi)u
+ \frac{m}{2}\varphi^{\frac{m}{2}-1}(\partial_i\partial_j\varphi)u\\
\qquad\qquad+ \frac{m}{2}\varphi^{\frac{m}{2}-1}(\partial_i\varphi)(\partial_ju)
+ \frac{m}{2}\varphi^{\frac{m}{2}-1}(\partial_j\varphi)(\partial_iu)
+ \varphi^{\frac{m}{2}}\partial_i\partial_ju.
\end{cases}
\end{equation}
Therefore,
\begin{align*}
\partial_tw + A(t)w
&=\varphi^{\frac{m}{2}}F + \frac{m}{2}\varphi^{\frac{m}{2}-1}(\partial_t\varphi)u+ \frac{m}{2}\left( \frac{m}{2}-1 \right)\sum_{i,j=1}^n a_{ij}
  \varphi^{\frac{m}{2}-2}(\partial_i\varphi)(\partial_j\varphi)u
\\&\quad+ \frac{m}{2} \varphi^{\frac{m}{2}-1}\sum_{i,j=1}^n a_{ij}(\partial_i\partial_j\varphi)u+ m\varphi^{\frac{m}{2}-1}\sum_{i,j=1}^n a_{ij}(\partial_i\varphi)(\partial_ju)  \quad \mbox{in $Q$}
\end{align*}
and
\begin{align*}
 \partial_{\nu_A}w =& \sum_{i,j=1}^n a_{ij}\partial_i(\varphi^{\frac{m}{2}}u)\nu_j\\
=& \varphi^{\frac{m}{2}}\sum_{i,j=1}^n a_{ij} (\partial_iu)\nu_j
+ \frac{m}{2} \varphi^{\frac{m}{2}-1}\left(\sum_{i,j=1}^n a_{ij}(\partial_i\varphi)\nu_j
\right)u\\
=& \varphi^{\frac{m}{2}}g
+ \left( \frac{m}{2}\varphi^{-1}\sum_{i,j=1}^n a_{ij} (\partial_i\varphi)\nu_j \right)w
\quad \mbox{on $\partial\Omega \times (0,T)$}.
\end{align*}
Hence,
\begin{equation}\label{1313}
\partial_{\nu_A}w - \widetilde{p}(x,t)w = \varphi^{\frac{m}{2}}g, \quad
\text{on } \partial\Omega\times (0,T),
\end{equation}
where
$$
\widetilde{p}(x,t):= p(x,t) + \frac{m}{2}\varphi^{-1}\sum_{i,j=1}^n a_{ij} (\partial_i\varphi)\nu_j
= p(x,t) + \frac{m\lambda}{2}\sum_{i,j=1}^n a_{ij}(\partial_i\eta)\nu_j.
$$
Moreover
\begin{equation}\label{1414}
\partial_tw + A(t)w = \varphi^{\frac{m}{2}}F + \widetilde{F} \quad \mbox{in $Q$},
\end{equation}
where we see
$$
| \widetilde{F}(x,t)| \le C(\varphi^{\frac{m}{2}+1}| u|
+ \varphi^{\frac{m}{2}}| \nabla u(x,t)|
\le C(\varphi| w(x,t)| + | \nabla w(x,t)|), \quad
\text{in } Q.
$$
We apply Lemma \ref{1} to \eqref{1313} and \eqref{1414} to
obtain
\begin{equation}\label{(2.12)}
\begin{split}
 &\int_Q \left( \frac{1}{s\varphi}\left( | \partial_tw|^2
+ \sum_{i,j=1}^n | \partial_i\partial_jw|^2\right)
+ s\varphi| \nabla w|^2 + s^3\varphi^3| w|^2\right)
e^{2s\varphi} \,dx\,dt\\
\le& C\int_Q \varphi^m| F|^2 e^{2s\varphi} \,dx\,dt
+ C\int_Q (\varphi^2| w|^2 + | \nabla w|^2) e^{2s\varphi} \,dx\,dt
                                     \\
&+  C\int_{(\partial\Omega\setminus \Gamma)\times (0,T)}
\left( \frac{| \partial_t(\varphi^{\frac{m}{2}}g)|^2}{s^2\varphi^2}
+ \frac{1}{\sqrt{s\varphi}}| \varphi^{\frac{m}{2}}g|^2 \right)
e^{2s\alpha} \,dS\,dt\\
&+  C\Vert \varphi^{\frac{m}{2}} ge^{s\alpha}\Vert^2_{L^2(0,T;
H^{\frac{1}{2}}(\partial\Omega))} + \tilde{C}(s)\Vert w\Vert_{H^1(\Gamma \times (0,T))}^2
\end{split}
\end{equation}
for all $s > s_0$.
Choosing $s>0$ large, we can absorb the second term on the right-hand side
into the left-hand side.

In terms of $u$ we rewrite as follows.
By \eqref{(2.11)}, we first have
$$
s^3\varphi^3| w|^2 = s^3\varphi^{m+3}| u|^2,
$$
$$
\varphi^{\frac{m}{2}}(\partial_iu) = \partial_iw - \frac{m}{2}\varphi^{\frac{m}{2}-1}
(\partial_i\varphi)u,
$$
and so
$$
s\varphi^{m+1}| \nabla u|^2 \le 2s\varphi| \nabla w|^2
+ Cs\varphi\left| m\varphi^{\frac{m}{2}-1}(\nabla \varphi)u
\right|^2
\le 2s\varphi| \nabla w|^2 + Cs^2\varphi^{m+2}| u|^2.
$$
Moreover, \eqref{(2.11)} implies
$$
\frac{1}{s} \varphi^{m-1}| \partial_tu|^2
\le \frac{C}{s\varphi}| \partial_tw|^2 + Cs^{-1}\varphi^{m+1} | u|^2.
$$
Finally, again \eqref{(2.11)} yields
\begin{align*}
\varphi^{\frac{m}{2}}\partial_i\partial_ju
=& \partial_i\partial_jw - \frac{m}{2}\left( \frac{m}{2}-1\right) \varphi^{\frac{m}{2}-2}
(\partial_i\varphi)(\partial_j\varphi)u
- \frac{m}{2}\varphi^{\frac{m}{2}-1}(\partial_i\partial_j\varphi)u\\
&- \frac{m}{2}\varphi^{\frac{m}{2}-1}(\partial_i\varphi)(\partial_ju)
- \frac{m}{2}\varphi^{\frac{m}{2}-1}(\partial_j\varphi)(\partial_iu),
\end{align*}
and so
\begin{align*}
\frac{1}{s\varphi} \left| \varphi^{\frac{m}{2}}\partial_i\partial_ju\right|^2
\le& \frac{C}{s\varphi}| \partial_i\partial_jw|^2+  \frac{C}{s\varphi}\left| \frac{m}{2}\left( \frac{m}{2} - 1\right)
\varphi^{\frac{m}{2}-2}(\partial_i\varphi)(\partial_j\varphi)u \right.
\\&\left.+ \frac{m}{2}\varphi^{\frac{m}{2}-1}(\partial_i\partial_j\varphi)u
+ \frac{m}{2} \varphi^{\frac{m}{2}-1}((\partial_i\varphi)\partial_ju
  + (\partial_j\varphi)\partial_iu)\right|^2.
\end{align*}
Hence,
$$
\frac{1}{s} \varphi^{m-1}| \partial_i\partial_ju|^2
\le \frac{C}{s\varphi}| \partial_i\partial_jw|^2
+ \frac{C}{s\varphi}(\varphi^m| u|^2 + \varphi^m| \nabla u|^2).
$$
In \eqref{(2.12)}, we can estimate $\Vert w\Vert_{H^1(\Gamma \times (0,T))}^2$
by means of Lemma \ref{3} (i):
$$
\Vert w\Vert_{H^1(\Gamma \times (0,T))}^2
\le \tilde{C}(s,m)\Vert u\Vert_{H^1(\Gamma \times (0,T))}^2
$$
for all $s\ge 1$.  Thus, the proof of Lemma \ref{2} is complete.
\end{proof}

In particular, setting $m=1$ in Lemma \ref{2}, we have

\begin{lem}\label{4}
Suppose that the conditions of Lemma \ref{2} hold true.
We can find constants $s_0 > 0$ and $C>0$ such that
\begin{align*}
& \int_Q \left(| \partial_tu|^2 + \sum_{i,j=1}^n | \partial_i\partial_ju|^2
+ s^2\varphi^2| \nabla u|^2 + s^4\varphi^4| u|^2\right)
e^{2s\varphi} \,dx\,dt\\
\le& C\int_Q s\varphi| \partial_tu + A(t)u|^2 e^{2s\varphi} \,dx\,dt+  C\int_{(\partial\Omega\setminus \Gamma)\times (0,T)}
\left( \frac{| \partial_t(\varphi^{\frac{1}{2}}g)|^2}{s^2\varphi^2}
+ {\root\of{s\varphi}}| g|^2 \right) e^{2s\alpha} \,dS\,dt\\
&+  C\Vert s^{\frac{1}{2}}\varphi^{\frac{1}{2}} ge^{s\alpha}\Vert^2_{L^2(0,T;
H^{\frac{1}{2}}(\partial\Omega))} + \tilde{C}(s)+\Vert u\Vert_{H^1(\Gamma \times (0,T))}^2
\end{align*}
for all $s > s_0$ and $u\in H^{2,1}(Q)$ satisfying \eqref{(2.3)} and
\eqref{(2.4)}.
\end{lem}

Now, noting that we have the Carleman estimates both for
$\partial_t + A(t)$ and $\partial_t - B(t)$ with the same weight $e^{2s\varphi}$,
we derive the main Carleman estimate for the mean field game
system \eqref{(1.3)} with \eqref{(1.4)}.
Setting $F:= c_0v+F$ in the first equation in \eqref{(1.4)}, we apply Lemma \ref{4} to obtain
\begin{equation}\label{(2.13)}
 \begin{split}&\int_Q \left(| \partial_tu|^2 + \sum_{i,j=1}^n | \partial_i\partial_ju|^2
+ s^2\varphi^2| \nabla u|^2 + s^4\varphi^4| u|^2 \right)
e^{2s\varphi} \,dx\,dt
 \\
\le & C\int_Q s\varphi| m|^2 e^{2s\varphi} \,dx\,dt
+ C\int_Q s\varphi | F|^2 e^{2s\varphi} \,dx\,dt+  C\int_{(\partial\Omega\setminus \Gamma)\times (0,T)}
\left( \frac{| \partial_tg|^2}{s\varphi} + \root\of{s\varphi}| g|^2
\right) e^{2s\varphi} \,dS\,dt
\\
&+ s\Vert \varphi^{\frac{1}{2}} ge^{s\alpha}\Vert^2_{L^2(0,T;
H^{\frac{1}{2}}(\partial\Omega))}  + \tilde{C}(s) \Vert u\Vert_{H^1(\Gamma \times (0,T))}^2 \\
\le& C\int_Q s\varphi| m|^2 e^{2s\varphi} \,dx\,dt
+ C\int_Q s\varphi | F|^2 e^{2s\varphi} \,dx\,dt+ \tilde{C}(s)(\Vert g\Vert_*^2 + \Vert u\Vert_{H^1(\Gamma \times (0,T))}^2)\end{split}
\end{equation}
for all $s > s_0$.

The application of Lemma \ref{1} to the second equation in
\eqref{(1.3)} with $G:= G + A_0(t)u$ yields
\begin{equation}\label{(2.14)}\begin{split}
 &\int_Q \biggl\{ \frac{1}{s\varphi}\left(
| \partial_tm|^2 + \sum_{i,j=1}^n | \partial_i\partial_jv|^2\right)
+ s\varphi| \nabla m|^2 + s^3\varphi^3| m|^2 \biggr\}
e^{2s\varphi} \,dx\,dt\\
\le& C\int_Q \sum_{i,j=1}^n | \partial_i\partial_ju|^2 e^{2s\varphi} \,dx\,dt
+ C\int_Q | G|^2 e^{2s\varphi} \,dx\,dt + \tilde{C}(s)(\Vert h\Vert_*^2 + \Vert m\Vert_{H^1(\Gamma \times (0,T))}^2).
\end{split}
\end{equation}
Estimating the first term on the right-hand side of \eqref{(2.14)} in
terms of \eqref{(2.13)}, we obtain
\begin{equation}\begin{split}
 &\int_Q \biggl\{ \frac{1}{s\varphi}\left(
| \partial_tm|^2 + \sum_{i,j=1}^n | \partial_i\partial_jv|^2\right)
+ s\varphi| \nabla m|^2 + s^3\varphi^3| m|^2 \biggr\}
e^{2s\varphi} \,dx\,dt\nonumber\\
\le& C\int_Q s\varphi| m|^2 e^{2s\varphi} \,dx\,dt
+ C\int_Q s\varphi| F|^2 e^{2s\varphi} \,dx\,dt
+ C\int_Q | G|^2 e^{2s\varphi} \,dx\,dt \nonumber\\
&+  \tilde{C}(s)(\Vert g\Vert_*^2 + \Vert h\Vert^2_*
+ \Vert u\Vert_{H^1(\Gamma \times (0,T))}^2
+ \Vert m\Vert_{H^1(\Gamma \times (0,T))}^2)\nonumber\end{split}
\end{equation}
for all large $s > 0$.
Hence, choosing $s>0$ sufficiently large, we can absorb
the first term on the right-hand side into the left-hand side, and we can
obtain
\begin{equation}\begin{split}
& \int_Q \biggl\{ \frac{1}{s\varphi}\left(
| \partial_tm|^2 + \sum_{i,j=1}^n | \partial_i\partial_jv|^2\right)
+ s\varphi| \nabla m|^2 + s^3\varphi^3| m|^2 \biggr\}
e^{2s\varphi} \,dx\,dt\nonumber\\
\le&  C\int_Q (s\varphi| F|^2 + | G|^2) e^{2s\varphi} \,dx\,dt
+ \tilde{C}(s)(\Vert g\Vert_*^2
+ \Vert h\Vert^2_*
+  \Vert u\Vert_{H^1(\Gamma \times (0,T))}^2
+ \Vert m\Vert_{H^1(\Gamma \times (0,T))}^2)
\end{split}
\end{equation}
for all large $s > 0$.  Adding with \eqref{(2.13)}, we absorb the term
$\int_Q s\varphi | m|^2 e^{2s\varphi} \,dx\,dt$ on the right-hand side into the
left-hand side, so that we proved

\begin{thm} \label{t3} (Carleman estimate for a generalized mean field game
equations)
Let $g,h\in L^2(0,T;H^\frac 12(\partial\Omega))$, $\partial_t g$,
$\partial_t h\in L^2(0,T;L^2(\partial\Omega\setminus \Gamma))$,
$F, G\in L^2(Q)$ and \eqref{lopukh1}, \eqref{lopukh}, \eqref{1.7} hold true.
We fix $\lambda>0$ sufficiently large.
Then we can find constants $s_0 > 0$ and $C>0$ such that
\begin{equation}\label{(2.17)}\begin{split}
 &\int_Q \biggl\{ | \partial_tu|^2 + \sum_{i,j=1}^n | \partial_i\partial_ju|^2
+ s^2\varphi^2| \nabla u|^2 + s^4\varphi^4| u|^2 \\
&+ \frac{1}{s\varphi}\left(
| \partial_tm|^2 + \sum_{i,j=1}^n | \partial_i\partial_jv|^2\right)
+ s\varphi| \nabla m|^2 + s^3\varphi^3| m|^2 \biggr\}
e^{2s\varphi} \,dx\,dt\\
\le& C\int_Q (s\varphi| F|^2 + | G|^2) e^{2s\varphi} \,dx\,dt
+ \tilde{C}(s)(\Vert g\Vert_*^2 + \Vert h\Vert^2_* +  \Vert u\Vert_{H^1(\Gamma \times (0,T))}^2
+ \Vert m\Vert_{H^1(\Gamma \times (0,T))}^2)\end{split}
\end{equation}
for all $s > s_0$ and $u, m\in H^{2,1}(Q)$ satisfying
\eqref{(1.3)} and \eqref{(1.4)}.
Here the constant $C>0$ depends continuously on $M$ (the bound of the coefficients) and $\lambda$ but
independent of $s \ge s_0$.
\end{thm}

Now we proceed to

\begin{proof}[Proof of Theorem \ref{t1}]
Since $\mu(t) \ge \mu(\varepsilon)$ for $\varepsilon \le t \le T-\varepsilon$, there exist
constants $C_2>0$ and $C_3>0$ such that
$$
\alpha(x,t) \ge \frac{e^{\lambda\eta(x)} - e^{2\lambda\Vert \eta\Vert_{C(\overline{\Omega})}}
}{\mu(\varepsilon)}
\ge \frac{-C_2}{\mu(\varepsilon)} =: -C_3
$$
for all $x\in \overline{\Omega}$ and $\varepsilon \le t \le T-\varepsilon$, and so
$$
e^{2s\alpha(x,t)} \ge e^{-2sC_3}, \quad x\in \overline{\Omega}, \,
\varepsilon\le t \le T-\varepsilon.
$$
Thus Theorem \ref{t3} completes the proof of Theorem \ref{t1}.
\end{proof}

\subsection{Proof of Theorem \ref{t2}}\label{Z3}

The proof is based on a similar idea to Theorem 3.1
in Imanuvilov and Yamamoto \cite{IY98}, where we have to
estimate extra second-order derivatives of $u$.

Without loss of generality, we can assume that $t_0 = \frac{T}{2}$ by
scaling the time variable.
Indeed, we choose small $\delta > 0$
such that $0<t_0-\delta < t_0 < t_0 + \delta < T$.
Then we consider a change of the variables
$t \mapsto \xi:= \frac{t-(t_0-\delta)}{2\delta}T$.  Then, the inverse problem
over the time interval $(t_0-\delta,\, t_0+\delta)$ can be transformed to
$(0,T)$ with $t_0= \frac{T}{2}$.
Thus, it is sufficient to assume that $t_0 = \frac{T}{2}$ and
$I = (0,T)$.

{\bf First Step}

We show
\begin{lem}\label{5}
Let $r \in \mathbb{R}$ and $w\in H^2(Q)$.  Then there exists a constant $C>0$,
independent of $w$, such that
\begin{equation}\label{(3.1)}
\int_{\partial\Omega} \varphi^{2r}| w|^2 e^{2s\alpha} \,dS
\le C\int_{\Omega}
(\varphi^{2r}| \nabla w|^2 + s^2\varphi^{2r+2}| w|^2)e^{2s\varphi} \,dx
                                     \end{equation}
and
\begin{equation}\label{(3.2)}
\int_{\partial\Omega} \varphi^{2r}| \nabla w|^2 e^{2s\alpha} \,dS
\le C\int_{\Omega} \left(\varphi^{2r}\sum_{i,j=1}^n| \partial_i\partial_j w|^2
+ s^2\varphi^{2r+2}| \nabla w|^2\right)e^{2s\varphi} \,dx            \end{equation}
for all $s>0$.
\end{lem}

\begin{proof}
Indeed, the trace theorem and \eqref{(2.10)} imply
\begin{align*}
\Vert \varphi^rw e^{s\alpha}\Vert^2_{L^2(\partial\Omega)}
&\le C(\Vert \varphi^r we^{s\alpha}\Vert^2_{L^2(\Omega)}
+ \Vert \nabla(\varphi^r we^{s\alpha})\Vert^2_{L^2(\Omega)})\\
&\le C\int_{\Omega}
(\varphi^{2r}| w|^2 + s^2\varphi^{2r+2}| w|^2
+ \varphi^{2r}| \nabla w|^2) e^{2s\varphi} \,dx\\
&\le C\int_{\Omega}
(\varphi^{2r}| \nabla w|^2 + s^2\varphi^{2r+2}| w|^2)e^{2s\varphi} \,dx.
\end{align*}
Thus \eqref{(3.1)} is seen.  Similarly we can prove \eqref{(3.2)}.
\end{proof}

{\bf Second Step: Carleman estimate for $(\partial_tu, \partial_tm)$}

Setting $y:= \partial_tu$ and $z:= \partial_t m$ in \eqref{(1.3)}, we have
\begin{equation}\label{(3.3)}
\begin{cases}
\partial_ty + A(t)y = c_0z + (\partial_tc_0)m - (\partial_tA(t))u + \partial_tF,\\
\partial_tz - B(t)z = A_0(t)y + (\partial_tA_0(t))u + (\partial_tB(t))m + \partial_tG
\end{cases}\quad \mbox{in $Q$}
\end{equation}
and
\begin{equation}\label{(3.4)}
\partial_{\nu_A}y - py = \partial_tg + g_1, \quad
\partial_{\nu_B}z - qz = \partial_th + h_1 \quad \mbox{on $\partial\Omega \times (0,T)$}.
                                                \end{equation}
Here
\begin{equation*}
g_1:= \sum_{i,j=1}^n (\partial_ta_{ij})(\partial_ju)\nu_i - (\partial_tp)u, \quad h_1:= \sum_{i,j=1}^n (\partial_tb_{ij})(\partial_jv)\nu_i - (\partial_tq)m \quad
\mbox{on $\partial\Omega \times (0,T)$}.
\end{equation*}

In this step, we will prove

\begin{lem}[Carleman estimate for $(\partial_tu, \partial_tm)$]\label{p1}
Let all the assumptions of Theorem \ref{t2} except \eqref{lopukh5}
and \eqref{(1.5)} hold true.
There exist constants $s_0>0$ and $C>0$ such that
\begin{align*}
& \int_Q \biggl\{ \frac{1}{s\varphi}\left( | \partial_t^2u|^2
+ \sum_{i,j=1}^n | \partial_i\partial_j\partial_tu|^2\right)
+ s\varphi| \nabla \partial_tu|^2 + s^3\varphi^3| \partial_tu|^2 \\
&+ \frac{1}{s^2\varphi^2}\left( | \partial_t^2m|^2
+ \sum_{i,j=1}^n | \partial_i\partial_j\partial_tm|^2\right)
+ | \nabla \partial_tm|^2 + s^2\varphi^2| \partial_tm|^2
\biggr\} e^{2s\varphi} \,dx\,dt  \\
\le& C\int_Q \left(s\varphi| F|^2 + | \partial_tF|^2
+ | G|^2 + \frac{1}{s\varphi}| \partial_tG|^2\right)
e^{2s\varphi} \,dx\,dt\\
&+ \tilde{C}(s)( \Vert \partial_tg\Vert^2_* + \Vert \partial_th\Vert^2_*
+ \Vert g\Vert^2_* + \Vert h\Vert^2_*+  \Vert u\Vert_{H^1(\Gamma \times (0,T))}^2\nonumber\\
&
+ \Vert m\Vert_{H^1(\Gamma \times (0,T))}^2
+ \Vert \partial_tu\Vert_{H^1(\Gamma \times (0,T))}^2
+ \Vert \partial_tm\Vert_{H^1(\Gamma \times (0,T))}^2)
\end{align*}
for all large $s > s_0$ and $u,m \in H^{2,1}(Q)$ satisfying
$\partial_tu, \partial_tm \in H^{2,1}(Q)$, \eqref{(1.3)} and \eqref{(1.4)}.
\end{lem}

\begin{proof}[Proof of Lemma \ref{p1}]
Since $A(t)$ and $B(t)$ depend on $t$, the time derivatives of
$u, m$ produces the first derivatives of $u,m$ in the Robin boundary
conditions of $\partial_tu, \partial_tm$, so that the estimation of
$\nabla u$ and $\nabla m$ on $\partial\Omega \setminus \Gamma$, is indispensable.
Such estimation can be done also by the Carleman estimate.

By \eqref{lopukh2}, we note that $\partial_tc_0$ and all the coefficients of
$\partial_tA(t)$, $\partial_tA_0(t)$ and $\partial_tB(t)$ are in $L^{\infty}(Q)$.
Therefore, we can apply Lemma \ref{1} and \ref{2} with $m=-1$ to
the first and the second equations in \eqref{(3.3)} respectively, and we have
\begin{equation}\label{(3.5)}\begin{split}
 &\int_Q \left(\frac{1}{s\varphi}\left( | \partial_ty|^2
+ \sum_{i,j=1}^n | \partial_i\partial_jy|^2\right)
+ s\varphi| \nabla y|^2 + s^3\varphi^3| y|^2\right)
e^{2s\varphi} \,dx\,dt\\
\le& C\int_Q | z|^2 e^{2s\varphi} \,dx\,dt
+ C\int_Q \left( \sum_{| \gamma|\le 2} | \partial_x^{\gamma}u|^2
+ | m|^2\right) e^{2s\varphi} \,dx\,dt    + C\int_Q | \partial_tF|^2 e^{2s\varphi} \,dx\,dt + \tilde{C}(s)\Vert \partial_tg\Vert^2_*
\\
&+ C\int_{(\partial\Omega\setminus \Gamma)\times (0,T)}
\left( \frac{| \partial_tg_1|^2}{s^2\varphi^2}
+ \frac{| g_1|^2}{\root\of{s\varphi}}
\right) e^{2s\alpha} \,dS\,dt
+ C\Vert g_1 e^{s\alpha}\Vert^2_{L^2(0,T;H^{\frac{1}{2}}(\partial\Omega))}
+ \tilde{C}(s) \Vert y\Vert_{H^1(\Gamma \times (0,T))}^2 \end{split}
\end{equation}
and
\begin{equation}\label{(3.6)}\begin{split}
 &\int_Q \left(\frac{1}{s^2\varphi^2}\left( | \partial_tz|^2
+ \sum_{i,j=1}^n | \partial_i\partial_jz|^2\right)
+ | \nabla z|^2 + s^2\varphi^2| z|^2\right)
e^{2s\varphi} \,dx\,dt\\
\le &C\int_Q \frac{1}{s\varphi}\sum_{i,j=1}^n | \partial_i\partial_jy|^2 e^{2s\varphi} \,dx\,dt
+ C\int_Q \frac{1}{s\varphi} \sum_{| \gamma|\le 2}
(| \partial_x^{\gamma}u|^2 + | \partial_x^{\gamma} m|^2) e^{2s\varphi} \,dx\,dt
\\
&+  C\int_Q \frac{1}{s\varphi} | \partial_tG|^2 e^{2s\varphi} \,dx\,dt
+ \tilde{C}(s)\Vert \varphi^{-\frac{1}{2}} \partial_th\Vert^2_*            \\
&+  C\int_{(\partial\Omega\setminus \Gamma)\times (0,T)}
\left( \frac{| \partial_t(\varphi^{-\frac{1}{2}}h_1)|^2}{s^3\varphi^2}
+ \frac{s^{-1}| \varphi^{-\frac{1}{2}} h_1|^2}{\root\of{s\varphi}} \right)
e^{2s\alpha} \,dS\,dt            \\
&+ Cs^{-1}\Vert \varphi^{-\frac{1}{2}}h_1 e^{s\alpha}\Vert^2
_{L^2(0,T;H^{\frac{1}{2}}(\partial\Omega))}
+ \tilde{C}(s)\Vert z\Vert_{H^1(\Gamma \times (0,T))}^2 .\end{split}
\end{equation}

Applying \eqref{(3.5)} for estimating the first term on the right-hand side of
\eqref{(3.6)}, we have
\begin{align*}
& \int_Q \left(\frac{1}{s^2\varphi^2}\left( | \partial_tz|^2
+ \sum_{i,j=1}^n | \partial_i\partial_jz|^2\right)
+ | \nabla z|^2 + s^2\varphi^2| z|^2\right)
e^{2s\varphi} \,dx\,dt\\
\le& C\int_Q | z|^2 e^{2s\varphi} \,dx\,dt
+ C\int_Q \left( \sum_{| \gamma|\le 2}
| \partial_x^{\gamma}u|^2 + | m|^2\right) e^{2s\varphi} \,dx\,dt\\
&+  C\int_Q | \partial_tF|^2 e^{2s\varphi} \,dx\,dt
+ \tilde{C}(s)(\Vert \partial_tg\Vert^2_*
+ \Vert y\Vert_{H^1(\Gamma \times (0,T))}^2)          \\
&+  C\int_{(\partial\Omega\setminus \Gamma)\times (0,T)}
\left( \frac{| \partial_tg_1|^2}{s^2\varphi^2}
+ \frac{| g_1|^2}{\root\of{s\varphi}}
\right) e^{2s\alpha} \,dS\,dt
+ C\Vert g_1e^{s\alpha}\Vert^2_{L^2(0,T;H^{\frac{1}{2}}(\partial\Omega))}\\
&
+ \tilde{C}(s)\Vert y\Vert_{H^1(\Gamma \times (0,T))}^2 
+ C\int_Q \frac{1}{s\varphi} \sum_{| \gamma|\le 2}
(| \partial_x^{\gamma}u|^2 + | \partial_x^{\gamma} m|^2) e^{2s\varphi} \,dx\,dt\\
&+ C\int_Q \frac{1}{s\varphi} | \partial_tG|^2 e^{2s\varphi} \,dx\,dt
+ \tilde{C}(s) (\Vert \partial_th\Vert^2_*
+ \Vert z\Vert_{H^1(\Gamma \times (0,T))}^2)\\
&+  C\int_{(\partial\Omega\setminus \Gamma)\times (0,T)}
\left( \frac{| \partial_t(\varphi^{-\frac{1}{2}}h_1)|^2}{s^3\varphi^2}
+ \frac{s^{-1}| \varphi^{-\frac{1}{2}}h_1|^2}{\root\of{s\varphi}} \right)
e^{2s\alpha} \,dS\,dt\\
&+  Cs^{-1}\Vert \varphi^{-\frac{1}{2}}h_1 e^{s\alpha}\Vert^2
_{L^2(0,T;H^{\frac{1}{2}}(\partial\Omega))}
+ \tilde{C}(s)\Vert z\Vert_{H^1(\Gamma \times (0,T))}^2 .
\end{align*}
Absorbing the first term on the right-hand side into the left-hand side,
adding \eqref{(3.5)} and noting
$$
\frac{1}{s\varphi} \sum_{| \gamma| \le 2}| \partial_x^{\gamma}u|^2
\le C\sum_{| \gamma| \le 2}| \partial_x^{\gamma}u|^2
$$
in $Q$, we obtain
\begin{align*}
& \int_Q \left(
\frac{1}{s\varphi}\left( | \partial_ty|^2
+ \sum_{i,j=1}^n | \partial_i\partial_jy|^2\right)
+ s\varphi| \nabla y|^2 + s^3\varphi^3| y|^2\right) e^{2s\varphi} \,dx\,dt\\
&+  \int_Q \left(\frac{1}{s^2\varphi^2}\left( | \partial_tz|^2
+ \sum_{i,j=1}^n | \partial_i\partial_jz|^2\right)
+ | \nabla z|^2 + s^2\varphi^2| y|^2\right)e^{2s\varphi} \,dx\,dt\\
\le& C\int_Q \left( | \partial_tF|^2 + \frac{1}{s\varphi}| \partial_tG|^2
\right) e^{2s\varphi} \,dx\,dt \\
&+ C\int_Q \left( \sum_{| \gamma|\le 2}
| \partial_x^{\gamma}u|^2 + \frac{1}{s\varphi} \sum_{| \gamma| \le 2}
| \partial_x^{\gamma}m|^2 + | m|^2\right) e^{2s\varphi} \,dx\,dt\\
&+ \tilde{C}(s)(\Vert \partial_tg\Vert^2_* + \Vert \partial_th\Vert^2_*
+ \Vert y\Vert_{H^1(\Gamma \times (0,T))}^2
+ \Vert z\Vert_{H^1(\Gamma \times (0,T))}^2 ) \\
&+ C\biggl\{ \int_{(\partial\Omega\setminus \Gamma)\times (0,T)}
\left( \frac{| \partial_tg_1|^2}{s^2\varphi^2}
+ \frac{| g_1|^2}{\root\of{s\varphi}}\right)
e^{2s\alpha} \,dS\,dt
+ C\Vert g_1e^{s\alpha}\Vert^2_{L^2(0,T;H^{\frac{1}{2}}(\partial\Omega))}\\
&+ C\int_{(\partial\Omega\setminus \Gamma)\times (0,T)}
\left( \frac{| \partial_t(\varphi^{-\frac{1}{2}}h_1)|^2}{s^3\varphi^2}
+ \frac{s^{-1}| \varphi^{-\frac{1}{2}}h_1|^2}{\root\of{s\varphi}}\right)
e^{2s\alpha} \,dS\,dt\\
&
+ Cs^{-1} \Vert \varphi^{-\frac{1}{2}}h_1 e^{s\alpha}\Vert^2_{L^2(0,T;
H^{\frac{1}{2}}(\partial\Omega))}\biggr\}.
\end{align*}
Here again we absorb the term $C\int_Q | z|^2
e^{2s\varphi} \,dx\,dt$ on the right-hand side, which results from \eqref{(3.5)},
into the left-hand side

Applying Theorem \ref{t3} to the second term on the right-hand side, we reach
\begin{equation}\label{(3.7)}
\begin{split}
 &\int_Q \left(
\frac{1}{s\varphi}\left( | \partial_ty|^2
+ \sum_{i,j=1}^n | \partial_i\partial_jy|^2\right)
+ s\varphi| \nabla y|^2 + s^3\varphi^3| y|^2\right) e^{2s\varphi} \,dx\,dt\\
&+  \int_Q \left(\frac{1}{s^2\varphi^2}\left( | \partial_tz|^2
+ \sum_{i,j=1}^n | \partial_i\partial_jz|^2\right)
+ | \nabla z|^2 + s^2\varphi^2| y|^2\right)e^{2s\varphi} \,dx\,dt\\
\le& C\int_Q \left( s\varphi| F|^2 + | \partial_tF|^2 + | G|^2
+ \frac{1}{s\varphi}| \partial_tG|^2 \right) e^{2s\varphi} \,dx\,dt \\
&+ \tilde{C}(s)(\Vert \partial_tg\Vert^2_* + \Vert \partial_th\Vert^2_*
 + \Vert g\Vert^2_* + \Vert h\Vert^2_*
+ \Vert u\Vert_{H^1(\Gamma \times (0,T))}^2
+ \Vert m\Vert_{H^1(\Gamma \times (0,T))}^2\\
&+ \Vert y\Vert_{H^1(\Gamma \times (0,T))}^2
+ \Vert z\Vert_{H^1(\Gamma \times (0,T))}^2)\\
&+ \biggl\{ C\int_{(\partial\Omega\setminus \Gamma)\times (0,T)}
\left( \frac{| \partial_tg_1|^2}{s^2\varphi^2}
+ \frac{| g_1|^2}{\root\of{s\varphi}}\right)
e^{2s\alpha} \,dS\,dt
+ C\Vert g_1e^{s\alpha}\Vert^2_{L^2(0,T;H^{\frac{1}{2}}(\partial\Omega))}\biggr\}
\\
&+  C\int_{(\partial\Omega\setminus \Gamma)\times (0,T)}
\left( \frac{| \partial_t(\varphi^{-\frac{1}{2}}h_1)|^2}{s^3\varphi^2}
+ \frac{s^{-1}| \varphi^{-\frac{1}{2}}h_1|^2}{\root\of{s\varphi}} \right)
e^{2s\alpha} \,dS\,dt
\\
&+ Cs^{-1} \Vert \varphi^{-\frac{1}{2}}h_1e^{s\alpha}\Vert^2_{L^2(0,T;
H^{\frac{1}{2}}(\partial\Omega))}\biggr\}.
\end{split}
\end{equation}
Here we set
\begin{equation}\label{(3.8)}\begin{split}
I:= &\int_{(\partial\Omega\setminus \Gamma)\times (0,T)}
\left( \frac{| \partial_tg_1|^2}{s^2\varphi^2}
+ \frac{| g_1|^2}{\root\of{s\varphi}}\right)
e^{2s\alpha} \,dS\,dt\\
&+  \int_{(\partial\Omega\setminus \Gamma)\times (0,T)}
\left( \frac{| \partial_t(\varphi^{-\frac{1}{2}}h_1)|^2}{s^3\varphi^2}
+ \frac{s^{-1}| \varphi^{-\frac{1}{2}}h_1|^2}{\root\of{s\varphi}} \right)
e^{2s\alpha} \,dS\,dt\\
&+  \Vert g_1e^{s\alpha}\Vert^2_{L^2(0,T;H^{\frac{1}{2}}(\partial\Omega))}
+ s^{-1} \Vert \varphi^{-\frac{1}{2}}h_1e^{s\alpha}\Vert^2_{L^2(0,T;
H^{\frac{1}{2}}(\partial\Omega))}
\\
=:& I_1 + I_2 + I_3 + I_4.\end{split}
                                         \end{equation}
Now we estimate $I_1, I_2, I_3, I_4$ separately.

{\bf Estimation of $I_1$}

We can represent
$$
g_1 = g_{11}(x,t)\cdot \nabla u + g_{10}(x,t)u \quad
\mbox{on $\partial\Omega\times (0,T)$},
$$
where $g_{11}, g_{10}$ can be extended to functions in
$C^1(\bar{Q})$ by \eqref{lopukh2}.

Therefore, \eqref{(3.1)} implies
$$
| I_1| = \left| \int_{\partial\Omega\times (0,T)}
\left( \frac{1}{s^2\varphi^2}| \partial_tg_1|^2
+ \frac{| g_1|^2}{\root\of{s\varphi}}\right)
 \,dx\,dt\right|.
$$
Since
$$
\partial_tg_1 = g_{11}\cdot \nabla y + g_{10}y + (\partial_tg_{11})\cdot \nabla u
+ (\partial_tg_{10})u \quad \mbox{on $\partial\Omega\times (0,T)$},
$$
we see
$$
| g_1|^2 \le C(| \nabla u|^2 + | u|^2)
$$
and
$$
| \partial_tg_1|^2 \le C(| \nabla y|^2 + | y|^2
+ | \nabla u|^2 + | u|^2)
\quad \mbox{on $\partial\Omega\times (0,T)$}.
$$
Hence,
$$
| I_1| \le C\int_{\partial\Omega\times (0,T)}
\left( \frac{1}{s^2\varphi^2}(| \nabla y|^2 + | y|^2)
+ \frac{1}{\root\of{s\varphi}}(| \nabla u|^2
+ | u|^2) \right) \,dS\,dt.
$$
Consequently \eqref{(3.1)} and \eqref{(3.2)} imply
\begin{equation}\label{(3.9)}
| I_1|
\le C\int_Q \biggl( \frac{1}{s^2\varphi^2}\sum_{|\gamma|\le 2}
| \partial_x^{\gamma}y|^2
+ \sum_{|\gamma|\le 2} | \partial_x^{\gamma}u|^2
+ | \nabla y|^2 + | y|^2 + s^2\varphi^2(| \nabla u|^2
+ | u|^2)\biggr) e^{2s\varphi} \,dx\,dt.    
\end{equation}

{\bf Estimation of $I_2$}

In view of \eqref{(2.10)}, we have
$$
| \partial_t(\varphi^{-\frac{1}{2}}h_1)| = | \varphi^{-\frac{1}{2}}\partial_th_1
 - \frac{1}{2}\varphi^{-\frac{3}{2}}(\partial_t\varphi)h_1|
\le C(\varphi^{-\frac{1}{2}}| \partial_th_1| + \varphi^{\frac{1}{2}}| h_1|),
$$
so that
$$
| I_2| \le C\int_{\partial\Omega\times (0,T)}
\left( \frac{1}{s^3\varphi^3}| \partial_th_1|^2
+ \frac{1}{(s\varphi)^\frac 32}| h_1|^2 \right) e^{2s\varphi} \,dS\,dt.
$$
Since we can represent $h_1 = h_{11}\cdot \nabla m + h_{10}m$,
it follows that
$$
| h_1|^2 \le C(| \nabla m|^2
+ | m|^2)
$$
and
$$
| \partial_th_1|^2
\le C(| \nabla z|^2 + | z|^2 + | \nabla m|^2
+ | m|^2) \quad \mbox{on $\partial\Omega\times (0,T)$}.
$$
Hence, \eqref{(3.1)} and \eqref{(3.2)} imply
\begin{equation}\label{(3.10)}
\begin{split}
| I_2| \le& C\int_{\partial\Omega\times (0,T)}
\left( \frac{1}{s^3\varphi^3}(| \nabla z|^2 + | z|^2)
+ \frac{1}{s\varphi}(| \nabla m|^2 + | m|^2)\right)
e^{2s\varphi} \,dS\,dt \\
\le& C\int_Q \biggl( \frac{1}{s^3\varphi^3}
\sum_{|\gamma|\le 2} | \partial_x^{\gamma}z|^2
+ \frac{1}{s\varphi} \sum_{| \gamma| \le 2} | \partial_x^{\gamma}m|^2
+ \frac{1}{s\varphi}(| \nabla z|^2 + | z|^2)
+ s\varphi(| \nabla m|^2 + | m|^2)\biggr) e^{2s\varphi} \,dx\,dt.
 \end{split}                                               
\end{equation}

{\bf Estimation of $I_3$}

By noting that $g_1= -\sum_{i,j=1}^n (\partial_ta_{ij})(\partial_iu)\nu_j
+ (\partial_tp)u$ on $\partial\Omega \times (0,T)$, the trace theorem yields
\begin{equation}\label{(3.11)}\begin{split}
| I_3| &= \Vert g_1e^{s\alpha}\Vert^2_{L^2(0,T;H^{\frac{1}{2}}(\partial\Omega))}
\\&\le C\Vert g_1e^{s\alpha}\Vert^2_{L^2(0,T;H^1(\Omega))}\\
&= C\int_Q (| \nabla(g_1e^{s\alpha})|^2 + | g_1e^{s\alpha}|^2)
\,dx\,dt
\\&\le C\int_Q (| \nabla g_1|^2 + s^2\varphi^2| g_1|^2)e^{2s\varphi} \,dx\,dt
\\
&\le C\int_Q \left( \sum_{i,j=1}^n | \partial_i\partial_j u|^2
+ s^2\varphi^2(| \nabla u|^2 + | u|^2) \right)
e^{2s\varphi} \,dx\,dt.
\end{split}
\end{equation}

{\bf Estimation of $I_4$}

The trace theorem yields
\begin{align*}
| I_4|
&= s^{-1}\Vert \varphi^{-\frac{1}{2}}h_1e^{s\alpha}\Vert^2
_{L^2(0,T;H^{\frac{1}{2}}(\partial\Omega))}\\
&\le Cs^{-1}\Vert \varphi^{-\frac{1}{2}}h_1e^{s\alpha}\Vert^2_{L^2(0,T;H^1(\Omega))}\\
&\le Cs^{-1}(\Vert \varphi^{-\frac{1}{2}}h_1e^{s\alpha}\Vert^2_{L^2(Q)}
+ \Vert \nabla(\varphi^{-\frac{1}{2}}h_1e^{s\alpha})\Vert^2_{L^2(Q)})\\
&\le Cs^{-1}\biggl(\Vert \varphi^{-\frac{1}{2}}h_1e^{s\alpha}\Vert^2_{L^2(Q)}
+ \left\Vert \frac{1}{2}\varphi^{-\frac{3}{2}} (\nabla \varphi)h_1e^{s\alpha}
\right\Vert^2_{L^2(Q)} \\
&\quad+  \Vert s\varphi^{-\frac{1}{2}}h_1(\nabla\varphi)e^{s\alpha}\Vert^2_{L^2(Q)}
+ \Vert \varphi^{-\frac{1}{2}}(\nabla h_1)e^{s\alpha}\Vert^2_{L^2(Q)}\biggr)\\
&\le C\int_Q \left(s\varphi| h_1|^2 + \frac{| \nabla h_1|^2}
{s\varphi}\right) e^{2s\varphi} \,dx\,dt.
\end{align*}
Since $| h_1| \le C(| \nabla m| + | m|)$ and
$| \nabla h_1|
\le C\sum_{|\gamma|\le 2} | \partial_x^{\gamma}m|^2$
in $Q$, we have
\begin{equation}\label{(3.12)}
| I_4|
\le C\int_Q \left( s\varphi(| \nabla m|^2 + | m|^2)
+ \frac{1}{s\varphi}\sum_{|\gamma|\le 2} | \partial_x^{\gamma}m|^2
\right) e^{2s\varphi} \,dx\,dt.                        \end{equation}
We rewrite inequality \eqref{(2.17)} as
\begin{equation}\label{(3.13)}\begin{split}
 &\int_Q \biggl\{ | \partial_tu|^2 + \sum_{i,j=1}^n | \partial_i\partial_ju|^2
+ s^2\varphi^2| \nabla u|^2 + s^4\varphi^4| u|^2
\\
&+ \frac{1}{s\varphi}\left(
| \partial_tm|^2 + \sum_{i,j=1}^n | \partial_i\partial_jv|^2\right)
+ s\varphi| \nabla m|^2 + s^3\varphi^3| m|^2 \biggr\}
e^{2s\varphi} \,dx\,dt \le J,  \end{split}
\end{equation}
where we set
$$
 J := C\int_Q (s\varphi| F|^2 + | G|^2) e^{2s\varphi} \,dx\,dt
+ \tilde{C}(s)(\Vert g\Vert_*^2 + \Vert h\Vert^2_*
+ \Vert u\Vert_{H^1(\Gamma \times (0,T))}^2
+ \Vert m\Vert_{H^1(\Gamma \times (0,T))}^2 ).
$$
Applying \eqref{(3.13)} to \eqref{(3.9)} -- \eqref{(3.12)}, we have
\begin{equation}\label{(3.14)}
\begin{cases}
| I_1| \le J
+ C\int_Q \left( \frac{1}{s^2\varphi^2}\sum_{| \gamma|\le 2}
| \partial_x^{\gamma}y|^2 + | \nabla y|^2 + | y|^2
\right) e^{2s\varphi} \,dx\,dt, \\
| I_2| \le J
+ C\int_Q \left( \frac{1}{s^3\varphi^3}\sum_{| \gamma|\le 2}
| \partial_x^{\gamma}z|^2
+ \frac{1}{s\varphi}(| \nabla z|^2 + | z|^2)
\right) e^{2s\varphi} \,dx\,dt,\\
| I_3|\le J\quad \mbox{and} \quad | I_4| \le J.
\end{cases}
\end{equation}

Applying \eqref{(3.14)} to \eqref{(3.7)}, we reach
\begin{align*}
& \int_Q \left(
\frac{1}{s\varphi}\left( | \partial_ty|^2
+ \sum_{i,j=1}^n | \partial_i\partial_jy|^2\right)
+ s\varphi| \nabla y|^2 + s^3\varphi^3| y|^2\right) e^{2s\varphi} \,dx\,dt\\
&+  \int_Q \left(\frac{1}{s^2\varphi^2}\left( | \partial_tz|^2
+ \sum_{i,j=1}^n | \partial_i\partial_jz|^2\right)
+ | \nabla z|^2 + s^2\varphi^2| y|^2\right)e^{2s\varphi} \,dx\,dt\\
\le& C\int_Q \left( s\varphi| F|^2 + | \partial_tF|^2 + | G|^2
+ \frac{1}{s\varphi}| \partial_tG|^2 \right) e^{2s\varphi} \,dx\,dt \\
&+  \tilde{C}(s)(\Vert \partial_tg\Vert^2_* + \Vert g\Vert^2_*
+ \Vert \partial_th\Vert^2_* + \Vert h\Vert^2_*+\Vert u\Vert_{H^1(\Gamma \times (0,T))}^2+ \Vert m\Vert_{H^1(\Gamma \times (0,T))}^2   \nonumber\\
&+
 \Vert y\Vert_{H^1(\Gamma \times (0,T))}^2
+ \Vert z\Vert_{H^1(\Gamma \times (0,T))}^2) + J+ \int_Q \left( \frac{1}{s^2\varphi^2}\sum_{| \gamma|\le 2}
| \partial_x^{\gamma}y|^2 + | \nabla y|^2 + | y|^2\right.\\
&+ \left.
 \frac{1}{s^3\varphi^3}\sum_{| \gamma|\le 2}
| \partial_x^{\gamma}z|^2 + \frac{1}{s\varphi}(| \nabla z|^2
+ | z|^2)\right) e^{2s\varphi} \,dx\,dt
\end{align*}
for all large $s>0$.
Choosing $s>0$ large, we can absorb the final term on the right-hand side into
the left-hand side, we complete the proof of Lemma \ref{p1}.
\end{proof}

{\bf Third Step: Completion of the proof}

Since $F(x,t)= q_1(x,t)f_1(x)$ and $G(x,t) = q_2(x,t)f_2(x,t)$,
by Lemma \ref{p1}, we have
\begin{equation}\label{(3.15)}\begin{split}
&\int_Q \left(
\frac{1}{s\varphi}| \partial_t^2u|^2 + s^3\varphi^3| \partial_tu|^2
+ \frac{1}{s^2\varphi^2} | \partial_t^2m|^2
+ s^2\varphi^2 | \partial_tm|^2\right) e^{2s\varphi} \,dx\,dt
\\
&\le C\int_Q (s\varphi | f_1|^2 + | f_2|^2) e^{2s\varphi} \,dx\,dt
+ \tilde{C}(s)J_0   \end{split}
\end{equation}
for all large $s>0$.  Here we set
$$
J_0:= \sum_{k=0}^1 (\Vert \partial_t^kg\Vert^2_*
+ \Vert \partial_t^kh\Vert^2_* +\Vert \partial_t^ku\Vert_{H^1(\Gamma \times (0,T))}^2
+ \Vert \partial_t^kv\Vert_{H^1(\Gamma \times (0,T))}^2 ).
$$

We note \eqref{(2.10)} and
$\lim_{t\to 0} \alpha(x,t) = -\infty$ for $x\in \Omega$.
Then,
\begin{align*}
& \int_{\Omega} s\varphi(x,t_0) | \partial_tu(x,t_0)|^2
e^{2s\alpha(x,t_0)} \,dx
\\=& \int^{t_0}_0 \partial_t\left( \int_{\Omega} s\varphi| \partial_tu(x,t)|^2
e^{2s\alpha(x,t)} \,dx \right) \,dt \\
=& \int^{t_0}_0 \int_{\Omega} (2s\varphi (\partial_tu)(\partial_t^2u)
+ s(\partial_t\varphi)| \partial_tu|^2
+ 2s^2\varphi (\partial_t\alpha)| \partial_tu|^2) e^{2s\varphi} \,dx\,dt\\
\le& C\int_Q (s\varphi| \partial_tu| | \partial_t^2u|
+ s\varphi^2 | \partial_tu|^2 + s^2\varphi^3| \partial_tu|^2) e^{2s\varphi} \,dx\,dt \\
= & C\int_Q \{ (s^{-\frac{1}{2}}\varphi^{-\frac{1}{2}}| \partial_t^2u|)
(s^{\frac{3}{2}}\varphi^{\frac{3}{2}} | \partial_tu|)
+ (s\varphi^2 + s^2\varphi^3) | \partial_tu|^2\} e^{2s\varphi} \,dx\,dt \\
\le & C\int_Q \left(\frac{1}{s\varphi}| \partial_t^2u|^2
+ (s^3\varphi^3 + s\varphi^2 + s^2\varphi^3)| \partial_tu|^2\right) e^{2s\varphi} \,dx\,dt\\
\le & C\int_Q \left(\frac{1}{s\varphi}| \partial_t^2u|^2
+ s^3\varphi^3| \partial_tu|^2\right) e^{2s\varphi} \,dx\,dt.
\end{align*}
For the last inequality, we used $s\varphi^2 + s^2\varphi^3 \le Cs^3\varphi^3$ for
large $s>0$.
Hence, by \eqref{(3.15)}, we have
\begin{equation}\label{(3.16)}
\int_{\Omega} s\varphi(x,t_0) | \partial_tu(x,t_0)|^2
e^{2s\alpha(x,t_0)} \,dx 
\le C\int_Q \left(s\varphi | f_1|^2 + | f_2|^2\right)
e^{2s\varphi} \,dx\,dt + \tilde{C}(s)J_0
\end{equation}
for all large $s>0$.

Next, by \eqref{(3.15)}, we can similarly estimate
\begin{equation}\label{(3.17)}\begin{split}
& \quad\int_{\Omega} | \partial_tm(x,t_0)|^2
e^{2s\alpha(x,t_0)} \,dx\\
&= \int^{t_0}_0 \partial_t\left( \int_{\Omega} | \partial_tm(x,t)|^2
e^{2s\alpha(x,t)} \,dx \right) \,dt \\
&= \int^{t_0}_0 \int_{\Omega} (2 (\partial_tm)(\partial_t^2m)
+ 2s(\partial_t\alpha)| \partial_tm|^2) e^{2s\varphi} \,dx\,dt\\
&\le C\int_Q (| \partial_tm| | \partial_t^2m|
+ s\varphi^2 | \partial_tm|^2) e^{2s\varphi} \,dx\,dt \\
&= C\int_Q \left\{
\left( \frac{1}{s\varphi} | \partial_t^2m|\right) (s\varphi | \partial_tm|)
+ s\varphi^2 | \partial_tm|^2 \right\} e^{2s\varphi} \,dx\,dt             \\
&\le  C\int_Q \left( \frac{1}{s^2\varphi^2} | \partial_t^2m|^2
+ s^2\varphi^2| \partial_tm|^2 + s\varphi^2| \partial_tm|^2
\right) e^{2s\varphi} \,dx\,dt
\\
&\le  C\int_Q (s\varphi | f_1|^2 + | f_2|^2) e^{2s\varphi} \,dx\,dt
+ \tilde{C}(s)J_0
\end{split}
\end{equation}
for all large $s>0$.

On the other hand, assumption \eqref{(1.5)} yields
$$
\begin{cases}
f_1(x) = \frac{1}{q_1(x,t_0)}\partial_tu(x,t_0)
+ \frac{1}{q_1(x,t_0)}(A(t_0)u_0 - c_0(x,t_0)m_0),\\
f_2(x) = \frac{1}{q_2(x,t_0)}\partial_tm(x,t_0)
+ \frac{1}{q_2(x,t_0)}(-B(t_0)m_0 - A_0(t_0)u_0), 
\end{cases}\quad x\in \Omega
$$
and
$$
\begin{cases}
s\varphi(x,t_0)| f_1(x)|^2
\le Cs\varphi(x,t_0)| \partial_tu(x,t_0)|^2
+Cs\varphi(x,t_0) \left( \sum_{| \gamma|\le 2}
| \partial_x^{\gamma}u_0(x)|^2 + | m_0(x)|^2\right),\\
| f_2(x)|^2
\le C| \partial_tm(x,t_0)|^2
+ C\sum_{| \gamma|\le 2}(| \partial_x^{\gamma}u_0(x)|^2
+ | \partial_x^{\gamma} m_0(x)|^2), 
\end{cases}\quad \text{in } \Omega.
$$
Hence \eqref{(3.16)} and \eqref{(3.17)} yield
\begin{equation}\label{(3.18)}\begin{split}
 &\int_{\Omega} (s\varphi(x_0,t)| f_1(x)|^2 + | f_2(x)|^2)
e^{2s\alpha(x,t_0)} \,dx\\
\le& C\int_{\Omega} (s\varphi(x,t_0)| \partial_tu(x,t_0)|^2
+ | \partial_tm(x,t_0)|^2) e^{2s\varphi} \,dx\,dt+ \tilde{C}(s)(\Vert u_0\Vert^2_{H^2(\Omega)} + \Vert m_0\Vert^2_{H^2(\Omega)})
\\
\le& C\int_Q \left( s\varphi| f_1|^2 + | f_2|^2\right)
e^{2s\varphi} \,dx\,dt + \tilde{C}(s)J_1
\end{split}
\end{equation}
for all large $s>0$.
Here and henceforth we set
$$
J_1:= J_0 + \Vert u_0\Vert^2_{H^2(\Omega)}
  + \Vert m_0\Vert^2_{H^2(\Omega)}.
$$
On the other hand, we see
\begin{align*}
\varphi(x,t)| f_1(x)|^2 e^{2s\alpha(x,t)}
&= \varphi(x,t_0)| f_1(x)|^2 e^{2s\alpha(x,t_0)}
\times \frac{\varphi(x,t)}{\varphi(x,t_0)}e^{2s(\alpha(x,t)-\alpha(x,t_0))}\\
&= \varphi(x,t_0)| f_1(x)|^2 e^{2s\alpha(x,t_0)}
\times \frac{\mu(t_0)}{\mu(t)}
e^{-2s \xi(x)\left( \frac{1}{\mu(t)} - \frac{1}{\mu(t_0)}\right)}.
\end{align*}
Here and henceforth we set
$$
\xi(x):= e^{2\lambda\Vert {\eta}\Vert_{C(\overline{\Omega})}}
- e^{\lambda{\eta}(x)}> 0, \quad
C_1:= e^{2\lambda\Vert {\eta}\Vert_{C(\overline{\Omega})}}
- e^{\lambda\Vert {\eta}\Vert_{C(\overline{\Omega})}}
= \min_{x\in \overline{\Omega}}\xi(x) > 0.
$$
Therefore,
\begin{align*}
& \int_Q \varphi(x,t)| f_1(x)|^2 e^{2s\alpha(x,t)} \,dx\,dt\\
\le& C\int_Q \varphi(x,t_0)| f_1(x)|^2 e^{2s\alpha(x,t_0)}
\frac{1}{\mu(t)}
e^{-2s\xi(x)\left( \frac{1}{\mu(t)} - \frac{1}{\mu(t_0)}\right)} \,dx\,dt\\
\le& C\int_{\Omega} \varphi(x,t_0)| f_1(x)|^2 e^{2s\alpha(x,t_0)}
\left( \int^T_0 \frac{1}{\mu(t)}
e^{-2sC_1\left( \frac{1}{\mu(t)} - \frac{1}{\mu(t_0)}\right)}\,dt
\right) \,dx.
\end{align*}
We will estimate
$\int^T_0 \frac{1}{\mu(t)}
e^{-2sC_1\left( \frac{1}{\mu(t)} - \frac{1}{\mu(t_0)}\right)}\,dt$.
Indeed
$$
\lim_{s\to \infty}
\frac{1}{\mu(t)}
e^{-2sC_1\left( \frac{1}{\mu(t)} - \frac{1}{\mu(t_0)}\right)} = 0
$$
for each fixed $t \in (0,T) \setminus \{t_0\}$.  Next, since
$$
\frac{1}{\mu(t)}e^{-2sC_1\left( \frac{1}{\mu(t)} - \frac{1}{\mu(t_0)}\right)}
\le \frac{1}{\mu(t)}
e^{-2C_1\left( \frac{1}{\mu(t)} - \frac{1}{\mu(t_0)}\right)}
$$
for $s \ge 1$ and by $\mu(t) \le \mu(t_0)$, we have
\begin{align*}
& \sup_{s\ge 1}\sup_{0< t< T} \frac{1}{\mu(t)}
\exp\left( -2sC_1\left( \frac{1}{\mu(t)} - \frac{1}{\mu(t_0)} \right)\right)
\le \sup_{0<t<T} \left( \frac{1}{\mu(t)}
e^{-\frac{2C_1}{\mu(t)}}\right) e^{\frac{2C_1}{\mu(t_0)}}\\
\le &\sup_{\zeta>0} (\zeta e^{-2C_1\zeta}) e^{\frac{2C_1}{\mu(t_0)}} < \infty.
\end{align*}
Consequently, the Lebesgue convergence theorem yields
$$
\int^T_0 \frac{1}{\mu(t)}
e^{-2sC_1\left( \frac{1}{\mu(t)} - \frac{1}{\mu(t_0)}\right)} \,dt
= o(1) \quad \mbox{as $s \to \infty$}.
$$
Hence,
$$
\int_Q \varphi(x,t)| f_1(x)|^2 e^{2s\alpha(x,t)} \,dx\,dt
= o(1)\int_{\Omega} \varphi(x,t_0)| f_1(x)|^2 e^{2s\alpha(x,t_0)} \,dx
$$
as $s \to \infty$.

Similarly we can verify
$$
\int_Q | f_2(x)|^2 e^{2s\alpha(x,t)} \,dx\,dt
= o(1)\int_{\Omega} | f_2(x)|^2 e^{2s\alpha(x,t_0)} \,dx
$$
as $s \to \infty$.
Therefore, \eqref{(3.18)} yields
\begin{align*}
& \int_{\Omega} (s\varphi(x,t_0)| f_1(x)|^2 + | f_2(x)|^2)
e^{2s\alpha(x,t_0)} \,dx\\
= &o(1)\int_{\Omega} (s\varphi(x,t_0)| f_1(x)|^2 + | f_2(x)|^2)
e^{2s\alpha(x,t_0)} \,dx
+ \tilde{C}(s)J_1
\end{align*}
for all large $s>0$.  Choosing $s>0$ large, we can absorb the first term on
the right-hand side into the left-hand side to obtain
$$
\int_{\Omega} (s\varphi(x,t_0)| f_1(x)|^2 + | f_2(x)|^2)
e^{2s\alpha(x,t_0)} \,dx
\le \tilde{C}(s)J_1.
$$
With such fixed $s>0$, we have
$\min_{x\in \overline{\Omega}} s\varphi(x,t_0)e^{2s\alpha(x,t_0)} > 0$, and so
we complete the proof of Theorem \ref{t2}.
$\hfill\square$

\section{Second Stability Result}

\label{sec:2}

Let $\Omega \subset \mathbb{R}^{n}$ be a bounded domain with the
piecewise smooth boundary $\partial \Omega $. The second MFG form we consider possess a convolution term as follows:
\begin{equation}\label{MFGCarleman2}
\begin{cases}
\partial_t u(x,t)+\Delta u(x,t)-\frac{1}{2}\kappa(x)|\nabla u(x,t)|^2
+F\left( \int_{\Omega }M(x,y) m(y,t)
\,dy,m(x,t) \right) \\\qquad\qquad\qquad\qquad\qquad\qquad\qquad\qquad\qquad\qquad\qquad\qquad\qquad=G_1(x,t) &\text{in } \Omega \times (0,T),\\
\partial_t m(x,t)-\Delta m(x,t){-div(\kappa(x)m(x,t)\nabla u(x,t))}
=G_2(x,t) &\text{in } \Omega \times (0,T), \\ 
\partial_{\nu }u=\partial_{\nu }m=0&\text{on }\Sigma,  \end{cases}
\end{equation}
where $G_1,G_2\in
L^2\left( \Omega \times (0,T)\right)$, the function $F\left( y,z\right) :\mathbb{R}
^2\rightarrow \mathbb{R}$ has derivatives $F_{y},F_{z}\in C\left( 
\mathbb{R}^2\right) $ such that 
\begin{equation}
\max \left( \sup_{\mathbb{R}^2}\left| F_{y}\left( y,z\right)
\right| ,\sup_{\mathbb{R}^2}\left| F_{y}\left( y,z\right)
\right| \right) \leq D_1,  \label{4.1}
\end{equation}
and $M(x,y) $, $r\in C^{1}\left( \overline{\Omega }\right) $ be such that 
\begin{equation}
\sup_{\Omega \times \Omega }\left| M(x,y) \right|
,\left\Vert r\right\Vert_{C^{1}\left( \overline{\Omega }\right) }\leq D_2,
\label{4.2}
\end{equation}
for constants $D_1,D_2>0$.

For any number $\varepsilon \in \left( 0,T\right) $ define the
domain $Q_{\varepsilon ,T}$ as:
\begin{equation}
Q_{\varepsilon ,T}=\Omega \times \left( \varepsilon ,T\right) \subset \Omega \times (0,T).
\label{4.0}
\end{equation}

\begin{thm} \label{Carleman21Thm}
Let $D_3,D_4>0$. 
Define sets of functions $K_3(D_3)$ and $
K_4(D_4)$ as
\begin{equation}
K_3(D_3)=\left\{ u\in H_0^2\left( \Omega \times (0,T)\right)
:\sup_{\Omega \times (0,T)}\left| u\right| ,\sup_{\Omega \times (0,T)}\left| \nabla
u\right| ,\sup_{\Omega \times (0,T)}\left| \Delta u\right| \leq D_3\right\} ,
\label{4.3}
\end{equation}
\begin{equation}
K_4(D_4)=\left\{ u\in H_0^2\left( \Omega \times (0,T)\right)
:\sup_{\Omega \times (0,T)}\left| u\right| ,\sup_{\Omega \times (0,T)}\left| \nabla
u\right| \leq D_4\right\} .  \label{4.4}
\end{equation}
Let
\begin{equation}
D=\max \left( D_1,D_2,D_3,D_4\right) .  \label{4.5}
\end{equation}
Assume that two pairs of functions $\left( u_1,m_1\right)$ and $\left( u_2,m_2\right)$ satisfy equations \eqref{MFGCarleman2} with two pairs of functions $\left(G_{1,1},G_{2,1}\right)$ and $\left(G_{1,2},G_{2,2}\right)$ respectively and are such that 
\begin{equation}
(u_1,m_1),(u_2,m_2)\in K_3\left(
D_3\right) \times K_4(D_4).  \label{4.6}
\end{equation}
Assume that these two pairs of functions $\left( u_1,m_1\right)$ 
and $(u_2,m_2)$ have the following terminal
conditions 
\begin{equation}
u_1(x,T)=u_T^{(1)}(x),\quad
m_1(x,T)=m_T^{(1)}(x),\quad
\text{in }\Omega ,  \label{4.7}
\end{equation}
\begin{equation}
u_2(x,T)=u_T^{(2)}(x), \quad m_2(x,T)=m_T^{(2)}(x), \quad
\text{in }\Omega .  \label{4.8}
\end{equation}
Then there exists a number $C_4=C_4\left(D,T,\Omega ,\varepsilon \right)>0 $ and a sufficiently small number $\delta_0=$ $\delta_0\left(D,T,\Omega ,\varepsilon \right) \in \left( 0,1\right) $ depending only on listed parameters such that if $\delta \in \left(
0,\delta_0\right) $ and 
\begin{equation}
\left\Vert u_T^{(1)}-u_T^{(2)}\right\Vert
_{H^{1}\left( \Omega \right) },\quad\left\Vert m_T^{\left( 1\right)
}-m_T^{(2)}\right\Vert_{L^2\left( \Omega \right) }\leq
\delta ,  \label{4.9}
\end{equation}
\begin{equation}
\left\Vert G_{1,1}-G_{1,2}\right\Vert_{L^2\left( \Omega \times (0,T)\right)
},\quad\left\Vert G_{2,1}-G_{2,2}\right\Vert_{L^2\left( \Omega \times (0,T)\right) }\leq
\delta ,  \label{4.90}
\end{equation}
then there exists a number $\rho =\rho \left(D,T,\Omega
,\varepsilon \right) \in \left( 0,1/6\right) $ depending only on
listed parameters such that the following two H\"{o}lder stability estimates
are valid:
\begin{multline}
\left\Vert \partial_Tu_1-\partial_Tu_2\right\Vert_{L^2\left(
Q_{\varepsilon ,T}\right) }+\left\Vert \Delta u_1-\Delta u_2\right\Vert
_{L^2\left( Q_{\varepsilon ,T}\right) }+\left\Vert u_1-u_2\right\Vert_{H^{1,0}\left( Q_{\varepsilon ,T}\right) }\\ 
\leq C_4\left( 1+\left\Vert m_1-m_2\right\Vert_{H^2\left(
\Omega \times (0,T)\right) }\right) \delta ^{\rho },\quad \forall \delta \in \left(
0,\delta_0\right) ,
 \label{4.10}
\end{multline}
\begin{equation}
\left\Vert m_1-m_2\right\Vert_{H^{1,0}\left( Q_{\varepsilon,T}\right) }\leq C_4\left( 1+\left\Vert m_1-m_2\right\Vert
_{H^2\left( \Omega \times (0,T)\right) }\right) \delta ^{\rho }.  \label{4.11}
\end{equation}
\end{thm}

To prove this theorem, we first state two Carleman estimates. Introduce three parameters $b>0,\lambda >0$ and $k>2$. Also, introduce our
first Carleman Weight Function 
\begin{equation}
\varphi_{\lambda ,k}\left( t\right) =\exp \left( \lambda \left( t+b\right)
^k\right) ,t\in \left( 0,T\right) .  \label{3.1}
\end{equation}

\begin{thm}[\cite{klibanov2023lipschitz}]\label{Carleman21Est1}
There exists a number $C_1=C_1\left( b,T\right) >0,$ which depends only on listed parameters, such that the following Carleman estimate is valid:
\begin{multline}
\int_{\Omega \times (0,T)}\left( \partial_t u+\Delta u\right)^2\varphi_{\lambda,k}^2\,dx\,dt\geq C_1\int_{\Omega \times (0,T)}\left( \partial_t u^2+\left(\Delta u\right)^2\right) \varphi_{\lambda ,k}^2\,dx\,dt \\ 
+C_1\lambda k\int_{\Omega \times (0,T)}\left( \nabla u\right) ^2\varphi_{\lambda,k}^2\,dx\,dt+C_1\lambda ^2k^2\int_{\Omega \times (0,T)}u^2\varphi
_{\lambda ,k}^2\,dx\,dt \\ 
-e^{2\lambda \left( T+b\right) ^k}\int_{\Omega }\left[ \left(
\nabla_{x}u\right) ^2+\lambda k\left( T+b\right) ^ku^2\right] (x,T) \,dx, \\ 
\forall \lambda >0,\quad\forall k>2,\quad\forall u\in H_0^2\left( \Omega \times (0,T)\right) .
\label{3.2}
\end{multline}
\end{thm}

\begin{thm}[\cite{klibanov2023lipschitz}]
There exist a sufficiently large number $k_0=k_0\left(T,b\right) >2$ and a number $C=C\left( T,b\right)>0$ depending only on listed parameters such that the following Carleman estimate holds:
\begin{multline}
\int_{\Omega \times (0,T)}\left( \partial_t u- \Delta u\right) ^2\varphi_{\lambda
,k}^2\,dx\,dt  
\geq C_1\sqrt{k}\int_{\Omega \times (0,T)}\left( \nabla u\right)^2\varphi_{\lambda ,k}^2\,dx\,dt\\
+C_1\lambda k^2\int_{\Omega \times (0,T)}u^2\varphi_{\lambda,k}^2\,dx\,dt  
-C_1\lambda k\left( T+b\right) ^{k-1}e^{2\lambda \left( T+b\right)^k}\int_{\Omega }u^2(x,T) \,dx\\
-C_1e^{2\lambda b^k}\int_{\Omega }\left[ \left( \nabla u\right) ^2+\sqrt{\nu }u^2\right] (x,0)\,dx,  
\\
\forall \lambda >0,\quad\forall k\geq k_0\left(T,b\right) >2,\quad\forall
u\in H_0^2\left(\Omega \times (0,T)\right) ,
\label{3.3}
\end{multline}
where the number $C_1=C_1\left( a,T\right) >0$
depends on the same parameters as ones in Theorem \ref{Carleman21Est1}.
\end{thm}

\begin{proof}[Proof of Theorem \ref{Carleman21Thm}]
In this proof $\tilde{C}_4=\tilde{C}
_4\left(D,T,\Omega \right) >0$ denotes different numbers depending
only on $D,T,\Omega $ and $C_4=C_4\left(D,T,\Omega,\varepsilon\right) >0$ denotes different numbers depending not only on
parameters $D,T,\Omega $ but on $\varepsilon$ as well. Consider four
arbitrary numbers $y_1,z_1,y_2,z_2\in \mathbb{R}.$ Let $\tilde{y}=y_1-y_2$ and $\tilde{z}=z_1-z_2.$ Hence, 
\begin{equation}
y_1z_1-y_2z_2=\tilde{y}z_1+\tilde{z}y_2.  \label{4.14}
\end{equation}
Denote
\begin{equation}
\tilde{u}(x,t) =u_1(x,t) -u_2(x,t) ,\quad
\tilde{m}(x,t) =m_1(x,t) -m_2(x,t) ,\label{4.15}
\end{equation}
\begin{equation}
\tilde{u}_T(x)=u_T^{(1)}(x)-u_T^{\left(
2\right) }(x),\quad \tilde{m}_T(x)=m_T^{\left(1\right) }(x)-m_T^{(2)}(x),\label{4.16}
\end{equation}
\begin{equation}
\tilde{G}_1(x,t) =\left( G_{1,1}-G_{1,2}\right) (x,t) ,\quad \tilde{G}_2(x,t) =\left(
G_{2,1}-G_{2,2}\right) (x,t) .
\label{4.160}
\end{equation}

Using \eqref{4.1}--\eqref{4.6} and the fundamental theorem of Calculus, we obtain
\begin{multline}
F\left( \int_{\Omega }M(x,y) m_1(y,t) \,dy,m_1(x,t) \right)  -F\left( \int_{\Omega }M(x,y) m_2(y,t) \,dy,m_2(x,t) \right)  \\ 
=f_1(x,t) \int_{\Omega }M(x,y) \tilde{m}(y,t) +f_2(x,t) \tilde{m}(x,t) ,
\label{4.17}
\end{multline}
where functions $f_1,f_2$ are such that 
\begin{equation}
\left| f_1(x,t) \right| ,\left| f_2\left(
x,t\right) \right| \leq D,\quad \text{in }\Omega \times (0,T).
\label{4.18}
\end{equation}

Subtract equations \eqref{MFGCarleman2} for the pair $(u_2,m_2)$ from corresponding equations for the pair $(u_1,m_1)$. Then use \eqref{4.6}--\eqref{4.8}, \eqref{4.14}--\eqref{4.18} and recall that Carleman estimates can work with both equations and inequalities. Hence, it is convenient to replace resulting equations with two inequalities:
\begin{equation}
\left| \partial_t m+\Delta \tilde{u}\right| (x,t) \leq \tilde{C}_4\left( \left| \nabla \tilde{u}\right| +\int_{\Omega}\left| \tilde{m}(y,t) \right| \,dy+\left| \tilde{m}\right|+\left| \tilde{G}_1\right| \right) (x,t) ,\quad \text{in }\Omega \times (0,T),  \label{4.19}
\end{equation}
\begin{equation}
\left|\tilde{m}_T-\Delta \tilde{m}\right| (x,t) \leq \tilde{C}_4\left( \left| \nabla \tilde{m}\right| +\left| \tilde{m}\right|
+\left| \Delta \tilde{u}\right| +\left| \nabla \tilde{u}\right| +\left| \tilde{G}_2\right| \right) (x,t) ,\quad \text{in }\Omega \times (0,T),  \label{4.20}
\end{equation}
\begin{equation}
\partial_{\nu }\tilde{u}\mid_{\Sigma_T}=0,\quad \partial_{\nu }\tilde{m}\mid_{\Sigma_T}=0,
\label{4.21}
\end{equation}
\begin{equation}
\tilde{u}(x,T)=\tilde{u}_T(x),\quad \tilde{m}(x,T)
=\tilde{m}_T(x),\quad \text{on } \Omega .  \label{4.22}
\end{equation}

Squaring both sides of equation \eqref{4.19} and \eqref{4.20}, applying
Cauchy-Schwarz inequality, multiplying by the CWF $\varphi_{\lambda
,k}^2\left( t\right) $ defined in \eqref{3.1} and integrating over $\Omega \times (0,T),$
we obtain
\begin{multline}
\int_{\Omega \times (0,T)}\left( \partial_t m+\Delta \tilde{u}\right) ^2\varphi_{\lambda,k}^2\,dx\,dt\leq \tilde{C}_4\int_{\Omega \times (0,T)}\left( \nabla \tilde{u}\right)
^2\varphi_{\lambda ,k}^2\,dx\,dt \\ 
+\tilde{C}_4\int_{\Omega \times (0,T)}\left( \tilde{m}^2+\int_{\Omega
}\tilde{m}^2(y,t) \,dy+\tilde{G}_1^2\right) \varphi_{\lambda
,k}^2\,dx\,dt,
\label{4.23}
\end{multline}
\begin{multline}
\int_{\Omega \times (0,T)}\left( \tilde{m}_T-\Delta \tilde{m}\right) ^2\varphi_{\lambda
,k}^2\,dx\,dt\leq \tilde{C}_4\int_{\Omega \times (0,T)}\left( \left( \nabla
\tilde{m}\right) ^2+\tilde{m}^2\right) \varphi_{\lambda ,k}^2\,dx\,dt+ \\ 
+\tilde{C}_4\int_{\Omega \times (0,T)}\left( \left( \Delta \tilde{u}\right)^2+\left( \nabla \tilde{u}\right)^2+\tilde{u}^2+\tilde{G}_2^2\right) \varphi_{\lambda ,k}^2\,dx\,dt.
\label{4.24}
\end{multline}
Note that 
\begin{equation}
\int_{\Omega \times (0,T)}\left( \int_{\Omega }\tilde{m}^2(y,t)\,dy\right) \varphi_{\lambda ,k}^2\left( t\right) \,dx\,dt\leq \tilde{C}_4\int_{\Omega \times (0,T)}\tilde{m}^2\varphi_{\lambda ,k}^2\left( t\right) \,dx\,dt.
\label{4.25}
\end{equation}

Set 
\begin{equation}
b=1  \label{4.250}
\end{equation}
in the Carleman Weight Function $\varphi_{\lambda ,k}\left( t\right)$. 
Apply Carleman estimate \eqref{3.2} to the left hand side of \eqref{4.23}
and use \eqref{4.21}, \eqref{4.22} and \eqref{4.25}. We obtain 
\begin{multline}
\tilde{C}_4\int_{\Omega \times (0,T)}\left( \tilde{m}^2+\tilde{G}
_1^2\right) \varphi_{\lambda ,k}^2\,dx\,dt \\
\geq \int_{\Omega \times (0,T)}\left( \partial_t m^2+\left( \Delta \tilde{u}\right) ^2\right)
\varphi_{\lambda ,k}^2\,dx\,dt  
+\lambda k\int_{\Omega \times (0,T)}\left( \nabla \tilde{u}\right) ^2\varphi_{\lambda,k}^2\,dx\,dt
\\+\lambda ^2k^2\int_{\Omega \times (0,T)}\tilde{u}^2\varphi_{\lambda
,k}^2\,dx\,dt-  
-\tilde{C}_4e^{2\lambda (T+1) ^k}\int_{\Omega }
\left[ \left( \nabla_{x}\tilde{u}_T\right) ^2+\lambda k(T+1)
^k\tilde{u}_T^2\right] \,dx, \\ 
\forall \lambda >0,\quad\forall k>2.
\label{4.26}
\end{multline}
Choosing $\lambda_1=\lambda_1\left(D,T,\Omega \right) \geq 1$ so large that 
\begin{equation}
\lambda_1>2\tilde{C}_4  \label{4.27}
\end{equation}
and recalling that $k>2,$ we obtain from \eqref{4.26} 
\begin{multline}
\tilde{C}_4\int_{\Omega \times (0,T)}\tilde{m}^2\varphi_{\lambda ,k}^2\,dx\,dt+
\tilde{C}_4\int_{\Omega \times (0,T)}\tilde{G}_1^2\varphi_{\lambda
,k}^2\,dx\,dt \\ 
\geq \int_{\Omega \times (0,T)}\left( \partial_t m^2+\left( \Delta \tilde{u}\right) ^2\right)\varphi_{\lambda ,k}^2\,dx\,dt
+\lambda k\int_{\Omega \times (0,T)}\left( \nabla \tilde{u}\right) ^2\varphi_{\lambda,k}^2\,dx\,dt\\+\lambda ^2k^2\int_{\Omega \times (0,T)}\tilde{u}^2\varphi_{\lambda
,k}^2\,dx\,dt  
-\tilde{C}_4e^{2\lambda (T+1) ^k}\int_{\Omega }
\left[ \left( \nabla_{x}\tilde{u}_T\right) ^2+\lambda k(T+1)^k\tilde{u}_T^2\right] \,dx, 
\\\quad\forall \lambda \geq \lambda_1,\quad\forall k>2.
\label{4.28}
\end{multline}

We now apply Carleman estimate \eqref{3.3} to the left hand side of \eqref{4.24}. We obtain
\begin{multline}
\tilde{C}_4\int_{\Omega \times (0,T)}\left( \left( \Delta \tilde{u}\right)^2+\tilde{u}^2+\tilde{G}_2^2\right) \varphi_{\lambda ,k}^2\,dx\,dt+
\tilde{C}_4\int_{\Omega \times (0,T)}\left( \left( \nabla \tilde{m}\right)^2+\tilde{m}^2\right) \varphi_{\lambda ,k}^2\,dx\,dt \\ 
\geq \sqrt{k}\int_{\Omega \times (0,T)}\left( \nabla \tilde{m}\right) ^2\varphi_{\lambda,k}^2\,dx\,dt+\lambda k^2\int_{\Omega \times (0,T)}\tilde{m}^2\varphi_{\lambda ,k}^2\,dx\,dt \\ 
-\lambda k\left( T+b\right) ^{k-1}e^{2\lambda \left( T+b\right)
^k}\int_{\Omega }\tilde{m}_T^2(x)\,dx-e^{2\lambda
b^k}\int_{\Omega }\left[ \left( \nabla \tilde{m}\right) ^2+\sqrt{k}\tilde{m}^2
\right] (x,0)\,dx, \\
\forall \lambda >0,\quad\forall k\geq k_0=k_0\left(T\right) >2.
\label{4.29}
\end{multline}
Choose the number $k_0=k_0\left(T\right) $ so large that 
\begin{equation}
k_0>\max \left( 2,4\tilde{C}_4^2\right) ,  \label{4.290}
\end{equation}
and, until \eqref{4.410}, set $k=k_0.$ Also, let $\lambda \geq \lambda_1,$ where the number $\lambda_1=\lambda_1\left(D,T,\Omega\right) \geq 1$ is defined in \eqref{4.27}. Then \eqref{4.29} implies
\begin{multline}
\tilde{C}_4\int_{\Omega \times (0,T)}\left( \left( \Delta \tilde{u}\right)^2+\tilde{u}^2+\tilde{G}_2^2\right) \varphi_{\lambda ,k}^2\,dx\,dt 
\geq \int_{\Omega \times (0,T)}\left( \nabla \tilde{m}\right) ^2\varphi_{\lambda
,k}^2\,dx\,dt\\+\lambda \int_{\Omega \times (0,T)}\tilde{m}^2\varphi_{\lambda ,k}^2\,dx\,dt 
-\tilde{C}_4\lambda (T+1) ^{k-1}e^{2\lambda \left(
T+1\right) ^k}\int_{\Omega }\tilde{m}_T^2(x)\,dx\\-\tilde{C}_4e^{2\lambda }\int_{\Omega }\left[ \left( \nabla \tilde{m}\right)
^2+\tilde{m}^2\right] (x,0)\,dx, \quad 
\forall \lambda \geq \lambda_1.
\label{4.30}
\end{multline}
In particular, it follows from \eqref{4.30} that
\begin{multline}
\int_{\Omega \times (0,T)}\tilde{m}^2\varphi_{1,\lambda ,k}^2\,dx\,dt\leq \tilde{C}_4\lambda ^{-1}\int_{\Omega \times (0,T)}\left( \left( \Delta \tilde{u}\right)
^2+\tilde{u}^2+\tilde{G}_2^2\right) \varphi_{\lambda ,k}^2\,dx\,dt \\ 
+\tilde{C}_4e^{2\lambda (T+1) ^k}\int_{\Omega}\tilde{m}_T^2(x)\,dx+\tilde{C}_4e^{2\lambda}\int_{\Omega }\left[ \left( \nabla \tilde{m}\right) ^2+\tilde{m}^2\right](x,0)\,dx,\quad\forall \lambda \geq \lambda_1.
\label{4.31}
\end{multline}
Replacing the first term of the first line of \eqref{4.28} with the right hand side of inequality \eqref{4.31}, we obtain
\begin{multline}
\tilde{C}_4\lambda e^{2\lambda (T+1)^k}\int_{\Omega }\left[ \tilde{m}_T^2+\left( \nabla_{x}\tilde{u}_T\right)^2+\tilde{u}_T^2\right] (x)\,dx+\tilde{C}_4\int_{\Omega \times (0,T)}\left( \tilde{G}_1^2+\tilde{G}_2^2\right)
\varphi_{\lambda ,k}^2\,dx\,dt
\\+\tilde{C}_4e^{2\lambda }\int_{\Omega }\left[ \left( \nabla
\tilde{m}\right) ^2+\tilde{m}^2\right] (x,0)\,dx
+\frac{\tilde{C}_4}{\lambda }\int_{\Omega \times (0,T)}\left( \left( \Delta
\tilde{u}\right) ^2+\tilde{u}^2\right) \varphi_{\lambda ,k}^2\,dx\,dt\\
\geq \int_{\Omega \times (0,T)}\left( \partial_t m^2+\left( \Delta \tilde{u}\right) ^2\right)
\varphi_{\lambda ,k}^2\,dx\,dt
+\lambda \int_{\Omega \times (0,T)}\left( \nabla \tilde{u}\right) ^2\varphi_{\lambda
,k}^2\,dx\,dt\\+\lambda ^2\int_{\Omega \times (0,T)}\tilde{u}^2\varphi_{\lambda
,k}^2\,dx\,dt,\quad\forall \lambda \geq \lambda_1.\label{4.32}
\end{multline}
By \eqref{4.27}, $\tilde{C}_4/\lambda <1/2,$ for all $\lambda \geq \lambda_1.$ Hence, terms in the 4th and 5th lines of \eqref{4.32} absorb terms in the 3rd line of \eqref{4.32}.
Hence,
\begin{multline}
\tilde{C}_4\lambda e^{2\lambda (T+1)
^k}\int_{\Omega }\left[ \tilde{m}_T^2+\left( \nabla_{x}\tilde{u}_T\right)
^2+\tilde{u}_T^2\right] (x)\,dx+\tilde{C}_4\int_{\Omega \times (0,T)}\left( \tilde{G}_1^2+\tilde{G}_2^2\right)
\varphi_{\lambda ,k}^2\,dx\,dt\\
+\tilde{C}_4e^{2\lambda }\int_{\Omega }\left[ \left( \nabla
\tilde{m}\right) ^2+\tilde{m}^2\right] (x,0)\,dx
\geq \int_{\Omega \times (0,T)}\left( \partial_t m^2+\left( \Delta \tilde{u}\right) ^2\right)
\varphi_{\lambda ,k}^2\,dx\,dt
\\+\lambda \int_{\Omega \times (0,T)}\left( \nabla \tilde{u}\right) ^2\varphi_{\lambda
,k}^2\,dx\,dt+\lambda ^2\int_{\Omega \times (0,T)}\tilde{u}^2\varphi_{\lambda
,k}^2\,dx\,dt,\quad\forall \lambda \geq \lambda_1.
\label{4.33}
\end{multline}

Comparing the last two lines of \eqref{4.33} with the first line of \eqref{4.30}, we obtain
\begin{multline}
\tilde{C}_4\lambda e^{2\lambda (T+1)
^k}\int_{\Omega }\left[ \tilde{m}_T^2+\left( \nabla_{x}\tilde{u}_T\right)
^2+\tilde{u}_T^2\right] (x)\,dx+\tilde{C}_4\int_{\Omega \times (0,T)}\left( \tilde{G}_1^2+\tilde{G}_2^2\right)
\varphi_{\lambda ,k}^2\,dx\,dt\\
+\tilde{C}_4e^{2\lambda }\int_{\Omega }\left[ \left( \nabla
\tilde{m}\right) ^2+\tilde{m}^2\right] (x,0)\,dx
\geq \int_{\Omega \times (0,T)}\left( \nabla \tilde{m}\right) ^2\varphi_{\lambda
,k}^2\,dx\,dt+\lambda \int_{\Omega \times (0,T)}\tilde{m}^2\varphi_{\lambda ,k}^2\,dx\,dt,
\\\quad\forall \lambda \geq \lambda_1.\label{4.34}
\end{multline}
Summing up \eqref{4.33} and \eqref{4.34}, we obtain
\begin{multline}
\int_{\Omega \times (0,T)}\left( \partial_t m^2+\left( \Delta \tilde{u}\right) ^2+\left(
\nabla \tilde{u}\right) ^2+\tilde{u}^2+\left( \nabla \tilde{m}\right) ^2+\tilde{m}^2\right) \varphi
_{\lambda ,k}^2\,dx\,dt
\\
\leq \tilde{C}_4e^{3\lambda (T+1)
^k}\int_{\Omega }\left[ \tilde{m}_T^2+\left( \nabla_{x}\tilde{u}_T\right)
^2+\tilde{u}_T^2\right] (x)\,dx  
+\tilde{C}_4e^{2\lambda }\int_{\Omega }\left[ \left( \nabla
\tilde{m}\right) ^2+\tilde{m}^2\right] (x,0)\,dx\\+\tilde{C}_4\int_{\Omega \times (0,T)}\left( \tilde{G}_1^2
+\tilde{G}_2^2\right) \varphi_{\lambda,k}^2\,dx\,dt,\quad\forall \lambda \geq
\lambda_1.
\label{4.35}
\end{multline}
Since by \eqref{4.0} $Q_{\varepsilon ,T}\subset \Omega \times (0,T),$ then replacing $
\Omega \times (0,T) $ with $Q_{\varepsilon ,T}$ in the first line of \eqref{4.35}, we
strengthen this inequality. Hence, 
\begin{multline}
\int_{Q_{\varepsilon ,T}}\left( \partial_t m^2+\left( \Delta \tilde{u}\right)
^2+\left( \nabla \tilde{u}\right) ^2+\tilde{u}^2+\left( \nabla \tilde{m}\right)
^2+\tilde{m}^2\right) \varphi_{1,\lambda ,k}^2\,dx\,dt
\\
\leq \tilde{C}_4e^{3\lambda (T+1)
^k}\int_{\Omega }\left[ \tilde{m}_T^2+\left( \nabla_{x}\tilde{u}_T\right)
^2+\tilde{u}_T^2\right] (x)\,dx  
+\tilde{C}_4e^{2\lambda }\int_{\Omega }\left[ \left( \nabla
\tilde{m}\right) ^2+\tilde{m}^2\right] (x,0)\,dx\\+\tilde{C}_4\int_{\Omega \times (0,T)}\left( \tilde{G}_1^2+\tilde{G}_2^2\right) \varphi_{\lambda ,k}^2\,dx\,dt,\quad\forall \lambda \geq
\lambda_1.
\label{4.36}
\end{multline}

Next, by \eqref{3.1}, \eqref{4.0} and \eqref{4.250}
\begin{equation}
\min_{\bar{Q}_{\varepsilon ,T}}\varphi_{\lambda ,k}^2\left( t\right)
=e^{2\lambda \left( \varepsilon +1\right) ^k},\text{ }  \label{4.37}
\end{equation}
\begin{equation}
\max_{\bar{Q}_T}\varphi_{\lambda ,k}^2\left( t\right) =e^{2\lambda
(T+1) ^k}.  \label{4.38}
\end{equation}
Also, by \eqref{4.9}, \eqref{4.90}, \eqref{4.16}, \eqref{4.160} and \eqref{4.38} 
\begin{equation}
\tilde{C}_4e^{3\lambda (T+1) ^k}\int_{\Omega }
\left[ \tilde{m}_T^2+\left( \nabla_{x}\tilde{u}_T\right) ^2+\tilde{u}_T^2\right] \left(
x\right) \,dx+\tilde{C}_4\int_{\Omega \times (0,T)}\left( \tilde{G}
_1^2+\tilde{G}_2^2\right) \varphi_{\lambda ,k}^2\,dx\,dt
\leq \tilde{C}_4e^{3\lambda (T+1) ^k}\delta ^2.
\label{4.39}
\end{equation}
By the trace theorem
\begin{equation}
\tilde{C}_4e^{2\lambda }\int_{\Omega }\left[ \left( \nabla
\tilde{m}\right) ^2+\tilde{m}^2\right] (x,0)\,dx\leq \tilde{C}
_4e^{2\lambda }\left\Vert \tilde{m}\right\Vert_{H^2\left( \Omega \times (0,T)\right) }^2.
\label{4.40}
\end{equation}
Hence, using \eqref{4.36}, \eqref{4.37}, \eqref{4.39} and \eqref{4.40}, we
replace $\tilde{C}_4$ with $C_4$ (defined at the beginning of this proof)
and obtain for all $\lambda \geq \lambda_1:$
\begin{multline}
\left\Vert \partial_t m\right\Vert_{L^2\left( Q_{\varepsilon ,T}\right)
}^2+\left\Vert \Delta \tilde{u}\right\Vert_{L^2\left( Q_{\varepsilon ,T}\right)
}^2+\left\Vert \tilde{u}\right\Vert_{H^{1,0}\left( Q_{\varepsilon ,T}\right)
}^2+\left\Vert \tilde{m}\right\Vert_{H^{1,0}\left( Q_{\varepsilon ,T}\right)
}^2
\\
\leq C_4e^{3\lambda (T+1) ^k}\delta ^2+C_4\exp \left[
-2\lambda \left( \varepsilon +1\right) ^k\left( 1-\frac{1}{\left(
\varepsilon +1\right) ^k}\right) \right] \left\Vert \tilde{m}\right\Vert
_{H^2\left( \Omega \times (0,T)\right) }^2,
\label{4.41}
\end{multline}
Recalling \eqref{4.290}, choose $k_1=k_1\left(T,\varepsilon
\right) \geq k_0\left(T\right) $ sufficiently large such that 
\begin{equation}
\frac{1}{\left( \varepsilon +1\right) ^{k_1}}<\frac{1}{2}  \label{4.410}
\end{equation}
and set $k=k_1.$ Then \eqref{4.41} implies
\begin{multline}
\left\Vert \partial_t m\right\Vert_{L^2\left( Q_{\varepsilon ,T}\right)}^2+\left\Vert \Delta \tilde{u}\right\Vert_{L^2\left( Q_{\varepsilon ,T}\right)}^2+\left\Vert \tilde{u}\right\Vert_{H^{1,0}\left( Q_{\varepsilon ,T}\right)}^2+\left\Vert \tilde{m}\right\Vert_{H^{1,0}\left( Q_{\varepsilon ,T}\right)}^2\\ 
\leq C_4e^{3\lambda (T+1) ^k}\delta ^2+C_4e^{-\lambda\left( \varepsilon +1\right) ^k}\left\Vert \tilde{m}\right\Vert_{H^2\left(\Omega \times (0,T)\right) }^2,\quad\lambda \geq \lambda_1.
\label{4.42}
\end{multline}
Choose $\lambda =\lambda \left( \delta \right) $ such that 
\begin{equation}
e^{3\lambda \left( \delta \right) (T+1) ^k}\delta ^2=\delta .
\label{4.43}
\end{equation}
Hence, 
\begin{equation}
\lambda \left( \delta \right) =\ln \left[ \delta ^{\left( 3(T+1)
\right) ^{-1}}\right] ,  \label{4.44}
\end{equation}
\begin{equation}
e^{-\lambda \left( \varepsilon +1\right) ^k}=\delta ^{2\rho },\quad 
2\rho =\frac{1}{3}\left( \frac{\varepsilon +1}{T+1}\right) ^k<\frac{1}{3}.
\label{4.45}
\end{equation}
Choose $\delta_0=\delta_0\left(D,T,\Omega ,\varepsilon \right)
\in \left( 0,1\right) $ so small that
\begin{equation}
\lambda \left( \delta_0\right) =\ln \left[ \delta ^{\left( 3\left(
T+1\right) \right) ^{-1}}\right] \geq \lambda_1.  \label{4.46}
\end{equation}
Then \eqref{4.41}--\eqref{4.46} imply that
\begin{multline}
\left\Vert \partial_t m\right\Vert_{L^2\left( Q_{\varepsilon ,T}\right)
}+\left\Vert \Delta \tilde{u}\right\Vert_{L^2\left( Q_{\varepsilon ,T}\right)
}+\left\Vert \tilde{u}\right\Vert_{H^{1,0}\left( Q_{\varepsilon ,T}\right)
}+\left\Vert \tilde{m}\right\Vert_{H^{1,0}\left( Q_{\varepsilon ,T}\right) } \\ 
\leq C_4\left( 1+\left\Vert \tilde{m}\right\Vert_{H^2\left( \Omega \times (0,T)\right)
}\right) \delta ^{\rho },\quad\forall \delta \in \left( 0,\delta
_0\right) ,\rho \in \left( 0,1/6\right) .
\label{4.460}
\end{multline}
The rest of the proof of the desired H\"{o}lder stability estimates \eqref{4.10}--\eqref{4.11} follows immediately from \eqref{4.7}, \eqref{4.8}, \eqref{4.15}, \eqref{4.16}, \eqref{4.460}.
\end{proof}

\end{document}

%% file: preface.tex
\preface 

	In this book, we present a curated collection of existing results on inverse problems for Mean Field Games (MFGs), a cutting-edge and rapidly evolving field of research. Our aim is to provide fresh insights, novel perspectives, and a comprehensive foundation for future investigations into this fascinating area. Mean Field Games, a class of differential games involving a continuum of non-atomic players, offer a powerful framework for analyzing the collective behavior of large populations of symmetric agents as the number of agents approaches infinity. This framework has proven to be an invaluable tool for quantitatively modeling the macroscopic dynamics of agents striving to minimize specific costs in complex systems, such as crowd dynamics, financial markets, traffic flows, and social networks.

The study of MFGs has traditionally focused on forward problems, where the goal is to determine the equilibrium behavior of agents given a set of model parameters, such as cost functions, interaction mechanisms, and initial conditions. However, the inverse problems for MFGs—which seek to infer these underlying parameters from observed data—have received comparatively less attention in the literature. This book seeks to address this gap by delving into the fundamental aspects of MFG inverse problems, with a particular emphasis on issues of unique identifiability, stability, and reconstruction of unknown parameters. These problems are not only mathematically challenging but also of immense practical significance, as they enable the calibration and validation of MFG models using real-world data.

The field of MFG inverse problems is still in its infancy, yet it has already witnessed a number of pioneering contributions that have laid the groundwork for further exploration. This book aims to synthesize these seminal results, offering a unified and accessible overview of the state-of-the-art while introducing new perspectives and innovative insights. A significant portion of the book is devoted to the latest developments in the field, particularly those arising from the authors' own research and collaborations. These contributions have led to a deeper understanding of the theoretical underpinnings of MFG inverse problems and have opened up new avenues for practical applications.

The motivation for this book stems from the growing recognition of the importance of inverse problems in the broader context of MFGs. While the forward problem has been extensively studied and is now well understood, the inverse problem remains a fertile ground for research, with many open questions and unexplored directions. Recent work by the authors and their collaborators has yielded a wealth of new results that are both theoretically profound and practically relevant. These findings have not only advanced our understanding of the inverse problem but have also demonstrated its potential to address real-world challenges in areas such as economics, engineering, and the social sciences.

Given the rapid pace of development in this field, we believe that now is an opportune time to compile and summarize these results in a single volume. This book is intended to serve as a valuable resource for researchers, graduate students, and practitioners who are interested in the theory and applications of MFGs, particularly those with a focus on inverse problems. By bringing together the latest research and providing a coherent narrative, we hope to inspire further exploration and innovation in this dynamic and interdisciplinary area of study.

Finally, we would like to acknowledge the contributions of our collaborators and the broader research community, whose work has been instrumental in shaping the field of MFG inverse problems. We are deeply grateful for their insights, encouragement, and support. It is our sincere hope that this book will serve as a catalyst for future research and will contribute to the continued growth and development of this exciting field.

\bigskip

\noindent Hongyu Liu, Catharine W. K. Lo and Shen Zhang
